\documentclass[11pt]{amsart}
\usepackage{a4wide}
\usepackage{amssymb,amsmath}
\usepackage{fullpage}
\usepackage{color}
\usepackage{amsmath}

\newcommand{\beq}{\begin{equation}}
\newcommand{\eeq}{\end{equation}}

\newcommand{\n}{{\bf n}}
\newcommand{\D}{\partial^{\varphi}} 
 
\newcommand{\V}{\mathcal{V}_{t}}

\newcommand{\N}{{\bf N}}

\def\eps{\varepsilon}

\Large

\newtheorem{prop}{Proposition}
\newtheorem{theoreme}[prop]{Theorem}
\newtheorem{lem}[prop]{Lemma}
\newtheorem{cor}[prop]{Corollary}
\newtheorem{defi}[prop]{Definition}
\newtheorem{rem}[prop]{Remark}

\numberwithin{equation}{section}
\numberwithin{prop}{section}
\begin{document}
\title{ Uniform regularity   and  vanishing viscosity limit for the free surface Navier-Stokes
 equations}

\author{Nader Masmoudi}
\address{Courant Institute, 251 Mercer Street, 
New York, NY 10012-1185, 
USA,}
\email{masmoudi@courant.nyu.edu  }
\author{Frederic Rousset}
\address{IRMAR, Universit\'e de Rennes 1, campus de Beaulieu,  35042  Rennes cedex, France}
\email{ frederic.rousset@univ-rennes1.fr  }
\date{}
\begin{abstract}
We study the inviscid limit of the free boundary Navier-Stokes 
equations. We  prove 
the existence of solutions  on  a uniform time interval
by  using a suitable functional framework based on Sobolev conormal spaces.
   This allows us  to use a  strong compactness argument to justify the  inviscid limit. 
  Our approach    does not rely  on  the justification of asymptotic expansions.  
 In particular, we  get a new   existence result   for the Euler equations 
   with free surface from the one
  for Navier-Stokes. 

\end{abstract}

\maketitle

\tableofcontents

\section{Introduction}


The study of fluid motions with free interfaces has attracted a lot of attention during the last thirty years.
 When  the fluid is  viscous, that is to say for the Navier-Stokes equation,  the local well-posedness for the Cauchy problem was  shown
  for example  in the classical works by   Beale
 \cite{Beale83}   and Tani  \cite{Tani96},   we also refer to the work by Solonnikov \cite{Solonnikov} and also  to \cite{Guo-Tice} 
  for example.  These results rely heavily  on the smoothing effect  due to the presence of viscosity.
  When viscosity is neglected, it is  much more difficult to control the regularity of the free surface. The first attempts to get local well-posedness
   for the free-surface Euler equation 
   go back to  the works 
by  Nalimov \cite{Nalimov74}  and H.~Yoshihara \cite{Yosihara82} for 
small data in two dimensions. A major 
 breakthrough was achieved in the nineties  by  S.~Wu \cite{Wu97,Wu99} 
who solved the irrotational  problem in  dimension two and three
(see also the works by   Craig \cite{Craig85},  Beale, Hou, and Lowengrub
\cite{BHL93},  Ogawa and Tani \cite{OT02}, Schneider and Wayne \cite{SW02},  Lannes \cite{Lannes05}, 
 Ambrose  and  Masmoudi \cite{AM05},  \cite{AM09}).   In this 
  irrotational case,  where dispersive effects are predominant,  an almost global existence result was achieved by S. Wu  \cite{Wu10} in the two-dimensional case and global existence results  in three dimensions were   proved  recently  by Germain, Masmoudi and 
Shatah \cite{GMS09cras,GMS11prep} and by 
 S.~Wu \cite{Wu11}.  We also refer to \cite{Alazard} for a local well-posedness result below the standard regularity threshold.
For the full system (without the irrotational assumption)
  there are well posedness results starting with the works by 
  Christodoulou and  Lindblad \cite{CL00} and Lindblad \cite{Lindblad05}. 
Alternative approach were proposed   
 by 
    Coutand   and  Shkoller   \cite{CS07}, 
  Shatah and  Zeng \cite{SZ08} and 
 P.~Zhang and Z.~Zhang \cite{ZZ08}. Note that the well-posedness of the Euler equation requires a Taylor sign condition for the pressure
  on the boundary.  
 
 A classical  problem  in  fluid mechanics,  since   Reynolds numbers for real flows are  very high,  
 is to  study the  behaviour of  the solutions of the Navier-Stokes equation at high Reynolds number which corresponds
 to small viscosity. A natural conjecture is of course to expect that the limit is given by a solution of  the Euler equation. 
 Nevertheless  in the presence of boundaries the situation can be more complicated as we shall describe below. 
 The main result of this work is the justification of the inviscid limit when we start from the free surface Navier-Stokes equation. 

The study of the vanishing viscosity limit  for the Navier-Stokes equation has  a long history.  We shall distinguish three 
different approach to the problem. The first one  consists  in proving 
 that strong solutions to the Navier-Stokes equation exist on an interval of time independent of the viscosity parameter
  and then passing to the limit by strong compactness arguments. The second  approach  is based on weak compactness arguments
   starting from Leray weak solutions of the Navier-Stokes equation.  Note that  in these  two approach  the well posedness
    of the Euler equation is not used. On the contrary, the third approach relies on the a priori knowledge that sufficiently smooth
    solutions of   the Euler equation exist: it consists either in  comparing directly
     by a modulated type energy argument  a Leray weak solution of the Navier-Stokes
     equation with  an approximate solution whose leading order is given by a smooth solution  the Euler equation or
       in directly justifying asymptotic expansions  by studying a modified Navier-Stokes equation for the remainder.
     
 When there is no boundary, the vanishing viscosity limit  can be justified by using the  three types of methods:     
  we refer to Swann \cite{Swann71} and Kato \cite{Kato72} (see also \cite{Masmoudi07cmp})  for  results with the   first approach, 
  to  DiPerna and Majda \cite{DM87cmp,DM87cpam} for the second approach  in the two-dimensional case
  and to the book \cite{Majda-Bertozzi} for example for the third approach. Note that the second approach yields the existence
   of global weak solution for the Euler equation and thus  in three dimensions  only results with the first and
   third approach are known. 
  
  In a fixed domain with boundary, the situation is more complicated,  the situation depends on the type of boundary conditions that
   we impose on the boundary.  
   
   Let us first describe the situation in the case of  the homogeneous  Dirichlet boundary condition.
  For the first approach, the main difficulty is that  for data of size of  order one,  all the known  existence
   results of strong solutions for the Navier-Stokes equation yield uniform estimates that  are valid only on a time interval which shrinks to zero  when
    the viscosity goes to zero.  This has  a physical explanation:   there is  formation of boundary layers
    in  a small vicinity  of the boundary. Indeed, in  a small vicinity of the boundary, the solution $u^\eps$, $\eps$
     being the viscosity parameter,  of the Navier-Stokes
    equation is expected to behave like 
    $ u^\eps \sim U(t,y, z/\sqrt{\eps})$,  
    if the boundary is locally given by $z=0$, where $U$ is some profile which is not described by the Euler equation.
     Because of this small scale  behaviour, one can see that there is no hope to get uniform estimates in all the  spaces
      for which the 3-D Euler equation is known to be well-posed ($H^s$, $s>5/2$ for example). Most of the works
       on the topic are thus concentrated on the justification of the third approach: namely the construction of
        an high order approximate solution  (involving two spatial scales thus of matched asymptotics type)  and the control of the remainder.  Nevertheless,   the equation that 
        governs the boundary layers that appear in the formal expansion,  the Prandtl equation, can be ill-posed  in Sobolev spaces  \cite{GD10},
         thus it is not always possible to even  constuct an approximate solution. Moreover, even when this is possible,  some 
          instabilities  make impossible  the control of  the remainder on a suitable interval of time \cite{Grenier}, \cite{Guo-Nguyen}.
                     Consequently, in a domain with  fixed boundary with homogeneous  Dirichlet boundary condition, the only  full  result is the work by 
           Sammartino  and Caflisch \cite{SC98a} where the asymptotic expansion is justified  in the analytic framework.
           There is also a result by  Kato \cite{Kato86} where  some necessary condition 
for the convergence to hold is exhibited  and the work  \cite{Masmoudi98arma}  where the 
vertical viscosity is  assumed much smaller than the horizontal one.
The case of a non-homogeneous Dirichlet boundary condition 
 which  happens
 for example when there is  injection or succion on the  boundary  is  also  studied.   In this case the size of the boundary layer is $\eps$
 (against $\sqrt{\eps}$ in the previous case) which makes the situation very different. The constuction and  justification of the asymptotic expansion, even for 
 general hyperbolic-parabolic systems  have  been widely studied recently, we refer to 
  \cite{Gisclon-Serre, Grenier-Gues, Grenier-Rousset, Metivier-Zumbrun, 
 Gues-Metivier-Williams-Zumbrun, Rousset, Temam-Wang}.

Let us finally describe the situation 
for the Navier-Stokes equation with a free surface.
At first we note that  there is a minimal amount of regularity which is needed in order to define the motion of the free surface, 
therefore there is up to our knowledge  
  no suitable notion of weak solution for this equation.
   Consequently, it seems that  to justify the inviscid limit one can only use either the first approach or  the justification
    of matched asymptotic expansions. The main drawbacks of this last approach is  first that a lot of regularity is needed
     in order to construct the expansion and second that the a priori knowledge of the well-posedness  of the free surface Euler equation is needed.
   From a mathematical point of view, this is not very satisfactory since the local theory for the Euler equation is much more difficult
    to establish than the one for the Navier-Stokes equation.  Moreover,   because of the parabolic behaviour of the equation
     one do not expect to need more regularity on the data in order to solve the viscous equation than the inviscid one.  
    In this work, we shall thus  avoid to deal with the justification of asymptotic expansions as it is usually done and deal with the first approach.
     In this approach, one of the main difficulty is  
     to  prove  
     that the local  smooth solution of the Navier-Stokes equation (which comes from the classical works  
     \cite{Beale83},   \cite{Tani96}) can be continued on an interval of time independent of the viscosity.
     In the vicinity of the free surface the expected behaviour of the solution  is 
   $  u^\eps \sim  u(t,x)+ \sqrt{\eps}U(t,y, z/\sqrt{\eps})$. Again, we  observe that the $H^s$ norm $s>5/2$ still cannot be bounded
    on an interval of time independent of $\eps$. 
   Nevertheless, the situation seems more favorable since  one can expect that  the Lipschitz norm  which is the crucial quantity
    for the inviscid problem  remains bounded.
    We shall consequently  use  a functional framework based on conormal Sobolev spaces  which minimizes the needed
     amount of normal regularity  but which gives a control of the Lipschitz norm of the solution in order to get that
      the solution of the Navier-Stokes equation exists and is uniformly bounded on an interval of time independent of $\eps$.
      Such an approach to the inviscid  limit was recently used in our  work \cite{MR11prep}
       in the much simpler case of the Navier-Stokes equation with the Navier boundary condition on a fixed boundary.
      As a consequence of  our  estimates, we shall justify the inviscid limit by strong compactness arguments. 
      Note that   we shall also  get as a corollary  the existence
       of a solution  to the free surface Euler equation.
       In particular this gives an approach  by a physical regularization  to the construction of solutions of  the Euler equation which does not use the Nash-Moser
       iteration scheme.   Of course, in order to get our result, we shall need to assume the 
 Taylor sign condition for the pressure 
        on the boundary which is needed to have local well-posedness  for the limit  Euler system.


\subsection{The free surface Navier-Stokes equations}

Let us now write down the system that we shall study.
We consider the motion of an incompressible viscous fluid at high Reynolds number submitted to the influence of gravity.
 The  equation for motion  is   thus  the incompressible Navier-Stokes equation
\beq
\label{NS}
\partial_{t} u + u \cdot\nabla u + \nabla p = \eps \Delta u, \quad \nabla \cdot u = 0 ,  \quad  x \in \Omega_{t}, \, t>0
\eeq
where $u(t,x)\in \mathbb{R}^3$ is the velocity of the fluid and  $p\in \mathbb{R}$ is the pressure. 
 We shall assume that  the fluid domain $\Omega_{t}$  is given by 
\beq
\label{omegat} \Omega_{t}= \big\{ x \in \mathbb{R}^3, \quad  x_{3} < h(t, x_{1}, x_{2}) \big\}\eeq
 thus the surface is given by the  graph
 $ \Sigma_{t}=  \big\{ x \in \mathbb{R}^3, \quad  x_{3} = h(t, x_{1}, x_{2}) \big\}$
 where   $h(t, x_{1}, x_{2})$  which  defines the free surface is also an unknown of the problem.
  The parameter $\eps>0$ which is the inverse of a Reynolds number is small. 
 In this work, we shall focus on the 3-D equation set in the simplest  domain. Note that this framework is relevant
  to describe for example atmospheric or oceanic motions. All our results are also valid in the two-dimensional case
   and can be  extended to more general domains.
   
On the boundary, we impose two boundary conditions. The first one is of kinematic nature and states
 that fluid particles do not cross the  surface:
 \beq
 \label{bord1}
 \partial_{t} h = u\big(t, x_{1}, x_{2}, h(t, x_{1}, x_{2})\big) \cdot \N,  \quad (x_{1}, x_{2}) \in \mathbb{R}^2
 \eeq
 where $\N$ is the outward normal given by
 \beq
 \label{Ndef} \N= \left( \begin{array}{ccc} - \partial_{1} h  \\ -\partial_{2} h  \\ 1  \end{array} \right).
 \eeq
 The second one expresses the continuity of  the stress tensor on the surface.  Assuming
  that there is vacuum above the fluid, we get
  \beq
  \label{bord2}
   p \, {\bf n}-  2 \, \eps \, S u \,{\bf n} =  {g}\, h\, {\bf n}, \quad x \in \Sigma_{t}
   \eeq
where $S$ is the  symmetric part of the gradient
$$ Su = {1 \over 2} \big( \nabla u + \nabla u^t\big),$$
  ${\bf n}$ is the outward unit normal  (${\bf n}= \N/|\N|$) and $g$ is a positive number which measures the influence of the gravity  on the motion of the fluid
  (in the dimensional form of the equation this would be the acceleration of gravity constant).
Note that the term involving the gravity force in  \eqref{NS} has been incorporated in the pressure term,
this is why  it appears in the right-hand side of \eqref{bord2}. We neglect surface tension effects.

At $t=0$,  we  supplement the system with the initial condition
\beq
\label{initintro}
h_{/t=0} = h_{0}, \quad u_{/t=0}= u_{0}.
\eeq 
Note that this means in particular that  the initial domain $\Omega_{0}$ is given.
 
  As usual in free boundary problems, we shall need to fix the fluid domain. A convenient choice in our case
   is to consider a  family of diffeomorphism $\Phi(t, \cdot)$ under the form
   \begin{eqnarray}
   \label{diff}
   \Phi(t, \cdot) : \,   \mathcal{S}= \mathbb{R}^2 \times (-\infty, 0) & \rightarrow  &  \Omega_{t} \\ 
 \nonumber x= (y,z) &\mapsto &   (y, \varphi(t, y, z))
 \end{eqnarray}
   and we have to choose  $\varphi(t, \cdot)$. Note that we have to ensure that $\partial_{z} \varphi>0$
   in order to get that $\Phi(t, \cdot)$ is a diffeomorphism.
   
   Once the choice of $\varphi$ is made, we reduce the problem into   the fixed domain 
  $ \mathcal{S}  $ by setting
   \beq
   \label{vdef}
   v(t,x)= u(t, \Phi(t,x)), \quad q(t,x)= p(t, \Phi(t,x) ) \quad x \in \mathcal{S}.
   \eeq
   
    Many choices are possible. 
    A basic choice is to take  $\varphi(t,y,z) =  z+ \eta(t,y)$.  Such a choice would  fit in the inviscid case 
     i.e. for the Euler equation when $\eps= 0$. Indeed,  the energy estimate  and the physical condition yield  that  $h$  and $u$ 
      have the same regularity  and hence  this choice yields that $\Phi$  has  in $\mathcal{S}$
       the same regularity  as $h$. In the Navier-Stokes case,  the dissipation term in the energy estimate yields that
        $ \sqrt{\eps }\, u$ is  one derivative smoother than $u$ and $h$
        and hence it is natural to choose a diffeomorphism such that $\Phi$ has
         the same regularity as $u$ in $\mathcal{S}$. This is not given by our previous choice since
          by using the boundary condition \eqref{bord1}, we shall get that $ \sqrt{\eps }\, h$ and hence
           $ \sqrt{\eps}\, \varphi$ gains only $1/2$ derivative. 
          The easiest way  
          to fix this difficulty is to  to take  a smoothing diffeomorphism 
             as in \cite{Schweizer05}, \cite{Lannes05}.               
         We thus choose $\varphi$  such that
        \beq
        \label{eqphi}
         \varphi(t,y,z)= A z + \eta(t, y, z)
        \eeq
        where $\eta$ is given  by  the extension of $h$ defined by
       \beq
       \label{eqeta}
       \hat{\eta}(\xi, z)= \chi \big(z {\xi  } \big) \hat{h}(\xi)
       \eeq
       where $\hat{\cdot}$ stands for the Fourier transform with respect  to the $y$ variable and
        $\chi$ is a smooth, even,  compactly supported  function such that $ \chi= 1$ on $B(0, 1)$.
      The  number $A>0$  is chosen so that
      \beq
      \label{Adeb}
       \partial_{z} \varphi(0,y,z) \geq  1, \quad  \forall x \in \mathcal{S}
       \eeq  
      which ensures that $\Phi(0, \cdot)$  is a diffeomorphism. 
      
      Another way  to reduce the problem to a fixed  domain would be to use standard  Lagrangian coordinates. This  coordinate system also
      has the advantage that $\Phi$ in $\mathcal{S}$  has the same regularity as the velocity in $\mathcal{S}$. Nevertheless, here since  the velocity
       will  only  have conormal regularity in $\mathcal{S}$,  the Lagrangian map will also have only conormal regularity.
        Consequently one main advantage  of the choice \eqref{eqeta} compared to Lagrangian coordinates is that we still  have for
         $\eta$ in $\mathcal{S}$ a standard Sobolev regularity even if $v$ only has a conormal one. 
      
         To transform the equations by using \eqref{vdef}, we 
   introduce the  operators
     $\partial_{i}^\varphi$, $i=t, \, 1, \, 2, \, 3$
    such that 
   $$   \partial_{i}u \circ \Phi(t, \cdot)= \D_{i}v, \quad i= t, \, 1, \, 2, \, 3
     $$
      we shall give the precise expression of these operators later.
         We  thus   get that  by using the  change of variable \eqref{diff}, the equation
       \eqref{NS} for $u$, $p$ becomes an equation for $v$, $q$ in the fixed domain $\mathcal{S}$
        which reads:
       \beq
       \label{NSv}
       \D_{t}v + \big(v \cdot \nabla^\varphi \big)v + \nabla^\varphi q = \eps   \Delta^\varphi v =2 \eps\nabla^\varphi \cdot S^\varphi v ,  \quad
        \nabla^\varphi  \cdot  v = 0, \quad x \in \mathcal{S}\eeq
       where we have  naturally set
       $$ \nabla^\varphi q = \left( \begin{array}{ccc} \D_{1} q \\ \D_{2} q \\ \D_{3} q \end{array} \right), \quad
        \Delta^\varphi= \sum_{i=1}^3 (\D_{i})^2, \quad \nabla^\varphi\cdot  v= \sum_{i=1}^3 \D_{i} v_{i}, \quad
         v \cdot \nabla^\varphi= \sum_{i=1}^3 v_{i}\D_{i}$$
          and $S^\varphi v= {1 \over 2}\big( \nabla^\varphi v + (\nabla^\varphi v)^t)$.
      The two  boundary conditions \eqref{bord1}, \eqref{bord2} become on $z=0$
      \begin{eqnarray}
      \label{bordv1} & & \partial_{t}h = v \cdot \N= -v_{y}(t,y,0) \cdot \nabla h(t,y) + v_{3}(t,y,0), \\
      \label{bordv2} & &  q\,\n - 2 \eps S^\varphi v \cdot \n = gh \,\n
      \end{eqnarray}
      where we use the notation $v=(v_{y}, v_{3})^t$.  We shall also  use  set  $\nabla= (\nabla_{y}, \partial_{z})^t)$
     From now on, we shall thus study in $\mathcal{S}$ the system
     \eqref{NSv}, \eqref{bordv1}, \eqref{bordv2}  with $\varphi$ given by \eqref{eqphi}, \eqref{eqeta}. We add to this system
      the initial conditions 
      \beq
      \label{initv}
      h_{/t=0}= h_{0}, \quad v_{/t=0}= v_{0}.
      \eeq

      To  measure the regularity of functions 
       defined in $\mathcal{S}$, we shall use Sobolev conormal spaces.
     Let us introduce the  vector fields    
      $$Z_{i}= \partial_{i}, \, i=1, \, 2, \quad  Z_{3}= { z \over  1 - z } \partial_{z}.$$
                          We  define the Sobolev conormal space  $H^{m}_{co}$ as   
   $$ H^m_{co}(\mathcal{S)}= \big\{ f \in L^2(\mathcal{S}), \quad  Z^\alpha  f  \in L^2 (\mathcal{S}), \quad | \alpha| \leq m \big\}$$
   where $ Z^\alpha = Z_{1}^{\alpha_{1}} Z_{2}^{\alpha_{2}} Z_{3}^{\alpha_{3}}  $  
    and we set
 $$ \|  f  \|_{m}^2 =\sum_{|\alpha| \leq m} \| Z^\alpha f \|^2_{L^2}, \quad  \|  f   \| = \|  f   \|_0 = \|  f   \|_{L^2}.  
  $$
  In a similar way, we set 
   $$ W^{m, \infty}_{co}(\mathcal{S})= \big\{ f \in L^\infty(\mathcal{S}), \quad  Z^\alpha  f  \in L^\infty (\mathcal{S}), \quad | \alpha| \leq m \big\}$$
   and 
   $$ \|  f  \|_{m, \infty} =\sum_{|\alpha| \leq k} \| Z^\alpha f \|_{L^\infty}.$$
    For vector fields,  we  also take the  sums over  its components.
   Note that the use of these spaces is classical in (hyperbolic) boundary value problems, see \cite{Gues,Hormander,Rauch,Tartakoff} for example.
    
    Finally for functions defined on  $\mathbb{R}^2$ (like $h$ in our problem), we use the notation $|\cdot |_{m}$ for
     the standard Sobolev norm.
    
    \subsection{Main results}
  
  Our aim is to get a local well-posedness result for strong  solutions of   \eqref{NSv}, \eqref{bordv1}, \eqref{bordv2} 
   which is valid on an interval of time independent of $\eps$ for $\eps \in (0,1]$.  Note that such a result will also  imply  the  local
    existence  of strong solutions for the Euler equation. As it is well-known, 
     \cite{Wu99,AM05,CS07,SZ08}, when there is no surface tension, a Taylor sign condition is needed
      to get local well-posedness for the Euler  equation. For the Euler equation  in a domain of 
 the form \eqref{omegat}, 
       the Taylor sign condition reads
   \beq
   \label{Taylor}    -\partial_{N}p + g \geq c_{0}>0, \quad x \in \Sigma_{t}.
   \eeq
   Before stating our main result, we need to understand what kind of Taylor sign condition
    (necessary in order to get  a uniform with respect to $\eps$ local  existence result) we have to impose
     for the Navier-Stokes equation.
 By using the divergence free condition, we get as usual that the pressure $q$ solves in $\mathcal{S}$  the elliptic equation
 $$ \Delta^\varphi q = - \nabla^\varphi \cdot (v\cdot \nabla^\varphi v).$$
 Moreover, by using the  boundary condition \eqref{bord2}, we get that on the boundary
 $$q_{/z=0} =  2 \eps S^\varphi v  \,{\bf n} \cdot \,{\bf n} + gh.$$
 We shall thus decompose the pressure into an "Euler" part and a "Navier-Stokes" part by setting $q= q^E + q^{NS}$ with
 $$  \Delta^\varphi q^E = - \nabla^\varphi \cdot (v\cdot \nabla^\varphi v), \quad q^E_{/z=0} =   gh$$
  and
  $$  \Delta^\varphi q^{NS} =0, \quad  q^{NS}_{/z=0} =  2 \eps S^\varphi v  \,{\bf n} \cdot \,{\bf n}.$$
 The main idea is that the part $q^{NS}$ is small   when $\eps$ is small  while $q^{E}$ which is of order one is the part which should  converge to the pressure of the Euler equation
   when $\eps$ goes to zero. Consequently, the Taylor sign condition has to be imposed on $q^E$. After the change
    of coordinates, this becomes
    \beq
    \label{Taylor2}
    g- \partial_{z}^\varphi q^E_{/z=0} \geq c_{0}>0.
    \eeq
   Note that  since we shall indeed  prove that the part $q^{NS}$ of the pressure is small  when $\eps$ is small,   we shall actually obtain
    that for $\eps$ sufficiently small, the total pressure verifies the Taylor condition. 
    
    Finally, let us introduce a last notation, we denote by $\Pi= {Id} - {\bf n}\otimes {\bf n}$ the projection on the tangent space of the boundary.
   
   \bigskip 
    
    Our main result reads:

  \begin{theoreme}
    \label{main}
For $m \geq 6$, consider inital data $(h_{0}^\eps, v_{0}^\eps)$ such that
\beq
\label{hypinit} \sup_{\eps \in (0, 1]} \Big(
  |h_{0}^\eps|_{m}+ \eps^{1\over 2} |h_{0}^\eps|_{m+{1\over 2}}
   + \|v_{0}^\eps\|_{m} + \| \partial_{z} v_{0}^{\eps}\|_{m-1} + \| \partial_{z} v_{0}^\eps\|_{1, \infty} + \eps^{1 \over 2}  \|\partial_{zz}^2 v_{0}^\eps \|_{L^\infty} \Big) \leq R_{0},   \eeq
 and assume   that  the Taylor sign condition \eqref{Taylor2} is verified 
     and  that   the compatibility condition
   $ \Pi S^\varphi v_{0}^\eps {\bf n}_{/z=0}= 0$ holds. 
   Then,    
there exists $T>0$ and $C>0$  such that for every  $\eps\in (0, 1]$, there exists a unique solution $(v^\eps, h^\eps)$  of  \eqref{NSv}, \eqref{bordv1}, \eqref{bordv2},
\eqref{eqphi}, \eqref{eqeta} which is 
 defined on $[0, T]$  and satifies the estimate:
\beq
\label{mainborne1}\sup_{[0, T]} \big(\|v^\eps\|_{m}^2 + |h^\eps|_{m}^2 + \|\partial_{z}v^\eps\|_{m-2}^2  + \|\partial_{z} v^\eps\|_{1, \infty}^2 \big)  + \| \partial_{z} v^\eps\|_{L^4([0,T], H^{m-1}_{co})}^2\leq   C.\eeq
Moreover, we also have the estimate
\beq
\label{mainborne2} \sup_{[0, T]}\big(  \eps |h^\eps|_{m+{1 \over 2}}^2 + \eps \|\partial_{zz}v^\eps\|_{L^\infty}^2\big)  +  \eps \int_{0}^T\big( \| \nabla v^\eps \|_{m}^2+
 \|\nabla \partial_{z}v^\eps\|_{m-2}^2\big) \leq C. \eeq
\end{theoreme}

Note that in the above result, we have separated 
 the estimates \eqref{mainborne1} that are independent of $\eps$ from  the ones in
 \eqref{mainborne2} that depend on $\eps$ and 
 are useful to get  a closed estimate for the Navier-Stokes case.
 This is why  we have stated an estimate of $\|\partial_{z} v^\eps\|_{m-1}$
  which is  not pointwise in time but only $L^4$. The reason  for which we do not expect $\sup_{[0,T]}\|\partial_{z} v^\eps\|_{m-1} $
   to be uniformly bounded  with respect to $\eps$ will be explained below.  
It is related  to the boundary  control of 
     the vorticity   for the Navier-Stokes system.
      Nevertheless we point out that there  is  no loss
    of regularity, namely  
we  have that $\partial_{z} v^\eps (t, \cdot)\in H^{m-1}_{co}$ for every  $t \in [0, T]$ with bounds that 
may blow up with $\eps$.
     We shall also provide estimates for  the pressure during  the proof.
    
 Also,  note that the uniform existence  time $T$ is a priori also limited by the validity of the Taylor sign condition \eqref{Taylor2}. 
  For the Euler equation, it was pointed out by S. Wu \cite{Wu97} that  in our infinite depth framework with 
 zero  vorticity in the bulk, 
    a maximum principle applied to the equation for the pressure yields that the Taylor sign condition always holds. 
Nevertheless, this argument breaks down when vorticity is not zero.
Finally, let us  point out  that the compatibility condition $ \Pi S^\varphi v_{0}^\eps {\bf n}_{/z=0}= 0$ that we assume at the initial time 
 is exactly the same as 
(1.8) in  \cite{Beale83}.
    
   The main part of the paper will be devoted to the proof  of Theorem \ref{main}. We shall explain the main steps and the main difficulties
    slightly below. We immediately see 
    that by  using the compactness provided by the uniform estimates of Theorem \ref{main}, we shall easily get as a corollary,
     the justification of the inviscid limit and the existence of a  solution to  the limit 
 Euler system:
     
   \begin{theoreme}
   \label{theoinviscid}
   Under the assumptions of Theorem \ref{main}, if we assume in addition 
  that there exists $(h_{0}, v_{0})$  such that
   \beq
   \label{hypinviscid} \lim_{\eps \rightarrow 0} \|v^\eps_{0} - v_{0}\|_{L^2(\mathcal{S})} + |h^\eps_{0} - h_{0} |_{L^2(\mathbb{R}^2)}= 0. \eeq
     Then there exists  $(h(t,x), v(t,x))$   with  $  Z^\alpha \nabla v \in L^{\infty}([0, T] \times \mathcal{S})$, $| \alpha | \leq 1$ and 
    $$v \in L^\infty([0, T], H^m_{co}(\mathcal{S})), \,  
      \partial_{z}v \in L^\infty([0, T], H^{m-2}_{co}(\mathcal{S})),  \,
    h \in L^\infty([0, T], H^m(\mathbb{R}^2))$$ 
    such that
 $$ \lim_{\eps \rightarrow 0}  \sup_{[0,T]}\Big( \|v^\eps -v\|_{L^2(\mathcal{S})} + \| v^\eps -v\|_{L^\infty(\mathcal{S})} +
  |h^\eps -h |_{L^2(\mathbb{R}^2)} + |h^\eps -h |_{W^{1, \infty}(\mathbb{R}^2)}\Big)= 0
$$
and   which is the unique solution to  the free surface Euler equation
     in the sense that
 \beq
 \label{eulerint} \D_{t}v + \big(v \cdot \nabla^\varphi \big)v + \nabla^\varphi q= 0, \quad \nabla^\varphi \cdot v= 0, \quad x \in \mathcal{S}
 \eeq
 with  the boundary conditions
 \beq
 \label{eulerb}\partial_{t}h= v \cdot \N \quad  \hbox{and} \quad  q= gh\eeq
 at  $z=0$, $\varphi$ being still defined by \eqref{eqphi}, \eqref{eqeta}. 
    \end{theoreme}

We thus provide the justification of the inviscid limit in $L^2$ and  $L^\infty$ norms.  The convergence in higher  norms (conormal for $v$, 
 standard Sobolev for $h$) 
follows  by interpolation  and   the uniform estimate \eqref{mainborne1}. Note that we do not obain by compactness the convergence
 of $v^\eps$ in Lipschitz norm.  This is expected since in that norm, the boundary layer profile cannot be neglected  when we pass to  the limit.
 The above result also provides an existence and uniqueness  result  of solutions  
(with minimal normal regularity of the velocity) to  the free surface Euler system
  which is  new. 
  Note that by using the  equation for the vorticity, one can easily propagate higher normal regularity and thus recover
   the results of \cite{Lindblad05,SZ08,ZZ08}.   In particular, we  get that
   $ \partial_{z}v \in L^\infty([0, T], H^{m-1}_{co}(\mathcal{S}))$. 
   
   \subsection{Sketch of the proof and organization of the paper}
   In order to prove Theorem \ref{main}, since classical local existence results of smooth solutions are available in the litterature
    \cite{Beale81,Tani96}, the main difficulty is to get a priori estimates
     on a time interval small but independent of $\eps$  of the quantities that appear in \eqref{mainborne1}, \eqref{mainborne2} in terms of the initial data only for  a sufficiently smooth solution of the equation.  We shall suppress the subscript $\eps$ in the proof
      of the a priori estimates  for notational convenience, 
      we shall simply denote $(v^\eps, h^\eps)$  by $(v,h)$. Let us define:
       \begin{align*}
        \mathcal{N}_{m, T}(v,h) &  = \sup_{[0, T]}\big( \|v(t) \|_{m}^2 + \|\partial_{z}v\|_{m-2}^2 + |h(t)|_{m}^2+ \eps  |h(t)|_{m+{1 \over 2}}^2
        + \eps \|\partial_{zz}v(t)\|_{L^\infty}^2 +  \|\partial_{z} v(t)\|_{1, \infty}^2 \big) \\
        \nonumber & \quad  \quad \quad   + \|\partial_{z}v\|_{L^4([0,T], H^{m-1}_{co})}^2 
        + \eps \int_{0}^T \|\nabla v\|_{m}^2 + \eps \int_{0}^T \|\nabla \partial_{z}v\|_{m-2}^2<+ \infty. 
       \end{align*}
      The crucial point is thus to get a closed  control of the above quantity on a sufficiently small
 time interval which is  independent of $\eps$. 
The main part of the paper will be devoted to  these a priori estimates.
   Many steps will be needed in order to get the result.

   \subsubsection*{Step 1: Estimates of $v$ and $h$}   
   The  first step will be to estimate  $Z^\alpha v$ and $Z^\alpha h$ for
    $0 \leq |\alpha| \leq m$. When $\alpha=0$, the estimate is a consequence of the energy identity for free surface Navier-Stokes
     equation which reads 
     $$  {d \over dt} \Big(  \int_{\mathcal{S}} |v|^2 \, d\V + g \int_{z=0} |h|^2\, dy  \Big) + 4 \eps  \int_{\mathcal{S}} |S^\varphi v|^2\, d\V=0.$$
  Here $d\V$ stands for the natural volume element induced by the change of variable \eqref{eqphi}: $d\V=\partial_{z}\varphi(t,y,z)\, dydz$.
    
    The difficulty to get estimates for higher order derivatives is that the coefficients (which depend on $h$) in the equation \eqref{NSv} are not smooth enough
    (even with the use of the smoothing diffeomorphism that we have taken) to control 
 the commutators in the  usual way.
    For example, for the transport term  which reads:  
   $$ \partial_{t}^\varphi + v \cdot \nabla^\varphi = \partial_{t} + v_{y}\partial_{y} + { 1 \over \partial_{z} \varphi}\big( v\cdot N - \partial_{t} \eta)\partial_{z}, 
  \quad N= (-\partial_{1}\varphi, -\partial_{2} \varphi, 1)^t, $$
 the commutator between $Z^\alpha$ and this term  in the equation involves in particular the term 
 $ (v \cdot Z^\alpha N)\partial_{z}v$ which can be estimated only with the help of   $\|Z^\alpha N\|_{L^2} \sim |h|_{m+{1 \over 2}}$.
  This yields
  a loss of $1/2$ derivative. 
   We also get similar problems when we compute for the pressure term the commutator between $Z^\alpha$ and $\nabla^\varphi q$. 
  This difficulty was solved  by Alinhac in \cite{Alinhac}.   
The main idea is  that some cancellation occurs
   when we use  the  good unknown
  $V^\alpha= Z^\alpha v - {\partial_{z}^\varphi v} Z^\alpha \eta$.
Indeed, let us write our  equation under the abstract  form  
$$ \mathcal{N}(v,q, \varphi)= \D_{t}v + \big(v \cdot \nabla^\varphi \big)v + \nabla^\varphi q -2\eps \nabla^\varphi \cdot \big( S^\varphi v\big).$$
Then,  if  $ \mathcal{N}(v,q, \varphi )=0, $  the linearized equation can be written as
  \begin{eqnarray*}
   & &    D\mathcal{N}(v, q, \varphi) \cdot(\dot v, \dot q, \dot \varphi)  = \\ 
   &  &  \quad \quad  \quad \quad  \big( \partial_{t}^\varphi +( v \cdot \nabla^\varphi)-  2 \eps \nabla^\varphi \cdot \big( S^\varphi \cdot  \big) \big)\big(\dot v- \D_{z} v \, \dot \varphi \big)
    + \nabla^\varphi\big( \dot q -   \D_{z} q \, \dot \varphi\big) \\
    & & \quad \quad   \quad \quad  \quad \quad  \quad \quad   +  \big( \dot v \cdot \nabla^\varphi\big) v - \dot \varphi (\D_{z} v \cdot \nabla^\varphi) v. 
   \end{eqnarray*}
This means that the fully linearized equation has the same structure as the equation linearized with respect to the $v$ variable only
 thanks to the introduction of the good unknown. The justification of this identity will be recalled in
 section \ref{prelim}.

By using this crucial remark, we get that the equation for $(Z^\alpha v, Z^\alpha q, Z^\alpha \eta)$ can be written as
$$ \D_{t} V^\alpha + v \cdot \nabla^\varphi V^\alpha + \nabla^\varphi Q^\alpha- 2 \eps \nabla^\varphi \cdot S^\varphi V^\alpha
  =l.o.t.$$
  with 
$V^\alpha = Z^\alpha v - {\partial_{z}^\varphi v} Z^\alpha \eta$, $Q^\alpha=  Z^\alpha q - {\partial_{z}^\varphi q} Z^\alpha \eta$
 and the $l.o.t$ means lower order terms with respects to $v$ and $h$ that can be controlled by 
$ \mathcal{N}_{m, T}(v,h) $.
 Consequently,  we can  perform an $L^2$ type energy estimate for this equation.
  By using  energy estimates, we shall get  an identity under the form 
  $$ {d \over dt }{1 \over 2} \int_{\mathcal{S}} |V^\alpha|^2 \, d\V +  {d \over dt }{1 \over 2} \int_{z=0} (g - \partial_{z}^\varphi q^E) |Z^\alpha h|^2
  = \cdots$$
    which yields a  good control of  the regularity of the surface only if  the sign condition  $g-\partial_{z}^\varphi q^E \geq c_{0}>0$
   is matched.
   
   The main conclusion of this step will be  that 
   \begin{equation} \label{step1est}
  \big\|\big(Z^m v - \partial_{z}^\varphi v Z^m \eta\big)(t)\big\|^2 + |h(t)|_{m}^2  \leq C_{0}+ t \Lambda(R)+ \int_{0}^t \|\partial_{z}v\|_{m-1}^2
\end{equation}
  where $C_{0}$ depends only on the initial data and $\Lambda$ is some continuous  increasing function 
 in all its arguments  (independent of $\eps$)
  as soon as 
   $$  Q_{m}(t) = \|v\|_{m}^2 + |h|_{m}^2 + \|\partial_{z}v\|_{m-2}^2  +  \|v\|_{2, \infty}^2 + \|\partial_{z}v\|_{1, \infty}^2 + \eps \|\partial_{zz}v\|_{L^\infty}^2 \leq R$$
 for $t \in [0, T^\eps]$.

\subsubsection*{Step 2: Normal derivative estimates I}
In order to close the argument, we need   estimates  for 
 $\partial_{z}v$. We shall first estimate $\|\partial_{z}v \|_{L^\infty_{t}(H^{m-2}_{co})}$.
 This is not sufficient to control the right-hand side in \eqref{step1est}, 
but this will be important in order to get the $L^\infty$ estimates. 
 The main idea is to use the equivalent quantity
 $$ S_{N}= \Pi S^\varphi v\, N$$
 which vanishes on the boundary.
    This allows to perform conormal estimates on the convection-diffusion type equation with homogeneous  Dirichlet boundary condition
    satisfied by $S_{N}$.
  This yields again an estimate under the form 
   $$ \|\partial_{z}v(t)\|_{m-2}^2  \leq C_{0}+ t  \Lambda(R)+ \int_{0}^t \|\partial_{z}v\|_{m-1}^2.$$
   
 \subsubsection*{Step 3: $L^\infty$ estimates}
 We also have to estimate the $L^\infty$ norms  that occur in the definition of $Q_{m}$. Note that we 
can not use Sobolev embedding (as it is classically done) since the conormal spaces do not control 
normal derivatives. 
 The estimate of $\|v\|_{2, \infty}$ is a consequence of  an  anisotropic Sobolev estimate and
 thus,  the difficult part is to estimate $\|\partial_{z}v\|_{1, \infty}$. Again, it is more convenient to estimate
  the equivalent quantity $\| S_{N}\|_{1,\infty}$ since $S_{N}$ solves
   a convection diffusion equation with homogeneous boundary condition.  The estimate of $\|S_{N}\|_{L^\infty}$
   is a consequence of the maximum principle for this equation. The control 
 of   $\| Z_{i}S_{N}\|_{L^\infty}$
    is  more difficult. The main reason is  that a crude estimate of the commutator  between $Z_{i}$
     and the variable coefficient operator $\Delta^\varphi$ involves  terms  with more normal derivatives:   two normal derivatives of $S_{N}$
      and hence three normal derivatives of $v$. To overcome 
 this difficulty, we note that at this step, the regularity of the surface
       is not really a problem: we want to estimate a fixed    number of derivatives of $v$ in $L^\infty$ while
        $m$ can be considered as large as we need. Consequently, the idea is to change the coordinate system
         into  a normal geodesic one in order to get  the simplest possible expression for the Laplacian. 
    By neglecting all the terms that can be estimated  by the previous steps, we get a simple 
 equation
     under the form
     $$  \partial_{t}  \tilde S_{N} + z \partial_{z}w_{3}(t,y,0) \partial_{z} 
         \tilde S_{N} + w_{y}(t,y,0) \cdot \nabla_{y} \tilde S_{N} - \eps \partial_{zz} \tilde S_{N}  =
          l.o.t$$
          where $\tilde{S}_{N}$ stands for $S_{N}$ expressed in the new coordinate system and $w$ is  the vector field that we obtain from
           $v$ by the change of variable. 
           This is a one-dimensional Fokker Planck type equation  (with an additional  drift term in the tangential direction that can be eliminated
            by using lagrangian coordinates in this direction)  for which the Green function is explicit
             and hence, we can use it to estimate $\|Z_{i}\tilde S_{N}\|_{L^\infty}$. 
         
         Again the conclusion of this step is an estimate of  the form
        $$ \|\partial_{z} v\|_{1, \infty}^2+ \eps \|\partial_{zz}v\|_{L^\infty}^2 \leq C_{0}+t \Lambda(R) + \int_{0}^t \|\partial_{z}v\|_{m-1}^2.$$
             
   \subsubsection*{Step 4: Normal derivative estimate II}
   In order to close our estimate, we still need to estimate
   $ \| \partial_{z}v\|_{m-1}$. For this estimate it  does not seem a good idea to use $S_{N}$ as an equivalent quantity for $\partial_{z}v$.
   Indeed, the equation for $Z^{m-1}S_{N}$ involves $Z^{m-1} D^2 p$  as a source term and we note that since  the Euler part of the pressure
 involves a harmonic function that   verifies $p^E =gh$ on the boundary, we have that
    $$ Z^{m-1}D^2 p^E \sim Z^{m-1} D^{3\over 2} h \sim |h|_{m+{1 \over 2}}$$
     and hence we do not have enough regularity of the surface. To get a better estimate,  
      it is natural to try  to use the vorticity  $\omega= \nabla^\varphi \times v$ in order to eliminate the pressure. We shall get that
      indeed $Z^{m-1} \omega$  solves an  equation under the form
    $$\partial_{t} Z^{m-1}\omega + V \cdot  \nabla Z^{m-1}\omega - \eps \Delta^\varphi Z^{m-1} \omega= l.o.t$$
    Nevertheless, while for the Euler equation the vorticity which solves a transport equation with a 
  characteristic boundary  is  easy to estimate,  for the Navier-Stokes system
     in a  domain with boundaries it is much more difficult to estimate. 
 The difficulty  in the case of  the Navier-Stokes  system is that we need an estimate of  the value
      of the vorticity at  the boundary to estimate it in the interior.
       Since on the boundary we have roughly $Z^{m-1}\omega \sim  Z^m  v +  Z^m h$,
        we only have by using a trace estimate 
         a (uniform) control by known quantities (and in particular the energy dissipation of the Navier-Stokes equation) of 
    $$ \sqrt{\eps}\int_{0}^t |Z^{m-1} \omega_{/z=0}|_{L^2(\mathbb{R}^2)}^2.$$
    In this case, by a simple computation on the heat equation which is  given in section \ref{sectionheat}, 
   we see that  the best estimate that we can expect,  when we study the problem with zero initial datum
    $$ \partial_{t} f - \eps \Delta f= 0,  \, z<0, \quad f_{/z=0}= f^b$$
and with  the boundary data $f^b$ that satisfies an estimate as above  is 
   $$ \int_{0}^{+ \infty} e^{-2\gamma t } \big\|( \gamma + |\partial_{t}|)^{1\over 4}) f \big\|_{L^2(\mathcal{S})}^2 \, dt
     \leq \sqrt{\eps} \int_{0}^{+ \infty} e^{-2 \gamma t}|f^b|_{L^2(\mathbb{R}^2)}^2 \, dt\leq C.$$
Consequently, we see that we get a control of $f$ in $H^{1\over 4}((0, T), L^2)$ which gives by Sobolev embedding an estimate of $f$
 in  $L^4([0, T], L^2(\Omega))$ only.
 
 Motivated by this result  on the heat equation, 
  we shall  get an  estimate  of  $\|Z^{m-1}\omega\|_{L^4((0,T), L^2)}$.
     Note that the transport term in the equation  has an important effect. Indeed, in the previous example of the heat equation, 
   if we add a constant  drift $c \cdot \nabla f $ in the equation, we obtain a smoothing effect under the form
   $$ \int_{0}^{+ \infty} e^{-2\gamma t } \big\|( \gamma + |\partial_{t}+c \cdot \nabla |)^{1\over 4}) f \big\|^2 \, dt.$$
 In order to estimate $\|Z^{m-1}\omega\|_{L^4((0,T), L^2)}$, we shall consequently 
   first switch into Lagrangian coordinates in order  to eliminate the transport term  and  hence look  for an estimate of 
     $ \|(Z^{m-1} \omega ) \circ X\|_{H^{1 \over 4}([0,T],  L^2)}$ since $(Z^{m-1}\omega)\circ X$ solves a purely parabolic
      equation. Note that the price to pay is that the parabolic operator that  we get  after the change of variable has coefficients with limited
       uniform  regularity (in particular with respect to the time and normal variables).  To get this estimate, we use a microlocal symmetrizer
      based on a "partially"  semiclassical   paradifferential calculus i.e. based on the weight       $(\gamma^2 + |\tau|^2 + |\sqrt{\eps}\, \xi|^4)^{1 \over 4})$.
      The use of  paradifferential calculus in place of pseudodifferential one  is needed  because of the only  low regularity uniform estimates of the coefficients that are known.
    The main properties of this calculus can be seen as a consequence of the general quasi-homogeneous calculus studied in \cite{Metivier-Zumbrun} 
    and is  recalled at 
 the end of the paper. 
  By Sobolev embedding this yields an estimate of  $\|(Z^{m-1} \omega ) \circ X\|_{L^{4}([0,T],  L^2)}$
     and thus of    $\|Z^{m-1} \omega\|_{L^{4}([0,T],  L^2)}$ by  a change of variable.
         This finally allows to get an estimate of  $\|Z^{m-1}\partial_{z} v\|_{L^4((0,T), L^2)}$.
       
      The  estimate  of $   \mathcal{N}_{m, T}(v,h)$ 
on some uniform time   follows  by combining the estimates of the four steps. Note that at 
  the end, we also have to check that
       the Taylor sign condition and the condition that $\Phi(t, \cdot)$ is a diffeomorphism remain true.
      
  \subsubsection*{Organisation of the paper}
  The paper is organized as follows. In section \ref{prelim}, we recall the main properties of Sobolev conormal spaces, in particular,  the product laws, 
  embeddings and trace estimates that we shall use. We also state  some geometric  identities that are linked to the change of variable
   \eqref{id0} and the Korn inequality that we shall use to control the energy dissipation term in \eqref{NSv}.
  In section \ref{sectionprelimeta}, we  study the regularity properties of  $\eta$ given by  \eqref{eqeta}.
  Next, the main part of the paper is devoted to the proof of  the a priori estimates leading to the proof  of Theorem \ref{main}.
    We first  recall the energy identity for the Navier-Stokes equation \eqref{NSv} in section \ref{sectionbasic}.
     The aim of  the nexts three sections is to get the estimates of higher order conormal derivatives. 
     In section \ref{sectionconorm}, we study the equations satisfy by $(Z^\alpha v, Z^\alpha h)$  and  prove the estimates for
      lower order commutators and boundary values. In section \ref{sectionpressure}, we prove the needed estimates for the pressure using elliptic regularity in conormal spaces 
       and  in section \ref{sectionconorm2}, we get the  desired estimates for  $\|v\|_{m}$ and $|h|_{m}.$
       Next, in section \ref{sectionnorm1}, we start the study of the estimates  for  normal derivatives, we prove an estimate
        for $\|\partial_{z} v \|_{m-2}$. In section \ref{sectionLinfty}, we prove the needed $L^\infty$ estimates.
         Finally, in section \ref{sectionnorm2}, we prove the estimate for  
$\|\partial_{z} v \|_{L^4 ([0,T], H^{m-1}_{co})}$.
           The results of paradifferential calculus that are needed for this section are recalled in section \ref{sectionparadiff}. 
            The proof of the existence
            part of  Theorem \ref{main} is given in section \ref{sectionexist} and the uniqueness part is briefly explained in 
            section \ref{sectionunique}. The proof of Theorem \ref{theoinviscid} is given in section \ref{sectioninviscid}  and  section \ref{sectiontech} is devoted to the proof of some technical lemmas.
   
   \bigskip
   
   In all the paper, the notation $\Lambda(\cdot, \cdot)$ stands for a continuous increasing function  in all its arguments, independent of $\eps$ and  that may change
    from line to line.

  \section{Functional framework and geometric identities}
 \label{prelim}
 
 \subsection{Some properties of Sobolev conormal spaces}

   We have already recalled the notations that we shall use for Sobolev conormal spaces.
   It  will be convenient to also  use the following notation. We set for $m \geq 1$:
   $$ E^m= \big\{ f \in H^m_{co}, \quad \partial_{z}f \in H^{m- 1}_{co}\big\}, \quad E^{m, \infty}= \big\{ f \in W^{m, \infty}_{co}, \quad \partial_{z}f \in W^{m-1, \infty}_{co}\big\}$$
   and we define the norms:
   $$ \|f\|_{E^m}^2= \|f\|_{m}^2 + \|\partial_{z} f \|_{m-1}^2, \quad \|f\|_{E^{m,\infty}}= \|f\|_{m, \infty}+ \|\partial_{z} f \|_{m-1, \infty}.$$

   We shall also use Sobolev tangential spaces defined for $s\in \mathbb{R}$   by
  $$ H^s_{tan}(\mathcal{S})= \big\{ f \in L^2(\mathcal{S}), \quad  \Lambda^s  f  \in L^2 (\mathcal{S})\big\}$$
  where $\Lambda^s$ is the tangential Fourier multiplier by $\big(1+ |\xi|^2 \big)^{s \over 2}$ i.e.
   $\mathcal{F}_{y} \big( \Lambda^s f \big)(\xi, z)= \big(1+ |\xi|^2 \big)^{s \over 2} \mathcal{F}_{y}( f) (\xi, z)$
   where $\mathcal{F}_{y}$ is the partial Fourier transform in the $y$ variable. We set
   $$ \|f \|_{H^{s}_{tan}}= \| \Lambda^s f \|_{L^2}.$$
   Note that we have
   $$ \| f\|_{H^s_{tan}} \lesssim  \| f\|_{m}$$
    for $m \in \mathbb{N}$, $m \geq s$.
   
  In this paper, we shall use repeatedly the following  Gagliardo-Nirenberg-Moser type properties of the Sobolev conormal spaces:
  
  \begin{prop}
  \label{sob}
  We have the following products,  and commutator  estimates:
  \begin{itemize}
  \item 
       For $u, \,v\in L^ \infty \cap H^k_{co}$, $k \geq 0$:
            \beq
     \label{gues}
      \| Z^{\alpha_{1}} u\, Z^{\alpha_{2}} v \| \lesssim  \|u\|_{L^\infty} \, \|v\|_{k} + \|v\|_{L^\infty} \|u\|_{k}, \quad
       |\alpha_{1}| + |\alpha_{2}|=k.
      \eeq
  \item For  $ 1 \leq | \alpha | \leq k $, $g \in H^{k-1}_{co} \cap L^\infty$,  $f \in H^k_{co}$ such
   that $Zf  \in L^\infty$,  we have
  \beq
  \label{com}
  \| [Z^\alpha, f]g \| \lesssim  \|Zf\|_{k-1}  \| g \|_{L^\infty} +  \|Z f \|_{L^\infty} \|g \|_{k-1}.
  \eeq
  \item For $| \alpha |=k \geq 2$, we  define the symmetric commutator  $[Z^\alpha, f, g]= Z^\alpha (f g) - Z^\alpha f \, g - f  Z^\alpha \,g$.
   Then, we have the estimate
   \beq
   \label{comsym} \| [Z^\alpha, f, g] \| \lesssim  \|Zf \|_{L^\infty}\|Zg\|_{k-2} +  \|Z g \|_{L^\infty} \|Z f \|_{k-2}.
   \eeq
   \end{itemize}
   \end{prop}
     \begin{proof}
  The proof of \eqref{gues} is classical and can be found for example in \cite{Gues}. The commutator estimates \eqref{com}, \eqref{comsym}
   follow from \eqref{gues} and the Leibnitz formula. Indeed, we have
   $$[Z^\alpha, f]g =  \sum_{  \tiny{ \begin{array}{ll}\beta + \gamma = \alpha, \\
    \beta \neq 0 \end{array} } } C_{\beta, \gamma}   Z^\beta f  Z^\gamma g$$
     and hence since $\beta \neq 0$, we can write $Z^\beta =  Z^{\tilde \beta} Z_{i}$, with $ |\tilde \beta |= | \beta |-1, $  to get that
     $$ \|Z^{\tilde \beta} Z_{i}f\,  Z^\gamma g \| \lesssim  \|Zf\|_{L^\infty} \|g\|_{k-1} + \|g\|_{L^\infty} \|Zf\|_{k-1}$$
      thanks to \eqref{gues}. The proof of \eqref{comsym} can be obained by a similar argument.
        \end{proof}

   We shall also need embedding and trace estimates for these spaces:
  
\begin{prop}
\label{proptrace}
\begin{itemize}
     \item For $s_{1}\geq 0 $, $s_{2} \geq 0$ such that $s_{1}+ s_{2} >2$ and   $f$ such that $f \in H^{s_{1}}_{tan}$, $\partial_{z} f \in  H^{s_{2}}_{tan}$, we have the anisotropic Sobolev embedding:
  \beq
  \label{emb}
  \|f \|_{L^\infty}^2 \lesssim   \|\partial_{z}f \|_{H^{s_{2}}_{tan}} \, \|f\|_{H^{s_{1}}_{tan}}.
  \eeq
  \item  For $ f \in H^{1 }(\mathcal{S})$, we have the trace estimates:
  \beq
  \label{trace} 
 |f(\cdot, 0)|_{H^s(\mathbb{R}^2)} \leq C \| \partial_{z} f \|^{1\over 2}_{H^{s_{2}}_{tan}}\, \|f\|^{1\over 2}_{H^{s_{1}}_{tan}},
 \eeq 
 with  $s_{1}+ s_{2}= 2 s \geq 0$.

   \end{itemize}
  \end{prop}   
 As a consequence of  \eqref{sob}, we  shall use very often that :
  \begin{rem}
  \label{remLinfty}
   For $k \geq 5$, we have :
  $$  \| f\|_{2, \infty}^2 \lesssim  \|\partial_{z} f\|_{k-2} \, \| f\|_{k}, \quad k \geq 5.$$
    \end{rem} 
  
  \begin{proof}


 To get  \eqref{emb}, we first note that thanks to the one-dimensional Sobolev estimate that we have
 $$ |\hat f (\xi, z) | \leq  \Big( \int_{-\infty}^0 |\partial_{z} \hat f(\xi, z ) |\, | \hat f (\xi, z)| \, dz  \Big)^{1 \over 2 }$$
  and hence, we obtain  from Cauchy-Schwarz and the fact that $s_{1}+ s_{2}>2$  that
  $$  \|f\|_{L^\infty} \leq \int_{\xi}| \hat f(\xi, z)| \, d\xi \lesssim\Big( \int_{\mathbb{R}^2}\big( 1 + |\xi |)^{s_{1}+ s_{2}}
   \int_{-\infty}^0 |\partial_{z} \hat f(\xi, z ) |\, | \hat f (\xi, z)| \, dz \, d \xi \Big)^{1 \over 2 }  \lesssim \| \partial_{z} \Lambda^{s_{1}} f \|^{1\over 2} \, \| \Lambda^{s_{2}} f\|^{1 \over 2}.$$

%
  The   trace estimates  \eqref{trace}  are also  elementary in $\mathcal{S}$.  To get the estimate,  it suffices
   to write that
  $$ |f(\cdot, 0)|_{H^s(\mathbb{R}^2)}^2 = 2   \int_{\mathbb{R}^2} \int_{-\infty}^0   \partial_{z}\Lambda^s f(z,y)\, \Lambda^sf(z,y) dz dy$$
  and the result follows from Cauchy-Schwarz and the fact that
  $$ \int_{\mathbb{R}^2 }\partial_{z}\Lambda^sf(z,y)\, \Lambda^sf(z,y) \, dy= \int_{\mathbb{R}^2 }\partial_{z}\Lambda^{s_{1}}f(z,y)\, \Lambda^{2s - s_{1}}f(z,y) \, dy.$$

  \end{proof}
  
  For later use, we also recall the classical tame Sobolev-Gagliardo-Nirenberg-Moser  and commutator estimates
   in $\mathbb{R}^2$:
  \begin{prop}
  \label{sobbord}
  For $s \in \mathbb{R}$, $s\geq 0$, we have
  \begin{eqnarray}
  \label{sobr2} & & |\Lambda^s(fg) |_{L^2(\mathbb{R}^2)}
   \leq C_{s}  \big( | f|_{L^\infty(\mathbb{R}^2} |g|_{H^s(\mathbb{R}^2)} +   | g|_{L^\infty(\mathbb{R}^2)} |f|_{H^s(\mathbb{R}^2)} \big) \\
 & &   \label{comr2}
  \| [\Lambda^s, f ] \nabla g\|_{L^2(\mathbb{R}^2)}
   \leq C_{s}  \big( | \nabla f|_{L^\infty(\mathbb{R}^2)} |g|_{H^s(\mathbb{R}^2)} +   |\nabla  g|_{L^\infty(\mathbb{R}^2)} |f|_{H^s(\mathbb{R}^2)} \big)\\
   & & \label{cont2D} |u v|_{1 \over 2} \lesssim  |u|_{1, \infty} |v|_{1\over 2}
   \end{eqnarray}  
   where $\Lambda^s$ is the Fourier multiplier by $ \big(1+ |\xi|^{2}\big)^{s \over 2}$.
   \end{prop}
     These estimates are also classical, we refer for example to \cite{Chemin95}. Note that the last estimate can be obtained  as
      a consequence of (1) and  (5) in Theorem \ref{symbolic}.

   We shall also need to use results about 
 semiclassical paradifferential calculus in section \ref{sectionnorm2}. They are  described
    in section \ref{sectionparadiff}.              
     
    \subsection{The equations in the fixed domain}
    
    By using the change of variables \eqref{vdef} and the definition \eqref{eqphi}, \eqref{eqeta},  we obtain that
    \begin{eqnarray*}
    \big( \partial_{i} u\big)(t, y, \varphi) & =   & \big( \partial_{i} v - { \partial_{i} \varphi \over \partial_{z} \varphi} 
     \partial_{z} v\big)(t,y,z), \quad i= 0,\,  1,\,  2  \\
      \big( \partial_{3} u\big)(t, y, \varphi) & = &\big( {1 \over \partial_{z} \varphi} \partial_{z} v\big)(t,y,z)
    \end{eqnarray*}
    where we set $\partial_{0}= \partial_{t}$.
    We thus introduce the  following operators
    $$ \partial^\varphi_{i} =  \partial_{i}  - { \partial_{i} \varphi \over \partial_{z} \varphi} 
     \partial_{z}, \quad i=0, \, 1,\, 2, \quad \D_{3}= \D_{z}=  {1 \over \partial_{z} \varphi} \partial_{z} $$
     in order to have
     \beq
     \label{id0}
      \partial_{i}u \circ \Phi= \D_{i}v, \quad i= 0, \, 1, \, 2, \, 3.
      \eeq
      This  yields  by using the  change of variable \eqref{diff}  that $(v,q)$ solves in the fixed domain $\mathcal{S}$  the system \eqref{NSv}, \eqref{bordv1}, \eqref{bordv2} that was introduced previously.
    Note that thanks to this definition,  since the operators $\partial_{i}$ commute, we immediately 
     get that
     \beq
     \label{comD}
     [ \D_{i}, \D_{j}]= 0, \quad \forall i, \, j.
     \eeq

     Since the Jacobian of the change of variable \eqref{diff} is $\partial_{z} \varphi$, it is natural to 
      use on $\mathcal{S}$ when performing energy estimates  the following weighted $L^2$  scalar products:
      $$ \int_{\mathcal{S}}  f\, g  \, d \V, \quad d\V = \partial_{z}\varphi(t,y,z) \, dy dz 
       $$
       and for   vector fields on $\mathcal{S}$.
       $$ \int_{\mathcal{S}}  v \cdot  w  \, d \V $$
       where $\cdot$ stands for the standard
       scalar product of $\mathbb{R}^3$.
       
     With this notation, we have the following integration by parts identities for the operators $\partial_{i}^\varphi$ and 
     the above weighted scalar products:
     \begin{lem}
     \label{lemipp}
     \begin{eqnarray}
     \label{ipp1}
     & &\int_{\mathcal{S}}\D_{i}f\,  g \, d\V = - \int_{\mathcal{S}} f\, \D_{i}g\, d\V  +  \int_{z=0} fg \, \N_{i}\, dy, \quad i= 1, \, 2, \, 3, \\
   & &  \label{ipp2} \int_{\mathcal{S}}\D_{t}f\,  g \, d\V =  \partial_{t}  \int_{\mathcal{S}}f\, g \, d\V - \int_{\mathcal{S}}f\, \D_{t} g \, d\V  - \int_{z=0} f g\, \partial_{t} h
     \end{eqnarray}
     where the outward normal $\N$ is given by \eqref{Ndef}.
     \end{lem}
     Note that in the above formulas, though it is not always explicitely mentionned, in each occurence
      of $\varphi$, this function has to be taken at the time $t$.
      
  We recall that by the choice \eqref{eqphi}, we have that on $z=0$, $\varphi=h$. 
     These formulas are  straightforward consequences of the standard integration by parts formulas.
     Also, as a straightforward corollary, we get: 
     \begin{cor}
     \label{coripp}
     Assume that $v(t, \cdot)$ is a vector field on $\mathcal{S}$ such that $\nabla^\varphi \cdot v=0$,
      then for every smooth functions $f$, $g$ and smooth vector  fields $u$, $w$, we have the identities:
   \begin{eqnarray}
   \label{ippt}  & &   \int_{\mathcal{S}}\big( \D_{t} f + v \cdot \nabla^\varphi f\big) f \, d\V= {1 \over 2} \partial_{t}  \int_{\mathcal{S}}
       |f|^2\, d\V  - {1 \over 2} \int_{z= 0} |f|^2 \big( \partial_{t}h - v \cdot \N \big) dy, \\
 \label{ippD}       & & \int_{\mathcal{S}} \Delta^\varphi f\, g \, d\V= - \int_{\mathcal{S}} \nabla^\varphi f \cdot \nabla^\varphi g\, 
        d\V + \int_{z=0} \nabla^\varphi f \cdot \N\, f\, dy, \\
     & &    \label{ippS}
         \int_{\mathcal{S}} \nabla^\varphi \cdot (S^\varphi u)  \cdot w  \, d\V= - \int_{\mathcal{S}} S^\varphi u    \cdot S^\varphi w\, 
        d\V + \int_{z=0} (S^\varphi u\, \N) \cdot  w\, dy. 
        \end{eqnarray}
      \end{cor}
     \subsection{Alinhac good unknown}
     In order to perform high order energy estimates, we shall need to study the equation that we get
      after  applying  conormal derivatives to the equation \eqref{NSv}.  The standard way to proceed is
       to  say that  the obtained equation has the structure
       \beq
       \label{naif} \D_{t}Z^\alpha v + \big(v \cdot \nabla^\varphi \big) Z^\alpha v + \nabla^\varphi Z^\alpha q = \eps   \Delta^\varphi Z^\alpha v + l.o.t\eeq
        where $l.o.t$ stands for lower order terms and then to perform the natural energy estimate for this equation.
         The difficulty in free boundary problems is that very often $\varphi$ is not smooth enough to consider all the standard commutators
          as lower order terms. Nevertheless there is a crucial  cancellation pointed out by Alinhac \cite{Alinhac} which
           allows to perform high order energy estimates.  Since applying a derivative to the equation is like linearizing the equation, 
            this cancellation can be explained in terms of a link between full an partial linearization.      
     Let us set
    \begin{eqnarray*}
    & &  \mathcal{N}(v,q, \varphi) =
    \D_{t}v + \big(v \cdot \nabla^\varphi \big)v + \nabla^\varphi q -2\eps \nabla^\varphi \cdot \big( S^\varphi v\big) ,  \\
    & & d(v,\varphi)= \nabla^\varphi \cdot v,\\
   & &  \mathcal{B}(v,q,\varphi)=  2 \eps S^\varphi v \, \N -  (q-gh)\N.
    \end{eqnarray*}
     It is more convenient to  use  the above form of   equation \eqref{NSv}
     in view of the boundary condition \eqref{bord2}.
     Note that by formal differentiation with respect to all the unknowns, we obtain the linearized  operators:
    \begin{eqnarray}
    \label{full1}
     D\mathcal{N} (v, q, \varphi) \cdot(\dot v, \dot q, \dot \varphi)& = & 
    \D_{t}\dot v + \big(v \cdot \nabla^\varphi \big)\dot v + \nabla^\varphi \dot q - 2 \,\eps \nabla^\varphi \cdot\big( S^\varphi  \dot v\big)   
     +\big( \dot v \cdot  \nabla^\varphi \big)v \\
  \nonumber    & &
       - \D_{z} v \big( \D_{t} \dot \varphi + v \cdot \nabla^\varphi \dot \varphi \big)  + \partial_{z}^\varphi q \, \nabla^\varphi \dot
        \varphi \\
      \nonumber   & & + 2  \eps \nabla^\varphi\big(  \D_{z} v  \otimes \nabla^\varphi \dot \varphi + \nabla^\varphi \dot \varphi \otimes \D_{z} v \big)
       + 2 \eps \D_{z} \big( S^\varphi v\big)\, \nabla^\varphi \dot \varphi, \\
   \label{full3}     D d (v, \varphi) \cdot (\dot v, \dot \varphi)& =  & \nabla^\varphi \cdot \dot v -  \nabla^\varphi \dot \varphi
         \cdot \D_{z} v, \\
         \label{full4} D \mathcal{B}(v,q,\varphi) \cdot(\dot v, \dot q, \dot \varphi) & = &  2 \eps S^\varphi \dot v \N - \D_{z} v \otimes \nabla^\varphi \dot \varphi \N
          - \nabla^\varphi \dot \varphi \otimes  \D_{z} v \N- (\dot q -g \dot h) \N  \\
          \nonumber  & &  +\big( 2 \eps S^\varphi v - (q- gh)\big)\dot  \N,
    \end{eqnarray}
    where  $\dot \N= (-\partial_{1} \dot \varphi, -\partial_{2} \dot \varphi, 0)^t$.
     
    We have the following crucial identity first observed by Alinhac \cite{Alinhac} which relates
     full and partial linearization.
     \begin{lem}
     \label{lemal}
     We have the following identities:
     \begin{eqnarray*}
      D\mathcal{N}(v, q, \varphi) \cdot(\dot v, \dot q, \dot \varphi)  & = &
   \big( \partial_{t}^\varphi +( v \cdot \nabla^\varphi)-  2 \eps \nabla^\varphi \cdot \big( S^\varphi \cdot  \big) \big)\big(\dot v- \D_{z} v \, \dot \varphi \big)
    + \nabla^\varphi\big( \dot q -   \D_{z} q \, \dot \varphi\big) \\
    & &  \big( \dot v \cdot \nabla^\varphi\big) v + \dot \varphi \big(  \D_{z}\big( \mathcal{N}(v,q, \varphi )\big) 
       - (\D_{z} v \cdot \nabla^\varphi) v \big)  \\
        D d(v, \varphi)\cdot(\dot v, \dot \varphi) &= & 
         \nabla^\varphi \cdot \big( \dot v - \D_{z}v\, \dot \varphi \big) + \dot \varphi \,\D_{z}  \big( d(v, \varphi) \big),\\
       D \mathcal{B}(v,q,\varphi) \cdot(\dot v, \dot q, \dot \varphi) & = &  2 \eps S^\varphi\big( \dot v -\D_{z} v \,\dot \varphi \big) \N
        +  2 \eps \, \dot \varphi \, \partial_{z} \big( S^\varphi v\big)  \N - (\dot q -g \dot h) \N  \\
          \nonumber  & &  +\big( 2 \eps S^\varphi v - (q- gh)\big)\dot  \N,      
        \end{eqnarray*}
     
     \end{lem}
    As a consequence of this lemma, if $(v,q, \varphi)$ solves \eqref{NSv} we get that
         \begin{align}
     \label{NSal1}
       D\mathcal{N}(v, q, \varphi) \cdot(\dot v, \dot q, \dot \varphi) 
     = \big( \partial_{t}^\varphi +( v \cdot \nabla^\varphi)- \eps \Delta^\varphi \big)\big(\dot v- \D_{z} v \, \dot \varphi \big)+ \nabla^\varphi\big( \dot q -   \D_{z} v \, \dot \varphi\big) \\ \nonumber+   \big( \dot v \cdot \nabla^\varphi\big) v - \dot \varphi  \, (\D_{z} v \cdot \nabla^\varphi) v
      \end{align}  
      and
      \beq
      \label{NSal2}
       \nabla^\varphi \cdot \big(  \dot v- \D_{z} v \, \dot \varphi \big) = 0.
       \eeq
       The main consequence of these identies is that  even if the naive  form \eqref{naif} of the equation for high order derivatives
        cannot be used,  we can almost use it in the sense that if we replace $Z^\alpha v$ and $Z^\alpha q$ by the corresponding
         "good unknowns" $Z^\alpha v - \partial_{z}^\varphi v Z^\alpha \eta$, $Z^\alpha q - \partial_{z}^\varphi v Z^\alpha \eta$
          in the left hand side, then we indeed get lower order terms in the right hand side of \eqref{naif}.    
    \begin{proof}
    The proof  follows from simple algebraic manipulations.  
    There are many ways to explain this cancellation. It can be seen as a consequence
     of \eqref{comD}.
     
     Let us set
    $$ \mathcal{A}_{i}(v,\varphi)= \D_i  v, \quad \mathcal{F}_{ij}(v,\varphi)= \D_{i} \D_{j} v$$
     for $i=0, \, 1, \, 2, \, 3$.
    We note that for $i=0, \, 1, \, 2$
    $$ D \mathcal{A}_{i}(v,\varphi)\cdot(\dot v,  \dot \varphi)=
      \D_{i} \dot v  -  \D_{z}v\,  \D_{i} \dot \varphi= \D_{i}\big( \dot v - \D_{z}v \, \dot \varphi)
       +\D_{i} \D_{z}v  \, \dot \varphi. $$ 
      Next,  since $\D_{i}$ and $\D_{z}$ commute,  we also have
       $$ \D_{i} \D_{z}v= \D_{z} \D_{i} v = \D_{z}\big( \mathcal{A}_{i}(v, \varphi) \big) $$
       and consequently, we find that
    \beq
    \label{al11}
    D \mathcal{A}_{i}(v,\varphi)\cdot(\dot v,  \dot \varphi) =   
    \D_{i}\big( \dot v - \D_{z}v \, \dot \varphi)  + \dot \varphi \, \D_{z} \big( \mathcal{A}_{i}(v, \varphi) \big).
    \eeq  
      A similar computation shows that this relation is also true for $i=3$.
    In a similar way, we have that for $i=1, \, 2$, $j=1, \, 2$
    \begin{eqnarray*} D \mathcal{F}_{ij}(v, \varphi) \cdot(\dot v, \dot \varphi) & =  & 
     \D_{i}\big( \D_{j} \dot v  -  \D_{z}v\,  \D_{j} \dot \varphi \big) - \D_{i}\dot \varphi\, \D_{z}\big(\D_{j}v \big) \\
       & = &   \D_{ij} \big(  \dot v - \D_{z} v \, \dot \varphi \big) + \dot \varphi \D_{i}\D_{j}\D_{z}v
        \end{eqnarray*}
        and hence we find by using again \eqref{comD} that
     \beq
     \label{al12}
      D \mathcal{F}_{ij}(v,\varphi)\cdot(\dot v,  \dot \varphi) =    \D_{ij} \big(  \dot v - \D_{z} v \, \dot \varphi \big) + \dot \varphi \, \D_{z}\big( \mathcal{F}_{ij}(v, \varphi) \big).
      \eeq
%

 A similar computation shows that this relation also holds true  when  $i= 3$ or $j=3$.
      
      The proof of Lemma \ref{lemal} easily follows by using the two relations
       \eqref{al11}, \eqref{al12}. 
       
    \end{proof}
    
    \subsection{Control from  the dissipation term}
    In view of the integration by parts formula \eqref{ippD}, \eqref{ippS} we  need to  prove that the control of quantities 
      like $\int_{\mathcal{S}} | \nabla^\varphi  f|^2 d \V$
      yield a control  of the standard $H^1$ norm of  $f$. We also 
      need  Korn type inequalities to control the energy dissipation term. This is the aim of this final part of this preliminary section.
      
      \begin{lem}
      \label{mingrad}
     Assume that  $\partial_{z} \varphi \geq c_{0} $ and   $\|\nabla \varphi \|_{L^\infty} \leq 1/c_{0}$ for some $c_{0}>0$,  
      then there exists $\Lambda_{0}= \Lambda({1\over c_{0}})$ such that 
     $$
      \| \nabla f  \|_{L^2(\mathcal{S})}^2 \leq  \Lambda_{0}
      \int_{\mathcal{S}} | \nabla^\varphi f |^2\, d \V. 
     $$  
      \end{lem}
      
      \begin{proof}
      By using the definition of  the operators  $\partial_{i}^\varphi$, 
       we first note that
       $$ |\partial_{z}f| \leq | \partial_{z} \varphi|\,  | \partial_{z}^\varphi f|. $$
       Hence we find
       $$ \| \partial_{z} f\|_{L^2(\mathcal{S})} \leq \Lambda_{0} \| \partial_{z}^\varphi f\|_{L^2(\mathcal{S})}$$
       and since $ d \V= \partial_{z} \varphi dx \geq c_{0} dx$ by assumption, this yields
       $$ \| \partial_{z} f\|_{L^2(\mathcal{S})}^2 \leq \Lambda_{0}  \int_{\mathcal{S}} |\partial_{z}^\varphi f |^2 d\V.$$
       Next, since for $i=1,\, 2$, we have
       $$ | \partial_{i} f| \leq | \partial_{i}^\varphi f| + |\partial_{i} \varphi| \, |\partial_{z}^\varphi f| \leq \Lambda_{0} | \nabla^\varphi f|, $$
       we also obtain that
       $$  \| \partial_{i} f\|_{L^2(\mathcal{S})}^2 \leq  \Lambda_{0}  \int_{\mathcal{S}} |\nabla^\varphi f |^2 d\V, \quad i=1, \, 2.$$
        This ends the proof of Lemma \ref{mingrad}.   
      \end{proof}
      
      In the next proposition we state an adapted  version of the classical Korn inequality in $\mathcal{S}$:
      
      \begin{prop}
      \label{Korn}
      If $\partial_{z} \varphi \geq  c_{0} $,   $\|\nabla \varphi \|_{L^\infty} + \|\nabla^2 \varphi\|_{L^\infty} \leq {1 \over c_{0}}$ for some $c_{0}>0,$ 
      then there exists $\Lambda_{0}= \Lambda(1/c_{0})>0$,  such that  for every $v \in H^1(\mathcal{S})$, we have
 \beq
 \label{estKorn}
 \|\nabla v\|_{L^2(\mathcal{S})}^2  \leq \Lambda_{0} \Big( \int_{\mathcal{S}} |S^\varphi v |^2 \, d\V + \| v\|_{L^2(\mathcal{S})}^2 \Big)
   \eeq
  where $$S^\varphi v=  {1 \over 2 }\big({ \nabla^\varphi v + \nabla^\varphi v^t}\big). $$ 
      \end{prop}
      For the sake of completeness, we shall give a proof of this estimate in subsection \ref{sectionKorn}.

    \section{Preliminary estimates  of $\eta$}
    \label{sectionprelimeta}
     In this section, we shall  begin our a priori estimates for  a sufficiently smooth solution
     of  \eqref{eqphi},  \eqref{NSv}. 
      We assume that  $A$  is chosen such that $\partial_{z} \varphi_{0}(y, z) \geq 1$ at the initial time.
      
       We shall work on an interval of time $[0, T^\eps]$ such that
      \beq
      \label{min}
      \partial_{z} \varphi(t,y,z) \geq {c_{0}}, \,  \forall t \in [0, T^\eps]
      \eeq
   for some  $c_{0}>0$. This   ensures that for every $t \in [0, T^\eps]$, $\Phi(t, \cdot)$ is a diffeomorphism.
    We shall first estimate  $\eta$ given by \eqref{eqeta} in terms of $h$ and $v$.
    

    Our first result is: 
    \begin{prop}
    \label{propeta}   We have the following  estimates for  $\eta$ defined by \eqref{eqeta}
    \begin{eqnarray}
    \label{etaharm} & & \forall s \geq 0, \quad  \| \nabla \eta (t)\|_{H^s(\mathcal{S})} \leq C_{s} |h(t)|_{s+{1 \over 2} }, \\
    \label{dtetaharm} & & \forall s \in \mathbb{N}, \quad  \| \nabla \partial_{t} \eta \|_{H^s(\mathcal{S})}
     \leq C_{s}\big( 1+ \|v\|_{L^\infty} + |\nabla h|_{L^\infty} \big)( \|v\|_{E^{s+1}} + |\nabla h|_{s+ {1 \over 2}}\big)
     \end{eqnarray}
     and moreover, we also have the $L^\infty $ estimates  
 \begin{eqnarray}
     \label{etainfty}
   & & \forall s \in \mathbb{N}, \quad   \|  \eta \|_{W^{s, \infty}} \leq C_{s} |h |_{s, \infty },\\ 
     \label{dtetainfty}
   & & \forall s \in \mathbb{N}, \quad   \| \partial_{t} \eta \|_{W^{s, \infty}}
      \leq  C_{s}\big( 1+   |\nabla h|_{s, \infty} \big) \|v\|_{s, \infty } 
     \end{eqnarray}
     where $C_{s}$ depends only on $s$.

    \end{prop}
    Note that the above estimate ensures that  $\eta$ has a  standard Sobolev regularity in $\mathcal{S}$
     and not only a conormal one. This is one of the main advantage  in the choice of  the diffeomorphism
    given by \eqref{diff} and  \eqref{eqphi}, \eqref{eqeta}.
    \begin{proof}
    From the explicit expression  \eqref{eqeta}, we get that
    $$ \int_{-\infty}^0  \big( |\xi|^2\, |\hat{\eta}(\xi, z) |^2  +  |\partial_{z}\hat{\eta}(\xi, z) |^2\, dz \lesssim  | \xi| \,|\hat{h}(\xi)|^2$$
    and hence 
      \eqref{etaharm} follows by using the Bessel identity.

      By using \eqref{etaharm}, we  get that 
    $$  \|\nabla \partial_{t} \eta \|_{s} \lesssim  |\partial_{t}h |_{s+{1 \over 2 }}.$$
    By using \eqref{bordv1} and  \eqref{sobr2}, we  get  by setting  $v^b(t,y)= v(t,y,0)$ that 
    $$    |\partial_{t}h |_{s+{1 \over 2 }} \leq |v^b\cdot \N|_{s+ {1 \over 2}}
     \lesssim  \|v\|_{L^\infty}  |\nabla h|_{s+{1\over 2}} + (1+ |\nabla h|_{L^\infty}) |v^b|_{s+ {1 \over 2 }}$$
     and hence \eqref{dtetaharm} follows by  using the trace inequality \eqref{trace}.
     
     For the $L^\infty$ estimates, we  observe that we can write
     \beq
     \label{etaconv} \eta(y,z)= {1 \over z^2} \psi( {\cdot \over z})\star_{y} h:= \psi_{z} \star_{y} h
     \eeq
     where $\star_{y}$ stands for a convolution in the $y$ variable and $\psi$ is  in  $L^1(\mathbb{R}^2)$. Consequently, 
      from the Young inequality for convolutions, we  get that
      $$ \| \eta \|_{L^\infty} \lesssim |h|_{L^\infty}, \quad  \|\partial_{i} \eta \|_{L^\infty} \lesssim |\nabla h|_{L^\infty}, \quad i=1, \, 2.$$
      For  $\partial_{z} \eta$, we note that
     $$ \partial_{z} \hat \eta= \big(  \xi_{1} \partial_{1} \chi + \xi_{2} \partial_{2} \chi\big)(z \xi)\, \hat h (\xi) =  \nabla \chi\big (\xi z) \cdot \mathcal{F}_{y}( \nabla h)(\xi).$$
 This yields that
 $$ \partial_{z} \eta =  {1 \over z^2} \psi^{(1)}( {\cdot \over z})\star_{y} \nabla  h$$
 where $\psi^{(1)}$ is again an $L^1$ function and hence we obtain that
 $$ \|\partial_{z} \eta \|_{L^\infty} \lesssim | \nabla h |_{L^\infty}.$$
 To get \eqref{etainfty}, the estimates of higher derivatives follow by induction.
 
 To prove \eqref{dtetainfty},  we write thanks to $\eqref{etainfty}$ and \eqref{bord1} that
 $$ \|  \partial_{t} \eta \|_{W^{s, \infty}} \lesssim |v^b \cdot \N|_{s, \infty} \lesssim \|v\|_{s, \infty}\big( 1 +  |\nabla h|_{s, \infty}\big).$$

    \end{proof}
 Next, we shall  study how   the smoothing effect of  the Navier-Stokes equation on the velocity can be used to 
   improve  the regularity of the surface. 
  
Since  in the following we  need to estimate very often expressions like  $f/\partial_{z} \varphi$ where
  $f \in H^s(\mathcal{S})$ or $H^s_{co}(\mathcal{S})$, we shall first state a general Lemma
  \begin{lem}
  For every $m \in \mathbb{N}$, we have
 \beq
 \label{quot}  \big\| {f \over \partial_{z} \varphi} \big\|_{m} \leq \Lambda\big( {1 \over c_{0}},  |h|_{1, \infty} + \|f \|_{L^\infty} \big)\big(|h|_{m+{1 \over 2}} +
     \|f\|_{m} \big) \eeq 
  \end{lem}
  \begin{rem}
 Of course, the above estimate is also valid in standard Sobolev spaces:
 $$ \big\| {f \over \partial_{z} \varphi} \big\|_{H^s} \leq \Lambda\big( {1 \over c_{0}},  |h|_{1, \infty} + \|f \|_{L^\infty} \big)\big( |h|_{
 s+{1 \over 2}} +
     \|f\|_{H^s} \big)$$
     for $s \in \mathbb{R}$, $s \geq {1 \over 2}.$
  \end{rem}
  \begin{proof}
  Since $\partial_{z} \varphi= A + \partial_{z} \eta$, we note that
  $$ {f \over \partial_{z} \varphi}= { f\over A}  -  {f\over A }{  \partial_{z} \eta \over  A+  \partial_{z} \eta } = {f \over A} - {f\over A}\, F(\partial_{z} \eta)$$
  where $F(x)= x/(A+ x)$ is a smooth function which is  bounded together with all  its derivatives on
    $A+x \geq c_{0}>0$ and such that $F(0)=0$.
   Consequently, by using \eqref{gues}, we get that 
   $$ \|F(\partial_{z} \eta) \|_{m} \lesssim \Lambda( {1 \over c_{0}}, \|\nabla \eta \|_{L^\infty}\big)  \| \partial_{z} \eta\|_{m}$$
   and further that
   $$ \|{f \over \partial_{z} \varphi } \|_{m} \lesssim  \|f /A \|_{m}+ \Lambda\big({1 \over c_{0}},  \| \nabla \eta \|_{L^\infty}  +  \|f/A \|_{L^\infty} \big)\big(\| \partial_{z} \eta \|_{m}
    + \|f/A \|_{m} \big).$$
    The result follows by using \eqref{etainfty} and \eqref{etaharm}. 
  \end{proof}

  Next, we study the gain in   the regularity of the surface which is induced  by the gain of regularity on the velocity for
   the Navier-Stokes equation.
    \begin{prop}
    \label{heps}
     For every $m\in \mathbb{N}$,  $\eps \in(0,1)$, we have the estimate 
     $$  \eps\, |h(t)|_{m+{1 \over 2}}^2
      \leq   \eps\, |h_{0}|_{m+{1 \over 2}}^2
       +  \eps \int_{0}^t |  v^b |_{m+ {1\over 2 }}^2 +  \int_{0}^t \Lambda_{1, \infty}
        \big( \|v\|_{m}^2 + \eps\,|h|_{m+{1 \over 2}}^2  \big)\, d\tau$$
    where 
  \begin{equation} \label{Lam1inf}
\Lambda_{1, \infty}=  \Lambda(
      |\nabla  h |_{L^\infty(\mathbb{R}^2)} + \|  v \|_{1,\infty}  \big)
\end{equation}
       and $v^b= v_{/z=0}$.
    \end{prop}


 Note that by  using the trace inequality \eqref{trace} (with $s_{1}=0, \,  s_{2}= 1,\, s= {1/2}$),  we can write that
    $$ \eps \int_{0}^t   |  v^b |_{m+ {1\over 2 }}^2 \leq \eps  \int_{0}^t \| \nabla v \|_{m}^2 + \eps \int_{0}^t \|v\|_{m}^2$$
    hence the right hand side  in the estimate of Proposition \ref{heps} can indeed be absorbed by  an energy dissipation term.
     Nevertheless, since the exact form of the energy dissipation term for our high order estimates will be more complicated, 
      we shall give the exact way to  control this term   later.
    
    \begin{proof}
    By using \eqref{bord1}, we get that
    $$ \partial_{t} \Lambda^{m+{1 \over 2}} h + v_y(t,y,0) \cdot \nabla  \Lambda^{m+{1 \over 2}} h -
     \Lambda^{m+ {1 \over 2 }} v_{3}(t,y,0) + [\Lambda^{m+{1 \over 2}}, v_{y}(t,y,0)]\cdot \nabla h=0.$$
From a standard energy estimate for this transport equation,  we thus obtain
     \begin{align*} {d \over dt } {1 \over 2} \eps \, |h|_{m+{1 \over 2}}^2
     &  \lesssim \eps\,  |\nabla_{y} v^b |_{L^\infty(\mathbb{R}^2)}   |h|_{m+{1 \over 2}}^2
        \\  &  +\eps \big(   |v ^b|_{m+ {1 \over 2}} + | [\Lambda^{m+{1 \over 2}}, v_{y}(t,y,0)]\cdot \nabla h |_{L^2(\mathbb{R}^2)}\big) |h|_{m+{1 \over 2}}
       \end{align*}  
       where $v^b= v(t,y,0)$. 
       By using the commutator estimate \eqref{comr2}, 
           we  have
     $$  | [\Lambda^{m+{1 \over 2}}, v_{y}(t,y,0)]\cdot \nabla h |_{L^2(\mathbb{R}^2)} 
        \lesssim  |\nabla_{y} v^b |_{L^\infty(\mathbb{R}^2)} |h|_{m+{1 \over 2}}
         +   |\nabla   h |_{L^\infty(\mathbb{R}^2)} |v^b|_{m+{1 \over 2}}$$
and hence, we obtain
   \begin{align}
   \label{h1} {d \over dt } {1 \over 2} \eps \, |h|_{m+{1 \over 2}}^2
       &  \lesssim   \eps\,  |\nabla_{y} v^b |_{L^\infty(\mathbb{R}^2)}   |h|_{m+{1 \over 2}}^2  \\    \nonumber & +
       \big( 1  +   |\nabla  h |_{L^\infty(\mathbb{R}^2)}\big) \eps \,  |v ^b|_{m+ {1 \over 2}}
       |h|_{m+{1 \over 2}}.
       \end{align}
       The result follows from the Young inequality
       \beq
       \label{young}
        ab \leq \delta a^p + C_{\delta} b^q, \quad {1\over p } + {1 \over q}= 1, \, a\, b\geq 0,
        \eeq
        with $p=q=2$ and an integration in time.
%
     This ends  the proof of Proposition \ref{heps}.
  \end{proof}
  By combining  Proposition \ref{propeta} and Proposition \ref{heps}, we also  obtain
   that:
 \begin{cor}
 \label{coreta}
  For every $m\in \mathbb{N}$,  $\eps \in(0,1)$, we have
     $$  \eps\, \|\nabla \eta(t)\|_{H^{m}(\mathcal{S})}^2
      \leq   \eps\, C_{m}\, \|h_{0}\|_{m+{1 \over 2}}^2
       + \eps \int_{0}^t |  v |_{m+{1\over 2  }}^2 +  \int_{0}^t \Lambda_{1,\infty}
        \big( \|v\|_{m}^2 + \eps\,\|h\|_{m+{1 \over 2}}^2  \big)\, d\tau$$
    where $\Lambda_{1,\infty} $  is defined in     \eqref{Lam1inf}. 
 \end{cor}

  \section{Basic $L^2$ estimate}
  \label{sectionbasic}
  We  now start the first main part of  our a priori estimates, namely    estimates  for  $Z^m v$ and $Z^m h$.
  The  easiest case is when  $m=0$ as corresponds to the physical energy. 
  \begin{prop}
  \label{basicL2}
  For any smooth solution of \eqref{NSv}, we have the energy identity: 
$$  {d \over dt} \Big(  \int_{\mathcal{S}} |v|^2 \, d\V + g \int_{z=0} |h|^2\, dy  \Big) + 4 \eps  \int_{\mathcal{S}} |S^\varphi v|^2\, d\V=0.$$
  \end{prop}
  
  \begin{proof}
 By using  \eqref{NSv} and the boundary condition \eqref{bord1}, we  get that
 $$ {d \over dt} \int_{\mathcal{S}} |v|^2 \, d\V=  2 \int_{\mathcal{S}}  \nabla^\varphi   \cdot(  2 \eps S^\varphi v - q \,  \mbox{Id}\big)
  \cdot v\, d \V$$
  and hence by using   the integration by parts formula \eqref{ipp2}, we find that
  $$  {d \over dt} \int_{\mathcal{S}} |v|^2 \, d\V + 4 \eps  \int_{\mathcal{S}} |S^\varphi v|^2\, d\V= 2 \int_{\mathcal{S}} q\, \nabla^\varphi \cdot v \, d\V +
  2 \int_{z=0} \big(2 \eps S^\varphi v - q \mbox{Id} \big) \N \cdot v \, dy.$$
  Next, by using successively   \eqref{bord2} and \eqref{bord1}, we observe that
  $$ 2 \int_{z=0} \big(2 \eps S^\varphi v - q \mbox{Id} \big) \N \cdot v \, dy =  - 2 \int_{z=0} g h \, v \cdot N \, dy
   = - \int_{z=0} g {d \over dt } |h|^2 \,dy$$
   and the result follows.  
  \end{proof}
  
  \begin{cor}
  \label{corL2}
   If $\partial_{z} \varphi\geq c_{0}>0$, $|h|_{2, \infty} \leq {1 \over c_{0}}$ for  $t \in [0, T^\eps]$, then we have
   $$ \|v(t)\|^2 +  \eps \int_{0}^t \| \nabla v \|^2 \leq \Lambda({1 \over c_{0}}) \Big(\|v_{0}\|^2+ \int_{0}^t \| v\|^2 \Big), \quad \forall t \in [0, T^\eps].$$
   \end{cor}
   \begin{proof}
   It suffices to  combine Proposition \ref{basicL2} and Propositions  \ref{Korn}, \ref{mingrad}.
   \end{proof}
  
\section{Equations satisfied by $(Z^\alpha v, Z^\alpha h, Z^\alpha q)$}
\label{sectionconorm}
  
  \subsection{A commutator estimate}

  The next step in order to perform higher order conormal  estimates  is to compute  the equation satisfied by $Z^\alpha v$. We thus need
   to commute $Z^\alpha$ with each term in the equation \eqref{NSv}.
   It is thus usefull to establish  the following general   expressions and estimates  for commutators that we shall use many times.
  
  We first notice that  for $i=1, \, 2,\, 3$,  we have for any smooth function $f$
  \beq
  \label{com1}
   Z^\alpha \D_{i} f=  \D_{i} Z^\alpha f  -  \D_{z}f \D_{i} Z^\alpha \eta+
       \mathcal{C}^\alpha_{i}(f)   \eeq
   where the commutator $\mathcal{C}^\alpha_{i}(f)$ is given for $\alpha \neq 0$ and $i \neq 3$ by 
 \beq
 \label{Cialpha}
  \mathcal{C}^\alpha_{i}(f)= \mathcal{C}^\alpha_{i, 1}(f)+ \mathcal{C}^\alpha_{i,2}(f)+ 
   \mathcal{C}^\alpha_{i, 3}(f)
   \eeq
   with
  \begin{align*}
  \mathcal{C}^{\alpha}_{i,1}& = -\big[ Z^\alpha, {\partial_{i} \varphi \over \partial_{z} \varphi }, \partial_{z} f  \big], \\
     \mathcal{C}^{\alpha}_{i,2}& =   - \partial_{z} f \big[ Z^\alpha, \partial_{i} \varphi, {1 \over \partial_{z} \varphi}\big]     - \partial_{i} \varphi \Big( Z^\alpha\big( {1 \over \partial_{z} \varphi}\big) + {Z^\alpha \partial_{z}
       \eta \over (\partial_{z} \varphi)^2 } \Big) \partial_{z} f, \\
        \mathcal{C}^\alpha_{i, 3}& =  - {\partial_{i} \varphi \over \partial_{z} \varphi} [Z^\alpha, \partial_{z}]f
         + {\partial_{i} \varphi \over (\partial_{z} \varphi)^2}\, \partial_{z} f\,  [Z^\alpha, \partial_{z}] \eta.
     \end{align*}
     Note that for $i=1, \, 2$ we have $\partial_{i}\varphi= \partial_{i}\eta$ and that for $\alpha \neq 0$, 
     $Z^\alpha \partial_{z}\varphi= Z^\alpha \partial_{z}\eta$. This is why we  have replaced  $\varphi$ by $\eta$ in 
       some terms of the above expressions.         
      
        For $i=3$,  we have the same kind of  decomposition  for
      the commutator  (basically, it suffices to replace $\partial_{i}\varphi$ by  $1$ in the above expressions).

  The commutators $\mathcal{C}^\alpha_{i}(f)$ enjoy the following estimate
  \begin{lem}
  \label{comi}
    For $1 \leq | \alpha |\leq m$, $i=  1, \, 2, \, 3$,  we have 
  \begin{align}
  \label{comiest}  \| \mathcal{C}_{i}^\alpha(f) \|   &  \leq \Lambda\big( {1 \over c_{0}}, |h|_{2, \infty } + \|\nabla f \|_{1, \infty}  \big) \big( \|\nabla  f \|_{m-1} +  |h
   |_{m- {1 \over 2 }}
   \big).
   \end{align}
  \end{lem}
  The meaning  of this lemma combined with \eqref{com1} is that by using the notations of the proof of Lemma
  \ref{lemal}, we have that
  \beq
  \label{com2}
   Z^\alpha \D_{i}f = D \mathcal{A}_{i}(f, \varphi) \cdot (Z^\alpha f, Z^\alpha \eta) + \mathcal{C}_{i}^\alpha
   \eeq
   where the commutator $\mathcal{C}_{i}^\alpha$ is of lower order in both $f$ and $\eta$.
  \subsubsection*{Proof of Lemma \ref{comi}}
  We give the proof for $i=1, \, 2$, the last case being similar and slightly easier.
  
  We shall first estimate $\mathcal{C}_{i,1}^\alpha$.
  Thanks to \eqref{comsym}, we have
    $$  \|  \mathcal{C}^\alpha_{i,1} \|_{L^2} \leq  \| Z\big( {\partial_{i }  \varphi \over \partial_{z} \varphi } \big) \|_{L^\infty}
     \, \|\partial_{z} f \|_{m-1} +   \| Z \partial_{z} f \|_{L^\infty} \, \|
  {\partial_{i }  \varphi \over \partial_{z} \varphi }  \|_{m-1}.$$
    Consequently, by using \eqref{quot}, we first  get that
  \begin{align*}   \| \mathcal{C}^\alpha_{i,1} \|_{L^2} \leq  \Lambda({1\over c_{0}}, 
     \|\nabla \varphi \|_{1, \infty} + |h|_{1, \infty} + \|Z \partial_{z} f \|_{L^\infty}\big)
     \big( \|\partial_{i}\eta \|_{m-1} +|h|_{m-{1 \over 2 }} +  \|\partial_{z} f \|_{m-1}\big),
     \end{align*}
     next, by using  \eqref{eqphi}, we obtain  $$
    \| \mathcal{C}_{i, 1}^\alpha(f) \| \leq \Lambda\big( {1 \over c_{0}}, \|\nabla  \eta \|_{1, \infty}+ |h|_{1, \infty} + \|\nabla f \|_{1, \infty}  \big) \big( \|\nabla  f \|_{m-1} +  |h|_{m-{1\over 2 }} + \|\nabla \eta \|_{m-1}\big)
 $$
  and finally, by using \eqref{etaharm}, \eqref{etainfty}, we get
  \beq
  \label{Calpha1}
    \| \mathcal{C}_{i, 1}^\alpha(f) \| \leq \Lambda\big( {1 \over c_{0}}, |h|_{2, \infty} + \|\nabla f \|_{1, \infty}  \big) \big( \|\nabla  f \|_{m-1} +  |h|_{m-{1\over 2 }} \big).
    \eeq
  To estimate the first term in  $\mathcal{C}_{i,2}^\alpha$,  we use the same  kind of arguments: by using \eqref{comsym}
   and \eqref{eqphi}, we first get that
  $$
 \| \partial_{z} f \big[ Z^\alpha, \partial_{i} \varphi, {1 \over \partial_{z} \varphi}\big] \|
  \leq \Lambda ({1 \over c_{0}}, \|\partial_{z} f \|_{L^\infty} +  \| \nabla \eta \|_{1, \infty} \big)\big( \| \nabla \eta\|_{m-1} + \|{ Z \partial_{z}\eta \over (\partial_{z} \varphi)^2}\|_{m-2} \big)$$
   and hence by using \eqref{quot} and \eqref{etaharm}, \eqref{etainfty}, we find
  $$ \| \partial_{z} f \big[ Z^\alpha, \partial_{i} \varphi, {1 \over \partial_{z} \varphi}\big] \|
  \leq \Lambda ({1 \over c_{0}}, \| \nabla  f \|_{L^\infty} +  | h|_{2, \infty} \big) | h|_{m-{1\over 2}}.$$
  To estimate the second  type of terms in   $\mathcal{C}_{i, 2}^\alpha$, we  note that
   for $|\alpha| \geq 1$, we can write
   \beq
   \label{Z1/dzphi} Z^\alpha \big({ 1 \over \partial_{z} \varphi} \big) = - Z^{\tilde \alpha} \big({ Z_{j} \partial_{z} \eta \over (\partial
   _{z } \varphi)^2 } \big), \quad |\tilde \alpha | = |\alpha |-1,\eeq
  hence, we obtain for $|\alpha | \geq 2 $  that 
 $$  \partial_{i} \varphi \Big( Z^\alpha\big( {1 \over \partial_{z} \varphi}\big) + {Z^\alpha \partial_{z}
       \eta \over (\partial_{z} \varphi)^2 } \Big) \partial_{z} f
       = - \partial_{i}\varphi\, \partial_{z} f\,  [  Z^{\tilde \alpha }, {1 \over (\partial_{z} \varphi)^2} ] Z_{j}\partial_{z}\eta $$
   and by using again \eqref{com}, \eqref{quot} and \eqref{etaharm}, \eqref{etainfty}, we also  obtain that
             \beq
   \label{Calpha2}
    \| \mathcal{C}_{i, 2}^\alpha(f) \| \leq \Lambda\big( {1 \over c_{0}}, | h |_{2, \infty} + \|\nabla f \|_{L^\infty}  \big)  | h  |_{m-{1 \over 2 } }.
  \eeq
  
  It remains to estimate $\mathcal{C}_{i, 3}^\alpha$.
    Since  
    $$[Z_{3}, \partial_{z}]= - \partial_{z} \big( { z \over  1+ z } \big) \partial_{z}, $$
    we can prove by induction that
    \beq
    \label{idcom} [Z^\alpha, \partial_{z}] h= \sum_{| \beta | \leq m-1} c_{\beta } \partial_{z} ( Z^\beta h)\eeq
     for some harmless smooth bounded functions $c_{\beta}.$
     This yields  
        \begin{align*} \big\| {\partial_{i}\varphi \over \partial_{z} \varphi} [Z^\alpha, \partial_{z}] f \big\|_{L^2} 
    + \big\|   {\partial_{i} \varphi \over (\partial_{z}
       \varphi)^2}  [Z^\alpha, \partial_{z} ] \varphi \, \partial_{z} f  \big\|_{L^2}
        \leq \Lambda( {1 \over  c_{0}}, \|\partial_{i} \varphi \|_{L^\infty} + \|\partial_{z} f \|_{L^\infty}
         )\big( \|\partial_{z} f \|_{m-1}  \\+ \| \partial_{i} \varphi \| +  \|\partial_{z} \eta \|_{m-1} \big)
         \end{align*}
       Indeed, the last term comes from the fact that  thanks to \eqref{eqphi},
        the commutator $ [Z^\alpha,  \partial_{z}]\varphi$ can be decomposed into an harmless bounded
         term and  the commutator $ [Z^\alpha,  \partial_{z}]\eta $ which is in $L^2$.
       This yields  by using again \eqref{etaharm}
    \begin{align}
    \nonumber
    \| \mathcal{C}^\alpha_{i, 3} \| &  \leq 
    \Lambda( {1 \over  c_{0}}, \|\partial_{i} \varphi \|_{L^\infty} + \|\partial_{z} f \|_{L^\infty}
         )\big( \|\partial_{z} f \|_{m-1}  + \| \nabla  \eta \|_{m-1} \big) \\ 
    \label{Calpha3}     &   \leq   \Lambda( {1 \over  c_{0}}, | h |_{1, \infty} + \|\nabla  f \|_{L^\infty}
         )\big( \|\partial_{z} f \|_{m-1}  + |h|_{m-{1 \over 2 }} \big).
   \end{align}
  
  To end the proof of Lemma \ref{comi}, it suffices to collect \eqref{Calpha1}, \eqref{Calpha2}, 
   \eqref{Calpha3}.
  
  The estimate for $\mathcal{C}_{3}^\alpha$ can be obtained through very similar arguments. 
  This ends the proof of Lemma \ref{comi}.

  \subsection{Interior equation satisfied by $(Z^\alpha v, Z^\alpha q, Z^\alpha \varphi)$} 
  We shall prove the following:
  \begin{lem}
  \label{lemValpha}
  For $1 \leq | \alpha | \leq m$, let us set $V^\alpha = Z^\alpha v - \D_{z}v\, Z^\alpha \eta, $ $Q^\alpha = Z^\alpha q - \D_{z}q\, Z^\alpha \eta,$
   then we get the system
   \begin{eqnarray}
  \label{eqValpha} & &  \D_{t} V^\alpha + v \cdot \nabla^\varphi V^\alpha + \nabla^\varphi Q^\alpha- 2 \eps \nabla^\varphi \cdot S^\varphi V^\alpha
    + \mathcal{C}^\alpha(q)
       + \mathcal{C}^\alpha(\mathcal{T})  \\
    \nonumber   & & \quad \quad \quad \quad \quad  \quad \quad  =  \eps  \mathcal{D^\alpha}\big( S^\varphi v \big) + \eps  \nabla^\varphi \cdot \big( \mathcal{E}^\alpha (v)\big) +(\D_{z}v  \cdot \nabla^\varphi v)
        Z^\alpha \eta
       , \\
   \label{divValpha}
     & &    \nabla^\varphi \cdot V^\alpha + \mathcal{C}^\alpha (d)= 0.
       \end{eqnarray}
       where the commutators $\mathcal{C}^\alpha(q)$, $\mathcal{C}^\alpha(d)$ and $\mathcal{E}^\alpha(v)$ satisfy the estimates:
       \begin{eqnarray}
       \label{Cq} & & \|\mathcal{C}^\alpha (q) \| \leq \Lambda\big( { 1\over c_{0}},  |h|_{2, \infty} + \| \nabla q \|_{1, \infty} \big) \big(
        \|\nabla q \|_{m-1}+ |h |_{m-{1\over 2}} \big), \\
        \label{Cd}& &  \|\mathcal{C}^\alpha (d) \| \leq \Lambda\big( { 1\over c_{0}},  |h|_{2, \infty} + \| \nabla v \|_{1, \infty} \big) \big(
        \|\nabla v \|_{m-1}+ |h |_{m-{1\over 2}} \big), \\
  \label{CT} & &\| \mathcal{C}^\alpha (\mathcal{T}) \| + \|\mathcal{E}^\alpha(v) \| \leq \Lambda \big( {1 \over c_{0}}, |h|_{2, \infty} +  \|v\|_{E^{2, \infty}}
   \big)  \big( \|v\|_{E^m}+ |h|_{m-{1\over 2} }\big)
 \end{eqnarray}
and $\mathcal{D}^\alpha (S^\varphi v ) $ is given by
$$\mathcal{D}^\alpha\big( S^\varphi v \big)_{ij}= 2 \,  \mathcal{C}_{j}^\alpha \big( S^\varphi v)_{ij}. $$

  \end{lem}
  Note that we will control the commutator $\mathcal{D}^\alpha(S^\varphi v)$ later by using integration by parts.
  \begin{proof}
   We need to compute 
   the equation solved by $Z^\alpha v$.
  
  By using the previous notations, we   first note  that
  \beq
  \label{comp1}
   Z^\alpha \nabla^\varphi q= \nabla^\varphi Z^\alpha q  -\partial_{z}^\varphi q\,  \nabla^\varphi Z^\alpha \varphi
   + \mathcal{C}^\alpha(q)
   \eeq
   where
   $$ \mathcal{C}^\alpha(q)=  \left(\begin{array}{lll} \mathcal{C}_{1}^\alpha(q) \\ \mathcal{C}_{2}^\alpha (q) \\
   \mathcal{C}_{3}^\alpha(q) \end{array}\right)$$
    and thus thanks to lemma \ref{comi}, we have the estimate
   \beq
   \label{comp2}
   \| \mathcal{C}^\alpha(q) \| \leq \Lambda\big( {1 \over c_{0}}, |h|_{2,\infty } +  \|\nabla  q \|_{1, \infty} \big) \big( \|\nabla  q \|_{m-1} 
   +  | h |_{m-{1 \over 2 } }\big).
   \eeq  
  
  In  a similar way,  we get that
\beq
\label{div1}
Z^\alpha \big( \nabla^\varphi \cdot v \big)= \nabla^\varphi \cdot Z^\alpha v - \nabla^\varphi Z^\alpha \varphi \cdot
 \partial_{z}^\varphi v + \mathcal{C}^\alpha(d)
\eeq 
where
$$
\mathcal{C}^\alpha(d) =  \sum_{i=1}^3 \mathcal{C}^\alpha_{i}(v)
$$
   and hence, thanks to Lemma \ref{comi}, we also have that
 \beq
   \label{comd2}
   \| \mathcal{C}^\alpha(d) \| \leq \Lambda\big( {1 \over c_{0}}, |h|_{2,\infty } +  \|\nabla  v \|_{1, \infty} \big) \big( \|\nabla  v \|_{m-1} 
   +  |  h  |_{m-{1 \over 2 } }\big).
   \eeq
   
   Next, we shall expand the transport part of \eqref{NSv}.
   $$ Z^\alpha \Big( \D_{t} + v \cdot \nabla^\varphi   \Big)v.$$
   We first  note that 
   \beq
   \label{transportW} \D_{t} + v \cdot \nabla^\varphi
   =  \partial_{t}+ v_{y} \cdot \nabla_{y}v +  V_{z} \partial_{z}\eeq


 where $V_{z}$ is defined by 
   \beq
   \label{Wdef}
    V_{z}= {1\over \partial_{z} \varphi} v_{z}=  {1 \over \partial_{z} \varphi} \big(  v\cdot \N- \partial_{t} \varphi\big)= {1\over \partial_{z} \varphi}\big(v\cdot\N- \partial_{t} \eta\big).
   \eeq
and    where $\N(t,y,z)$ is defined as
   $$ \N(t,y,z)=\big(-\partial_{1} \varphi(t,y,z), -\partial_{2}\varphi(t,y,z), 1\big)^t
   =\big(-\partial_{1} \eta(t,y,z), -\partial_{2}\eta(t,y,z), 1\big)^t.$$
   Note that  $\N$ is defined in the whole $\mathcal{S}$ and that $\N(t,y,0)$ is indeed the outward normal to the boundary
    that we have used before.
   For the vector fields $v_{z}$ and  $V_{z}$, we have  the estimates:
   \begin{lem}
   \label{lemestW}
   We have  for $m \geq 2$, the estimates 
   \begin{eqnarray}
    & & \label{vzinfty} \|v_{z}\|_{1, \infty} +   \|V_{z}\|_{1, \infty}\leq \Lambda\big({1\over c_{0}}, \|v\|_{1, \infty} + |h|_{2, \infty}\big), \\
      & & \label{Zvzm-2} \|Zv_{z}\|_{m-2} +  \|ZV_{z}\|_{m-2} \leq  \Lambda\big(  {1 \over c_{0}}, \|v\|_{1, \infty} +  |h|_{2, \infty}   \big) \big(\|v\|_{E^{m-1}}  + |h|_{m-{1\over 2 } }\big).
     \end{eqnarray} 
    \end{lem}
   {\bf Proof of Lemma \ref{lemestW}}
   
   For the first estimate, we use that
   $$ \|v_{z}\|_{L^\infty} + \|V_{z}\|_{L^\infty} \leq \Lambda\big({1\over c_{0}} , \|v\|_{L^\infty}+ \| \nabla \eta \|_{L^\infty} +\| \partial_{t} \eta \|_{L^\infty} \big)
    \leq \ \Lambda\big({1\over c_{0}} , \|v\|_{L^\infty} +|h|_{1, \infty}\big)$$
    thanks to \eqref{etainfty}, \eqref{dtetainfty}.
    In the same way, we get that
    $$ \|Zv_{z}\|_{L^\infty} +   \|Z V_{z} \|_{L^\infty} \leq \Lambda({1 \over c_{0}},  \|v \|_{1, \infty} +  \|\nabla \eta \|_{W^{1, \infty}} + \|  \partial_{t} \eta \|_{W^{1, \infty}} \big)
       \leq \Lambda({1 \over c_{0}},  \|v \|_{1, \infty} +|h|_{2, \infty}  \big)$$
     by using also \eqref{dtetainfty}.
     
     By using \eqref{gues} and  Proposition \ref{propeta},  we also obtain that
     \begin{eqnarray}
    \nonumber
      \|Z v_{z}\|_{m-2} &  \lesssim &  \|\nabla \partial_{t} \eta \|_{m-2}+ \|v\|_{L^\infty}( 1+ \|\nabla \eta \|_{m-1}) + \|\nabla \eta \|_{L^\infty} \|v\|_{m-1}\\ 
    \label{Zpetitvz}  & \lesssim & 
       \Lambda\big(  \|v\|_{L^\infty} +  |h|_{1, \infty}   \big) \big(\|v\|_{E^{m-1}} + |h|_{m-{1\over 2 } }\big).\end{eqnarray}
       Note that $\partial_{t} \eta$ is not in $L^2$, this why we are obliged to apply one  derivative to $v_{z}$ before estimating
        it in $L^2$ in order to use Proposition \ref{propeta}.
      Finally, we note that we have 
      $$\| Z V_{z}\|_{m-2} \lesssim \|Z\big( {1\over \partial_{z} \varphi} \big)  v_{z}\|_{m-2} +  \|{1\over \partial_{z} \varphi} Zv_{z} \|_{m-2}$$
      and hence, by using \eqref{gues}, \eqref{quot}, we obtain that
      \beq
      \label{Zvzm-2prov}
       \| Z V_{z} \|_{m-2} \leq \Lambda\big( {1 \over c_{0}},  \|\nabla \eta \|_{1, \infty} +  \|v_{z}\|_{1, \infty} \big) \big(
       \| Zv_{z}\|_{m-2}+ \|\partial_{z} \eta \|_{m-2}\big).\eeq
%
        Finally, \eqref{Zvzm-2} follows  from \eqref{Zvzm-2prov} by using \eqref{Zpetitvz}, \eqref{vzinfty}  and  Proposition \ref{propeta}.
       
    This ends the proof of Lemma \ref{lemestW}.  
  \begin{rem}
  \label{remdzVz}
  By using similar arguments, we also get that for $m \geq 3$, we have
  $$ \|\partial_{z} Z V_{z}\|_{m-3}  \leq   \Lambda\big(  \|v\|_{E^{1, \infty}} +  |h|_{2, \infty}   \big) \big(\|v\|_{E^{m-1}}  + |h|_{m-{1\over 2 } }\big)$$
   and
   $$ \|\partial_{z} V_{z}\|_{m-3} \leq \Lambda\big(  \|v\|_{E^{1, \infty}} +  |h|_{2, \infty}   \big) \big(\|v\|_{E^{m-2}}  + |h|_{m-{3\over 2 } }\big).$$ 
  \end{rem}  
   \bigskip
 
  By using the  identity \eqref{transportW}, we thus get that
\begin{eqnarray}\nonumber  & & \!\!\!\! Z^\alpha \big( \D_{t} + v \cdot \nabla^\varphi   \big)v \\
\nonumber& &    =   \big(\partial_{t}+ v_{y} \cdot \nabla_{y} +  V_{z} \partial_{z} \big) Z^\alpha v + \big(v\cdot Z^\alpha \N- 
\partial_{t} Z^\alpha \eta \big) \D_{z} v - \D_{z} Z^\alpha \eta \big( v \cdot \N - \partial_{t} \eta\big)
 \D_{z}v  
+ \mathcal{C}^\alpha(\mathcal{T}) \\
\label{T1}& & =  \big( \D_{t}+ v \cdot \nabla^\varphi\big)Z^\alpha v - \D_{z}v \big( \D_{t}+ v \cdot \nabla^\varphi) Z^\alpha \eta
 +  \mathcal{C}^\alpha(\mathcal{T}) 
 \end{eqnarray}
 where the commutator $\mathcal{C}^\alpha(\mathcal{T})$ is defined by 
 $$ \mathcal{C}^\alpha(\mathcal{T})= \sum_{i=1}^5 \mathcal{T}_{i}^\alpha,$$
 \begin{eqnarray*}
 & & \mathcal{T}_{1}^\alpha=  [Z^\alpha, v_{y}]\partial_{y} v, \quad  \mathcal{T}_{2}^\alpha = [Z^\alpha, V_{z},\partial_{z}v], 
 \quad  \mathcal{T}_{3}^\alpha= {1 \over \partial_{z}\varphi} [Z^\alpha, v_{z}]\partial_{z}v,\\
     & & \mathcal{T}_{4}^\alpha =\Big( Z^{\alpha}\big( {1 \over \partial_{z} \varphi} \big) + { \partial_{z}\,  Z^{\alpha} \eta \over
       (\partial_{z} \varphi)^2} \Big) v_{z} \partial_{z} v, 
      \quad  \mathcal{T}_{5}^\alpha= v_{z}\partial_{z}v { [Z^\alpha, \partial_{z}]\eta \over (\partial_{z} \varphi)^2 }
       + V_{z} [Z^\alpha, \partial_{z}]v,\\
      & & \mathcal{T}_{6}^\alpha =
   [Z^\alpha, v_{z}, {1 \over \partial_{z} \varphi}] \partial_{z} v.
 \end{eqnarray*}
To estimate these commutators, we use similar arguments as in the proof of Lemma \ref{comi}. 
In particular,  we use the estimates \eqref{com}, \eqref{comsym} and \eqref{quot} combined with Lemma \ref{lemestW}
and also again the estimates \eqref{etaharm}, \eqref{etainfty}. This yields
    \beq 
    \label{CalphaTproof}
     \| \mathcal{C}^\alpha(\mathcal{T}) \| \leq 
       \Lambda\big( {1 \over c_{0}}, |h|_{2, \infty}+ \|v\|_{E^{2, \infty}}\big)\big( \|v\|_{E^{m}} + |h|_{m-{1\over 2}}\big)
     \eeq
      and hence the estimate \eqref{CT} for $\mathcal{C}^\alpha(\mathcal{T})$ is proven.
    \bigskip

  It remains to  compute $ \eps Z^\alpha \Delta^\varphi v$.  Since $\nabla^\varphi \cdot v = 0$, it is more  convenient to use that
   $Z^\alpha \Delta^\varphi v=  2 Z^\alpha \nabla^\varphi \cdot  (S^\varphi v).$  By using \eqref{com1}, we thus use the following expansion
   $$ 2  Z^\alpha \nabla^\varphi \cdot  (S^\varphi v) =  2 \nabla^\varphi \cdot \big( Z^\alpha \, S^\varphi v \big) - 2 \big( \D_{z}\, S^\varphi v \big) \nabla^\varphi
     ( Z^\alpha \varphi) + \mathcal{D^\alpha}\big( S^\varphi v \big)$$
    where 
    \beq
    \label{Ddef} \mathcal{D}^\alpha\big( S^\varphi v \big)_{i}= 2 \,  \mathcal{C}_{j}^\alpha \big( S^\varphi v)_{ij}\eeq
    by using the summation convention over repeated indices.
    Next, since we can again expand
    \beq
    \label{comS} 2 Z^\alpha  \big( S^\varphi v \big) =  2 S^\varphi \big( Z^\alpha v \big) - \D_{z} v \otimes \nabla^\varphi Z^\alpha \varphi 
     - \nabla^\varphi Z^\alpha \phi 
  \otimes \D_{z} v  + \mathcal{E}^\alpha(v)\eeq
     with
     $$\big( \mathcal{E}^\alpha v\big)_{ij}= \mathcal{C}^\alpha_{i}(v_{j})+ \mathcal{C}^\alpha_{j}(v_{i}), $$  
     we obtain that
     \begin{eqnarray}
     \label{deltaexp}
      & &\eps Z^\alpha \Delta^\varphi v   =  \\ 
    \nonumber    & & 2 \, \eps\, \nabla^\varphi \cdot  S^\varphi(Z^\alpha v )  - 2 \eps \nabla^\varphi \cdot \Big(  \D_{z} v \otimes \nabla^\varphi Z^\alpha \varphi 
     - \nabla^\varphi Z^\alpha { \mathbf \varphi }  \otimes \D_{z} v \Big) - 
      2  \eps \big( \D_{z}\, S^\varphi v \big) \nabla^\varphi
     ( Z^\alpha \varphi) \\
  \nonumber     
       & & +   \eps  \mathcal{D^\alpha}\big( S^\varphi v \big) + \eps  \nabla^\varphi \cdot \big( \mathcal{E}^\alpha v\big).\end{eqnarray}

       Thanks to Lemma \ref{comi}, we also have the estimate
       \beq
       \label{Ealphap}\|\mathcal{E}^\alpha(v) \| \leq \Lambda \big( {1 \over c_{0}}, |h|_{2, \infty} +  \|\nabla v\|_{1, \infty}
   \big)\big( \|v\|_{m}+ \|\partial_{z} v \|_{m-1}+ |h|_{m-{1\over 2} }\big).\eeq
       Consequently, by collecting
       \eqref{comp1}, \eqref{div1}, \eqref{T1} and \eqref{deltaexp}, we get  in view of \eqref{full1}, \eqref{full3} that $(Z^\alpha v, Z^\alpha q, Z^\alpha \varphi)$ solves
        the equation
      $$ D\mathcal{N}(v,q,\varphi) \cdot\big( Z^\alpha v, Z^\alpha q, Z^\alpha \eta\big) - \big(Z^\alpha v \cdot \nabla^\varphi \big)v + \mathcal{C}^\alpha(q)
       + \mathcal{C}^\alpha(\mathcal{T})=  \eps  \mathcal{D^\alpha}\big( S^\varphi v \big) + \eps  \nabla^\varphi \cdot \big( \mathcal{E}^\alpha v\big)$$
       and the constraint
       $$ D d( v,  \varphi) \cdot \big(Z^\alpha v, Z^\alpha \eta\big) + \mathcal{C}^\alpha(d)= 0.$$
  Consequently, by using Lemma \ref{lemal} and \eqref{NSal1}, \eqref{NSal2},  we get  that $(V^\alpha, Q^\alpha)$ with $V^\alpha = Z^\alpha v - \D_{z} v\, Z^\alpha \eta$, $Q^\alpha= Z^\alpha q- \D_{z}q Z^\alpha \eta$
   solves
  \begin{align*} \D_{t} V^\alpha + v \cdot \nabla^\varphi V^\alpha + \nabla^\varphi Q^\alpha- 2 \eps \nabla^\varphi \cdot S^\varphi V^\alpha
    + \mathcal{C}^\alpha(q)
       + \mathcal{C}^\alpha(\mathcal{T}) \quad \quad \quad \quad  \\
        =  \eps  \mathcal{D^\alpha}\big( S^\varphi v \big) + \eps  \nabla^\varphi \cdot \big( \mathcal{E}^\alpha v\big) + \big(\partial_{z}^\varphi v  \cdot
         \nabla^\varphi v\big)Z^\alpha \eta
       \end{align*}
       with
       $$ \nabla^\varphi \cdot V^\alpha + \mathcal{C}^\alpha (d)= 0.$$ 
       Since we have proven the estimates \eqref{Ealphap}, \eqref{comp2}, \eqref{comd2}, this ends the proof of Lemma
        \ref{lemValpha}.
     
  \end{proof}   
  
 \subsection{Estimates of the boundary values}
 
 We shall also need to compute the boundary condition satisfied   by $Z^\alpha v$ when $\alpha_{3} = 0$
 (for $\alpha_{3}\neq 0$, we have $Z^\alpha v= 0$ on the boundary). As a preliminary, in order to control
  the commutators, we first establish the following
 
 \begin{lem}
 \label{lembord}
 For every $s \geq 0$, $s \in \mathbb{R}$,  we have the following estimates at  $z=0$:
 \begin{eqnarray}
 \label{dzvb}
  | \nabla v(\cdot, 0) |_{s} &  \leq  &  \Lambda\big( \|v\|_{{1, \infty}}  +  |h|_{2, \infty}  \big) \big(  |v(\cdot, 0)|_{s+1}
   +|h|_{s+1}\big). 
   \end{eqnarray} 
 \end{lem}
 \begin{proof}

 Note  that the estimate is obvious for $|\partial_{i} v(\cdot, 0)|_{s}$, $i=1, \,2$. Consequently, the only difficulty
  is to estimate $|\partial_{z} v(\cdot, 0)|_{s}$.
  Since $\nabla^\varphi \cdot  v=0$, we get that
 $$ \partial_{z} v \cdot  \N= \partial_{z} \varphi \big(\partial_{1} v_{1}+ \partial_{2} v_{2})= \big( A + \partial_{z} \eta \big) \big( \partial_{1} v_{1}
 + \partial_{2} v_{2} \big)$$
 and hence  that
 \beq
 \label{eqdzvn} \partial_{z}v \cdot \n={1 \over   |\N| }  \big( A + \partial_{z} \eta \big) \big( \partial_{1} v_{1} + \partial_{2} v_{2}\big), \quad |\N|= \big( 1 + (\partial_{1} \eta^2) + (\partial_{2} \eta)^2 \big)^{1 \over 2} .\eeq
 Consequently, by using \eqref{sobr2} and \eqref{quot} (or actually its version on the boundary), we get that
 $$ | \partial_{z}v \cdot \n |_{s} \leq \Lambda\big( \|v\|_{1, \infty} +  |h|_{{2, \infty}(\mathbb{R}^2)} + \|\partial_{z} \eta \|_{L^\infty} \big)
 \big( |v(\cdot, 0)|_{s+1}+  |Z\N|_{s-1}  + |\partial_{z} \eta(\cdot, 0) |_{s}\big) $$
and hence, by using Proposition \ref{propeta} and  the  trace estimate \eqref{trace} which yields
$$ |\partial_{z} \eta(\cdot, 0) |_{s}\lesssim \| \partial_{z} \eta  \|_{H^{ s +{1\over 2}}(\mathcal{S})} \lesssim |h|_{s+1},$$ 
we  get 
\beq
\label{dzvnb1}
 | \partial_{z}v \cdot \n |_{s} \leq \Lambda\big( \|v\|_{1, \infty} +  |h|_{2, \infty}  \big)
 \big( |v(\cdot, 0)|_{s+1}+  |h|_{s+1}\big). 
\eeq
It remains to control $\Pi\big( \partial_{z}v\big)$ where $\Pi= Id - \n\otimes \n$ is the orthogonal projection  on $(\n)^\perp$
 which is the tangent space to the boundary.
 We shall use the boundary condition \eqref{bord2} which  yields
 \beq
 \label{bordpi} \Pi\big( S^\varphi v \, \n) =0.\eeq
 To expand $\Pi (S^\varphi v\, \n\big)$, we shall use  the local basis in $\Omega_{t}$ induced by  \eqref{diff} that we denote
  by $(\partial_{y^1}, \partial_{y^2}, \partial_{y^3})$.
 Note that by definition, we have by using \eqref{vdef}
  that $(\partial_{y^i}u)(t,  \Phi(t, \cdot))= \partial_{i}v.$
   We also consider the induced riemannian metric defined by $g_{ij}= \partial_{y^i}\cdot \partial_{y^j}$
   and we define as usual the metric $g^{ij}$ as the inverse of $g_{ij}$.
    Then we obtain that
   \beq
   \label{Su}  2 Su \, \n= \partial_{\n}u + g^{ij} \big(\partial_{y^j} u \cdot \n\big) \partial_{y^i}.\eeq
   Consequently, since
    \begin{align*}
    \partial_{\N}u  & =   - \partial_{1}  \varphi \partial_{1} u -\partial_{2} \varphi \partial_{2}u + \partial_{3}u \\
    &  =  - \partial_{1}  \varphi \D_{1} v -\partial_{2} \varphi \D_{2}v + \D_{3}v \\    & = {1\over \partial_{z} \varphi} \big( 1+ | \nabla h|^2 \big)\partial_{z}v
     -\partial_{1}\varphi \partial_{1}v -\partial_{2}\varphi \partial_{2}v,
     \end{align*}
    we get from \eqref{bordpi} that
   \beq
   \label{eqdzvPi} \Pi \partial_{z}v=  {\partial_{z} \varphi \over 1 + | \nabla h |^2}\Big( \Pi \big(  \partial_{1} h \partial_{1} v + \partial_{2} h \partial_{2} v \big)  - \big(
     g^{ij} \big( \partial_{j}v \cdot \n \big) \Pi \partial_{y^i}\Big).
     \eeq
     Consequently, by using the same product estimates as above, we find that
   \begin{align*}
    | \Pi \partial_{z} v|_{s}  & \leq  \Lambda\big( \|v\|_{1, \infty} +  \|\partial_{z}v\cdot \n \|_{L^\infty} +  |h|_{2, \infty}  \big) \big(  |v(\cdot, 0)|_{s+1}
     + |h|_{s+1}+ | \partial_{z}v \cdot \n |_{s} \big) \\
      & \leq  \Lambda\big( \|v\|_{1, \infty} +   |h|_{2, \infty}  \big) \big(  |v(\cdot, 0)|_{s+1} 
     + |h|_{k+1}\big)
     \end{align*}
     where the last estimate comes from \eqref{dzvnb1} and the fact that from \eqref{eqdzvn} we obviously have
     $$ \|\partial_{z}v \cdot \n \|_{L^\infty} \leq \Lambda(  \| \nabla \eta \|_{L^\infty}  + \|v\|_{1, \infty}\big).$$
     
      To conclude, it suffices to combine the last estimate and  \eqref{dzvnb1} since
      $$ | (\partial_{z}v)(\cdot, 0) |_{s} \leq | \Pi \partial_{z}v |_{s} + |\partial_{z}v \cdot \n|_{s}.$$
   \end{proof}  
 
 \bigskip
 
 Next, we shall  obtain the  boundary condition for $Z^\alpha v$. Again, note that the only interesting  case occurs when $\alpha_{3}=0$. Indeed, 
  otherwise $Z^\alpha v=0$, $Z^\alpha \eta= 0$  on the boundary. We start with the study of the boundary condition \eqref{bordv2}.
  \begin{lem}
  \label{lembordV}
  For every $\alpha$, $1 \leq |\alpha | \leq m$ such that $\alpha_{3}= 0$ we have that on $\{z=0\}$
  \beq
  \label{bordV}
  2 \eps S^\varphi V^\alpha\, \N  -\big( Z^\alpha q   - g Z^\alpha h\big)\N  +  \big(2 \eps  S^\varphi v - \big( q - gh) \big) Z^\alpha \N= \mathcal{C}^\alpha(\mathcal{B}) - 2 \eps Z^\alpha h \D_{z}\big( S^\varphi v \big)\N
  \eeq
  where $V^\alpha = Z^\alpha v - \D_{z} v\, Z^\alpha \eta$, $Q^\alpha = Z^\alpha Q- \D_{z}q \, Z^\alpha \eta$
  and the commutator $\mathcal{C}^\alpha(\mathcal{B})$  enjoys the estimate:
  \beq
  \label{CalphaB}
  | \mathcal{C}^\alpha(\mathcal{B}) |_{L^2(\mathbb{R}^2)} \leq    \Lambda\big( {1 \over c_{0}} , | h|_{2, \infty} + 
   \|v\|_{E^{2, \infty}} \big) \big( \eps  |v^b|_{m}+ \eps  | h |_{m} \big)
     \big).
  \eeq

  \end{lem}
 
 \begin{proof}
 By applying $Z^\alpha$ to the boundary condition \eqref{bord2}  and by using the expansion  \eqref{comS}, we get  that
 \beq
 \label{bordV1}
  \eps \big( 2  S^\varphi \big( Z^\alpha v \big)  -  \D_{z} v \otimes \nabla^\varphi Z^\alpha \varphi  
     -  \nabla^\varphi Z^\alpha v \otimes \D_{z} v   \big) \N- \big( Z^\alpha q - g Z^\alpha h\big) \N
      + \big( 2 \eps S^\varphi v - (q- gh ) \big) Z^\alpha \N = \mathcal{C}^\alpha(\mathcal{B}) 
 \eeq
 where
 $$  \mathcal{C}^\alpha(\mathcal{B})= - \eps \mathcal{E}^\alpha(v) - \mathcal{C}^\alpha (\mathcal{B})_{1}   + \mathcal{C}^\alpha(\mathcal{B})_{2}$$ 
 with $\mathcal{E}$ defined after \eqref{comS} and
 \begin{align*}
 &  \mathcal{C}^\alpha (\mathcal{B})_{1}= \sum_{  \tiny{ \begin{array}{ll}\beta + \gamma = \alpha, \\
    0 <| \beta | <|\alpha| \end{array} }} \eps Z^\beta \big(S^\varphi v \big) \, Z^\gamma \N, \\
&  \mathcal{C}^\alpha(\mathcal{B})_{2}= \sum_{ \tiny{ \begin{array}{ll}\beta + \gamma = \alpha, \\
    0 <| \beta | <|\alpha| \end{array} }} Z^\beta \big( q- g h\big)\, Z^\gamma \N.
    \end{align*} 
   Next, by using the product and commutator estimates of Proposition \ref{sobbord} on the boundary, we obtain that
   \begin{align*} | \mathcal{C}^\alpha(\mathcal{B})_{1}|_{L^2(\mathbb{R}^2)}  & \lesssim \eps\,  \| S^\varphi v\|_{1, \infty} |Z \N|_{m-2}
    + \eps   | S^\varphi v|_{m-1} |Z \N|_{L^\infty(\mathbb{R}^2)} \\
     &\leq  \Lambda\big( {1 \over c_{0}},  \|\nabla \eta \|_{1, \infty} +  \|\nabla v \|_{1, \infty} \big) \big( \eps  |h|_{m} +
      \eps |\nabla v |_{m-1}\big)
      \end{align*}
      and hence, thanks to \eqref{etainfty} and \eqref{dzvb}, we get that
    $$ | \mathcal{C}^\alpha(\mathcal{B})_{1}|_{L^2(\mathbb{R}^2)}   \lesssim
     \Lambda\big( {1 \over c_{0}},  |h|_{ 2, \infty} + \|v\|_{E^{2, \infty}} \big)
      \big( \eps |h|_{m}+  \eps |v^b|_{m}\big).$$
   To estimate $\mathcal{C}^\alpha(\mathcal B)_{2}$, we first note that thanks to \eqref{bordv2}, we have
    $ Z^\beta (q- gh)= Z^\beta\big(  \eps S^\varphi v \n \cdot \n) $ and we get 
      in a similar way that 
    \begin{align*}
     | \mathcal{C}^\alpha(\mathcal{B})_{2}|_{L^2(\mathbb{R}^2)}  & \lesssim 
       \big( |  Z q^{NS} |_{L^\infty} + |h|_{L^\infty} \big) | \nabla h |_{m-1} + | \nabla h |_{{1, \infty}} \big(  | Z q^{NS}|_{m-2} \big) \\
        & \leq \Lambda({1 \over c_{0}},   |h|_{2, \infty} + \|v\|_{E^{2,\infty}} \big) \big(  \eps |v^b|_{m} + \eps |h|_{m})
       \end{align*}
       where we have set  $q^{NS}= \eps S^\varphi v \n \cdot \n$
        and the last estimate comes from  \eqref{dzvb}.
      
      It remains to estimate $\mathcal{E}^\alpha(v)$ and hence $|\mathcal{C}_{i}^\alpha (v_{j})|_{L^2(\mathbb{R}^2)}.$
      At first, we note that $\mathcal{C}_{i, 3}^\alpha(v_{j})= 0$ since we only consider the case that $\alpha_{3}= 0$.
      Then,  we get as in the proof of Lemma \ref{comi} that
    $$ |\mathcal{C}_{i}^\alpha (v_{j})|_{L^2(\mathbb{R}^2)} \leq  \Lambda\big( {1 \over c_{0}} , \| \nabla \eta \|_{1, \infty}
     + \| \nabla v \|_{1, \infty}\big) \big( | \nabla v |_{m-1} + |\nabla \eta |_{m-1}\big).$$
     Finally, we can use  Lemma \ref{lembord} and \eqref{etainfty}, \eqref{etaharm} to get that 
   $$  |\mathcal{C}_{i}^\alpha (v_{j})|_{L^2(\mathbb{R}^2)} \leq  \Lambda\big( {1 \over c_{0}} , | h|_{2, \infty} + 
   \|v\|_{E^{2, , \infty}}  \big) \big( |v^b|_{m} + | h |_{m}\big).$$
   This yields
 $$ | \eps \mathcal{E}^\alpha(v)|_{L^2(\mathbb{R}^2)} \leq  \Lambda\big( {1 \over c_{0}} , | h|_{2, \infty} + 
   \|v\|_{E^{2, \infty}}  \big)\big( \eps  |v^b|_{m} + \eps  | h |_{m}\big)$$
   and the estimate \eqref{CalphaB} follows.
   
   To get \eqref{bordV} from \eqref{bordV1}, it suffices to use Lemma \ref{lemal}.
 \end{proof}

It remains to study  the kinematic  boundary condition \eqref{bordv1}
\begin{lem}
\label{lembordh}
  For every $\alpha$, $1 \leq |\alpha | \leq m$ such that $\alpha_{3}= 0$ we have that on $\{z=0\}$
  \beq
  \label{bordh}
  \partial_{t} Z^\alpha h - v^b \cdot  Z^\alpha \N - V^\alpha \cdot \N = \mathcal{C}^\alpha(h)
  \eeq
  where
  \beq
  \label{bordhC}
 | \mathcal{C}^\alpha(h)|_{L^2(\mathbb{R}^2)}
  \leq \Lambda\big( {1 \over c_{0}},   | v|_{E^{1, \infty}} +  |h|_{2, \infty}  \big) \big(  |h|_{m} + \|v\|_{E^{m}}  
\big). 
\eeq
\end{lem}
\begin{proof}
We immediately get from \eqref{bordv1} that
$$ \partial_{t} Z^\alpha h +  v^b \cdot Z^\alpha \N +  V^\alpha \cdot \N = - [ Z^\alpha, v^b_{y}, \nabla_{y}h]  - {(\partial_{z} v)^b \over \partial_{z} \varphi} \cdot \N\,
 Z^\alpha h:= \mathcal{C}^\alpha(h)$$
 where we recall that we use the notation $f^b= f_{/z=0}$ for the trace on $z=0$
  Consequently, by using the  formula  \eqref{comsym}  on the boundary, we get that
  $$ \| \mathcal{C}^\alpha(h) \| \leq \Lambda\big( {1 \over c_{0}},   |Z v|_{L^\infty} +  |h|_{2, \infty} + |\partial_{z}v |_{L^\infty} \big) \big(  |h|_{m} + |Z v^b|_{m-2}
\big)$$
 and the result follows by using the trace estimate \eqref{trace}. 

\end{proof}

\section{Pressure estimates}
\label{sectionpressure}
In view of the equation \eqref{eqValpha}, in order to estimate the right hand side,  we  need to control the pressure.
 This is the aim of this section. 
By  applying $\nabla^\varphi \cdot$ to  the equation \eqref{NSv}, we get that the pressure $q$ solves in $\mathcal{S}$ the system
\beq
\label{eqpression}
\Delta^\varphi q = - \nabla^\varphi(v\cdot \nabla^\varphi v), \quad q_{/z=0} =  2 \eps S^\varphi \n \cdot \n + gh.
\eeq
We shall split the pressure into  an "Euler" and a "Navier-Stokes" part with the following decomposition:
\beq
\label{pressuredec}
q= q^E + q^{NS}
\eeq
where $q^E$ solves 
\beq
\label{peuler}
\Delta^\varphi q^E = - \nabla^\varphi(v\cdot \nabla^\varphi v), \quad q^E_{/z=0}= gh, 
\eeq
and $q^{NS}$ solves
\beq
\label{qNS}
\Delta^\varphi q^{NS} = 0, \quad q^{NS}_{/z=0}=  2 \eps S^\varphi v \,\n \cdot \,\n.
\eeq
The idea behind this decomposition is that  we shall  need more regularity on $v$ to estimate $q^{NS}$ but thanks to the gain
 of the $\eps$ factor,  this term will be controlled by the viscous regularization. The  Euler pressure solves
  the same equation with the same boundary condition as in the free boundary Euler equation.
Thanks to the explicit expressions of the operators  $\D_{i}$,  we get that the operator $\Delta^\varphi$ can be expressed as
\beq
\label{deltaphi} \Delta^{\varphi}f =  {1 \over \partial_{z} \varphi} \nabla \cdot \big( E \nabla f \big),\eeq
 with the matrix $E$ defined by 
 $$ E= \left( \begin{array}{ccc} \partial_{z} \varphi & 0 & {- \partial_{1} \varphi} \\
  0 & \partial_{z} \varphi & -\partial_{2} \varphi  \\ -\partial_{1} \varphi & -\partial_{2} \varphi &  
   {1 + (\partial_{1} \varphi)^2} + (\partial_{2} \varphi)^2 \over \partial_{z} \varphi \end{array} \right) 
   =  \frac1{\partial_{z} \varphi} PP^*  $$
  and where  we have  for the gradient and the divergence the expressions:
 \beq
 \label{graddiv} \nabla^\varphi \cdot v= {1 \over \partial_{z}\varphi}  \nabla\cdot \big( P v \big) , \quad  \nabla^\varphi f= {1 \over \partial_{z} \varphi} P^* \nabla f,
  \quad P= \left( \begin{array}{ccc}  \partial_{z} \varphi & 0 & 0 \\ 0 & \partial_{z} \varphi & 0 \\ -\partial_{1} \varphi  & - \partial_{2} \varphi & 1 \end{array}
   \right).\eeq
   
   Note that $E$ is symmetric positive and that if 
   $ \| \nabla_y \varphi \|_{L^\infty} \leq { 1 \over c_{0}} $ and   
 $  \partial_{z} \varphi \geq c_{0}>0$ then
   there exists  $\delta(c_{0})>0$ such that
   \beq
   \label{Eminor}
   E X \cdot X \geq \delta |X|^2, \quad \forall X \in \mathbb{R}^3.
   \eeq
   Moreover, we note that
   \beq
   \label{ELinfty}
   \| E\|_{W^{k, \infty}} \leq \Lambda({1 \over c_{0}}, |h|_{k+1, \infty})
   \eeq
   and that by using the decomposition
   $$ E= \mbox{Id}_{A} + \tilde{E}, \quad 
\tilde  E= \left( \begin{array}{ccc}  \partial_{z}\eta & 0 & {- \partial_{1} \eta} \\
  0 &  \partial_{z}\eta  & -\partial_{2} \eta  \\ -\partial_{1} \eta & -\partial_{2} \eta &  
   {A( (\partial_{1} \eta)^2} + (\partial_{2} \eta)^2) - \partial_{z} \eta  \over \partial_{z} \varphi \end{array} \right), \quad
    \, \mbox{Id}_{A}=\mbox{diag} (1, 1, 1/A),$$
   we have the estimate
 \beq
 \label{EHs} \| \tilde{E}\|_{H^s} \leq \Lambda({1 \over c_{0}}, |h|_{1, \infty} \big) |h|_{s+ {1 \over 2}} .
 \eeq
 Before proving the estimates that we need for $q^E$ and $q^{NS}$, we shall establish
  general elliptic estimates for the Dirichlet problem  of 
 the operator $\nabla\cdot\big( E \nabla \cdot \big)$.
  \begin{lem}
  \label{lemelgen}
  Consider the elliptic equation in $\mathcal{S}$
  \beq
  \label{elgen}
  - \nabla \cdot \big( E \nabla \rho)= \nabla \cdot F, \quad \rho_{/z=0}= 0.
  \eeq
  Then we have the estimates:
  \begin{eqnarray}
  \label{ElH2} & &   \| \nabla \rho \| \leq \Lambda\big( {1 \over c_{0}}, |h|_{1, \infty}\big) \|F\|_{L^2},  \quad 
  \| \nabla^2 \rho \| \leq \Lambda\big( {1 \over c_{0}}, |h|_{2, \infty}\big)\big( \|\nabla \cdot F \| + \|F\|_{1} \big),  \\
  \label{ElHmco} & &   \| \nabla \rho \|_{k} 
   \leq \Lambda\big( {1 \over c_{0}}, |h|_{2, \infty} + |h|_{3} + \|F \|_{H^2_{tan}}+ \| \nabla \cdot F \|_{H^1_{tan}}  \big)\big( |h|_{k+{1 \over 2}}+  \|F\|_{k} \big), \quad
   k \geq 1, \\ 
   \label{Eldzzco} & & \| \partial_{zz} \rho \|_{k-1} \leq  \Lambda\big( {1 \over c_{0}}, |h|_{2, \infty} + |h|_{3} + \|F \|_{H^2_{tan}}+ \| \nabla \cdot F \|_{H^1_{tan}}  \big) \\
   \nonumber  & & \quad \mbox{\hspace{6.5cm}} \big( |h|_{k+{1 \over 2}}+  \|F\|_{k} + \| \nabla \cdot F \|_{k-1} \big), \quad k \geq 2.
  \end{eqnarray}
  \end{lem}
    Note that the estimate \eqref{ElHmco} is tame in the sense that it is linear in  the highest 
 norm with respect to the source term and $h$.
     Nevertheless,  in the statement, the function  $\Lambda$   depends on $k$.
    
    \begin{proof}
   We can construct by standard variational arguments  a  solution of \eqref{elgen} 
which satisfies
    the homogeneous $H^1$ estimate
  \beq
  \label{rhoH1} \| \nabla \rho \| \leq \Lambda( {1 \over c_{0}}) \| F\|. 
  \eeq
 Next, the classical elliptic regularity result  for \eqref{elgen} (see \cite{Gilbarg},  \cite{Lannes05} for example)  yields
\beq
\label{rhoH2}
 \| \nabla \rho \|_{H^1} \leq   \Lambda\big( {1 \over c_{0}}, |h|_{2, \infty} \big) \big( \| \nabla \cdot F  \| + \|F \|_{H^1_{tan}} \big).
 \eeq
 Before proving \eqref{ElHmco}, we shall need another auxiliary estimate for low order derivatives in order to control $\| \nabla \rho\|_{L^\infty}$.
  Indeed, from \eqref{emb}, we get that
  \beq
  \label{qinfty0}
  \| \nabla \rho \|_{L^\infty}^2 \lesssim \|\partial_{z} \nabla \rho \|_{H^{1}_{tan}} \, \| \nabla \rho \|_{H^2_{tan}}
  \eeq
By applying the tangential  derivation $\partial_{y}^\alpha $  to \eqref{elgen} for $| \alpha |= 2$, we get
  that  $$- \nabla \cdot  \big( E \nabla \partial_{y}^\alpha \nabla \rho\big) = \nabla \cdot\big( \partial_{y}^{\alpha} F) + \nabla \cdot \big( [\partial_{y}^{\alpha}, E] \nabla \rho), \quad
   \partial_{y}^\alpha \rho_{/z= 0}= 0$$
   and the standard energy estimate gives that
   $$ \| \nabla \partial_{y}^\alpha  \rho \| \leq \Lambda({1 \over c_{0}}) \big(  \| \partial_{y}^\alpha F \| + \| [\partial_{y}^\alpha, E] \nabla \rho\|\big).$$
   To control the commutator, we can expand it to get that 
   $$ \| [\partial_{y}^\alpha, E] \nabla \rho\| \lesssim \|E\|_{1, \infty} \| \nabla \rho \|_{H^1_{tan}} +   \|(\partial_{y}^\alpha  \tilde E)\nabla \rho\|
    \leq\Lambda\big({1 \over c_{0}}, |h|_{2, \infty} \big)\| \nabla \rho \|_{H^1_{tan}} + \|(\partial_{y}^\alpha  
   \tilde E)\nabla \rho\| $$
 where we have used  \eqref{ELinfty} for the second estimate. To control the above last term, we use successively the Holder
  inequality and the Sobolev inequality to get that
  \beq
  \label{L4}  \|(\partial_{y}^\alpha  
   \tilde E)\nabla \rho\| \lesssim \| \partial_{y}^\alpha \tilde E \|_{L^3}  \| \nabla \rho \|_{L^6} \lesssim  \|\partial_{y}^\alpha \tilde E\|_{H^{1\over2}}
    \, \| \nabla \rho \|_{H^1}, \quad | \alpha |= 2.\eeq
    Consequently, by using \eqref{EHs} and \eqref{rhoH2}, we obtain 
    \beq
   \label{rhos0}
    \| \nabla  \rho \|_{H^2_{tan}} \leq \Lambda({1 \over c_{0}}, |h|_{2, \infty} +  |h|_{3} \big) \|\nabla \cdot F\| + \|F\|_{H^{2}_{tan}}).
   \eeq
    Since the  equation \eqref{elgen}   gives
      \beq
      \label{eqdzzrho} \partial_{zz}  \rho ={1 \over E_{33}} \Big(  \nabla \cdot F -  \partial_{z}\big( \sum_{j<3}E_{3,j}\partial_{j} \rho  \big)- 
       \sum_{i<3, \, j} \partial_{i}\big( E_{i j} \partial_{j} \rho\big) \Big)\eeq
     we also obtain by using \eqref{rhos0},   \eqref{rhoH2} and \eqref{L4} that 
     \beq
     \label{dzzrho1} \|\partial_{zz} \rho \|_{H^1_{tan}} \leq \Lambda\big({1 \over c_{0}}, |h|_{2, \infty} + |h|_{3} \big)\big( \| \nabla \cdot F \|_{H^1_{tan}}
      + \|F\|_{H^2_{tan}} \big).\eeq
     By using  the anisotropic Sobolev embedding \eqref{qinfty0} and \eqref{rhos0}, \eqref{dzzrho1}, we finally get that
      \beq
      \label{qinfty1}
      \|\nabla q \|_{L^\infty} \leq \Lambda\big({1 \over c_{0}}, |h|_{2, \infty} + |h|_{3} \big)\big( \| \nabla \cdot F \|_{H^1_{tan}}
      + \|F\|_{H^2_{tan}} \big).
      \eeq

 We  can now establish the estimate \eqref{ElHmco}  by induction.
 Note that the case $k=1$ is just a consequence of  \eqref{ElH2}.  Next, we want to apply derivatives  to  the equation,
  use the induction assumption to control  commutators and use the  energy estimate \eqref{rhoH1} for the obtained equation.
   Note that for the last part of the argument  we need the source term to be in divergence form. Consequently, 
    a difficulty arises since $Z_{3}$ does not commute with $\partial_{z}$.
     To solve this difficulty, we can  introduce a modified field such that
     \beq
     \label{comtilde}
      \tilde Z_{3} \partial_{z}= \partial_{z} Z_{3}.
      \eeq
      This is achieved by setting
      $$ \tilde Z_{3} f = Z_{3} f  + {1 \over (1- z) ^2} f.$$
      Let us also set $\tilde Z ^\alpha = Z_{1}^{\alpha_{1}} Z_{2}^{\alpha_{2}} \tilde{Z}_{3}^{\alpha_{3}}.$ Note that
      $$\|\tilde Z^\alpha f - Z^\alpha f \| \lesssim \|f\|_{|\alpha| - 1}.$$
      By applying $\tilde Z^\alpha$ to \eqref{elgen}, we first get
      $$  \nabla \cdot \big( Z^\alpha \big( E \nabla \rho )\big) = \nabla \cdot\Big( Z^\alpha F  +\big( \tilde{Z}^\alpha - Z^\alpha \big)  F_{h}
       -  \big( \tilde Z^\alpha - Z^\alpha\big) (E\nabla \rho)_{h}\Big)$$
 where for a vector $X \in \mathbb{R}^3$, we set $X_{h}= (X_{1}, X_{2}, 0)^t$.Ê  Next, this yields
  \beq
  \label{elgender}
   - \nabla \cdot \big( E \cdot \nabla Z^\alpha\rho\big) = \nabla \cdot\Big( Z^\alpha F  +\big( \tilde{Z}^\alpha - Z^\alpha \big)  F_{h}
       -  \big( \tilde Z^\alpha - Z^\alpha\big) (E\nabla \rho)_{h}\Big) + \nabla \cdot  \mathcal{C}
   \eeq
         where the commutator $\mathcal{C}$ is given by 
     $$
     \mathcal{C}= - \big( E [Z^\alpha,  \nabla] \rho\big)  -  \big(  \sum_{ \beta + \gamma= \alpha, \beta \neq 0}
    c_{\beta, \gamma} Z^\beta E \cdot Z^\gamma \nabla \rho \big).$$
    From the standard energy estimate since $Z^\alpha \rho$ vanishes
       on the boundary, we get that
      \beq
      \label{elprov3} \| \nabla \rho\|_{k} \leq \Lambda({1 \over c_{0}} ) \big( \|F\|_{k} + \|E \nabla \rho \|_{k-1} + \| E [Z^\alpha , \nabla] \rho \| +  \big\|  \sum_{ \beta + \gamma= \alpha, \beta \neq 0} c_{\beta, \gamma}
      Z^\beta E \, Z^\gamma \nabla \rho \big\|\big).\eeq
      To estimate the right hand-side, we first use  \eqref{gues} and \eqref{EHs} to get
    \begin{eqnarray*} \|E \nabla \rho \|_{k-1}  &  \lesssim &  (1 +   \|\tilde{E}\|_{L^\infty}) \|\nabla \rho \|_{k-1} +  \| \nabla \rho \|_{L^\infty} \|\tilde{E}\|_{k-1} \\
      &  \lesssim &  \Lambda ({1 \over c_{0}},  |h|_{1, \infty }) \| \nabla \rho \|_{k-1} +  \Lambda( {1 \over c_{0}}, |h|_{1, \infty}) |h|_{k- {1 \over 2}} \| \nabla \rho \|_{L^\infty}.
      \end{eqnarray*}
       Next,  $\|\nabla \rho \|_{k-1}$ is controlled by the induction assumption and  we use \eqref{qinfty1} to get  
           $$ \|E \nabla \rho \|_{k-1}  \lesssim  \Lambda \big({1 \over c_{0}}, |h|_{2, \infty} + |h|_{3} +  \|F\|_{H^2_{tan}} + \| \nabla \cdot F\|_{H^1_{tan}}  \big)\big(
        |h|_{k- {1 \over 2}} + \|F\|_{k- 1} \big).$$
      In a similar way,  we can use \eqref{idcom} to get that
    $$[Z^\alpha, \partial_{z}]= \sum_{ | \beta | \leq | \alpha |-1} \varphi_{\alpha, \beta} \partial_{z}\big( Z^\beta  \cdot \big)$$
     for some harmless bounded functions $\varphi_{\alpha, \beta}$ and hence
     $$  \| E [Z^\alpha , \nabla] \rho \| \leq \Lambda\big( {1 \over c_{0}}, |h|_{1, \infty}\big) \| \nabla \rho \|_{k-1} \big)$$
       and  we get by using \eqref{gues} and \eqref{EHs} that
       $$\big\|  \sum_{ \beta + \gamma= \alpha, \beta \neq 0} c_{\beta, \gamma}
      Z^\beta E \, Z^\gamma \nabla \rho \big\| \lesssim \Lambda( |h|_{2, \infty} + \|\nabla \rho \|_{L^\infty} )(|h|_{k+ {1 \over 2}} + \|\nabla \rho\|_{k-1}).$$
      To conclude, we plug the two last estimates in \eqref{elprov3} and we use the induction assumption to  estimate $\| \nabla \rho\|_{k-1}$ and \eqref{qinfty1}. This proves \eqref{ElHmco}.     
      
      Finally, to get \eqref{Eldzzco}, it suffices to use the equation \eqref{elgen} in the form \eqref{eqdzzrho}         and to use again the  product estimates \eqref{gues}, and   \eqref{ELinfty}, \eqref{EHs}, \eqref{qinfty1}  together  with the  previous estimate
         \eqref{ElHmco}.

      \end{proof}
      
   Also,  
it will  be convenient to  have at  our disposal  the same kind of estimates as in Lemma \ref{lemelgen} but with
    nonhomogeneous Dirichlet condition.
  
  \begin{lem}
  \label{Elnh}
  Consider the elliptic equation in $\mathcal{S}$
  \beq
  \label{eqelgenbis}
  - \nabla \cdot \big( E \nabla \rho)= 0, \quad \rho_{/z=0}=  f^b.
  \eeq
  Then we have the estimates:
  \begin{eqnarray}
  \label{estElnh1}
 & &  \| \nabla  \rho \|_{H^k(\mathcal{S})}  \leq \Lambda\big( {1 \over c_{0}}, |h|_{2, \infty}  + |h|_{3}+|f^b|_{1, \infty}  + |f^b|_{5 \over 2} \big)
  \big( |h|_{k+{1 \over 2}} + |f^b|_{k+{1 \over 2}}\big).
  \end{eqnarray}
  \end{lem}
  The regularity  involving  $|h|_{3}$ in the estimate \eqref{estElnh1} is not optimal, nevertheless, since at
 the end we shall
   use the Sobolev embedding to control $|h|_{2, \infty}$, this is sufficient.
  \begin{proof}
   We write $\rho$ under the form
  $\rho= \rho^H + \rho^r$
  where $\rho^H$ will absorb the boundary data  and $\rho^r$ solves
  \beq
  \label{eqqr}
   - \nabla \cdot \big( E \nabla \rho^{r} \big)= \nabla \cdot\big( E \nabla \rho^H\big), \quad \rho^{r}_{/z= 0}= 0.
   \eeq 
 For $\rho^H$, we choose as in \eqref{eqeta}
 $$ \hat \rho^H(\xi, z)= \chi(z \, \xi) \hat{f}^b.$$
 Consequently, the estimates of Proposition \ref{propeta}  are also valid for $\rho^H$ in particular, we have for every $k \geq 0$:
 \beq
 \label{estqrbis} \|\nabla \rho^H \|_{H^k} \lesssim |f^b|_{k+ {1 \over 2 }}, \quad \|\rho^H\|_{k, \infty} \lesssim \| f^b\|_{k, \infty}.
 \eeq
 Next, we can estimate the solution of \eqref{eqqr}. In Lemma \ref{lemelgen}, we have studied the elliptic problem
 \eqref{elgen} in Sobolev conormal spaces. Nevertheless, by using the same approach one can also get 
  the estimate corresponding to \eqref{ElHmco} in standard Sobolev spaces. This yields 
 \begin{align*}
  \| \nabla  \rho^r \|_{H^k} \leq \Lambda\big( {1 \over c_{0}}, |h|_{2, \infty} + |h|_{3} + \|E\nabla \rho^H\|_{H^2}  \big)
 \big(  |h|_{k+{1 \over 2}} + \| E \nabla \rho^H \|_{H^k} \big).
 \end{align*}
 To estimate the right hand side,  since   $\tilde E$ and $q^H$ have a standard Sobolev regularity in $\mathcal{S}$,
  we get  
 thanks to \eqref{estqrbis}, \eqref{EHs} and  \eqref{ELinfty}  that 
 \begin{eqnarray*}
 & &  \| E \nabla q^H \|_{H^k} \lesssim \Lambda\big( |h|_{1, \infty} + |f^b|_{1, \infty}\big) \big( |h|_{k+ {1 \over 2 }} +  |f^b|_{k+ {1 \over 2 }}\big).
 \end{eqnarray*}
 This yields
 \beq
 \label{estqr}\| \nabla  q^r \|_{k} \leq \Lambda\big( {1 \over c_{0}}, |h|_{2, \infty} +|f^b|_{1, \infty} + |h|_{3} + |f^b|_{5 \over 2} \big)
  \big( |h|_{k+{1 \over 2}} + |f^b|_{k+{1 \over 2}}\big).
  \eeq
  Consequently, \eqref{estElnh1} is proven. 
  \end{proof}
We shall first  establish regularity estimates for $q^{NS}$
\begin{prop}
\label{propPNS}
For $q^{NS}$, we have for $m \geq 1$,  the estimate:
\beq
\label{estqNS}
\| \nabla q^{NS} \|_{H^{m-1}}  \leq   \Lambda\big( {1 \over c_{0}},|h|_{2, \infty} + \|v\|_{E^{2, \infty}} +   |h|_{4} + \|v\|_{E^4} \big)
\big( | \eps v^b |_{m+{1 \over 2}} + \eps |h|_{m+{1 \over 2 }}\big),
\eeq
and the $L^\infty$ estimate
\beq
\label{qNSLinfty} \|\nabla q^{NS} \|_{L^\infty} \leq \eps  \Lambda\big( {1 \over c_{0}},|h|_{2, \infty} + \|v\|_{E^{2, \infty}} +   |h|_{4} + \|v\|_{E^4} \big).
\eeq

\end{prop}
 
\begin{proof}
We need to estimate the solution of 
\beq
\label{qNSbis}
 - \nabla \cdot \big( E \nabla q^{NS} \big)= 0, \quad q^{NS}_{/z= 0}= 2 \eps S^\varphi v \n \cdot \n,  
 \eeq
 therefore,  it suffices to  use Lemma \ref{Elnh} with  $f^b =2 \eps S^\varphi v \n \cdot \n $. Since, we have 
   $$ |f^b|_{1, \infty}  \leq \Lambda \big({1 \over c_{0}},  |v|_{E^{2, \infty}} + |h|_{2, \infty} \big)$$
    and hence by using \eqref{eqdzvn} and \eqref{eqdzvPi}, we find
    $$ |f^b|_{1, \infty}  \leq \Lambda \big({1 \over c_{0}},  |v|_{E^{1, \infty}} + |h|_{2, \infty} + \eps \| v \|_{2, \infty}\big)$$
   and 
   \begin{align*} | f^b|_{m-{1 \over 2}} &  \leq \Lambda\big( {1 \over c_{0}},  |v|_{E^{1, \infty}} + |h|_{2, \infty} \big) \big( \eps | \nabla v(\cdot, 0)|_{m-{1 \over2}}
    + \eps |h|_{m+{1 \over 2}}\big)  \\ 
     & \leq\Lambda\big( {1 \over c_{0}},  |v|_{E^{1, \infty}} + |h|_{2, \infty} \big) \big( \eps | v(\cdot, 0)|_{m+{1 \over2}}
    + \eps |h|_{m+{1 \over 2}}\big) 
    \end{align*}
    where the last estimate comes from \eqref{dzvb}
    In particular, since the above estimate holds for every $m$,  we  also have
    $$ |f^b|_{5\over 2 } \leq   \Lambda\big( {1 \over c_{0}},  |v|_{E^{1, \infty}} + |h|_{2, \infty}  + \eps |h|_{7\over 2} +  \eps | v^b|_{7\over 2}\big).$$
    To conclude and get \eqref{estqNS}, we use the trace estimate \eqref{trace} which yields
    \beq
    \label{trace7/2}  | v^b|_{7\over 2} \lesssim  \| \partial_{z}v \|_{3} +   \|v\|_{4} .  \eeq
     Finally, to obtain  \eqref{qNSLinfty},  it suffices to  use  the standard Sobolev embedding in $\mathcal{S}$ to write
    $$ \|\nabla q^{NS} \|_{L^\infty} \lesssim  \| \nabla  q^{NS}\|_{2} $$
     combined with \eqref{estqNS} and the trace estimate \eqref{trace7/2}.
    This ends the proof of Proposition \ref{propPNS}.

   \end{proof}
   
   It remains to estimate $q^E$.
   \begin{prop}
     \label{proppE}
   For $q^{E}$, we have the estimate:
\begin{eqnarray}
 \label{estqE}
 & & \| \nabla q^{E} \|_{m-1} + \|\partial_{zz} q^E \|_{m-2} \leq  \Lambda \big( {1 \over c_{0}}, |h|_{2, \infty}+ \|v\|_{E^{1 , \infty}} + |h|_{3} +
 |v|_{E^3}\big)
  \big( |h|_{m-{1 \over 2}}  + |v|_{E^m} \big), \\
\label{estqElinfty}  & & \| \nabla q^E \|_{1, \infty} + \| \partial_{z}^2 q^E \|_{L^\infty} \leq   \Lambda \big( {1 \over c_{0}}, |h|_{2, \infty}+ \|v\|_{E^{1 , \infty}}  +
 |h|_{4} + \|v\|_{E^{4}}\big), \\
\label{estqElinfty2} & &  \| \nabla q^E \|_{2, \infty} \leq   \Lambda \big( {1 \over c_{0}}, |h|_{2, \infty}+ \|v\|_{E^{1 , \infty}}  +
 |h|_{5} + \|v\|_{E^{5}}\big)
\end{eqnarray}
   \end{prop}
  
   \begin{proof}
   By using \eqref{peuler} and \eqref{graddiv}, we see that we have to solve the elliptic problem
   \beq
   \label{peuler2}
     - \nabla \cdot \big( E \nabla q^E\big) =  \nabla \cdot \big( P ( v \cdot \nabla^\varphi v\big) =  \partial_{z}\varphi \nabla^\varphi v \cdot \nabla^\varphi v,
      \, (y,z) \in \mathcal{S}, 
      \quad q^E_{/z=0} =  gh.
      \eeq 
   We can split this equation in two parts by setting $q^E = q^b + q^i$ where $q^b$ solves the homogeneous equation
 $$  - \nabla \cdot \big( E \nabla q^b\big) = 0$$
 in $\mathcal{S}$  with  nonhomogeneous boundary condition $ q^b_{/z=0}= gh$
  and $q^i$ solves
 $$   - \nabla \cdot \big( E \nabla q^i\big) =  \nabla \cdot \big( P ( v \cdot \nabla^\varphi v)\big) =  \partial_{z}\varphi \nabla^\varphi v \cdot \nabla^\varphi v,
      \, (y,z) \in \mathcal{S}, 
      \quad q^i_{/z=0} = 0.$$
      We get the estimate of $q^b$  as  a consequence of Lemma \ref{Elnh} with $f^b=  gh$.   We find 
   $$ \| \nabla q^b \|_{m-1} + \|\partial_{zz} q^b \|_{m-2} \leq \Lambda \big( {1 \over c_{0}}, |h|_{2, \infty} + |h|_{3 }\big) |h|_{m-{1 \over 2}}.$$
    To estimate  $q^i$ we can use  Lemma \ref{lemelgen}. This yields
    \begin{multline*} \| \nabla q^i \|_{m-1}   \leq   \Lambda \big( {1 \over c_{0}}, |h|_{2, \infty}  + |h|_{3} + \|P ( v \cdot \nabla^\varphi v)\|_{2}
     +   \|\partial_{z}\varphi \nabla^\varphi v \cdot \nabla^\varphi v\|_{1} \big) \\  \big( |h|_{m-{1 \over 2 }}
 +   \|P ( v \cdot \nabla^\varphi v)\|_{m-1} \big).
 \end{multline*} 
 By using again product estimates we thus  find
 $$ \| \nabla q^i \|_{m-1}   \leq   \Lambda \big( {1 \over c_{0}}, |h|_{2, \infty}+ \|v\|_{E^{1 , \infty}}  + |h|_{3 }  + \|v\|_{E^3} \big)
  \big( |h|_{m-{1 \over 2}}  + |v|_{E^m} \big).$$
  In a similar way, using  Lemma \ref{lemelgen},  we  get the estimate for $ \partial_{zz} q^i$.
  
  To get \eqref{estqElinfty}, we use  \eqref{sob} to write that
  $$ \| \nabla q^E \|_{k, \infty} \lesssim \| \partial_{z} \nabla q^{E}\|_{k+1} \| \nabla q^E \|_{k+2}, \quad k=1, \, 2$$
  and hence we get  the estimate for $\| \nabla q^E \|_{k,\infty},$  $k=1, \, 2$  by using \eqref{estqE} with $m= 4$ or $m=5$.  Finally, to estimate $\|\partial_{zz} q^E \|_{L^\infty}$, it suffices
    to rewrite the  equation \eqref{peuler2} under the form  \eqref{eqdzzrho} and to use the  estimate for   $\| \nabla q^E \|_{1, \infty}$ just established.
  
    \end{proof}

     This  ends the proof of Proposition \ref{proppE}.     
   It remains a last estimate for  $q^E$ that   will be  used for  the control of the Taylor sign condition.
   \begin{prop}
   \label{proptaylor}
   For $T\in (0, T^\eps)$, we have  the estimate
   $$ \int_{0}^T | (\partial_{z} \partial_{t} q^E)^b |_{L^\infty} \leq  \int_{0}^T \Lambda \big( {1 \over c_{0}}, \|v\|_{6}+ \|\partial_{z}v \|_{4} +  \|v\|_{E^{2, \infty}} +|h|_{6}+ |h|_{3,  \infty}\big) \\ 
       \cdot \big( 1 +  \eps \|\partial_{zz}v\|_{L^\infty} + \eps \| \partial_{zz} v\|_{3}\big).$$
   
   \end{prop}
  
  The proof of proposition \ref{proptaylor} is  a consequence of the following general  lemma:
\begin{lem}
\label{lemtaylor}
Consider $\rho$ the solution of the equation 
\beq
\label{taylorel}
 - \nabla \cdot \big(E \nabla \rho)= \nabla \cdot F, \quad \rho_{/z=0}=0
 \eeq 
 then  for every $k>1$,  $\rho$ satisfies the estimate
 $$ \|(\partial_{z} \rho)^b \|_{L^\infty} \leq    \Lambda \big( {1 \over c_{0}}, |h|_{k+2, \infty} \big)
   \big(  \| F\|_{H^{k+ 1}_{tan}}  +  |(F)^b|_{L^\infty(\mathbb{R}^2)}\big).$$
\end{lem}
We recall that we denote 
 the trace of   a function $f$ on the boundary $z=0$ by $f_{/z=0}$  or $f^b$.
  \begin{proof}
   The first step is  to establish that $\rho$ satisfies for every $m\geq 0$ the estimate:
  \beq
  \label{eltaylor1}
   \| \nabla \rho \|_{H^m_{tan}} \leq \Lambda \big( {1 \over c_{0}}, |h|_{m+1, \infty}\big) \|F\|_{H^m_{tan}}.
  \eeq 
   To  get this estimate, it suffices to apply $\partial_{y}^\alpha$ with $| \alpha| \leq m$ to \eqref{taylorel}
    and to use the standard energy estimate in a classical way.
    Next, we can rewrite \eqref{taylorel} as
    \beq
    \label{taylorel2}
    - E_{33}\partial_{zz} \rho =  \partial_{z} F + R 
    \eeq
    with $R$ given by 
    \beq
    \label{Rtaylor}
     R= \sum_{i<3} \partial_{i} F_{i}  + \partial_{z} E_{33} \partial_{z} \rho + \sum_{i<3}\Big( \partial_{i}\big( E \nabla \rho)_{i} +\partial_{z}(E_{3i}\partial_{i}\rho)\Big).
    \eeq
  For the moment, we see \eqref{taylorel2} as an ordinary differential equation in $z$, the variable $y$ being only a parameter.
   We multiply \eqref{taylorel2} by $\partial_{z} \rho$ and we integrate in $z$ to get after integration by parts:
   \beq
   \label{eltaylor2}| \partial_{z} \rho(0)|^2 \leq \Lambda \big( {1 \over c_{0}}, |h|_{2, \infty} \big)
   \big(  |\partial_{z} \rho|_{L^2_{z}}^2 +  |R|_{L^2_{z}}^2 + |F(0)|\, | \partial_{z} \rho(0)| + \big| \int_{-\infty}^0  F\, \partial_{zz} \rho\, dz \big|\big).\eeq
Note that here $| \cdot |$ stands for the absolute value. To estimate the last term, we use again 
 the equation \eqref{taylorel2} to obtain
 $$  \big| \int_{-\infty}^0  F\, \partial_{zz} \rho\, dz \big| \leq \big|  \int_{-\infty}^0  F\, {\partial_{z}  F \over E_{33}} \, dz \big|
  + \Lambda\big({1 \over c_{0}}\big) |F|_{L^2_{z}}\, |R|_{L^2_{z}}$$
  and hence from an integation by parts we find
 $$  \big| \int_{-\infty}^0  F\, \partial_{zz} \rho\, dz \big| \leq \Lambda\big({1 \over c_{0}}, |h|_{2, \infty}\big)
 \big(  |F(0)| +  |F|_{L^2_{z}}^2 +  |R|_{L^2_{z}}^2 \big).$$ 
 Consequently, we get from \eqref{taylorel2} that
 $$| \partial_{z} \rho(0)| \leq \Lambda \big( {1 \over c_{0}}, |h|_{2, \infty} \big)
   \big(  |\partial_{z} \rho|_{L^2_{z}} +  |R|_{L^2_{z}} + |F(0)| \big).$$
   Now, by taking the supremum in $y$  and by using the two-dimensional Sobolev embedding for the right hand side (except for the last term), we find that
   $$ | (\partial_{z} \rho)^b|_{L^\infty(\mathbb{R}^2)} \leq  \Lambda \big( {1 \over c_{0}}, |h|_{2, \infty} \big)
   \big(  \|\partial_{z} \rho\|_{H^k_{tan}} +  \|R\|_{H^k_{tan}} + |F^b|_{L^\infty(\mathbb{R}^2)} \big)$$
    for $k>1$.
    To conclude, we see from the definition of $R$ that
    $$ \|R\|_{H^k_{tan}} \lesssim   \Lambda\big( {1 \over c_{0}}, |h|_{k+2,\infty}\big) \|\nabla \rho\|_{H^{k+1}_{tan}}+ \|F\|_{H^{k+1}_{tan}}$$
    and hence, we get  from \eqref{eltaylor1} that 
    $$  | (\partial_{z} \rho)^b|_{L^\infty(\mathbb{R}^2)} \leq  \Lambda \big( {1 \over c_{0}}, |h|_{k+2, \infty} \big)
   \big(  \| F\|_{H^{k+ 1}_{tan}}  +  |F^b|_{L^\infty(\mathbb{R}^2)}\big).$$
   This ends the proof of Lemma \ref{lemtaylor}.
  \end{proof}
  We are now in position to give the proof of Proposition \ref{proptaylor}:
 
  {\bf Proof of proposition \ref{proptaylor}.}
  
  We note that $\partial_{t} q^E$ solves the elliptic equation
  $$ \nabla\cdot( E \nabla \partial_{t} q^E) = \nabla \cdot\big(  \partial_{t} \big( P (v\cdot \nabla^\varphi v) )\big)-
   \nabla \cdot \big( \partial_{t} E \, \nabla q^E\big), \quad  \partial_{t} q^E_{/z=0} = g \partial_{t} h $$
   consequently, we can again split $\partial_{t} q^E=  q^i + q^B$ where $q^B$ absorbs the boundary term:
   $$ \nabla\cdot( E \nabla  q^B)= 0, \quad q^B_{/z=0}= g \partial_{t} h$$
    and $q^i$ solves
   \beq
   \label{eqtaylorqi}
    \nabla\cdot( E \nabla \partial_{t} q^i) = \nabla \cdot\big(  \partial_{t} \big( P v\cdot \nabla^\varphi v) \big)-
   \nabla \cdot \big( \partial_{t} E \, \nabla q^E\big), \quad   q^i_{/z=0} = 0. 
   \eeq
  The estimate of $\|\nabla q^B\|_{L^\infty}$ is a consequence of Lemma \ref{Elnh}. Indeed,  for $k>3/2$, we have
  $$ \| \nabla q^B\|_{L^\infty} \lesssim \| \nabla q^B \|_{H^k}$$
   and from the estimate  of Lemma \ref{Elnh}, we  find that
   $$
   \| \nabla q^B\|_{L^\infty} \leq  \Lambda\big({ 1 \over c_{0}}, |h|_{2, \infty} + |h|_{3} +|\partial_{t} h|_{3}\big).
 $$
  From the boundary condition \eqref{bordv1}, we obtain
  $$ |\partial_{t}h|_{3} \leq \Lambda( \|v\|_{L^\infty} + \|h\|_{1, \infty}\big) \big( |h|_{4}+  |v^b|_{3}\big)
   \leq   \Lambda( \|v\|_{L^\infty} + \|h\|_{1, \infty}+ |h|_{4} + \|v\|_{E^4}\big)$$
   where the last estimate comes from the trace inequality \eqref{trace}.  We thus finally get for $q^B$ that
   \beq
   \label{qBtaylor}
    \| \nabla q^B\|_{L^\infty} \leq  \Lambda\big({ 1 \over c_{0}}, |h|_{2, \infty} + \|v\|_{L^\infty} + |h|_{4} +  \|v\|_{E^4}\big).
   \eeq
  
     It remains to estimate $q^i$. We shall use Lemma \ref{lemtaylor}, but we need to be carefull with the structure of the right hand side
      in \eqref{eqtaylorqi}. We first note thanks to the identities \eqref{graddiv} that
   $$ \nabla \cdot\big(  \partial_{t} \big( P (v\cdot \nabla^\varphi v) )\big)= \nabla \cdot \big( (\partial_{t} P) v\cdot \nabla^\varphi v\big)
    + \nabla \cdot \big( P (\partial_{t}v \cdot \nabla^\varphi v) \big) + \partial_{z}\varphi \nabla^\varphi\cdot  \big( v \cdot \partial_{t}(\nabla^\varphi v)\big).$$
    For the last term, by using again \eqref{graddiv} and the summation convention on repeated indices,  we  can write
    $$ \partial_{z}\varphi \nabla^\varphi \cdot \big( v \cdot \partial_{t}(\nabla^\varphi v)\big)=
    \partial_{z}\varphi \partial_{i}^\varphi\big( v \cdot  \partial_{t}\big( {1 \over \partial_{z}\varphi }P^*\nabla v_{i}\big)
    =  \partial_{z}\varphi \partial_{i}^\varphi\big( v \cdot \partial_{t}\big( {1 \over \partial_{z}\varphi }P^*\big)  \nabla v_{i}\big) + 
       \partial_{z}\varphi \, \partial_{i}^\varphi\big( v \cdot \nabla^\varphi \partial_{t} v_{i}\big)$$
       and the crucial observation is that since $\nabla^\varphi \cdot v=0$, we   have
       $$  \partial_{i}^\varphi\big( v \cdot \nabla^\varphi \partial_{t} v_{i}\big)= \partial_{i}^\varphi \big( v_{j}\partial_{j}^\varphi \partial_{t}v_{i}\big)
       = \partial_{j}^\varphi \big( \partial_{t}v_{i}\partial_{i}^\varphi v_{j}\big) +
        \partial_{j}^\varphi \big( \D_{i} \partial_{t}v_{i} \, v_{j}\big).$$
         Again since $\nabla^\varphi \cdot v= 0$, we can write
         $$\D_{i} \partial_{t}v_{i} = - {1 \over \partial_{z}\varphi} \nabla \cdot(\partial_{t} P v)
       $$
       and hence, we obtain
     $$ \partial_{i}^\varphi\big( v \cdot \nabla^\varphi \partial_{t} v_{i}\big)= \partial_{j}^\varphi \big( \partial_{t}v_{i}\partial_{i}^\varphi v_{j}\big) 
     - \partial_{j}^\varphi \Big( v_{j}  
      {1 \over \partial_{z}\varphi} \nabla \cdot(\partial_{t} P v)  \Big).$$
       Consequently, we have proven that
     \begin{multline*} \nabla \cdot\big(  \partial_{t} \big( P (v\cdot \nabla^\varphi v) )\big)=
      \nabla\cdot \big( (\partial_{t} P) v\cdot \nabla^\varphi v\big)+ 2  \nabla \cdot \big( P (\partial_{t}v \cdot \nabla^\varphi v) \big)
       +  \nabla \cdot \big( v \cdot \partial_{t}\big( {1 \over \partial_{z}\varphi }P^*\big)  \nabla v\big) \\
        - \nabla \cdot\Big( {1 \over \partial_{z}\varphi} \nabla \cdot(\partial_{t} P v) \, v \Big).
       \end{multline*}
       This allows to observe that
        \eqref{eqtaylorqi} is under the form \eqref{taylorel} with
        \beq
        \label{defFtaylor} F= (\partial_{t} P) v\cdot \nabla^\varphi v + 2  P (\partial_{t}v \cdot \nabla^\varphi v)+  v \cdot \partial_{t}\big( {1 \over \partial_{z}\varphi }P^*\big)  \nabla v  - {1 \over \partial_{z}\varphi} \nabla \cdot(\partial_{t} P v) + \partial_{t} E \nabla q^E.\eeq
        Consequently, by using Lemma \ref{lemtaylor} for $k=2$, we get that
        $$ | \partial_{z}q^i_{/z=0}|_{L^\infty} \leq \Lambda\big({ 1 \over c_{0}}, |h|_{4, \infty}\big)
         \big( \|F\|_{3} + |F^b|_{L^\infty}\big)$$
          and hence that
       \beq
       \label{taylorqi3} \int_{0}^T | \partial_{z}q^i_{/z=0}|_{L^\infty} \leq \int_{0}^T
       \Lambda\big({ 1 \over c_{0}}, |h(t)|_{4, \infty}\big) \big( \|F\|_{3} + |F^b|_{L^\infty}\big).
       \eeq 
         Next, by using Proposition \ref{propeta} and \eqref{gues}, we get that
       $$   \|F\|_{3} \leq \Lambda \big( {1 \over c_{0}}, \|v\|_{E^4}+ \|v\|_{E^{1, \infty}} +|h|_{4}+ |h|_{1,  \infty}+ |\partial_{t} h|_{5} + 
        |\partial_{t}h|_{2, \infty} + \| \nabla q^E \|_{3}\big)\big( 1 + \|\partial_{t} v\|_{L^\infty} + \|\partial_{t}v\|_{3}\big).$$
        Therefore, by using again the boundary condition \eqref{bordv1} and Proposition \ref{proppE}, we find
        $$  \|F\|_{3} \leq \Lambda \big( {1 \over c_{0}}, \|v\|_{E^5}+ \|v\|_{E^{2, \infty}} +\|v\|_{6}+|h|_{6}+ |h|_{3,  \infty}\big)( 1 + \| \partial_{t}v \|_{L^\infty}
         + \|\partial_{t}v \|_{3}\big).$$
         By using again  the expression \eqref{defFtaylor}, we also find
        \begin{align*}
         |F^b|_{L^\infty} &  \leq \Lambda\big( {1 \over c_{0}}, |h|_{3, \infty} +  \|v\|_{E^{2, \infty}}   + \|\nabla q^E \|_{L^\infty}
          \big)\big( 1 + \| \partial_{t} v\|_{L^\infty}\big) \\
          & \leq  \Lambda\big( {1 \over c_{0}}, |h|_{3, \infty} +  \|v\|_{E^{2, \infty}}  + |h|_{4}+ \|v\|_{E^4}\big)\big( 1 + \|\partial_{t}v \|_{L^\infty}\big)
          \end{align*}
          where the last estimate comes from Proposition \ref{proppE}. Consequently, by plugging these last two estimates
           in \eqref{taylorqi3}, we find
           $$   \int_{0}^T | \partial_{z}q^i_{/z=0}|_{L^\infty} \leq \int_{0}^T
          \Lambda \big( {1 \over c_{0}}, \|v\|_{E^5}+ \|v\|_{E^{2, \infty}} +  \|v\|_{6}+|h|_{6}+ |h|_{3,  \infty}\big)( 1 + \| \partial_{t}v \|_{L^\infty}
         + \|\partial_{t}v \|_{3}\big).$$
         To conclude, we can use the equation \eqref{NSv} to get 
       \begin{multline*}  \|\partial_{t}v \|_{3} + \| \partial_{t}v \|_{L^\infty} \leq   \Lambda \big( {1 \over c_{0}}, \|v\|_{E^5}+ \|v\|_{E^{2, \infty}} +|h|_{5}+ |h|_{2,  \infty}\big) \\ 
       \cdot \big( 1 +  \eps \|\partial_{zz}v\|_{L^\infty} + \eps \| \partial_{zz} v\|_{3} + \|\nabla q\|_{L^\infty} + \|\nabla q \|_{3}\big)
         \end{multline*}
         and then Proposition \ref{propPNS} and Proposition \ref{proppE} combined with the trace estimate \eqref{trace} to find
         $$  \|\partial_{t}v \|_{3} + \| \partial_{t}v \|_{L^\infty} \leq   \Lambda \big( {1 \over c_{0}}, \|v\|_{E^5}+ \|v\|_{E^{2, \infty}} +|h|_{5}+ |h|_{2,  \infty}\big) \\ 
       \cdot \big( 1 +  \eps \|\partial_{zz}v\|_{L^\infty} + \eps \| \partial_{zz} v\|_{3}\big).$$
       This ends the proof of Proposition \ref{proptaylor}. 
            \section{Conormal estimates for $v$ and $h$}
            \label{sectionconorm2}    
  \subsection{Control given by  the good unknown}
  To perform  our  higher order energy estimates, we shall use  the good unknown
     $ V^\alpha= Z^\alpha  v - \D_{z}v  Z^\alpha  \varphi$.  A crucial point is therefore to establish that
      the control of  this type of  quantity and  $Z^\alpha   h $  yield a control  of all
       the needed quantities.
    
    We shall perform a priori estimates on an interval of time $[0, T^\eps]$ for which we assume that 
         \beq
    \label{apriori}
    \partial_{z} \varphi \geq {c_{0}}, \quad |h|_{2, \infty} \leq {1 \over c_{0}}, \quad g  -  (\partial_{z}^\varphi q^E)_{/z=0} \geq {c_{0}\over 2},  \quad \forall t \in [0, T^\eps]
    \eeq
for some $c_{0}>0$. In the following, we shall use the notation $\Lambda_{0}= \Lambda(1/c_{0})$.
    Note  in particular that  this will allow us to use the Korn inequality recalled in Proposition \ref{Korn}.
    
    Let us introduce a few notations. As we have already used, we set
     $$ V^\alpha= Z^\alpha v - \partial_{z}^\varphi v Z^\alpha h, \, \alpha \neq 0, \quad V^0= v$$
      (and $V^0= v$), 
     and  we shall use the norms:
     $$ \|V^m(t) \|^2 = \sum_{| \alpha| \leq m} \|V^\alpha (t)\|^2, \quad  \| S^\varphi V^m(t) \|^2= \sum_{| \alpha| \leq m } \| S^\varphi V^\alpha(t) \|^2.$$
     By using \eqref{apriori}, we  get that for every $t \in [0, T^\eps]$, we have the equivalence
     \beq
     \label{equiv1} 
      \|v\|_{m} \lesssim   \|V^m\| +  \Lambda({1 \over c_{0}} , \| \nabla v\|_{L^\infty}) |h|_{m-{1 \over 2 }},  \quad \|V^m\| \lesssim  
       \|v\|_{m} +   \Lambda({1 \over c_{0}} , \| \nabla v\|_{L^\infty}  )|h|_{m-{1 \over 2 }}.
     \eeq  

 \subsection{Main estimate}

 \begin{prop}
  \label{conormv}
  Let us define for $t \in [0, T^\eps]$, 
 \beq
 \label{deflambdainfty1}   \Lambda_{\infty}(t)= \Lambda \big( {1 \over c_{0}},  \|v(t)\|_{E^{2, \infty}} + \eps^{1 \over 2} \|\partial_{zz}v\|_{L^\infty}+|h(t)|_{4} + |v(t)|_{E^4} \big),
 \eeq
 then, for every $t \in [0, T]$, a sufficiently smooth solution of \eqref{NSv}, \eqref{bordv1}, \eqref{bordv2} satisfies
  for every $m \geq 0$  the estimate
 \begin{align*}
   &    \|V^m(t)\|^2 + |h(t)|_{m}^2 + \eps |h(t)|_{m+{1 \over 2}}^2  + \eps \int_{0}^t \| \nabla V^m\|^2  \\
     &   \leq \Lambda_{0}\big(\| V^m(0)\|^2 + |h(0)|_{m}^2 + \sqrt{\eps}|h(0)|_{m+{1 \over 2}}^2\big) \\
      & \quad + \int_{0}^t\Lambda_{\infty} \big( 1 +   | (\partial_z \partial_{t}q^E)^b |_{L^\infty}
     \big) \big(\|V^m\|^2   +    |h(t)|_{m}^2 + \eps |h(t)|_{m+{1 \over 2}}^2\big)  +   \int_{0}^T \Lambda_{\infty} \|\partial_{z} v\|_{m-1}^2.
 \end{align*}
     where as before $\Lambda_{0}= \Lambda(1/c_{0})$.
  \end{prop}

  \begin{proof}
  By using the equations \eqref{eqValpha}, \eqref{divValpha} and  the boundary condition \eqref{bordV},  the same integrations by parts as in the proof of Lemma \ref{basicL2}, yield 
  that 
  \beq
  \label{en1}
   {d \over dt} \int_{\mathcal{S}} |V^\alpha |^2 d\V + 4 \eps \int_{\mathcal{S}} |S^\varphi V^\alpha |^2\, d\V \\
    = \mathcal{R}_{S}+ \mathcal{R}_{C}  + 2  \int_{z= 0} \big( 2 \eps S^\varphi V^\alpha  - Q^\alpha \mbox{Id} \big) \N \cdot V^\alpha\, dy  
   \eeq
  where we set
  \begin{align}
  \label{RS} & \mathcal{R}_{S}= 2 \int_{\mathcal{S}} \Big( \eps D^\alpha(S^\varphi v)  + \eps \nabla^\varphi \cdot \big( \mathcal{E}^\alpha (v) \big) \cdot V^\alpha
   \, d\V, \\
   \label{RC} & \mathcal{R}_{C}=- 2 \int_{\mathcal{S}} \Big( \big(\mathcal{C}^\alpha(\mathcal{T})+ \mathcal{C}^\alpha(q)\big) \cdot V^\alpha - \mathcal{C}^\alpha (d) Q^\alpha \Big)\, d\V.
  \end{align}
    
  Let us start with the analysis of the last term in \eqref{en1}. Note that this is the crucial term for which we shall need
   to use the physical condition.  We first notice that if $\alpha_{3} \neq 0$ since $V^\alpha_{/z=0}=0$, 
   this term  vanishes. Consequently, we only need to study the case $\alpha_{3}= 0$.
  By using \eqref{bordV}, we get that
  \begin{align}
  \label{enbor} \int_{z= 0} \big( 2 \eps S^\varphi V^\alpha  - Q^\alpha \mbox{Id} \big) \N \cdot V^\alpha & =  \int_{z=0}\big( -g Z^\alpha h + \D_{z} q\, Z^\alpha h \big)  \N\cdot
   V^\alpha \\
    & \nonumber - \int_{z= 0} \big(2 \eps S^\varphi v - \big( q-gh){\mbox{Id}} \big) Z^\alpha \N \cdot V^\alpha \, dy + \mathcal{R}_{B}
    \end{align}
    where
    \beq
    \label{RB}
    \mathcal{R}_{B}= \int_{z=0} \big( \mathcal{C}^\alpha(\mathcal{B}) - 2 \eps Z^\alpha h\, \D_{z} \big( S^\varphi v\big) \N) \cdot V^\alpha \, dy.
    \eeq
  \end{proof}
  To estimate $\mathcal{R}_{B}$, we  can use \eqref{CalphaB} to get that
  \begin{align*} 
   | \mathcal{R}_{B}| &  \leq \Lambda\big(  {1 \over c_{0}}, |h|_{2, \infty}  +\|v \|_{E^{2, \infty}}
   \big) \big( \eps(1 + \| \partial_{zz} v \|_{L^\infty}) |h|_{m}+   \eps |v^b|_{m}  \big) | (V^\alpha)^b |_{L^2} \\
   &\leq  \Lambda_{\infty} 
  \big( \eps(1 + \| \partial_{zz} v \|_{L^\infty}) |h|_{m}+   \eps |v^b|_{m}  \big) | (V^\alpha)^b |_{L^2}.
   \end{align*}
   Note that in the proof, when we do not specify the time variable, this means that  all the quantities that appear
    are evaluated at time $t$.
  To estimate the second term in the right hand side of  \eqref{enbor}, we first note that 
  $$ \int_{z= 0} \big(2 \eps S^\varphi v - \big( q-gh){\mbox{Id}} \big) Z^\alpha \N \cdot V^\alpha \, dy
  =  \int_{z= 0} \big(2 \eps S^\varphi v -  q^{NS}{\mbox{Id}} \big) Z^\alpha \N \cdot V^\alpha \, dy$$
   since $q= q^E+ q^{NS}$ and  $q^E= gh$ on the boundary. This yields
   \begin{align*}\Big|  \int_{z= 0} \big(2 \eps S^\varphi v -  q^{NS}{\mbox{Id}} \big) Z^\alpha \N \cdot V^\alpha \, dy\Big|
     & \leq |  Z^\alpha \nabla h |_{-{1 \over 2 }} \, | \big(2 \eps S^\varphi v - q^{NS}{\mbox{Id}} \big)(V^\alpha)^b|_{{1 \over 2}}.
      \end{align*}
    Since $q^{NS}= S^\varphi v \n \cdot \n$ on the boundary, we obtain by using \eqref{cont2D} that 
    $$\Big|  \int_{z= 0} \big(2 \eps S^\varphi v -  q^{NS}{\mbox{Id}} \big) Z^\alpha \N \cdot V^\alpha \, dy\Big|
     \leq \Lambda_{\infty}\, \eps |h|_{m+{1 \over 2}} |(V^\alpha)^b |_{1\over 2}. $$
  To  express the first term in the right hand side of \eqref{enbor}, we use  first use the decomposition $q= q^E+ q^{NS}$
   and we write
   $$ \int_{z=0}\big( -g Z^\alpha h + \D_{z} q\, Z^\alpha h \big)  \N\cdot
   V^\alpha =   \int_{z=0}\big( -g Z^\alpha h + \D_{z} q^E\, Z^\alpha h \big)  \N\cdot
   V^\alpha \, dy + \int_{z=0} \D_{z}q^{NS} Z^\alpha h  V^\alpha \cdot \N.$$
For the second term, we get thanks to \eqref{qNSLinfty}   that
\begin{align*} \Big| \int_{z=0} \D_{z}q^{NS} Z^\alpha h  V^\alpha \cdot \N \Big| \leq  \|\partial_{z}^\varphi q^{NS}\|_{L^\infty}
 |h|_{m} |(V^\alpha)^b| &  \lesssim \Lambda_{\infty}\,  \eps |h|_{m}|(V^\alpha)^b| 
  \end{align*}
 For the first term, we use the kinematic boundary condition under the form given by Lemma \ref{lembordh}:
 \begin{align*}  \int_{z=0}\big( -g Z^\alpha h + \D_{z} q^E\, Z^\alpha h \big)  \N\cdot
   V^\alpha \, dy   & = \int_{z=0}\big( -g Z^\alpha h + \D_{z} q^E\, Z^\alpha h \big)\partial_{t} Z^\alpha h
    \\
    &  - \int_{z=0}\big( -g Z^\alpha h + \D_{z} q^E\, Z^\alpha h \big)v^b \cdot (Z^\alpha \nabla_{y}h - \mathcal{C}^\alpha(h))\, dy.
    \end{align*}
    Thanks to an integration by parts,   Proposition \ref{proppE} and  \eqref{bordhC}, we obtain
   \begin{align*} \Big| \int_{z=0}\big( -g Z^\alpha h + \D_{z} q^E\, Z^\alpha h \big)v^b \cdot (Z^\alpha \nabla_{y}h - \mathcal{C}^\alpha(h))\, dy\Big|
  &  \lesssim \|v\|_{1, \infty}( 1 +  \|\D_{z} q^E\|_{1, \infty}) \big( |Z^\alpha h| + | \mathcal{C}^\alpha(h) | \big) \\
  &  \leq   \Lambda_{\infty}\big( |h|_{m} + \|v\|_{E^m})|h|_{m}
  \end{align*}
     and  we finally write 
     $$  \int_{z=0}\big( -g Z^\alpha h + \D_{z} q^E\, Z^\alpha h \big)\partial_{t} Z^\alpha h
      = - {1 \over 2} {d \over dt}  \int_{z=0} ( g- \D_{z} q^E) |Z^\alpha h |^2 - \int_{z=0} \partial_{t}\big( \D_{z} q^E \big) |Z^\alpha h|^2.$$
      Consequently,  gathering all the previous estimates, we have proven from \eqref{enbor} that
      \beq
      \label{enbor2}   \int_{z= 0} \big( 2 \eps S^\varphi V^\alpha  - Q^\alpha \mbox{Id} \big) \N \cdot V^\alpha
       = - {1 \over 2} {d \over dt}  \int_{z=0} ( g- \D_{z} q^E) |Z^\alpha h |^2 +  \tilde{ \mathcal{R}}_{B}\eeq
       where
   \begin{align*}
  | \tilde{\mathcal{R}}_{B}|  &  \leq  \Lambda_{\infty} \Big(  \eps \big(1+   \|\partial_{zz} v \|_{L^\infty} \big)|h|_{m} + \eps |v^b |_{m}\big) |(V^\alpha)^b| \\
  & \quad   + \eps |h|_{m+{1 \over 2}} |(V^\alpha)^b|_{{1 \over 2}} + (1 + |(\partial_{z} \partial_{t} q^E)^b |_{L^\infty})|h|_{m}^2 + \|v\|_{E^m} |h|_{m}\Big)
  \end{align*}
 and we can use successively, \eqref{equiv1},   that
 \beq
 \label{vbm}|v^b|_{m} \leq \Lambda_{\infty}\big( |(V^m)(\cdot, 0)| +|h|_{m}\big), \eeq
 and  the trace inequality \eqref{trace} which yields
 $$ |(V^\alpha)^b|_{L^2}^2\lesssim  \| \partial_{z} V^\alpha \| \|V^\alpha \|, \quad   |(V^\alpha)^b|_{1\over 2 } \lesssim  \|\nabla V^\alpha\| + \|V^\alpha \|$$
  to get that
 \begin{align}
 \label{tildeRB}
   | \tilde{\mathcal{R}}_{B}|  & \leq   \eps \| \nabla V^m\|\, \|V^m\| + \Lambda_{\infty} \|\partial_{z}v\|_{m-1}|h|_{m} \\
  \nonumber & \quad \quad +
     \Lambda_{\infty} \big( 1  + |(\partial_{z} \partial_{t} q^E)^b |_{L^\infty}\big) \big( |h|_{m}^2
   + \eps |h|_{m+{1 \over 2}}^2  + \|V^m \|^2 \big)
   \end{align}
  
 By plugging \eqref{enbor2} into \eqref{en1}, we get that
 \beq
 \label{en2}
 {d \over dt }{1 \over 2 }\Big(
 \int_{\mathcal{S}} |V^\alpha |^2 d\V  +  \int_{z=0} ( g- \D_{z} q^E) |Z^\alpha h |^2 \, dy \Big)
  + 4 \eps \int_{\mathcal{S}} |S^\varphi V^\alpha |^2\, d\V= \mathcal{R}_{S} + \mathcal{R}_{C}+ \tilde{\mathcal R}_{B}  
 \eeq 
 and it remains to estimate the commutators $ \mathcal R_{S}$, $\mathcal{R}_{C}$.
 
 Let us begin with the estimate of $\mathcal{R}_{C}$.
 By using \eqref{Cd}, \eqref{CT} and the Sobolev embedding, we immediately get that
 \beq
 \label{RC1}|\mathcal{R}_{C}|\leq \Lambda_{\infty} \big( \|v\|_{m} +\|\partial_{z} v\|_{m-1} + |h|_{m} \big) \|V^m\| +\Big| \int \mathcal{C}^\alpha(q) \cdot V^m
  d\V \Big|.
 \eeq 
  To estimate the last term, we could use directly \eqref{Cq}.  Neverthess,  the estimate \eqref{Cq} implies that we need to control
   $\| \nabla q \|_{1, \infty}$ and the control of $\| \nabla q^{NS} \|_{1, \infty}$  through Sobolev embedding
    and \eqref{estqNS} would  involve   a  dependence  in  $\eps | v^b|_{s}$ with $s>4$ in the definition of $\Lambda_{\infty}$. 
     We can actually easily  get   a better estimate by using that
    $\mathcal{C}^\alpha (q) = \mathcal{C}^\alpha(q^E) + \mathcal{C}^\alpha(q^{NS})$ and by handling the two terms in two different ways.
    For the Euler pressure, we can use directly \eqref{Cq} and Proposition \ref{proppE}  to get that
    \beq
    \label{CqE}
    \| \mathcal{C}^\alpha(q^E) \| \, \| V^m\| \lesssim \Lambda_{\infty} \big( \|v\|_{E^{m-1}}+ |h|_{m} \big) \| V^m \|.
    \eeq
    To estimate $\mathcal{C}^\alpha(q^{NS})$ we can have a closer look at the structure of the commutator.
     By using the decomposition \eqref{Cialpha},  and \eqref{Calpha2}, \eqref{Calpha3} combined with Proposition \ref{propPNS}, we
     have that
    $$ \| \mathcal{C}^\alpha_{i, 2}(q^{NS})\| +\| \mathcal{C}^\alpha_{i,3}(q^{NS} )\| \leq \Lambda_{\infty} \big( \|v\|_{E^{m-1}}+ |h|_{m} \big).$$
    Consequently, it only remains to study $\mathcal{C}^\alpha_{i,1}(q^{NS})= - [Z^\alpha, {\partial_{i} \varphi \over \partial_{z} \varphi}, \partial_{z} q^{NS}]
   $(the case $i=3$ can be handled by similar arguments). For this term, it is actually better to write it under the form
    $$ \mathcal{C}_{i, 1}^{\alpha}(q^{NS})= - \Big( [Z^\alpha, {\partial_{i}\varphi \over \partial_{z} \varphi } ]\partial_{z} q^{NS} - Z^\alpha \big( {\partial_{i} \varphi
     \over \partial_{z}\varphi}\big) \partial_{z} q^{NS}\Big).$$ 
    We can use the commutator estimate \eqref{com} for the first term and Proposition \ref{propeta} and Proposition \ref{propPNS}
      to get that
  \begin{align*} \| [Z^\alpha, {\partial_{i}\varphi \over \partial_{z} \varphi } ]\partial_{z} q^{NS}\| + \big\| Z^\alpha \big( {\partial_{i} \varphi
     \over \partial_{z}\varphi}\big) \partial_{z} q^{NS}\big|
   &   \leq   \Lambda\big({1 \over c_{0}},  \eps^{-1}\|\partial_{z} q^{NS} \|_{L^\infty}\big)\big( \|\partial_{z} q^{NS}\|_{m-1} +\eps |h|_{m+{1 \over 2 }} \big) 
     \\  &  \leq   \Lambda_{\infty} \big(  \eps   |h|_{m+{1 \over 2 }} + \eps |v^b|_{m+{1 \over  2 } }\big).
   \end{align*}
  This yields
  $$ |\mathcal{R}_{C}|\leq \Lambda_{\infty} \big( \|v\|_{m} + \|\partial_{z}v \|_{m-1} + |h|_{m} + \eps^{1 \over 2 } | h|_{m+{1 \over 2 }}
   + \eps |v^b|_{m+{1 \over 2 }} \big) \|V^m\|.$$
   To estimate the last  term  $ \eps |v^b|_{m+{1 \over 2 }}  \|V^m\|$ in the above estimate, we can  write that
    $$  |v^b|_{m+{1 \over 2 }} \leq  \sum_{|\alpha| \leq m} |V^\alpha|_{1\over 2} + |\partial_{z}^\varphi v Z^\alpha h|_{1\over 2}$$
     and we use \eqref {cont2D} and the trace estimate \eqref{trace} to get that 
     \beq
     \label{vbm+}  |v^b|_{m+{1 \over 2 }} \leq \| \nabla V^m\| + \|V\|_{m} + \Lambda_{\infty} |h|_{m+{1 \over 2 }}.\eeq
   Hence,   we find
   \beq
   \label{RCfinal}
    |\mathcal{R}_{C}|\leq   \eps \|\nabla V^m\|\, \|V^m\| +\Lambda_{\infty} \|\partial_{z}v \|_{m-1} \|V\|_{m} + \Lambda_{\infty}\big( \|V^m\|^2 + |h|_{m}^2
     + \eps |h|_{m+{1 \over 2}}^2\big).
     \eeq
     It remains to estimate $\mathcal{R}_{S}$ which is given by \eqref{RS}. For the second term, we use an integration by parts:
     $$ \int_{\mathcal{S}} \eps \nabla^\varphi( \mathcal{E}^\alpha(v) \big) \cdot V^\alpha d \V=
      - \int_{\mathcal{S}} \eps\, \mathcal{E}^\alpha(v) \cdot \nabla V^\alpha d \V+ \int_{z=0}\eps\,  \mathcal{E}^\alpha(v) \N \cdot V^\alpha \, dy.$$
      The estimate of the first term comes directly from \eqref{CT}, we get
      $$ \Big| \int_{\mathcal{S}} \eps\, \mathcal{E}^\alpha(v) \cdot \nabla V^\alpha d \V \Big| \leq \Lambda_{\infty}
      \big( \|v\|_{E^m}+  |h|_{m} \big)  \eps\, \|\nabla V^\alpha \|.$$
      To estimate the boundary term, we, need an estimate of $\mathcal{E}^\alpha(v)$ on the boundary.  The same arguments
       yielding the proof of \eqref{CT} can be used by using the commutator estimates on the boundary to get that
       $$ | \mathcal{E}^\alpha(v)(\cdot, 0)|_{L^2} \leq \Lambda_{\infty}\big(  |h|_{m}+ |v^b|_{m} +|\partial_{z} v^b|_{m-1} \big)$$
       and hence  by using \eqref{dzvb}, we find
       $$ | \mathcal{E}^\alpha(v)(\cdot, 0)|_{L^2} \leq \Lambda_{\infty}\big(  |h|_{m}+ |v^b|_{m}  \big).$$
       Consequently, we have proven that
       $$  \Big| \int_{\mathcal{S}} \eps \nabla^\varphi( \mathcal{E}^\alpha(v) \big) \cdot V^\alpha d \V \Big|
        \leq \Lambda_{\infty} \Big(  \big(  \|v\|_{E^m}+  |h|_{m} \big)  \eps\, \|\nabla V^\alpha \| +  \big(  |h|_{m}+ |v^b|_{m}  \big) \eps\,  |(V^\alpha)^b| \Big).
        $$
        and by using again the trace inequality, we find that
        \beq
        \label{RS1}
        \Big| \int_{\mathcal{S}} \eps \nabla^\varphi( \mathcal{E}^\alpha(v) \big) \cdot V^\alpha d \V \Big|
        \leq \Lambda_{\infty}  \eps \| \nabla V^m \|\big( \|V^m \| + \|\partial_{z}v\|_{m-1}+  |h|_{m} \big)        \eeq
        It remains  to estimate  $\eps \int_{\mathcal{S}} \mathcal D^\alpha(S^\varphi v) \cdot V^\alpha d \V$
        This is given in the following lemma:
      \begin{lem}
      \label{comDS}
      We have the estimate:
       \begin{align*}
       \eps\Big| \int_{\mathcal{S}} \mathcal D^\alpha(S^\varphi v) \cdot V^\alpha d \V\Big| \leq 
       \Lambda_{\infty} \Big( \eps  \|\nabla V^m\| \big(\|V^m \| + \|\partial_{z}v \|_{m-1}+ |h|_{m+{1 \over2}}\big) + \eps\big( \|v\|_{E^m}^2 + |h|_{m+{1 \over 2}}^2 \big) \\ 
        + \eps \|\partial_{zz} v\|_{L^\infty}\big( |h|_{m}^2 + \| V^m\|^2\Big).
        \end{align*}
      \end{lem}
      We shall first finish the proof of Proposition  \ref{conormv} and then prove Lemma \ref{comDS}.
      
     From  \eqref{en2}, we can sum over $\alpha$ (for $\alpha=0$, we use the basic estimate \ref{corL2}),  integrate in time, use our a priori assumption \eqref{apriori} on the Taylor condition  and use 
     \eqref{RCfinal}, \eqref{tildeRB} and Lemma \ref{comDS} to get that
     \begin{align}
     \label{presquefini}
     &   \|V^m(t) \|^2 + |h(t)|_{m}^2 + \int_{0}^t \| S^\varphi V^m \|^2 
      \leq \Lambda_{0}\big(
       \|V^m(0)\|^2 + |h(0)|_{m}^2\big)  \\ 
   \nonumber  & \quad     + \int_{0}^t \Big(
       \eps  \, \Lambda_{\infty} \|\nabla V^m\| \big(\|V^m \| + \|\partial_{z}v \|_{m-1}+ |h|_{m+{1 \over2}}\big)  \\
  \nonumber     &    \quad \quad   \quad     +  \Lambda_{\infty}(1 +  |(\partial_{z}\partial_{t}q^E)^b |_{L^\infty}\big) \big( |h|_{m}^2 + \| V^m\|^2
        + \eps |h|_{m+{1 \over 2}}^2\big) +    \Lambda_{\infty} \|\partial_{z}v \|_{m-1}^2 
     )\Big).       
     \end{align}
      Now,  we can  use Proposition \ref{mingrad} and the Korn inequality of Proposition \ref{Korn} to get that
      $$\|\nabla V^m\| \leq \Lambda\big({1\over c_{0}}\big) \big( \| S^\varphi V^m\| + \|v\|_{m}\big)$$
     and the estimate of Proposition \ref{conormv} follows from \eqref{presquefini} and the Young inequality \eqref{young}.
 This ends  the proof of Proposition \ref{conormv}.
 
 \bigskip
 
       It remains the: 
      \subsubsection*{Proof of Lemma \ref{comDS}}
       By using \eqref{Ddef}, we actually have to estimate
       \begin{align}
     \nonumber \mathcal{R}_{Si}  & =  \eps \int_{\mathcal{S}} \mathcal C^\alpha_{j}(S^\varphi v)_{ij}  V^\alpha_{j} d \V \\
   \label{RSi}     & = 
      \eps \int_{\mathcal{S}} \mathcal C^\alpha_{j, 1}(S^\varphi v)_{ij}  V^\alpha_{j} d \V 
      + \eps \int_{\mathcal{S}} \mathcal C^\alpha_{j, 2}(S^\varphi v)_{ij}  V^\alpha_{j} d \V  
       +  \eps \int_{\mathcal{S}} \mathcal C^\alpha_{j, 3}(S^\varphi v)_{ij}  V^\alpha_{j} d \V \\
        &: = \mathcal{R}_{Si}^1 + \mathcal{R}_{Si}^2 + \mathcal{R}_{Si}^3
         \end{align}
       where we have used the decomposition \ref{Cialpha}. For the first term, by using the definition of the symmetric commutator
        we see that we need to estimate terms like
        $$\eps  \int_{\mathcal{S}}  Z^\beta \big( {\partial_{j} \varphi \over \partial_{z} \varphi}\big) \big(Z^{\tilde{\gamma}} \partial_{z}(S^\varphi v)_{ij} \big)V^\alpha_{j} d \V$$
         where $\beta $ and $\tilde\gamma$ are such that $\beta \neq 0, \, \tilde \gamma \neq 0$ and $|\beta | + |\tilde \gamma|=m.$
          By using \eqref{idcom}, we can reduce the problem to the estimate of 
       $$ \eps  \int_{\mathcal{S}}  c_{\gamma} Z^\beta \big( {\partial_{j} \varphi \over \partial_{z} \varphi}\big)  \partial_{z}\big( Z^\gamma (S^\varphi v)_{ij} \big)V^\alpha_{j} d \V$$
       with $\beta$ as before (thus $| \beta | \leq m-1$) and $| \gamma | \leq | \tilde \gamma|\leq m-1.$ By using an integration by parts, we are lead to the estimate
        of  three types of terms:
        $$  \mathcal{I}_{1}=  \eps\int_{\mathcal{S}} Z^\beta \big( {\partial_{j} \varphi \over \partial_{z} \varphi } \big)\,Z^\gamma (S^\varphi v)_{ij}\, 
        \partial_{z} V_{j}^\alpha d \V, \quad \mathcal{I}_{2}= 
      \eps\int_{\mathcal{S}} \Big( \partial_{z}Z^\beta \big( {\partial_{j} \varphi \over \partial_{z} \varphi }  \big) \Big)Z^\gamma (S^\varphi v)_{ij} V_{j}^\alpha d \V, \quad $$
       and the term on the boundary
       $$ \mathcal{I}_{3}= \eps \int_{y=0} Z^\beta \big( {\partial_{j} \varphi \over \partial_{z} \varphi } \big)\, Z^\gamma (S^\varphi v)_{ij} V_{j}^\alpha dy.$$
       For the first term,  since $\beta \neq 0$, we get by using again \eqref{gues}, \eqref{quot} and Proposition \ref{propeta} that
       $$| \mathcal{I}_{1}| \leq  \eps \| \nabla V^m\|\, \Lambda_{\infty}\Big(   \|  v\|_{m} + \| \partial_{z}v\|_{m-1}
         +  |h|_{m+{1 \over 2}} \Big).$$
         To estimate $\mathcal{I}_{2}$, we can first use \eqref{idcom} to  expand it as a sum of terms under the form
         $$ \tilde{\mathcal{I}}_{2} =  \eps\int_{\mathcal{S}} \Big( c_{\tilde \beta }Z^{\tilde \beta} \partial_{z} \big( {\partial_{j} \varphi \over \partial_{z} \varphi }  \big) \Big)Z^\gamma (S^\varphi v)_{ij} V_{j}^\alpha d \V$$
         with $|\tilde{\beta}|\leq \beta.$ If $\gamma  = 0$, since $|\tilde \beta | \leq m-1$,  we just write
         $$|\tilde{\mathcal{I}}_{2}| \leq  \eps  \big\| \partial_{z} \big(  {\partial_{j} \varphi \over \partial_{z} \varphi }\big)\big\|_{m-1} \|S^\varphi\|_{L^\infty}
          \, \| V_{j}^\alpha \|$$
          while for for $\gamma \neq 0$, we use \eqref{gues}  to get
          $$  |\tilde{\mathcal{I}}_{2}| \leq  \eps \big\|   \partial_{z} \big(  {\partial_{j} \varphi \over \partial_{z} \varphi }\big)\big\|_{L^\infty} \|S^\varphi v\|_{m}
          \, \| V_{j}^\alpha \| + \eps \| S^\varphi v \|_{1, \infty} \big\|   \partial_{z} \big(  {\partial_{j} \varphi \over \partial_{z} \varphi }\big)\big\|_{m-1}.$$
          Consequently, by using \eqref{quot} and Proposition \ref{propeta}, we obtain
          $$ \| \mathcal{I}_{2}\| \leq \Lambda_{\infty} \, \|V^m\| \big( \eps \| S^\varphi v\|_{m} + \eps \|h\|_{m+{1 \over 2}} \big).$$
          To conclude, we need to relate $\| S^\varphi v\|_{m}$ to the energy dissipation term. By using \eqref{com1} and
           Lemma \ref{comi} combined with the identity  \eqref{al11}, we get that
          $$ \| S^\varphi v\|_{m} \leq \|S^\varphi V^m \| +  \Lambda_{\infty}\big( 1 + \| \partial_{zz}v\|_{L^\infty}\big) |h|_{m} +
           \Lambda_{\infty}\big(\| \nabla v\|_{m-1} + |h|_{m-{1\over 2}}\big) $$
           and hence we finally obtain that
          $$  | \mathcal{I}_{2}| \leq \Lambda_{\infty} \, \|V^m\| \big( \eps \| S^\varphi V^m\| + \eps \|h\|_{m+{1 \over 2}} + \eps(1+ \|\partial_{zz}v \|_{L^\infty})
           |h|_{m} + \|v\|_{E^m}) \big).$$
         It remains to estimate $\mathcal{I}_{3}$. By using product estimates on the boundary we first  find by using that $\beta \neq 0$ that
         $$ | \mathcal{I}_{3}| \leq \Lambda_{\infty}\, \eps\big( |h|_{m+{1 \over 2 }} +  |(S^\varphi v)^b|_{m-1} \big)|(V^\alpha)^b|\leq 
       \Lambda_{\infty}\, \eps\big( |h|_{m+{1 \over 2 }} +  |v^b|_{m} \big)|(V^\alpha)^b|   $$
       where we have used \eqref{dzvb} in the second estimate. By using  again \eqref{vbm} and the trace inequality, this yields
       $$ | \mathcal{I}_{3}| \leq \Lambda_{\infty} \Big( \eps  \| \nabla V^m\|\, \|V^m\| + \eps   |h|_{m+{1 \over 2}}^2  + \eps  \|v\|_{E^m}^2\big).$$
       Consequently,  we get from  the previous estimates    that
       \beq
       \label{RSi1}
       | \mathcal{R}_{Si}^1 | \leq \Lambda_{\infty} \Big( \eps \| \nabla V^m \| \big(\|V^m\|+ |h|_{m+{1 \over2}}\big) + \eps\big( \|v\|_{E^m}^2 + |h|_{m+{1 \over 2}}^2 \big)
        + \eps \|\partial_{zz} v\|_{L^\infty}( |h|_{m}^2+ \|V^m\|^2\Big).
        \eeq
        The estimate of $\mathcal{R}_{Si}^2$ is straightforward, from the definition after \eqref{Cialpha}, we immediately get that
        \beq
        \label{RSi2}
      | \mathcal{R}_{Si}^2 | \leq \eps \Lambda_{\infty}(1 + \| \partial_{zz} v \|_{L^\infty}) \|V^m\| \, |h|_{m-{1 \over 2}}.
        \eeq
        It remains to estimate $\mathcal{R}_{Si}^3$. By using the definition of $\mathcal{C}_{i, 3}^{\alpha}(S^\varphi v)$,
       we  get by using again \eqref{idcom} that
        $$
        \eps \Big| \int_{\mathcal{S}}  {\partial_{i}\varphi \over (\partial_{z}\varphi)^2} \partial_{z}\big( S^\varphi v)  [Z^\alpha, \partial_{z}]\varphi d\V\Big| \leq \eps \Lambda_{\infty} ( 1 + \|\partial_{zz}v\|_{L^\infty}) |h|_{m-{1 \over 2}} \|V^m\|.
        $$
         For the term
         $$\eps \int_{\mathcal{S}}{\partial_{i}\varphi \over \partial_{z} \varphi}  V^\alpha \, [Z^\alpha, \partial_{z}](S^\varphi v)  d\V, $$
         we perform an integration by parts as in the estimate of $\mathcal{R}_{Si}^{1}$. We obtain by similar arguments that
         $$ \Big|  \int_{\mathcal{S}}{\partial_{i}\varphi \over \partial_{z} \varphi}  V^\alpha \, [Z^\alpha, \partial_{z}](S^\varphi v)  d\V\Big| \leq 
          \Lambda_{\infty}\, \eps \big( \|V^m\| + \| \nabla V^m \|\big) \|v\|_{E^m}.$$      
       From the two last estimates and \eqref{RSi1}, \eqref{RSi2}, we finally obtain that
       \begin{align*} |\mathcal{R}_{Si}| \leq 
       \Lambda_{\infty} \Big( \eps  \|\nabla V^m\| \big(\|V^m \| + \|\partial_{z}v \|_{m-1}+ |h|_{m+{1 \over2}}\big) + \eps\big( \|v\|_{E^m}^2 + |h|_{m+{1 \over 2}}^2 \big) \\ 
        + \eps \|\partial_{zz} v\|_{L^\infty}\big( |h|_{m}^2 + \| V^m\|^2\Big).
        \end{align*}
        This end the proof of Lemma \ref{comDS}.
       
    \section{Normal derivative estimates part I}
    \label{sectionnorm1}
    
    In order to close the estimate of Proposition \ref{conormv}, we need to estimate $\|\partial_{z}v \|_{m-1}^2$.
         For the normal component of $v$, this is given for free, we have
     \begin{lem}
     \label{lemvnorm}
     For every $m \geq 1$, we have
     \beq
     \label{vnorm}
    \| \partial_{z}v \cdot \n \|_{m-1} \leq \Lambda\big( {1 \over c_{0},} \|\nabla v\|_{L^\infty}\big)\big( \|V^m\| + |h|_{m-{1 \over 2}}\big)
    \eeq
    where $\N= (- \partial_{1} \varphi, -\partial_{2} \varphi, 1)^t$, $\n= \N/|\N|$.
        \end{lem}
        Note that in the above definition $\N$ is defined in the whole $\mathcal{S}$ via $\varphi$.
        \begin{proof}
       From \eqref{NSv}, the divergence free condition yields
        \beq
        \label{dzVid}
        \partial_{z}v \cdot \N= \partial_{z}\varphi \big( \partial_{1} v_{1} + \partial_{2} v_{2} \big)
        \eeq
        and hence the estimate follows from \eqref{gues} and  Proposition \ref{propeta}.
        
        \end{proof}
        
        The next step is to estimate the tangential components of $\partial_{z} v $.
     We shall proceed in two steps. As a first step, we shall estimate   $\sup_{[0, T]} \| \partial_{z} v \|_{m-2}$. 
      This   estimate  
       will be important  in order  to control the $L^\infty$ norms that occur in the definition of  $\Lambda_{\infty}$ in 
        in Proposition \ref{conormv}  in terms of known quantities  via Sobolev embeddings and $L^\infty$ estimates. We shall prove in the end 
         of the paper the  more delicate estimate  which allows to control  
       $\int_{0}^T \| \partial_{z} v \|_{m-1}^2.$
              
       We begin with the following remark:
       \begin{lem}
       \label{lemdzS}
For every $k \geq 0$, we have
       \begin{eqnarray}
       \label{dzvk} \| \partial_{z} v\|_{k} \leq \Lambda\big( {1 \over c_{0}}, \|\nabla v\|_{L^\infty}\big) \big( \| S_{\n}\|_{k} + |h|_{k+{1 \over 2 }} +  \| v\|_{k+1}\big),  \\
       \label{dzzvk} \|\partial_{zz} v \|_{k} \leq \Lambda\big( {1 \over c_{0}}, |v|_{E^{2, \infty}}\big)\big( \|\nabla S_{\n} \|_{k} + |h|_{k+{ 3 \over 2} } + \|v\|_{k+2}\big)
       \end{eqnarray}
       where
       \beq
       \label{Sn} S_{\n}= \Pi S^\varphi v \, \n, \quad  \Pi = Id- \n\otimes\n\eeq 
       and $\n$ is defined by  $\n= \N/|\N|.$
       \end{lem}
        
       The main consequence of  \eqref{dzvk} is that for $k \leq m-1 $, since $\|v\|_{k+1}$ is estimated in Proposition \ref{conormv}, we can look for an estimate of  $\|S_{\n}\|_{k}$
        instead of $\| \partial_{z}v \|_{k}$.
      \begin{proof}
      For the proof of Lemma \ref{lemdzS}, we can combine the identity \eqref{Su} which can be written in the equivalent form 
    \beq
    \label{Subis}
      2 S^\varphi v \ n= \partial_{n}u + g^{ij} \big(\partial_{j} v \cdot \n\big) \partial_{y^i}
      \eeq
      and the fact that
      \beq
      \label{Subis1} \partial_{\N} u = {1 + |\partial_{1}\varphi|^2 + |\partial_{2} \varphi |^2 \over \partial_{z}\varphi} \partial_{z} v - \partial_{1}\varphi \partial_{1}v
       - \partial_{2} \varphi \partial_{2} v\eeq
       to obtain
       $$ \| \partial_{z} v\|_{k} \leq \Lambda\big( {1\over c_{0}}, \|\nabla v\|_{L^\infty}\big)\big( \|V^{k+1}\| + |h|_{k+{1 \over 2 }} + \| \partial_{z}v \cdot \n \|_{k} + \|S^\varphi v\, \n \|_{k}\big). $$
   Recall  that  we have the estimate $|h|_{2,\infty} \leq 1/c_{0}$ thanks to \eqref{apriori}.
       To estimate $  \| \partial_{z}v \cdot \n \|_{k}$, we use   Lemma \ref{lemvnorm} and to estimate the last term, we note that
       \beq
       \label{Pisu} S^\varphi v \, \n= S_{\n} + \big( S^\varphi v \, \n \cdot \n \big)\, \n= S_{\n}+  D_{\n}^\varphi v\,  \cdot \n \, \n\eeq
        and we use Lemma \ref{lemvnorm}. This yields \eqref{dzvk}.
        
        To obtain \eqref{dzzvk}, we first observe that by using again \eqref{dzVid}, we obtain that
        $$ \| \partial_{zz}v \cdot \n \|_{k} \leq \Lambda\big( {1 \over c_{0}}, |v|_{E^{2, \infty}}\big)\big(  |h|_{k+{3 \over 2}} + \|v\|_{k+1} +  \|\partial_{z}v\|_{k+1}\big)$$
         and hence, we find by using \eqref{dzvk} that
         \beq
         \label{dzzvk1}
         \| \partial_{zz}v \cdot \n \|_{k} \leq  \Lambda\big( {1 \over c_{0}}, |v|_{E^{2, \infty}}\big)\big(  |h|_{k+{3 \over 2}} + \|v\|_{k+2} +  \| S_{\n}\|_{k+1}\big).
         \eeq
       To get  the estimate of the  other components of $\partial_{zz}v$, it suffices  to  use again \eqref{Subis}, \eqref{Subis1}
       combined with the previous estimates. This ends the proof of Lemma \ref{lemdzS}.

      \end{proof}
      
     The aim of the following Proposition is to prove the  estimate on $S_{\n}$.
     Let us  recall that $\Lambda_{\infty}$ is defined  in  \eqref{deflambdainfty1} in Proposition \ref{conormv} and that $\Lambda_{0}= \Lambda(1/c_{0}).$
       \begin{prop}
       \label{propdzvm-2}
       For every $m \geq 2$, the solution of \eqref{NSv}, \eqref{bordv1}, \eqref{bordv2} satisfies,   
 for every $t \in [0, T^\eps]$,  the estimate
         \begin{align}
   \nonumber 
   &  \|S_{\n}\|_{m-2}^2    + \eps \int_{0}^t  \| \nabla S_{\n}\|_{m-2}^2 \\  
 \label{estSnm-2} & \leq  \Lambda_{0}\|S_{\n}(0)\|_{m-2}^2 + \int_{0}^t\Lambda_{\infty}\big(  \|V^m\|^2 + |h|_{m}^2  + \eps |h|_{m+{1 \over 2 }}^2
  + \|S_{\n}\|_{m-2}^2\big) \\
  \nonumber  &  + \int_{0}^t \Lambda_{\infty} \|\partial_{z} v\|_{m-1}^2 + \eps \int_{0}^t \|\nabla V^m\|^2.
 \end{align}  
         \end{prop}
         Note that the last term on the right-hand side of  \eqref{estSnm-2}
 can  be estimated  by using Proposition \ref{conormv}.
   We also point out  that in the above estimate, because of the dependence in $|h|_{m}$ in the right-hand side,  we cannot change $m-2$ into $m-1$.
     The occurence of this term  is mainly due to the Euler part of the pressure when we study the equation for $S_{\n}$.
      Indeed,  when we perform energy estimates on   the   equation for $S_{\n}$, the Hessian $D^2p^E$ of the pressure
        is handled as a source term.  Since $p^E= gh$ on the boundary, we get that the estimate of   $\|D^2 p^E\|_{k}$ necessarily involves
         $ |h|_{k+{3 \over 2}}$. This is why we are restricted to $k \leq m-{3/2}<m-1$. 
         
    \begin{proof}
         The first step is to find an equation for $S_{\n}$.  From \eqref{NSv}, we find the equation
    $$ \D_{t} \nabla^\varphi v+ \big(v \cdot \nabla^\varphi\big) \nabla^\varphi v  + (\nabla^\varphi v)^2 + (D^\varphi)^2 q  -\eps \Delta^\varphi \nabla^\varphi v = 0$$
     where $(D^\varphi)^2q$ stands for the Hessian matrix of the pressure,  $((D^\varphi)^2q)_{ij}= \D_{i}\D_{j}q$ and hence
      by taking the symmetric part of the equation, we get
     \beq
     \label{eqS}
      \D_{t} S^\varphi v+ \big(v \cdot \nabla^\varphi\big) S^\varphi v  + { 1\over 2}\big((\nabla^\varphi v)^2+ ((\nabla^\varphi v)^t)^2 +  (D^\varphi)^2 q  -\eps \Delta^\varphi (S^\varphi v)= 0.
     \eeq
     This allows to get an evolution equation for $S_{\n}= \Pi \big( S^\varphi v\, \n)$:
     \beq
     \label{eqPiS}
     \D_{t}  S_{\n}+ \big(v \cdot \nabla^\varphi\big) S_{\n}  -\eps \Delta^\varphi (S_{\n})= F_{S}
     \eeq
     where the source term $F_{S}$ is given by 
     \beq
     \label{FSdef} F_{S}= F_{S}^1 + F_{S}^2, 
     \eeq
     with 
     \begin{align}
  \label{FS1} &  F_{S}^1 =   - {1 \over 2} \Pi\big( (\nabla^\varphi v)^2+ ((\nabla^\varphi v)^t)^2\big)\, \n + \big( \partial_{t} \Pi + v \cdot \nabla^\varphi \Pi \big)
      S^\varphi v \, \n +    \Pi S^\varphi v \big( \partial_{t} \n + v \cdot \nabla^\varphi \n\big)  \\ 
    \label{FS2}   & F_{S}^2 =   - 2 \eps \, \D_{i} \Pi\, \D_{i} \big( S^\varphi v \, \n \big)  -  2 \eps  \Pi \big( \D_{i}\big( S^\varphi v \big) \, \D_{i} \n \big) - \eps \big(\Delta^\varphi \Pi  \big)S^\varphi v \, \n  - \eps  \Pi S^\varphi v \Delta^\varphi \n \\
\nonumber    & \quad \quad \quad    - \Pi\big( (D^\varphi)^2 q \big) \n.
     \end{align} 
     Note that we have used the summation convention for the last term. By using   Proposition \ref{propeta} and
       Propositions \ref{proppE}, \ref{propPNS}, and the product estimates \eqref{gues}, we have that the source term $F_{S}^1$
        is  bounded by 
   $$
    \| F_{S}^1 \|_{m-2} \leq \Lambda_{\infty}\big( \|\nabla^\varphi v \|_{m-2}  + |h|_{m-{1 \over 2}} + \|v\|_{m-2} \big).
 $$
 We recall that the above notation implicitely means that all the quantities are evaluated at time $t$;
    Note that by using  Lemma \eqref{dzvk}, we can rewrite this estimate as
   \beq
     \label{estFS}
    \| F_{S}^1 \|_{m-2} \leq \Lambda_{\infty}\big( \|S_{\n}  \|_{m-2}  + |h|_{m- {1 \over 2}} + \|v\|_{m-1} \big).
    \eeq 
  In  a  similar way, we get  for $F_{S}^2$ that 
  $$
   \| F_{S}^2 \|_{m-2} \leq \Lambda_{\infty}\,  \eps   \big(   \|\partial_{zz}v\|_{m-2} + \|\partial_{z}v\|_{m-1} + \|v\|_{m} + |h|_{m+{1 \over 2}}
    \big) + \Lambda_{\infty}\|\nabla q\|_{E^{m-1}}.
 $$
  Indeed, let us give more details on one  example:
$$ \| \eps \, \D_{i} \Pi\, \D_{i} \big( S^\varphi v \, \n \big) \|_{m-2} \lesssim 
 \eps \|\D_{i} \Pi\|_{L^\infty}\, \| \D_{i}\big(S^\varphi v \, \n \big) \|_{m-2} +  \eps \|\partial_{i}^\varphi\big(S^\varphi v \, \n \big) \|_{L^\infty} \,  \|\D_{i} \Pi\|_{m-2}$$
 and the result follows by using \eqref{gues} again (note that $\Lambda_{\infty}$ involves
  $\eps^{1 \over 2} |h|_{3, \infty}$  and $\eps^{1 \over 2} \|\partial_{zz}v\|_{L^\infty}.$) Actually, we get an estimate of $F_{S}^2$
   that involves  only $\eps \|\partial_{zz}v\|_{L^\infty}$, but this improvement does not seem useful.
   Next, we can use Proposition \ref{propPNS} and Proposition \ref{proppE} to estimate the pressure.
    This yields
    $$ \| F_{S}^2 \|_{m-2} \leq \Lambda_{\infty}\,  \eps \big(  \|\partial_{zz}v\|_{m-2}  
    + |h|_{m+{1 \over 2}} +|v^b|_{m+{1 \over 2}} \big) + \Lambda_{\infty}\big( |v|_{E^m} + |h|_{m - {{1 \over 2}}}\big)$$
    and by using Lemma \ref{lemdzS} and Lemma \ref{mingrad}, we can rewrite this last estimate in the alternative form
 \beq
  \label{estFS2}
    \| F_{S}^2 \|_{m-2} \leq \Lambda_{\infty}\,  \eps \big(  \|\nabla^\varphi S_{\n}\|_{m-2}  
    + |h|_{m+{1 \over 2}} +|v^b|_{m+{1 \over 2}} \big) + \Lambda_{\infty}\big( |v|_{E^m} + |h|_{m - {{1 \over 2}}}\big).
    \eeq

%
     Note that thanks to the boundary condition \eqref{bordv2}, we have that $S_{\n}$ verifies an homogeneous Dirichlet boundary condition
      on $z=0$, 
      \beq
      \label{bordSN}
       (S_{\n})_{/z=0}= 0
       \eeq
     hence we shall be able to estimate $S_{\n}$ through standard energy estimates. 
   We shall first prove for $m \geq 2$  the following estimate by induction:
   \begin{align}
   \label{estSnm-2pf}
   &  \|S_{\n}(t)\|_{m-2}^2    + \eps \int_{0}^t \| \nabla^\varphi S_{\n}|_{m-2}^2  \\  
 \nonumber  & \leq  \Lambda_{0}\|S_{\n}(0)\|_{m-2}^2 + \int_{0}^t\Lambda_{\infty}\big( \|v\|_{E^m}^2 + |h|_{m-{1 \over 2}}^2 + \eps|h|_{m+{1 \over 2 }}^2 + \|S_{\n}\|_{m-2}^2\big)+ \eps\int_{0}^t |v^b|_{m+{1 \over 2} } \|S_{\n}\|_{m-2}.
 \end{align}

      For the $L^2$ estimate (which corresponds to $m=2$),  by using the boundary condition \eqref{bordSN} and Lemma \ref{lemipp}, we obtain 
 $$   {d \over dt} {1 \over 2} \int_{\mathcal{S}}  | S_{\n}|^2 \, d\V    + \eps \int_{\mathcal{S}} | \nabla^\varphi S_{\n}|^2 \, d\V
 = \int_{\mathcal{S}}   F_{S}  \cdot  S_{\n}\, d\V.$$
 To conclude, we integrate in time and we use
   \eqref{estFS}, \eqref{estFS2} and the Young inequality to absorb the term $\| \nabla^\varphi S_{\n}\|$ in the left hand side. Note that
    thanks to \eqref{apriori}, we can again use that
    \beq
    \label{remS0}  \| \nabla^\varphi S_{\n}(t) \|^2 \leq \Lambda_{0} \int_{\mathcal{S}} |\nabla^\varphi S_{\n} |^2 \, d\V.
    \eeq

 Assuming that the estimate \eqref{estSnm-2pf} is proven for $k \leq m-3$, we shall prove that it is verified for $m-2$. 
     We first need to compute the equation
      satisfied by $Z^\alpha S_{\n}$ for $| \alpha | \leq m-2.$ By using the expression \eqref{transportW} for the transport part
       of the equation, we get that
       \beq
       \label{SNalpha}
       \D_{t}  Z^\alpha S_{\n}+ \big(v \cdot \nabla^\varphi\big)Z^\alpha S_{\n}  -\eps \Delta^\varphi Z^\alpha (S_{\n})= F_{S} + \mathcal{C}_{S}
       \eeq
       where  the commutator is given by 
       \beq
       \label{CS}
       \mathcal{C}_{S}= \mathcal  C_{S}^1 + \mathcal C_{S}^2\eeq
       with
       $$\mathcal{C}_{S}^1= [Z^\alpha v_{y}]\cdot  \nabla_{y} S_{\n} +  [Z^\alpha, V_{z}] \partial_{z} S_{\n}:=
        C_{Sy}+ C_{Sz}, \quad \mathcal{C}_{S}^2 =  - \eps [Z^\alpha, \Delta^\varphi] S_{\n}.
       $$
       Since $Z^\alpha S_{\n}$ vanishes on the boundary, we get by using Corollary \ref{coripp} and  a standard energy estimate 
        that
        \beq
        \label{m-21}
{d \over dt} {1 \over 2} \int_{\mathcal{S}}  |Z^\alpha S_{\n}|^2 \, d\V    + \eps \int_{\mathcal{S}} | \nabla^\varphi Z^\alpha S_{\n}|^2 \, d\V
 = \int_{\mathcal{S}}  \big( F_{S} + \mathcal{C}_{S} \big) \cdot Z^\alpha S_{\n}\, d\V
 \eeq
     and we need to estimate the right hand-side. We shall begin with the part involving $\mathcal{C}_{S}$.
      Thanks to the decomposition \eqref{CS}, we first observe that thanks  to  \eqref{com}
      \beq
      \label{CSy} \|C_{Sy}\|
      \leq \Lambda_{\infty}\big(  \|S_{\n}\|_{m-2} +  \|v\|_{E^{m-2}} \big).\eeq
      To estimate $\mathcal{C}_{Sz}$,  we need to study more carefully the commutator in order to avoid the appearance     of  $\| \partial_{z}  S_{\n}\|_{k}$
       which is  not controlled (and cannot be controlled uniformly in $\eps$  due to the presence of boundary layers). 
        By expanding the commutator and by using \eqref{idcom}, we see that we have to estimate terms like
        $$ \| Z^\beta V_{z} \, \partial_{z} Z^\gamma S_{\n}  \|$$
        with $ | \beta | +  | \gamma | \leq m-2$, $ | \gamma |  \leq m-3$. Next, we shall use that
        $$ Z^\beta V_{z} \, \partial_{z} Z^\gamma S_{\n}= {1 - z \over z } Z^\beta V_{z} \, Z_{3} Z^\gamma S_{\n}$$
        and we can finally  rewrite  the  above expression as a sum of terms under the form:
        \beq
        \label{CSz0} c_{\tilde \beta }Z^{\tilde \beta }\big(  {1- z  \over {z }} V_{z} \big)   \, Z_{3} Z^\gamma S_{\n}\eeq
        where $c_{\tilde \beta}$ are harmless bounded functions and $|\tilde \beta | \leq | \beta |$.  Indeed, this comes from
         the fact that $ Z_{3} ( (1-z)/z ) = \tilde{c} (1-z)/z$ for $\tilde c$ a smooth bounded function. 
        
        If $\tilde \beta = 0$, we  get 
                 that 
   $$   \left\| c_{\tilde \beta }Z^{\tilde \beta }\big(  {1- z  \over {z }} V_{z} \big)   \, Z_{3} Z^\gamma S_{\n}  \right\|
       \lesssim   \|S_{\n}\|_{m-2}$$
        
        When $\tilde \beta \neq 0$ we can use \eqref{gues} to obtain that 
        $$  \left\| c_{\tilde \beta }Z^{\tilde \beta }\big(  {1- z  \over {z }} V_{z} \big)   \, Z_{3} Z^\gamma S_{\n}  \right\|
         \lesssim \big\|  Z \big(  {1- z  \over {z }} V_{z} \big) \big\|_{L^\infty} \|S_{\n} \|_{m-2} +  \|S_{\n}\|_{L^\infty} 
          \big\|  Z \big(  {1- z  \over {z }} V_{z} \big) \big\|_{m-3}$$
        Next, we observe that         $$   \big\|  Z \big(  {1- z  \over {z }} V_{z} \big) \big\|_{L^\infty}  \lesssim \|V_{z} \|_{W^{1, \infty}} +  \|ZV_{z}\|_{W^{1, \infty}}\lesssim \|V_{z}\|_{E^{2, \infty}}$$
    since  thanks to \eqref{Wdef}  and \eqref{bordv1}, we have that $V_{z}$ vanishes on the boundary.    Moreover, 
     again   since $V_{z}$ vanishes on the boundary,  we shall prove that
    \beq
    \label{hardyfin}   \big\|  Z \big(  {1- z  \over {z }} V_{z} \big) \big\|_{m-3} \lesssim \|ZV_{z}\|_{m-3} + \|\partial_{z} Z V_{z} \|_{m-3} + \|\partial_{z} V_{z} \big \|_{m-3}.\eeq
    Since 
   $$  \big\|  Z \big(  {1- z  \over {z }} V_{z} \big) \big\|_{m-3} \lesssim  \big\|   {1- z  \over {z }} Z V_{z} \big\|_{m-3}  +  \big\|  {1 \over z (1-z)}V_{z}  \big\|_{m-3},
   $$
   we  have to estimate
    $$ \| {1 - z \over z} Z^\beta  ZV_{z}\|,  \quad \|  {1 \over z(1-z)} Z^\beta  V_{z}\|, \quad |\beta| \leq m-3.$$
       By using   the following
    variants of the Hardy inequality: 
    \begin{lem}
    \label{hardybis} If $f(0)$= 0, we have the inequalities, 
  \begin{align*}
  &   \int_{-\infty}^0 {1\over z^2(1-z)^2} |f(z)|^2 dz \lesssim    \int_{-\infty}^0 |\partial_{z}f(z)|^2 \, dz, \\
  &   \quad \int_{-\infty}^0 \big({ 1- z \over z}\big)^2|f(z)|^2
     \lesssim \int_{-\infty}^0\big( |f(z)|^2 + |\partial_{z}f(z)|^2 \big) dz.
     \end{align*}
    \end{lem}
     The estimate \eqref{hardyfin} follows. Note that we had to be careful in these estimates since $V_{z}$ is not in $L^2$ while
      its derivatives are. The proof  of this Lemma which only relies on an integration by parts is given in section \ref{sectiontech}.
      Looking at the previous estimates, 
    we have thus proven that
    $$ \| [ Z^\alpha, V_{z}\partial_{z}] S_{\n}\| \leq \|V_{z}\|_{E^{2, \infty}} \|S_{\n}\|_{m-2} + \|S_{\n}\|_{L^\infty}\big( \|ZV_{z}\|_{E^{m-2}}
     + \|\partial_{z}V_{z}\|_{m-3}\big).
   $$
    Thanks to Proposition \eqref{propeta}, we have that $\|V_{z}\|_{E^{2, \infty}}\leq \Lambda_{\infty}$.
     Moreover, by using Lemma \ref{lemestW} and Remark \ref{remdzVz} we also find that
    $$ \|ZV_{z}\|_{E^{m-2}} \leq \Lambda_{\infty}\big( \|v\|_{E^{m-1}} + |h|_{m-{1 \over 2}}\big).$$
    We have thus proven the commutator estimate
    \beq
    \label{CSz}
    \| [ Z^\alpha, V_{z}\partial_{z}] S_{\n}\| \leq \Lambda_{\infty} \big(   \|v\|_{E^{m-1}} + |h|_{m-{1 \over 2}} + \|S_{\n}\|_{m-2}\big)
    \eeq
      and hence, in view of \eqref{CS1},  this yields
    \beq
    \label{CS1}
   \|\mathcal{C}_{S}^1 \|
      \leq \Lambda_{\infty}\big(  \|S_{\n}\|_{m-2} +  \|v\|_{E^{m-1}} + |h|_{m-{1 \over 2}} \big).
    \eeq
     Now, we shall estimate the term involving $\mathcal{C}_{S}^2$.
      To expand this commutator, we can use the expression \eqref{deltaphi} of $\Delta^\varphi$. This yields
     $$ \mathcal{C}_{S}^2 = \mathcal{C}_{S1}^2 + \mathcal{C}_{S2}^2 + \mathcal{C}_{S3}^2$$
      where
      $$
       \mathcal{C}_{S1}^2= \eps [Z^\beta , {1 \over \partial_{z} \varphi}] \nabla \cdot \big( E \nabla S_{\n} \big), \quad
        \mathcal{C}_{S2}^2=\eps {1 \over \partial_{z} \varphi} [Z^\alpha, \nabla] \cdot \big( E \nabla S_{\n}\big), \quad
         \mathcal{C}_{S3}^2 =\eps {1 \over \partial_{z}\varphi} \nabla\cdot  \big( [Z^\alpha, E \nabla \big] S_{\n}\big).$$   
   By expanding the first commutator $\mathcal{C}_{S1}^2$, we get that we need to estimate terms under the form
   $$\eps \int_{\mathcal{S}} Z^{\beta}\big( {1 \over \partial_{z}\varphi} \big) Z^{\tilde\gamma} \big(\nabla \cdot ( E\nabla S_{\n})\big)\, \cdot Z^\alpha S_{\n} \, d\V$$
   with $\beta + \tilde \gamma = \alpha$, $\beta \neq 0$ and next we can commute $Z^{\tilde \gamma}$ and $\nabla$
    (we recall that $Z_{3}$ and $\partial_{z}$ does not commute)  
    to reduce the problem  to the estimate of 
    $$ \eps\int_{\mathcal{S}} Z^{\beta}\big( {1 \over \partial_{z}\varphi} \big)  \partial_{i} Z^\gamma\big( ( E\nabla S_{\n})_{j}\big)\, \cdot Z^\alpha S_{\n} \, d\V $$ 
        with $ |\gamma| \leq| \tilde \gamma|$. Since $S_{\n}$ vanishes on the boundary, we can integrate by parts  
      to get
   \begin{align*} &  \Big|  \eps\int_{\mathcal{S}} Z^{\beta}\big( {1 \over \partial_{z}\varphi} \big)  \partial_{i} Z^\gamma\big( ( E\nabla S_{\n})_{j}\big)\, \cdot Z^\alpha S_{\n} \, d\V\Big| \leq 
     \\
    &  \Lambda_{0} \Big(  \eps  \big\| Z^{\beta}\big( {1 \over \partial_{z}\varphi} \big)  Z^\gamma\big( ( E\nabla S_{\n})_{j} \big)\big\|\, 
   \big( \| \nabla^\varphi  Z^\alpha S_{\n}\| + \|S_{\n}\|_{m-2}\big) +  \eps   \big\|\partial_{i} Z^{\beta}\big( {1 \over \partial_{z}\varphi} \big)  Z^\gamma\big( ( E\nabla S_{\n})_{j} \big) \big\|\, 
    \|  S_{\n}\|_{m-2}\Big).
    \end{align*}
    Next, we can use \eqref{gues},  Proposition \ref{propeta} and \eqref{ELinfty}, \eqref{EHs} to obtain
  \begin{align*}
    \eps^{1 \over 2 } \big\| Z^{\beta}\big( {1 \over \partial_{z}\varphi} \big)  Z^\gamma\big( ( E\nabla S_{\n})_{j} \big)\big\|
    &  \leq       \eps^{1 \over 2}  \Big( \big\| Z\big({1 \over \partial_{z} \varphi}\big)\|_{L^\infty}\, \| E\nabla S_{\n}\|_{m-3} + 
    \| E \nabla S_{\n}\|_{L^\infty} \big\|Z\big( {1 \over \partial_{z} \varphi}\big) \big\|_{m-3}\Big) \\
     &   \leq    \eps^{1 \over 2} \Lambda_{0} \|\nabla S_{\n}\|_{m-3} +  \Lambda_{\infty}   \,  |h|_{m-{3 \over 2}}.  
  \end{align*}
  Note that $\Lambda_{\infty}$ contains a control of $\sqrt{\eps}|\partial_{zz}v\|_{L^\infty}$ that we  have used to  get the last line.
  From the same arguments, we also get that 
  \begin{align*}
   \eps  \big\|\partial_{i} Z^{\beta}\big( {1 \over \partial_{z}\varphi} \big)  Z^\gamma\big( ( E\nabla S_{\n})_{j} \big) \big\|
  \leq    \eps^{1 \over 2} \Lambda_{0} \|\nabla S_{\n}\|_{m-3} +  \Lambda_{\infty}   \, \eps^{1 \over 2}\, |h|_{m-{1 \over 2}}  
  \end{align*}
  We have thus proven the estimate
 \begin{multline}
 \label{CS21}
 \Big| \int_{\mathcal{S}} \mathcal{C}_{S1}^2 \cdot Z^\alpha S_{\n} d\V\Big|
  \leq  \Lambda_{0}\big( \eps^{1 \over 2} \| \nabla^\varphi  Z^\alpha S_{\n} \| + \|S_{\n}\|_{m-2}\big) \\  \big( \eps^{1\over 2 }\|\nabla S_{\n}\|_{m-3}+ \|S_{\n}\|_{m-2} + \Lambda_{\infty}( |h|_{m-{3 \over 2 }} +\eps^{1 \over 2} |h|_{m-{1 \over 2}}) \big).
  \end{multline}
  By using \eqref{idcom} we see that to handle the term involving $\mathcal{C}_{S2}^2$, we have to estimate
$$ \eps \int_{\mathcal{S}}\partial_{z} Z^\beta\big( E\nabla S_{\n}\big) \cdot Z^\alpha S_{\n} \,dy dz$$
with $|\beta | \leq m-3$, 
 hence we perform an integration by parts as above to get that
 \beq
 \label{CS22}
  \Big| \int_{\mathcal{S}} \mathcal{C}_{S2}^2 \cdot Z^\alpha S_{\n}\, d\V \Big|
  \leq  \Lambda_{0}\, \eps^{1 \over 2} \| \nabla^\varphi  Z^\alpha S_{\n} \|  \big( \eps^{1\over 2 }\|\nabla S_{\n}\|_{m-3}+ \Lambda_{\infty} |h|_{m-{3 \over 2 }}\big).
  \eeq
In a similar way,  by integrating by parts, we get that
\begin{align}
\label{CS23} \Big| \int_{\mathcal{S}} \mathcal{C}_{S3}^2 \cdot Z^\alpha S_{\n}  \, d\V\Big| &  \leq  \eps \| [Z^\alpha, E \nabla ] S_{\n} \| \, \| \nabla Z^\alpha S_{\n}\|  \\ \nonumber & \leq   \Lambda_{0} \eps^{1 \over 2} \| \nabla^\varphi  Z^\alpha S_{\n} \|  \big( \eps^{1\over 2 }\|\nabla S_{\n}\|_{m-3}+ \Lambda_{\infty} |h|_{m-{3 \over 2 }}\big).
  \end{align}
 In view of \eqref{CS21}, \eqref{CS22}, \eqref{CS23}, we have actually proven that
 \begin{multline}
 \label{CS2}
  \Big| \int_{\mathcal{S}} \mathcal{C}_{S}^2 \cdot Z^\alpha S_{\n} d\V\Big|
  \leq  \Lambda_{0}\big( \eps^{1 \over 2} \| \nabla   Z^\alpha S_{\n} \| + \|S_{\n}\|_{m-2}\big) \\  \big( \eps^{1\over 2 }\|\nabla S_{\n}\|_{m-3}+ \|S_{\n}\|_{m-2} + \Lambda_{\infty}( |h|_{m-{3 \over 2 }}+ \eps^{1 \over 2} |h|_{m-{1 \over 2}})\big).
  \end{multline}
  To conclude our energy estimate, we can use the identity  \eqref{m-21}, the estimates \eqref{FS1}, \eqref{FS2} and \eqref{CS1}, \eqref{CS2}, 
   to obtain
   \begin{align*}
   &  {d \over dt} {1 \over 2} \int_{\mathcal{S}}  |Z^\alpha S_{\n}|^2 \, d\V    +{ \eps\over 2} \int_{\mathcal{S}} | \nabla^\varphi Z^\alpha S_{\n}|^2 \, d\V \\
   & \leq \Lambda_{\infty}\big( \|v\|_{E^m} + |h|_{m-{1 \over 2}} + \eps^{1\over 2}|h|_{m+{1 \over 2 }} \big)\big( \|S_{\n}\|_{m-2} + |h|_{m-{1 \over 2 }}\big)
+ \eps |v^b|_{m+{1 \over 2} } \|S_{\n}\|_{m-2} + \Lambda_{0} \|\nabla S_{\n}\|_{m-3}^2.
 \end{align*}

 Note that we have used the Young inequality and \eqref{remS0} for $Z^\alpha S_{\n}$  to absorb the term $\|\nabla Z^\alpha S_{\n}\|_{m-2}$  by the energy dissipation term
  in the left hand side.  Next, we can integrate in time to obtain
 \begin{align*}
   &  \|S_{\n}\|_{m-2}^2    + \eps \int_{0}^t  | \nabla^\varphi S_{\n}|_{m-2}^2 \\  
   & \leq  \Lambda_{0}\|S_{\n}(0)\|_{m-2}^2 + \int_{0}^t\Big(\Lambda_{\infty}\big( \|v\|_{E^m} + |h|_{m-{1 \over 2}} + \eps^{1\over 2}|h|_{m+{1 \over 2 }} \big)\big( \|S_{\n}\|_{m-2} + \|h\|_{m-{1 \over 2 }}\big) \\
& \quad + \eps\int_{0}^t |v^b|_{m+{1 \over 2} } \|S_{\n}\|_{m-2} + \Lambda_{0}\eps \int_{0}^t \|\nabla S_{\n}\|_{m-3}^2
 \end{align*}
 and we finally get  \eqref{estSnm-2pf} by using the induction assumption to control $\Lambda_{0}\eps \int_{0}^t \|\nabla S_{\n}\|_{m-3}^2$.

  To end the proof of Proposition \ref{propdzvm-2}, we can use again \eqref{vbm+} 
    to get 
    $$ \eps\int_{0}^t |v^b|_{m+{1 \over 2} } \|S_{\n}\|_{m-2} \leq  \eps \int_{0}^t \|\nabla V^m\|^2    + \int_{0}^t \Lambda_{\infty}\big(\|V^m\|^2
     + \|S_{n}\|_{m-2}^2 + \eps |h|_{m+{1 \over }}^2  \big).$$
 and we can use  Lemma \ref{mingrad} to replace $ \|\nabla^\varphi S_{\n}\|_{m-2}$ by $\|\nabla S_{\n}\|_{m-2}$ in the left hand side.
     This ends the proof of Proposition \ref{propdzvm-2}.
     \end{proof}

   \section{$L^\infty$ estimates}
   \label{sectionLinfty}
   In view of the estimate of Proposition \ref{propdzvm-2}, to close the argument, we need to estimate  the $L^\infty$ norms  contained in $\Lambda_{\infty}$
    and $\int_{0}^t \|\partial_{z}v\|_{m-1}^2$. We shall first  estimate the $L^\infty$ norms in terms of the quantities 
     in the left hand side of the estimate of Proposition \ref{propdzvm-2} and Proposition \ref{conormv}.
    
   We shall begin with the estimates which can be easily  obtained through Sobolev embeddings.
   \begin{prop}
   \label{propLinfty1}
   We have the following estimates: 
   \begin{itemize}
 \item for $ k\in \mathbb{N}$,  
   \beq
   \label{hinfty} |h|_{k, \infty}  + \sqrt{\eps}|h|_{k+1, \infty} \lesssim  |h|_{2+k}+ \eps^{1 \over 2} |h|_{2+k+{1 \over 2}},
   \eeq
   \item  for the velocity, we have 
   \beq
   \label{u2infty}
   \|v(t)\|_{2, \infty} \leq \Lambda\big({ 1 \over c_{0}},  |h|_{4, \infty} +   \|V^4\| + \|S_{\n}\|_{3}\big)
   \eeq
   \item and also
   \beq
   \label{dzv1infty1}
    \|\partial_{z}v \|_{1, \infty} \leq  \Lambda_{0} \big( \|S_{\n}\|_{1, \infty}  + \|v\|_{2, \infty}\big), \quad \sqrt{\eps} \|\partial_{zz}v\|_{L^\infty}
     \leq  \Lambda_{0} \big(\sqrt{\eps} \|\partial_{z}S_{\n}\|_{L^\infty} + \|S_{\n}\|_{1, \infty}  + \|v\|_{2, \infty}\big ).
    \eeq
        \end{itemize}
   \end{prop} 
   
   \begin{proof}
   The estimate \eqref{hinfty} is a direct consequence of the Sobolev embedding in dimension $2$.
   
    To  get the estimate \eqref{u2infty}, we first  use  the anisotropic Sobolev embedding \eqref{emb} to obtain that
    \beq
    \label{dzv3inf} \| Z^\alpha v\|_{\infty}^2 \lesssim \|\partial_{z}  v\|_{3} \, \|v\|_{4}, \quad |\alpha| \leq 2\eeq
     and  next we use \eqref{dzVid},  \eqref{Subis} and  \eqref{Subis1} to estimate $ \|\partial_{z}  v\|_{3} $
      in terms of $\|S_{\n}\|_{3}$. Note that  here we estimate the product  by  putting always
       $v$ in $L^2$ and the  various coefficients in $L^\infty$ (and thus, we use Proposition \ref{propeta} to control them in $L^\infty$). 
        This yields
        \beq
        \label{dzv3fin} \|\partial_{z}v\|_{3} \leq \Lambda({1\over c_{0}}, |h|_{4, \infty}+ \|S_{\n}\|_{3} + \|v\|_{4}\big).\eeq
       Finally, by using the definition of $V^\alpha= Z^\alpha v -\partial_{z}^\varphi v Z^\alpha h$, we also get that
        $$ \|v\|_{4} \lesssim \|V^4\| + \Lambda\big( {1 \over c_{0}}, |h|_{4, \infty}\big) \|\partial_{z}v \| \leq  \Lambda\big( {1 \over c_{0}}, |h|_{4, \infty}\big)
         \big( \|V^4\| + \|S_{n}\| + \|v\|_{1}\big).$$
         Next, by crude interpolation, we can write
         $$  \|v\|_{1} \lesssim  \|v\|_{4}^{1 \over 2} \|v\|^{1\over 2}$$ and
         hence the Young inequality yields
         $$  \|v\|_{4} \lesssim  \Lambda\big( {1 \over c_{0}}, |h|_{4, \infty} + \|V^4\| + \|S_{\n}\| + \|v\|\big)$$
         and since $V^0=v$, we have actually proven that
         $$ \|v\|_{4} \lesssim  \Lambda\big( {1 \over c_{0}}, |h|_{4, \infty} + \|V^4\| + \|S_{\n}\|\big).$$
         By plugging this estimate in \eqref{dzv3fin} and \eqref{dzv3inf}, we finally obtain \eqref{u2infty}.

                The  estimates in  \eqref{dzv1infty1} are also a consequence of \eqref{dzVid} and \eqref{Subis}, \eqref{Subis1}.

   \end{proof}
     
    As a consequence of  Proposition \ref{propLinfty1}, we see that  the only  $L^\infty$  norms that
     appear in  the estimates of Proposition \ref{conormv} and Proposition \ref{propdzvm-2} that   remain to be  estimated 
     are $\|S_{\n}\|_{1, \infty}$ and $\eps^{1 \over 2} \| \partial_{z}S_{\n}\|_{L^\infty}$.
     Moreover, by using the following Lemma, we see that  actually
      we only have to prove   estimates for  
       $ \| \chi S_{\n}\|_{1, \infty}$ and  $\eps^{1 \over 2} \| \chi \partial_{z} S_{\n}\|_{L^\infty}$
        for $\chi(z)$   compactly supported  and  such that $\chi = 1$ in a  vicinity of $z=0$.
    \begin{lem} 
   For  any smooth  cut-off function $\chi$ 
   such that $\chi= 0$ in a vicinity of $z=0$, we have for $m > k+3/2$:
    \beq
    \label{inftyint}
     \| \chi  f\|_{ W^{k, \infty}} \lesssim  \|f\|_{m}.
    \eeq 
    \end{lem}
    To get this Lemma, it suffices to 
      use the Sobolev embedding of $H^k(\mathcal{S})$   in $L^\infty(\mathcal{S)}$  
 for $k>3/2$ and then to  use the fact that away from the boundary the conormal
       Sobolev norm $\| \cdot \|_{k}$ is equivalent to   the standard $H^k$ norm.
       
%

        To summarize the estimates of Proposition \ref{propLinfty1} and the last remark,  it is convenient to introduce the notation
        \beq
        \label{defQm}
        \mathcal{Q}_{m}(t)= \|h(t)\|_{m}^2 +  \eps  |h(t)|_{m+{1\over 2}}^2 + \|V^m \|^2 + \|S_{\n}(t)\|_{m-2}^2 + \| S_{\n}(t)\|_{1, \infty}^2 +
         \eps  \|\partial_{z} S_{\n} \big\|_{L^\infty}^2.
        \eeq
       and use it to state
       \begin{cor}
       \label{corLinfty}
       For $m\geq 6$,  we have the following estimates for the $L^\infty$ norms:
    $$   \|h(t)\|_{4, \infty} + \|v(t)\|_{2, \infty} + \|\partial_{z}v(t) \|_{1, \infty} + \sqrt{\eps}\|\partial_{zz}v(t)\|_{L^\infty}
     \leq \Lambda\big({1 \over c_{0}}, \mathcal Q_{m}(t)\big).$$
       \end{cor}
       The next step will be to estimate 
       $ \|  S_{\n}\|_{1, \infty}$ and  $\eps^{1 \over 2} \| \partial_{z} S_{\n}\|_{L^\infty}$.
        Let us set
        \beq
        \label{deflambdainfty} \Lambda_{\infty, m}(t)= \Lambda\big({1 \over c_{0}}, 
  |h|_{4, \infty} + \|v\|_{E^{2, \infty}} + \eps^{1 \over 2} \|\partial_{zz}v\|_{L^\infty} +|h|_{m} + \|v\|_{m} + \|\partial_{z}v\|_{m-2} + 
   \sqrt{\eps}|h|_{m+{1 \over 2}} \big).
   \eeq
    Note that thanks to Corollary \ref{corLinfty} and \eqref{dzvk}, we have
 \beq 
  \label{suplambdainfty}
     \Lambda_{\infty, m}(t) \leq  \Lambda\big({1 \over c_{0}}, \mathcal{Q}_{m}(t)\big), \quad m \geq 6.\eeq

    \begin{prop}
    \label{propdzvinfty}
   For $t \in [0, T^\eps], $ and $m \geq 6$,  we have the estimate
 \begin{align*}
       \| S_{\n} (t) \|_{1, \infty}^2 &  \leq \Lambda_{0}\Big( \| S_{\n}(0)\|_{1, \infty }^2  + \Lambda\big( {1 \over c_{0}}, |h(t)|_{m}+ \|V^m(t) \|+ \|S_{\n}(t)\|_{m-2}\big) \\
     & \quad  \quad \quad   
     + \int_{0}^t\big( \eps \|\nabla V^m \|^2 +  \eps \| \nabla S_{\n}\|_{m-2}^2\big)  +(1+t) \int_{0}^t \Lambda_{\infty, m}(t)\Big).
     \end{align*}
    \end{prop} 
    Note that in the above estimate, the terms
    $\Lambda\big( {1 \over c_{0}}, |h(t)|_{m}+ \|V^m(t) \|+ \|S_{\n}(t)\|_{m-2}\big)$ and
   $   \int_{0}^t\big( \eps \|\nabla V^m \|^2 +  \eps \| \nabla S_{\n}\|_{m-2}^2\big)$ can be estimated by using the estimates of  Proposition \ref{conormv}
    and Proposition \ref{propdzvm-2}. 
     
    \begin{proof}
         To  estimate  $\| S_{\n}\|_{1, \infty}$, we shall   perform
        directly $L^\infty$ estimates on the convection diffusion equation \eqref{eqPiS}  solved by $ S_{\n}.$ As before, 
         the main interest  in the use of $S_{\n}$ is the fact that it verifies   the  homogeneous Dirichlet boundary condition
         \eqref{bordSN} on the boundary. 
      
      The estimate of $\| S_{\n}\|_{L^\infty}$ is a direct consequence of the maximum principle for the convection diffusion
       equation \eqref{eqPiS}  with homogeneous Dirichlet boundary condition. We find
       $$ \|S_{\n}\|_{L^\infty} \leq \|S_{\n}(0)\|_{L^\infty} + \int_{0}^t \|F_{S}\|_{L^\infty}$$
        and from a  brutal estimate, we get in view of the expressions \eqref{FS1}, \eqref{FS2}  that
        $$ \| F_{S}\|_{L^\infty} \leq \Lambda_{\infty, m} + \Lambda_{0} \| \nabla q \|_{E^{1, \infty}}.$$
         To estimate  the pressure, we  use  \eqref{estqElinfty}  to get 
         $$ \| \nabla q^E\|_{E^{1, \infty}} \leq \Lambda_{\infty, m}$$
          and  \eqref{estqNS} and the Sobolev embedding in  $\mathcal{S}$ to get
          \beq
          \label{qNS1infty} \| \nabla q^{NS} \|_{E^{1, \infty}}  \lesssim \| \nabla q^{NS}\|_{H^{3}} \leq \Lambda_{\infty}\big( \|\eps v^b\|_{9\over 2} + \|\eps h\|_{9 \over 2}).\eeq
           To  control  $ \|\eps v^b\|_{9\over 2}$, we use again \eqref{vbm+}
            to get
          \beq
          \label{SnqNS}
           \|  \nabla q^{NS} \|_{E^{1, \infty}}  \lesssim \Lambda_{\infty, m}\, \eps  \big( \| \nabla  V^k \| +  \|V^k \|+ |h|_{k+{1 \over 2 }}\big), 
            \quad k \geq 4.
          \eeq
           Consequently, we  obtain  by using the Cauchy-Schwarz and Young inequalities that 
          \beq
          \label{Sninfty1}  \|S_{\n}(t)\|_{L^\infty} \leq \|S_{\n}(0)\|_{L^\infty} + \eps\int_{0}^t  \| S^\varphi V^m \|^2  + (1+ t) \int_{0}^t \Lambda_{\infty, m}.
          \eeq

       The next step is to estimate $\|\chi Z S_{\n}\|_{L^\infty}$. Indeed,  from the definition of Sobolev conormal spaces, we get that
       $$ \| Z S_{\n}\|_{L^\infty} \lesssim  \| \chi Z S_{\n}\|_{L^\infty} +  \|v\|_{2, \infty}  $$
        and hence thanks to \eqref{u2infty}, we find
        \beq
        \label{Slocalise}  \| Z S_{\n}\|_{L^\infty} \lesssim  \| \chi Z S_{\n}\|_{L^\infty} +  \Lambda\big( {1 \over c_{0}}, |h|_{m} +\|V^m\| + \|S_{\n}\|_{m-2}
        \big), \quad m \geq 6
        .\eeq
       The main difficulty  is  to handle the  commutator between the fields $Z_{i}$ and
          the Laplacian $\Delta^\varphi$. As in our previous work 
\cite{MR11prep},
           it is convenient for this estimate to use a  coordinate system where  the Laplacian has the simplest expression. We shall
            thus use a normal geodesic coordinate system in the vicinity of the boundary. Note that we have not used
             this coordinate system before because it requires more regularity  for the boundary:  to get an
              $H^m$ (or $\mathcal{C}^m$) coordinate system, we need  the boundary to be $H^{m+1}$ (or $\mathcal{C}^{m+1}$). 
              Nevertheless, at this step, this is not a problem since we want to estimate a fixed low number of derivatives
               of the velocity while we can assume that the boundary is $H^m$ for $m$ as large as we need. We shall choose the cut-off
                function $\chi$ in order to  get a well defined coordinate system in the vicinity of the boundary.
            
            We define a different parametrization of the vicinity of the boundary of  $\Omega_{t}$ by    
            \begin{eqnarray}
   \label{diffgeo}
   \Psi(t, \cdot) : \,   \mathcal{S}= \mathbb{R}^2 \times (-\infty, 0) & \rightarrow  &  \Omega_{t} \\ 
 \nonumber x= (y,z) &\mapsto &   \left( \begin{array}{ll}  y \\ h(t,y)   \end{array} \right) + z \n^b(t,y)
 \end{eqnarray}  
 where $\n^b$ is the unit exterior  normal $\n^b(t,y)= (-\partial_{1}h, - \partial_{2} h, 1)/ |\N|$.
  Note that $D\Psi(t, \cdot)$ is under  the form  $M_{0}+ R$ where
   $ |R|_{\infty} \lesssim z |h|_{2, \infty}$ and
    $$M_0=  \left( \begin{array}{ccc} 1 & 0 & -\partial_{1}h  \\ 0 & 1  & - \partial_{2} h  \\ \partial_{1} h & \partial_{2} h  & 1  \end{array}\right)$$
     is invertible. This yields that $\Psi(t, \cdot) $ is a diffeomorphism from $\mathbb{R}^2\times( - \delta, 0)$ to a vicinity of
      $\partial \Omega_{t}$  for some  $\delta$ which depends only on $c_{0}$,  and  for every $t \in [0, T^\eps]$
       thanks to \eqref{apriori}. By this parametrization,    the scalar product in $\Omega_{t}$ induces a Riemannian 
        metric on $T\mathcal{S}$  which has the block structure
        \beq
        \label{blockg} g(y,z)= \left( \begin{array}{cc} \tilde{g}(y,z) & 0  \\ 0  & 1 \end{array} \right)\eeq
        and hence, the Laplacian in this coordinate system reads:
        \beq
        \label{laplacien} \Delta_{g}f= \partial_{zz}f + {1 \over 2 } \partial_{z} \big( \ln |g| \big) \partial_{z} f  + \Delta_{\tilde g} f 
        \eeq
     where $|g|$ denotes the determinant of the matrix $g$ and $\Delta_{\tilde{g}}$ which is defined by 
           $$ \Delta_{\tilde{g}} f=  { 1 \over |\tilde{g}|^{1 \over 2} }  \sum_{1 \leq i, \, j \leq 2} \partial_{y^i}\big(
            \tilde{g}^{ij} |\tilde{g}|^{1 \over 2} \partial_{y^j} f \big)$$          
        involves only tangential derivatives.
        Note that the drawback of this system is that it depends on $h$ and $\nabla h$  via $\n$ thus it looses one degree of regularity.
        
        To use this coordinate system, we shall first  localize the equation for $S^\varphi v$ in a vicinity of the boundary.
         Let us set
         \beq
         \label{Schidef}
         S^\chi= \chi(z) S^\varphi v
         \eeq
         where $\chi(z) = \kappa({z \over \delta(c_{0}) })\in [0, 1]$ and $\kappa$  is smooth compactly supported and takes the value $1$ in a vicinity of $z=0$. Note that this choice implies that  
             \beq
    \label{chideriv}| \chi^{(k)}(z)| \leq  \Lambda_{0}\eeq
     for every $k \geq 1$.

           Thank  to \eqref{eqS}, we get for  $S^\chi$ the equation
           \beq
           \label{Schieq}
           \partial_{t}^\varphi S^\chi + (v \cdot \nabla^\varphi)S^\chi - \eps \Delta^\varphi S^\chi=  F_{S^\chi} 
           \eeq
            where the source term $F_{S^\chi}$ can be split into
          \beq
          \label{Fchidef}
          F_{S^\chi}=  F^\chi + F_{v}
           \eeq 
           with
           \begin{eqnarray*}
           & & F^\chi= \big(V_{z} \partial_{z} \chi\big) S^\varphi v  -  \eps \nabla^\varphi \chi \cdot \nabla^\varphi S^\varphi v - \eps \Delta^\varphi \chi \, S^\varphi v , \\
           & & F_{v} = - \chi \big( D^\varphi)^2 q -{\chi \over 2}\big( (\nabla ^\varphi v)^2 + ((\nabla^\varphi v)^t)^2\big).
           \end{eqnarray*}
           Note that thanks to Lemma \ref{lemestW} and \eqref{inftyint}, since all the terms in $F^\chi$ are  supported away
            from the boundary, we get that
           \beq
           \label{Fchiest}
           \|F^\chi\|_{1, \infty} \leq \Lambda\big( {1 \over c_{0}}, \|v\|_{1, \infty} + |h|_{2, \infty}\big)  \|v\|_{5} \leq \Lambda_{\infty, m}
           \eeq
          Next, we  define implicitely in $\Omega_{t}$, $\tilde{S}$ by $\tilde S (t,  \Phi(t,y,z))= S^\chi(t,y,z)$
         and  then $S^\Psi$ in $\mathcal{S}$ by  $S^\Psi(t,y,z)= \tilde S (t,\Psi(t,y,z))=  S^\chi(t, \Phi(t, \cdot)^{-1} \circ \Psi).$ The change of variable is well-defined since we can
          choose $\tilde S$ to be  supported in the domain where $\Psi^{-1}$ is well defined by taking $\delta$ sufficiently small.
          Since $S^\chi$ solves \eqref{Schieq}, we get that $\tilde{S}$ solves
          $$ \partial_{t}   \tilde S  + u \cdot \nabla \tilde S  - \eps \Delta  \tilde S = F_{S^\chi}(t, \Phi(t, \cdot)^{-1})$$
           in $\Omega_{t}$
           and hence by using \eqref{laplacien}, we get   that $S^\Psi$ solves in $\mathcal{S}$  the convection diffusion equation
          \beq
          \label{eqSPsi}
          \partial_{t}  S^\Psi + w \cdot \nabla  S^\Psi - \eps\big(  \partial_{zz}   S^\Psi  + {1 \over 2 } \partial_{z} \big( \ln |g| \big) \partial_{z} S^\Psi  + \Delta_{\tilde g}   S^\Psi \big)= F_{S^\chi} (t, \Phi^{-1} \circ \Psi)
         \eeq
          where the vector field $w$ is given by 
          \beq
          \label{defW}
           w=  \overline{\chi} (D \Psi)^{-1}\big( v(t, \Phi^{-1} \circ \Psi ) - \partial_{t} \Psi \big) 
          \eeq
          with $D \Psi$ the jacobian matrix of $\Psi$ (with respect to the space variables). Note that 
           $S^\Psi$ is compactly supported in $z$ in a vicinity of $z=0$.  The function $\overline{\chi}(z)$  is
           a function with a slightly larger  support  such that $\overline \chi  S^\Psi= S^\Psi$.
               The introduction of this function   allows to have $w$ also supported in a vicinity of the boundary. 
        We finally set 
           \beq
           \label{etadef}
           S_{\n}^\Psi(t,y,z)=\Pi^b(t,y) S^\Psi \n^b(t,y)
           \eeq          
         with   $\Pi^b= \mbox{Id}- \n^b \otimes \n^b$. Note that $\Pi^b$ and
             $\n^b$  are independent of $z$.
           This yields that $S_{\n}^\Psi$ solves
     \beq
     \label{eqSnPsi}     \partial_{t}  S_{\n}^\Psi + w \cdot \nabla  S_{\n}^\Psi - \eps\big(  \partial_{zz}  + {1 \over 2 } \partial_{z} \big( \ln |g| \big) \partial_{z} \big) S_{\n}^\Psi  = F_{\n}^\Psi\eeq
     where 
     \beq
     \label{Fetadef}
     F_{\n}^\Psi=\Big(  \Pi^b F_{S^\chi} \n^b + F_{\n}^{\Psi, 1}  + F_{\n}^{\Psi, 2} \Big). 
     \eeq
    with
     \begin{eqnarray}
   & &   \label{Feta1}
     F_{\n}^{\Psi, 1}=  \big(( \partial_{t} + w_{y} \cdot \nabla_{y}) \Pi^b\big)S^\Psi \n^b + \Pi^b S^\Psi\big( \partial_{t}+ w_{y}\cdot \nabla_{y}) n^b, \\
     & & \label{Feta2} F_{\n}^{\Psi, 2}=  - \eps \Pi^b\big( \Delta_{\tilde{g}} S^\Psi\big) \n^b.
     \end{eqnarray}
      
      We shall thus estimate $S_{\n}^\Psi$ which solves the equation \eqref{eqeta} in $\mathcal{S}$ with the
      homogeneous Dirichlet  boundary
      condition $(S_{\n}^\Psi)_{/z=0}= 0$ (indeed, on the boundary $z=0$, we have that $S_{\n}^\Psi=\Pi S^\varphi v \n= S_{\n}$.) 
     
                In the following, we shall use very often the following  observation:
    \begin{lem}
    \label{leminv}
    Consider $\mathcal{T}: \mathcal{S} \rightarrow \mathcal{S}$ such that $\mathcal{T}(y,0)= y, \quad \forall y \in \mathbb{R}^2$
     and let   $g(x)= f(\mathcal{T} x)$. Then, we have  for every $k \geq 1$, the estimate:
     $$ \|g\|_{k, \infty} \leq \Lambda\big( \| \nabla \mathcal{T}\|_{k-1,\infty} \big) \|f\|_{k,\infty}.$$  
    \end{lem}
   This is just a statement of the fact that Sobolev conormal spaces are invariant by diffeomorphisms
    which preserve the boundary.  The proof is just a consequence of the chain rule and the fact that the family  $(Z_{1}, \, Z_{2}, \, Z_{3})$  generates
      the set of vector fields tangent to the boundary. A similar statement holds for  the $\| \cdot \|_{m}$ norm.
        Note that it is equivalent to estimate $S_{\n}^\Psi$  or  $S_{\n}$.  Indeed,  since $S_{\n}^\Psi = \Pi^b S^\chi(t, \Phi^{-1}\circ \Psi) n^b$,
         the above lemma yields
         $$ \| S_{\n}^\Psi \|_{1, \infty} \leq \Lambda \big( |h|_{2, \infty} \big) \| \Pi^b S^\varphi v \, \n^b \|_{1, \infty}$$
          and we observe that $|\Pi- \Pi^b| + | \n- \n^b|= \mathcal{O}(z)$ in the vicinity of the boundary to get
          \beq
          \label{etaSn}
         \| S_{\n}^\Psi \|_{1, \infty} \leq \Lambda_{0}\big( \|S_{\n}\|_{1, \infty} + \|v\|_{2, \infty}\big).
         \eeq
         From the same arguments, we also obtain that
         \begin{align}
        \nonumber  \|  S_{\n}\|_{1, \infty} &  \leq \Lambda_{0} \big( \|S_{\n}^\Psi \|_{1, \infty}+ \|v\|_{2, \infty}\big)  \\
      \label{Sneta}  &   \leq \Lambda_{0}
           \Big( \|S_{\n}^\Psi \|_{1, \infty} + \Lambda\big( {1 \over c_{0}}, |h|_{m} + \|V^m\| + \|S_{\n}\|_{m-2}\big)\Big), \quad m \geq 6
          \end{align}
         where  the last estimate comes from \eqref{u2infty}.
           By using repeatidly these arguments, we also obtain
     \beq
     \label{w1infty} \|w\|_{1, \infty} \leq \Lambda\big( {1 \over c_{0}}, |h|_{3, \infty} + \|v\|_{1, \infty}+ \|\partial_{t}\Psi\|_{1, \infty} \big)    
     \leq  \Lambda\big( {1 \over c_{0}}, |h|_{3, \infty} + \|v\|_{2, \infty}+ |\partial_{t}h|_{2, \infty} \big)  \leq \Lambda_{\infty, m}
     \eeq
     where $\Lambda_{\infty, m}$ is defined by \eqref{deflambdainfty}.
     Note that we have used the boundary condition \eqref{bord1} to get the last part of the estimate.

     To estimate $Z_{i} S_{\n}^\Psi= \partial_{i} S_{\n}^\Psi$, $i=1, \, 2$,  we can proceed as in  the proof 
 of 
     estimate \eqref{Sninfty1}.  We first apply $\partial_{i}, \, i=1, \, 2$ to \eqref{eqSnPsi} to get
      \beq
      \label{diSnPsi}
        \partial_{t}   \partial_{i }S_{\n}^\Psi + w \cdot \nabla  \partial_{i }S_{\n}^\Psi - \eps\big(  \partial_{zz}  + {1 \over 2 } \partial_{z} \big( \ln |g| \big) \partial_{z} \big) \partial_{i }S_{\n}^\Psi   =\partial_{i}F_{\n}^\Psi- \partial_{i} w \cdot  \nabla S_{\n}\Psi -{\eps \over 2} \partial_{z} S_{\n}^\Psi \partial_{iz}^2 \big( \ln |g| \big).\eeq
        From the maximum principle, we thus get that
   \begin{eqnarray}
\nonumber  \|   \partial_{i }S_{\n}^\Psi (t)  \|_{L^\infty} &  \leq &  \| \partial_{i }S_{\n}^\Psi (0)\|_{L^\infty}
  + \int_{0}^t \Big( \| \partial_{i} F_{\n}^\Psi\|_{L^\infty}+  \|\partial_{i}w \cdot \nabla S_{\n}^\Psi \|_{L^\infty} +  \Lambda\big( {1 \over c_{0}}, |h|_{3, \infty}\big) 
   \eps \| \partial_{z} S_{\n}^\Psi\|\Big)\\
\label{diSnPSi1}   &\leq  & \| \partial_{i }S_{\n}^\Psi (0)\|_{L^\infty}
  + \int_{0}^t \Big( \| \partial_{i} F_{\n}^{\Psi}\|_{L^\infty}+  \|\partial_{i}w \cdot \nabla S_{\n}^\Psi \|_{L^\infty} + \Lambda_{\infty, m}\Big) .
   \end{eqnarray}
   Note that for the last term, we have used that $g$ involves two derivatives of $h$ but is $\mathcal{C}^\infty $ in $z$  in view
    of the definition \eqref{diffgeo} and hence that $\partial_{iz}^2 |g|$ has the regularity of $\partial_{i}|g|$. 
    
    It remains to estimate the right hand side in this last estimate. We first note that
    $$ \|\partial_{i}w \cdot \nabla S_{\n}^\Psi \|_{L^\infty} \leq \| w\|_{1, \infty} \| S_{\n}^\Psi \|_{1, \infty} + \|\partial_{i}w_{3} \partial_{z} S_{\n}^\Psi\|_{L^\infty}
     \leq \Lambda_{\infty, m} + \|\partial_{i}w_{3} \partial_{z} S_{\n}^\Psi\|_{L^\infty}.$$
     Next, thanks to the definition \eqref{defW} of  $w$, we note that on the boundary, we have
     \beq
     \label{wbord}
      w^b=( D \Psi(t,y,0))^{-1}\Big(v^b - \left( \begin{array}{cc} 0  \\ \partial_{t}h \end{array} \right) \Big).
      \eeq
       Since on the boundary, we have
       $$  D\Psi(t,y,0)= \left( \begin{array}{ccc}
        1  & 0 & \n_{1}^b \\ 0 & 1 & \n_{2}^b   \\ \partial_{1} h & \partial_{2} h  & \n_{3}^b \end{array} \right), $$
    we get  that 
    \beq
    \label{cdv} ((D\Psi(t,y,0))^{-1} Y)_{3}= Y \cdot \n^b
    \eeq for every $Y \in \mathbb{R}^3$ and hence in particular that on the boundary
    \beq
    \label{w3bord}
    w_{3}^b= v^b \cdot \n^b - \partial_{t} h \n_{3}^b= {1 \over |\N|} \big( v^b \cdot \N - \partial_{t}h)= 0 
    \eeq
    thanks to the boundary condition \eqref{bordv1}. Consequently, since $\partial_{i} w_{3}$  also vanishes on the boundary we  get
     that 
     \beq
     \label{diwS}  \|\partial_{i}w_{3} \partial_{z} S_{\n}^\Psi\|_{L^\infty} \leq \| \partial_{z} \partial_{i} w_{3}\|_{L^\infty} \|S_{\n}^\Psi \|_{1, \infty} \leq
      \Lambda\big({1 \over c_{0}},  |h|_{3, \infty} + \|v\|_{E^{2, \infty}}\big) \leq \Lambda_{\infty, m}.
      \eeq
     Note that for the last estimate, we have used again that  $\Psi$ is $\mathcal{C}^\infty$ in $z$.
          It remains to estimate $ \|F_{\n}^\Psi\|_{1, \infty}$ in the right-hand side of \eqref{diSnPSi1}. By using  \eqref{Fetadef} and
     the expressions   \eqref{Feta1}, \eqref{Feta2},  we get
         $$ \|F_\n^{\Psi, 1} \|_{1, \infty} \leq \Lambda\big( {1 \over c_{0}}, | \partial_{t} h|_{2, \infty}+  |h|_{3, \infty}
     + \|w\|_{1, \infty}\big)\big( \| v\|_{E^{2, \infty}} +  \eps \| \partial_{z}v \|_{2, \infty}\big) \leq \Lambda_{\infty, m}$$
     and by using again the fact that $|\Pi-\Pi^b |+ |\n-\n^b| =\mathcal{O}(z)$, we get that
     $$  \|F_{\n}^{\Psi, 2} \|_{1, \infty} \leq \Lambda_{\infty, m}\big( \eps  \| S_{\n}\|_{3, \infty}+ \eps \|v\|_{4, \infty} \big).$$
     Next, from the definitions \eqref{Fetadef}, \eqref{Fchidef}, we get  thanks to \eqref{Fchiest} that
     $$ \| F_{\n}^\Psi \|_{1, \infty} \leq \Lambda_{\infty, m} \big( 1 +  \eps  \| S_{\n}\|_{3, \infty}+ \eps \|v\|_{4, \infty}  +  \|Ê
      \Pi^b \big( (D^\varphi)^2 q \big)\n^b \big) \|_{1, \infty} \big).$$
      To estimate the pressure term, we  use   that $\Pi \nabla^\varphi$ involves only conormal derivatives 
       and the decomposition \eqref{pressuredec} to obtain
       $$   \| \Pi^b \big( (D^\varphi)^2 q \big)\n^b \big) \|_{1, \infty} \leq \Lambda_{0} \big( \|\nabla q^E \|_{2, \infty} + \|\nabla q^{NS}\|_{2, \infty}\big).$$
      For the Euler part $q^E$ of the pressure, we use \eqref{estqElinfty} to get
      $ \|\nabla q^E \|_{2, \infty} \leq \Lambda_{\infty}$ and  for $q^{NS}$, we can use again
      Proposition \ref{propPNS}  to get that an estimate analogous to  \eqref{SnqNS} holds for $ \| \nabla q^{NS}\|_{{2, \infty}}$
       but with the restriction $k \geq 5$ for the right hand side. 
      This yields
    $$     \| \Pi^b \big( (D^\varphi)^2 q \big)\n^b \big) \|_{1, \infty} \leq \Lambda_{0} \big(  1 + \eps \|S^\varphi V^m\|\big).$$           
  We have thus proven that
 \beq
 \label{Fetainfty}
  \| F_{\n}^\Psi\|_{1, \infty} \leq \Lambda_{\infty, m}
     \big(  1 + \eps \|\nabla  V^m\| + \eps \|S_{\n}\|_{3, \infty} + \eps \|v\|_{4, \infty}\big).
     \eeq
Consequently, by combining \eqref{diSnPSi1}, \eqref{diwS} and \eqref{Fetainfty}, we obtain
\beq
\label{diSnPsi1} \|   \partial_{i }S_{\n}^\Psi (t)  \|_{L^\infty}   \leq   \| \partial_{i }S_{\n}^\Psi (0)\|_{L^\infty}
  + \int_{0}^t \Lambda_{\infty, m}\big(  1 + \eps \|S^\varphi V^m\| + \eps \|S_{\n}\|_{3, \infty} + \eps \|v\|_{4, \infty}\big).
  \eeq
     
  The next step is to estimate $\|Z_{3} S_{\n}^\Psi \|_{L^\infty}$. This is more delicate due to the bad commutation
   between $Z_{3}$ and $\eps \partial_{zz}$.  We shall use a more precise description of the solution of \eqref{eqSnPsi}.
    First, it is convenient to eliminate the term $\eps \partial_{z} \ln |g| \partial_{z}$ in the equation \eqref{eqSnPsi}. 
     We  set 
           \beq
           \label{etadefbis}
           \rho(t,y,z)= |g|^{1 \over 4} S_{\n}^\Psi=  |g|^{1 \over 4} \Pi^b(t,y) S^\Psi \n^b(t,y)
           \eeq
           This yields that $\rho$ solves
     \beq
     \label{eqrhobis}     \partial_{t}  \rho + w \cdot \nabla   \rho - \eps  \partial_{zz}   \rho   = |g|^{ 1 \over 4 } \big(F_{\n}^\Psi + F_{g}\big)
     := \mathcal{H}\eeq
     where
     \begin{eqnarray}
     & & \label{Fg} F_{g}= {\rho \over |g|^{1 \over 2}} \Big( \partial_{t} + w \cdot \nabla - \eps \partial_{zz} \big) |g|^{1 \over 4}.
     \end{eqnarray}
      Since
      \beq
      \label{etaequiv} \|Z_{3}  S_{\n}^\Psi \|_{L^\infty} \leq \Lambda_{0} \| \rho \|_{1, \infty}, \quad   \| \rho \|_{1, \infty} \leq \Lambda_{0}  \|S_{\n}^\Psi \|_{1, \infty}\eeq
      it is equivalent to estimate $\|Z_{3} S_{\n}^\Psi\|_{1, \infty}$ or $\|\rho \|_{1, \infty}.$
     We shall thus estimate $\rho$ which solves the equation \eqref{eqrhobis} in $\mathcal{S}$ with the
      homogeneous Dirichlet  boundary
      condition $\rho_{/z=0}= 0$.
      
      This estimate will be a consequence of the following Lemma:
  
  \begin{lem}
        \label{lemFP0}
        Consider  $\rho$  a smooth solution of 
        \beq
        \label{eqetaFP0} 
        \partial_{t} \rho + w \cdot \nabla \rho = \eps \partial_{zz} \rho + \mathcal{H}, \quad  z<0, \quad 
         \rho(t,y,0)= 0, \quad  \rho(t=0)= \rho_0
         \eeq
          for  some smooth  vector field $w$ such that 
          $ w_{3}$ vanishes on the boundary. Assume that 
         $ \rho$ and $\mathcal{H}$  are  compactly supported in $z$.  Then, we have the estimate:
         $$ \| Z_{i} \rho (t) \|_{\infty}
          \lesssim \|  Z_{i}\rho_{0} \|_{\infty} +  \|\rho_{0}\|_{\infty}+ 
          \int_{0}^t \Big( \big(  \|w \|_{E^{2, \infty}} + \| \partial_{zz} w_{3} \|_{L^\infty} \big) \big( \| \rho \|_{1, \infty}
           + \|  \rho \|_{4} \big) + \| \mathcal{H} \|_{1, \infty} \Big), \quad i=1, \, 2, \, 3.$$
\end{lem} 
     
       We have already used an estimate of the same type in our previous work \cite{MR11prep}.
        The proof of the Lemma is given in  section \ref{sectionFP}.  
        From Lemma \ref{lemFP0}, we get that
   \beq
   \label{eta1infty1}   
\| Z_{3}\rho (t) \|_{\infty}
          \lesssim \| Z_{3} \rho_{0} \|_{\infty} + \| \rho_{0}\|_{\infty} +  
          \int_{0}^t \Big( \big(  \|w \|_{E^{2, \infty}} + \| \partial_{zz} w_{3} \|_{L^\infty} \big) \big( \| \rho \|_{1, \infty}
           + \|  \rho \|_{4} \big) + \| \mathcal{H} \|_{1, \infty} \Big).
           \eeq
        Next, by using \eqref{Fetainfty} and \eqref{Fg}, we obtain
\begin{eqnarray}
 \nonumber\|\mathcal{H}\|_{1, \infty} &\leq &  \Lambda_{\infty, m} \big(  1 + \eps \|S^\varphi V^m\| + \eps \|S_{\n}\|_{3, \infty} + \eps \|v\|_{4, \infty}\big)
      + \Lambda\big( {1 \over c_{0}},  |h|_{4, \infty} + \|v\|_{1, \infty} \big) \\
  \label{estH1infty}     & \leq & \Lambda_{\infty, m} \big(  1 + \eps \|S^\varphi V^m\| + \eps \|S_{\n}\|_{3, \infty} + \eps \|v\|_{4, \infty}\big).
       \end{eqnarray}
     Hence, from \eqref{etadefbis}, we get that
   $$ \| \rho \|_{4} \leq \Lambda\big({ 1 \over c_{0}}, |h|_{6} +  \|S_{\n}^\Psi\|_{4} \big)\leq \Lambda_{\infty, m}$$
   since we assume that $m\geq 6$ 
   and from  the definition \eqref{defW}, we  also have $\|w\|_{E^{2, \infty}} \leq \Lambda_{\infty, m}$.
    We only need to be more careful with the term $\| \partial_{zz} w_{3}\|_{L^\infty}$ in the right hand side of \eqref{eta1infty1} since it contains
     two normal derivatives.  Looking at the expression \eqref{defW} of $w$,   we first note that 
      since $\Psi$ is $\mathcal{C}^\infty$ in $z$,  we have
      $$\Big\| \partial_{zz}\big( \overline \chi \big( D\Psi^{-1}\partial_{t} \Psi \big) \big)\Big\|_{L^\infty} \leq \Lambda\big( {1 \over c_{0}},  |h|_{2, \infty} +
       |\partial_{t}h|_{1, \infty} \big) \leq \Lambda_{\infty, m}$$
       therefore, we only need to estimate $\partial_{zz}\big( \overline \chi  D\Psi^{-1} v(t,  \Phi^{-1}\circ \Psi)\big)_{3}.$ We can first use that
       \begin{align*}\big\| \partial_{zz}\big( \overline \chi  D\Psi^{-1} v(t, \Phi^{-1}\circ \Psi\big)_{3}\big\|_{L^\infty}
         & \leq \| \overline \chi  \partial_{zz} \big( (D\Psi(t,y,0))^{-1} v(t, \Phi^{-1} \circ \Psi)\big)_{3} \big\|_{L^\infty} + \Lambda_{\infty, m} \|v\|_{E^{2, \infty}} \\
        &  \leq  \| \overline \chi  \partial_{zz} \big( (D\Psi(t,y,0))^{-1} v(t, \Phi^{-1}\circ \Psi\big)_{3} \big\|_{L^\infty} + \Lambda_{\infty, m}
        \end{align*}
      and then use the observation \eqref{cdv} to get that
      $$ \big\| \partial_{zz}\big( \overline \chi  D\Psi^{-1} v(t, \Phi^{-1} \circ \Psi)\big)_{3}\big\|_{L^\infty}
       \leq  \big\| \overline \chi \partial_{zz}\big( v(t, \Phi^{-1}\circ \Psi) \cdot \n^b) \big\|_{L^\infty} + \Lambda_{\infty,m}.$$
            We thus  need to compute $$\overline{\chi} \partial_{zz} \big( v(t, \Phi^{-1} \circ \Psi) \cdot \n^b \big).$$
       Note that $v(t, \Phi^{-1} \circ \Psi)= u(t, \Psi)$ where $u$ is defined
        in $\Omega_{t}$ and let us set $u^{\Psi}(t,y,z)= u(t, \Psi)$. Since $u$ is divergence free in
         $\Omega_{t}$, the  expression in local coordinates of the divergence yields
     $$  \partial_{z} u^{\Psi} \cdot \n^b= -{ 1 \over 2} \partial_{z}\big( \ln |g|\big) u^\Psi \cdot \n^b
      -  \nabla_{\tilde{g}} \cdot u_{y}^\Psi $$
      where $u_{y}^\Psi= \Pi^b u^\Psi$  and $\nabla_{\tilde{g}}\cdot $ is the divergence operator associated
       to the metric $\tilde{g}$ in the definition \eqref{blockg} and hence involves only tangential
        derivatives of $u^\Psi$. This means as before that for $u^\Psi \cdot \n^b= v(t, \Phi^{-1}\circ \Psi) \cdot \n^b$, we can replace
         one normal derivative by tangential derivatives. Consequently, we get from \eqref{defW} that
     \beq
     \label{dzzw3} \| \partial_{zz} w_{3}\|_{L^\infty} \leq \Lambda\big( {1 \over c_{0}}, \|u^\Psi \|_{E^{1, \infty} } +|h|_{3, \infty} + | \partial_{t} h|_{L^\infty}\big)
       \leq \Lambda_{\infty, m}.
       \eeq
     Note that we have again used the fact that  $\Psi$ is $\mathcal{C}^\infty$ in $z$.
     Finally, we note that
     $$ \| \rho \|_{4} \leq \Lambda\big({1 \over c_{0}},   |h|_{6} + \|S_{\n}\|_{4} + \|v\|_{5} \big) \leq \Lambda_{\infty, m}, \quad \|Z_{3} \eta_{0} \|_{\infty}
      + \| \eta_{0} \|_{\infty} \leq \Lambda_{0} \| v(0) \|_{E^{2, \infty}},$$
      and hence we obtain from \eqref{eta1infty1} that
 \beq
 \label{eta1infty}   \| \rho (t) \|_{1, \infty}
          \leq \Lambda_{0} \|v\|_{E^{2, \infty}} + 
          \int_{0}^t \Lambda_{\infty} \Big(  1 + \eps \|S^\varphi V^m\| + \eps \|S_{\n}\|_{3, \infty} + \eps \|v\|_{4, \infty} \Big), \quad m\geq 5.
           \eeq  
 Consequently, by combining \eqref{eta1infty}, \eqref{diSnPsi1}  and \eqref{etaequiv}, \eqref{Sneta}, we get that
 \begin{multline}
 \label{dzvinfty2}
 \|S_{\n} \|_{1, \infty} \leq \Lambda_{0} \|v(0)\|_{E^{2, \infty}} +  \Lambda\big( {1 \over c_{0}}, \|V^m(t)\|+ \|S_{\n}(t) \|_{m-2}  + |h(t)|_{m}\big)
   \\+ \int_{0}^t \Lambda_{\infty, m}
    \Big(  1 + \eps \|\nabla  V^m\| + \eps \|S_{\n}\|_{3, \infty} + \eps \|v\|_{4, \infty} \Big).
    \end{multline}
  To conclude, we use Sobolev estimates to control the last two terms.
   At first, thanks to \eqref{emb}, we get that
   $$\sqrt{\eps}\|S_{\n}\|_{3, \infty} \lesssim\big( \sqrt{\eps}  \| \partial_{z} S_{\n}\|_{4}\big)^{1 \over 2} \big( \sqrt{\eps}\|S_{\n}\|_{5}\big)^{1 \over 2} \leq \Lambda_{\infty, m}\big( \sqrt{\eps}\|\nabla S_{\n}\|_{4}\big)^{1 \over 2} \big(
  \sqrt{\eps}  \| \nabla v\|_{4}\big)^{1 \over 2}
    $$
    and we use that for $\alpha \neq 0$, $|\alpha| \leq 4$, 
    $$ \sqrt{\eps}\| \nabla  Z^\alpha v \| \leq \sqrt{\eps} \| \nabla V^\alpha \| +  \Lambda_{\infty, m}.$$
      Indeed, note that  in particular, the term $\sqrt{\eps}\| \partial_{zz}v \|_{L^\infty}$ is in $\Lambda_{\infty, m}.$
       Hence, we have proven that
     \beq
     \label{Sn3infty} \sqrt{\eps}\|S_{\n}\|_{3, \infty} \leq \Lambda_{\infty, m} +  \sqrt{\eps} \| \nabla V^m \| + \sqrt{\eps}\|\nabla S_{\n}\|_{m-2}.\eeq
      From the same arguments, we  also obtain
    \beq  \label{v4infty} \sqrt{\eps}\|v\|_{4, \infty} \leq\sqrt{\eps }\big( \|\nabla v \|_{5}+ \|v\|_{6}\big) \leq 
       \Lambda_{\infty, m} + \sqrt{\eps }\| \nabla V^m\|.\eeq
       for $m\geq 6$.
     Consequently, we get from  \eqref{dzvinfty2} and the Cauchy-Schwarz inequality that
      for $m \geq 6$
      \begin{multline}
 \|S_{\n} \|_{1, \infty} \leq \Lambda_{0} \|v(0)\|_{E^{2, \infty}} +  \Lambda\big( {1 \over c_{0}}, \|V^m(t)\|+ \|S_{\n}(t) \|_{m-2}  + |h(t)|_{m}\big)
   \\+ (1+ t )  \int_{0}^t \Lambda_{\infty, m} + \int_{0}^t\big(
    \ \eps \|\nabla  V^m\|^2 + \eps \|\nabla S_{\n}\|_{m}^2 \big). 
    \end{multline}
     This ends the proof of Proposition \ref{propdzvinfty}.    
   \bigskip

    It will be useful  in the future  to remember the  estimates \eqref{Sn3infty}, \eqref{v4infty} that we have just establish:
    \begin{lem} 
    \label{Linfty0t}
    For $m \geq 6$, we have the estimate
 $$    \sqrt{\eps} \int_{0}^t \|\nabla v\|_{3, \infty} \leq   \int_{0}^t \Lambda_{\infty, m} +   \int_{0}^t\big( \eps  \| \nabla V^m \|^2 +\eps \| \nabla S_{\n}\|_{m-2}^2 \big).
   $$
      \end{lem}
      We recall that $\Lambda_{\infty,m}$ was defined in \eqref{deflambdainfty}.
  To get the proof of this lemma, it suffices to use  that
  $$  \| \nabla v\|_{3, \infty} \leq \Lambda_{\infty, m} \big( \|S_{n}\|_{3, \infty} +  \|v\|_{4, \infty}\big)$$ 
     thanks to \eqref{Subis}, \eqref{Subis1} and \eqref{dzVid} and then to use \eqref{Sn3infty}, \eqref{v4infty}
      and the Young inequality.
      
      \bigskip

    Our next  $L^\infty$ estimate will be the one of  $\sqrt{\eps} \|\partial_{z} S_{\n} \|_{L^\infty}$ which is still missing.
    
   \begin{prop}
   \label{dzzvLinfty}
   Assuming  that the initial data  satisfy  the  boundary condition \eqref{bordv2}, we have for $m \geq 6$ the estimate
    \begin{align*}
  \eps\|\partial_{zz} v(t)  \|_{L^\infty}^2
 &  \leq  \Lambda_{0}\big( \|v(0)\|_{E^{2, \infty}}^2  + \eps \|\partial_{zz}v_{0}\|^2 \big)  + \Lambda\big({ 1 \over c_{0}},  |h(t)|_{m} + \|V^m(t)\| + \|S_{\n}(t)\|_{m-2} \big)
    \\  & \quad +  \Lambda_{0} \int_{0}^t \eps \big(  \| \nabla  S_{\n}\|_{m-2}^2 + \| \nabla V^m\|^2\big)   +  (1 + t ) \int_{0}^t \big( 1 + {1 \over\sqrt { t - \tau}}\big)
    \Lambda_{\infty, m}\, d\tau
  \end{align*}
  for $t \in [0, T^\eps]$.
   \end{prop} 
   Note that  in our estimates, this is the only place where we use the compatibility condition on the initial data.
    We also again point out that the terms
    $ \Lambda\big({ 1 \over c_{0}},  |h(t)|_{m} + \|V^m(t)\| + \|S_{\n}(t)\|_{m-2} \big)$ and
    $\int_{0}^t \eps \big(  \| \nabla  S_{\n}\|_{m-2}^2 + \| \nabla V^m\|^2\big) $
     can be estimated by using  Proposition \ref{conormv} and Proposition \ref{propdzvm-2}.
  \begin{proof}
   As in the  proof of Proposition \ref{propdzvinfty}, we shall first reduce the problem to the estimate of 
    $\sqrt{\eps} \|\partial_{z} \rho \|_{L^\infty}.$
     By using   \eqref{etadef},
     and \eqref{Sneta}, 
      we obtain
    $$
 \sqrt{\eps}  \|\partial_{z} S_{\n}\|_{L^\infty} \leq \Lambda_{0} \big( \|v\|_{E^{2, \infty}} +\sqrt{\eps} \|\partial_{z} \rho \|_{L^\infty}\big)$$
 and hence, by using again Proposition \ref{propLinfty1}(and in particular \eqref{dzv1infty1} and \eqref{u2infty}), we obtain that
   \begin{align}
  \nonumber    \sqrt{\eps}  \|\partial_{z} S_{\n}\|_{L^\infty} &  \leq \Lambda_{0} \big( \|S_{\n}\|_{1, \infty} +  \|v\|_{2, \infty} +  \sqrt{\eps}  \|\partial_{z} \rho \|_{L^\infty}
     \\
    \label{dzSneta} & \leq   \Lambda_{0}\Big( \|S_{\n}\|_{1, \infty} + \Lambda\big({ 1 \over c_{0}},  |h|_{m} + \|V^m\| + \|S_{\n}\|_{m-2} \big)
      +  \sqrt{\eps}  \|\partial_{z} \rho \|_{L^\infty}\Big)
      \end{align}
      so that it only remains to estimate $\sqrt{\eps}\|\partial_{z} \rho \|_{L^\infty}$. 
        Again note that we use the fact that $g$ is $\mathcal{C}^\infty$ in $z$ to write the first estimate.
   Since $ \eta $ solves the convection-diffusion equation \eqref{eqrhobis} in $z>0$ with zero Dirichlet boundary condition on
    the boundary, we can use the one-dimensional  heat kernel of  $z>0$
    $$G(t,z,z')= {1 \over \sqrt{4\pi t}}\big( e^{- { (z-z')^2 \over 4t }} -  e^{- { (z+z')^2 \over 4t }} \big)
    $$  
    to write that
    \beq
    \label{duhamel} \eps^{1 \over 2} \partial_{z} \rho(t,y,z) = 
   \sqrt{\eps}  \int_{0}^{+ \infty} \partial_{z}G(t,z, z') \rho_{0}(y, z') dz'
      + \int_{0}^t \sqrt{\eps} \partial_{z} G(t-\tau,z,z')\big( \mathcal{H}(\tau,y,z') - w \cdot \nabla \rho\big) dz' d\tau.\eeq
    Since $\rho_{0}$  vanishes on the boundary thanks to the compatibility condition,  we can
     integrate by parts the first term to obtain
  \beq
  \label{dzzeta1}\eps^{1 \over 2} \|\partial_{z} \rho (t) \|_{L^\infty}
   \leq  \eps^{1 \over 2 } \|\partial_{z} \rho_{0}\|_{L^\infty}  + \int_{0}^t  {1 \over \sqrt{t- \tau}} \big(\|\mathcal{H}\|_{L^\infty}+ \| w \cdot \nabla \rho \|_{L^\infty}
   \big).
   \eeq
    Next, we use the definition of the source term $\mathcal{H}$ in  \eqref{eqrhobis} to get
   $$ \| \mathcal{H}\|_{L^\infty} \leq \Lambda_{\infty, m}\big( 1+ \|\nabla q \|_{E^{1, \infty}} +  \eps(\|S_{\n}\|_{2, \infty}+ \|v\|_{3, \infty}) \big)$$
   and hence, thanks to Proposition \ref{proppE} and  \eqref{qNS1infty}, we  get
  $$ \| \mathcal{H}\|_{L^\infty} \leq \Lambda_{\infty, m}\big( 1+  \eps \|v^b\|_{9\over 2}+  \eps(\|S_{\n}\|_{2, \infty}+ \|v\|_{3, \infty}) \big).$$
   By using the trace inequality \eqref{trace} and \eqref{emb}, we find for $m \geq 6$
   $$\eps \|v^b\|_{9\over 2} \lesssim  \eps \|\partial_{z}v\|_{4}^{1 \over 2} \|v\|_{5}^{1 \over 2} \leq \Lambda_{\infty, m}\, \eps \|\nabla v\|_{4}^{1 \over 2}
     \leq  \Lambda_{\infty, m} \big( 1+  \eps \| \nabla V^m\|^{1 \over 2 }\big)$$
      and
  \begin{align*}
  &  \eps \|S_{\n}\|_{2, \infty} \leq \eps \| \nabla S_{\n}\|_{3}^{1 \over 2} \|S_{\n}\|_{4} \leq \Lambda_{\infty, m} \eps   \| \nabla S_{\n}\|_{m-2}^{1 \over 2}, \\ 
 &  \eps \|v\|_{3, \infty} \leq \eps \|\nabla v\|_{4}^{1\over 2} \|v\|_{5}^{1\over 2} \leq \Lambda_{\infty, m}\big( 1 + \|\nabla V^m \|\big).
  \end{align*}
   Consequently, we have proven that
$$   \| \mathcal{H}\|_{L^\infty} \leq \Lambda_{\infty, m}\big( 1+ \|\nabla V^m\|^{1\over 2} + \|\nabla S_{\n}\|_{m-2}^{1 \over 2}\big)$$
  Finally, by using \eqref{w3bord}, we can write as we have used to get \eqref{diwS} that
  $$ \|w \cdot \nabla \rho \|_{L^\infty}  \leq \| w\|_{E^{1, \infty}} \| \rho \|_{1, \infty} \leq \Lambda_{\infty, m}.$$
  Consequently, by combining the two last estimates and \eqref{dzzeta1}, we get  that
  $$ \eps^{1 \over 2} \|\partial_{z} \rho (t) \|_{L^\infty}
   \leq \Lambda_{0} \big( \eps^{1 \over 2 } \|\partial_{zz}v (0)\|_{L^\infty} + \|v(0)\|_{E^{2, \infty}}\big)  + \int_{0}^t  {\Lambda_{\infty, m} \over \sqrt{t- \tau}}
  \big( 1 + \eps \| \nabla S_{\n}\|^{1 \over 2}_{m-2} +\eps \|\nabla V^m\|^{1 \over 2} \big).$$
  To conclude, we use  the H\"older inequality and in particular  that
 $$ \Big( \int_{0}^t  { \Lambda_{\infty} \over (t- \tau)^{1 \over 8} } \| \nabla S_{\n}\|_{m-2}^{1 \over 2}  {1 \over (t- \tau)^{ 3\over 8}}\Big)^2
   \leq \Big( \int_{0}^t \| \nabla S_{\n}\|_{m-2}^2 \Big)^{1 \over 2}\Big( \int_{0}^t {\Lambda_{\infty, m}^4 \over  (t- \tau)^{1 \over 2} } \Big)^{1\over 2}
    t^{1 \over 4}$$
    and a similar estimate  for the term involving $\|\nabla V^m\|^{1 \over 2} $
  to get the estimate 
  \begin{align*}
  \big(\eps^{1 \over 2} \|\partial_{z} \rho (t) \|_{L^\infty}\big)^2
   \leq \Lambda_{0} \big( \big(\eps^{1 \over 2 } \|\partial_{zz} v(0)\|_{L^\infty}\big)^2  + \|v(0)\|_{E^{2, \infty}}^2\big) 
   + \Lambda_{0}\|v(t)\|_{E^{2, \infty}}^2  \\+\Lambda_{0} \int_{0}^t\big( \eps  \| \nabla S_{\n}\|_{m-2}^2 +\eps \| \nabla V^m\|^2 \big)  +  (1 + t ) \int_{0}^t { \Lambda_{\infty} \over (t - \tau)^{1 \over 2}}\, d\tau.
  \end{align*}
To conclude  the proof of Proposition \ref{dzzvLinfty}, we can combine the last estimate and \eqref{dzSneta} with the
 estimate of Proposition \ref{propdzvinfty}.  
  \end{proof}
  
  \bigskip
  
  By using arguments close to the ones we have just used, we can also establish that
  \begin{lem}
  For $m \geq 6$, assume that 
    $\sup_{[0, T]} \Lambda_{\infty, m}(t) \leq M$.
   Then, there exists $\Lambda(M)$ such that 
  we have the estimate
  \beq
  \label{Linfty0t1}
  \int_{0}^t \sqrt{\eps} \| \partial_{zz}v\|_{1, \infty} \leq  \Lambda(M) \big( 1 + t) \sqrt{t}\Big(  1  +
   \int_{0}^t \eps \big( \|\nabla V^m \|^2 + \|\nabla S_{\n}\|_{m-2}^2 \big)\Big).
   \eeq
  
  \end{lem} 
   
  Note that, by combining the previous  lemma  and   Lemma \ref{Linfty0t}, we obtain
  \begin{cor}
  \label{corLinfty0t}
  For $m \geq 6$, assume that 
    $\sup_{[0, T]} \Lambda_{\infty, m}(t) \leq M$.
   Then, there exists $\Lambda(M)$ such that, we have the estimate
 $$
  \int_{0}^t \sqrt{\eps} \| \nabla^2v\|_{1, \infty} \leq  \Lambda(M) \big( 1 + t)^2\Big(  1  +
   \int_{0}^t \eps \big( \|\nabla V^m \|^2 + \|\nabla S_{\n}\|_{m-2}^2 \big)\Big).
  $$
  
  \end{cor}
  \begin{proof}
  We first note by  using \eqref{etadef}, \eqref{Sneta} and  again \eqref{dzVid}, \eqref{Subis} \eqref{Subis1} that
  $$ \int_{0}^t \sqrt{\eps} \| \partial_{zz}v\|_{1, \infty} \leq \Lambda(M)\int_{0}^t \sqrt{\eps} \big(\| \partial_{z} \rho \|_{1,\infty} + \sqrt{\eps} \| \nabla v\|_{2, \infty}\big)$$
  and hence, thanks to Lemma \ref{Linfty0t}, we obtain
  \beq
  \label{Linfty+1}  \int_{0}^t \sqrt{\eps} \| \partial_{zz}v\|_{1, \infty} \leq \Lambda(M)(1+t)\Big( \int_{0}^t \sqrt{\eps} \| \partial_{z} \rho \|_{1,\infty} 
   + \int_{0}^t  \eps \big( \|\nabla V^m \|^2 + \|\nabla S_{\n}\|_{m-2}^2 \Big).\eeq
   \end{proof}
  To estimate $\|\partial_{z}\rho\|_{1, \infty}$,  note that a brutal use of  the Duhamel formula \eqref{duhamel} is not sufficient.
  Indeed,   even for the fields  $Z_{i}, \, i=1, \, 2$  which commute with $G$,  we get
    that
    $$ \sqrt{\eps }\| \partial_{z} Z_{i} \rho \|_{L^\infty} \leq {1 \over \sqrt{t}} \| \rho _{0}\|_{1, \infty}
     + \int_{0}^t{1 \over \sqrt{t- \tau}}\big( \| \mathcal{H}\|_{1, \infty} + \| w \cdot \nabla \rho \|_{1, \infty} \big)$$
    and the problem is that 
     $$   \| w \cdot \nabla \rho \|_{1, \infty}$$ cannot be estimated in terms of controlled quantities ($ \|\rho \|_{2, \infty} \sim \| \partial_{z}v\|_{2, \infty}$
      is not controlled). Consequently, we need to be more precise and incorporate the transport term in  the argument.
       We shall use  the following lemma 
     \begin{lem}
     \label{lemFP1}
        Consider  $\rho$  a smooth solution of 
        \beq
        \label{eqetaFP1} 
        \partial_{t} \rho + w \cdot \nabla \rho = \eps \partial_{zz} \rho + \mathcal{H}, \quad  z<0, \quad 
         \rho(t,y,0)= 0 
         \eeq
          for  some smooth  vector field $w$ such that 
          $ w_{3}$ vanishes on the boundary. Assume that 
         $ \rho$ and $\mathcal{H}$  are  compactly supported in $z$ and that
         $ \sup_{[0, T]}\big(\|w\|_{E^{2, \infty}} + \|\partial_{zz} w_{3}\|\big) \leq M.$
          Then,  there exists $\Lambda(M)$ such that  we have the estimate:
         $$\sqrt{\eps } \| \partial_{z} \rho (t) \|_{1, \infty}
          \leq \Lambda(M) \Big( { 1 \over  \sqrt{t}} \|  \rho (0) \|_{1, \infty} + 
          \int_{0}^t {1 \over \sqrt{t - \tau}} \big( \| \rho \|_{1, \infty}
           + \|  \rho \|_{4}  + \| \mathcal{H} \|_{1, \infty} \big)\Big), \quad \forall t \in [0, T].$$
     \end{lem}
  The proof of this Lemma is given in section \ref{sectionFP1}. By using the previous Lemma,  for the equation   \eqref{eqrhobis}
   and the estimates   \eqref{estH1infty},  \eqref{dzzw3},  we find that
   $$\sqrt{\eps } \| \partial_{z} \rho (t) \|_{1, \infty}
          \leq \Lambda(M) \Big( { 1 \over  \sqrt{t}} \|  \rho (0) \|_{1, \infty} + 
          \int_{0}^t {1 \over \sqrt{t - \tau}} \big(  1   + \eps \|\nabla V^m\| + \eps \|S_{\n}\|_{3, \infty} +  \eps \|v\|_{4, \infty}\big)\Big) $$
           and hence by using \eqref{Sn3infty} and  \eqref{v4infty}, we obtain
     $$\sqrt{\eps } \| \partial_{z} \rho (t) \|_{1, \infty}
          \leq \Lambda(M) \Big( { 1 \over  \sqrt{t}} \|  \rho (0) \|_{1, \infty} + 
          \int_{0}^t {1 \over \sqrt{t - \tau}} \big(  1   + \eps \|\nabla V^m\| + \eps \|\nabla S_{\n}\|_{m-2}\big)\Big)$$
           for $m \geq 6$. From integration in time, we obtain
      $$   \int_{0}^t \sqrt{\eps } \| \partial_{z} \rho (t) \|_{1, \infty}
          \leq \Lambda(M) \Big( \sqrt{t} \|  \rho (0) \|_{1, \infty} +  \sqrt{t}
          \int_{0}^t \big(  1   + \eps \|\nabla V^m\| + \eps \|\nabla S_{\n}\|_{m-2}\big)\Big).$$
          We finally get \eqref{Linfty0t1} by combining the last estimate and \eqref{Linfty+1}.  Note that the terms involving the initial data
           can be  estimated by $\Lambda(M)$.
      \end{proof} 
   \section{Normal derivative estimates part II}
   \label{sectionnorm2}
     In the estimate of Proposition \ref{propdzvm-2}, we see that the left hand side is still insufficient to control the
      term $\int_{0}^t \|\partial_{z} v\|_{m-1}^2.$ 
      It does not seem possible to estimate it by  estimating
      $\|S_{\n}\|_{m-1}$. Indeed,  in the right hand side of \eqref{SNalpha}, the term involving the pressure
        in $F_{S}^2$ does not enjoy any additional regularity. The main obstruction comes from the Euler
         part $q^E$ of the pressure. Since $q^E= gh$ on the boundary the estimate given by Proposition \ref{estqE}
          is optimal in terms of the regularity of $h$: the estimate of $\|D^2 q^E \|_{m-1}$  necessarily involves
           $|h|_{m+{1\over 2}}$  which cannot be controlled uniformly in $\eps$.  The solution will be to use
            the vorticity instead of $S_{\n}$ to perform this estimate.  The main difficulty will be that
             the vorticity does not vanish on the boundary. Note that this difficulty  is not a problem  in the inviscid
              case since there is always a good energy estimate for the transport equation solved by the vorticity even
               if it does not vanish on the boundary. 
             
       Let us set $\omega = \nabla^\varphi \times v$ (note that we have in an equivalent way that $\omega= (\nabla \times u)(t, \Phi)$).
        Since 
        \beq
        \label{omega1} \omega \times \n= {1 \over 2}\big( D^\varphi v\, \n - (D^\varphi v)^t\n)= S^\varphi v\, \n -   (D^\varphi v)^t\n,
        \eeq
        we get as in \eqref{Subis} that
        $$ \omega \times \n =   {1 \over 2} \partial_{\n} u  - g^{ij}\big( \partial_{j}v \cdot \n \big) \partial_{y^i} .$$
        Consequently, we get by using also \eqref{Subis1} that
        \beq
        \label{dzvm-1O} \|Z^{m-1}\partial_{z} v\| \leq \Lambda_{\infty, 6}\big(\|v\|_{m} +  |h|_{m-{1 \over 2}} + \|\omega\|_{m-1}\big).
        \eeq   
     and hence we shall estimate $\|\omega\|_{m-1}$ in place of $\|\partial_{z}v\|_{m-1}$. By applying $\nabla^\varphi \times\cdot $ to the equation \eqref{NSv}, we get 
       the vorticity equation in $\mathcal{S}$
      \beq
      \label{eqomega}
      \D_{t} \omega + v \cdot \nabla^\varphi \omega - \omega \cdot \nabla^\varphi v = \eps \Delta^\varphi \omega.
      \eeq
     By using  \eqref{omega1} and the boundary condition \eqref{bordv2}, we note  that on the boundary, we have 
     $$ \omega \times \n = \Pi \big( \omega \times \n) =  - \Pi\big( g^{ij}\big( \partial_{j}v \cdot \n \big) \partial_{y^i}\big)$$
     and thus $\omega\times \n$ does not vanish on the boundary. Consequently, there is no gain to  consider
      $\omega \times \n$ in place of $\omega$ since the equation for $\omega \times \n$  is more complicated.
       In this subsection, we shall thus estimate $Z^{m-1} \omega$. 
       
    Thanks to \eqref{eqomega}, we get that $Z^\alpha \omega $ for $|\alpha| \leq m-1$ solves  in $\mathcal{S}$ the equation
   \beq
   \label{eqomegaalpha}\D_{t} Z^\alpha  \omega + v \cdot \nabla^\varphi \, Z^\alpha  \omega - \eps \Delta Z^\alpha \omega= F\eeq
   where the source term $F$ is given by 
   \beq
   \label{Fomega}
    F= Z^\alpha\big(\omega \cdot \nabla^\varphi v)  + \mathcal{C}_{S}
   \eeq 
   where $\mathcal{C}_{S}$  is given as in \eqref{CS} by 
    \beq
       \label{CSomega}
       \mathcal{C}_{S}= \mathcal  C_{S}^1 + \mathcal C_{S}^2\eeq
       with
       $$\mathcal{C}_{S}^1= [Z^\alpha v_{y}]\cdot  \nabla_{y} \omega +  [Z^\alpha, V_{z}] \partial_{z} \omega:=
        C_{Sy}+ C_{Sz}, \quad \mathcal{C}_{S}^2 =  - \eps [Z^\alpha, \Delta^\varphi] \omega.
       $$
       
     In addition, by using Lemma \ref{lembord}, we get that on the boundary 
     \beq
     \label{omegab1} |(Z^\alpha \omega)^b| \leq \Lambda_{\infty, 6}\big( |v^b|_{m} + |h|_{m} \big).
     \eeq
     Note that by using the trace Theorem, we get that
     $$ |(Z^\alpha \omega)^b| \leq \Lambda_{\infty, 6} \big( \|\nabla  V^{m}\|^{1 \over 2} \|V^m\|^{1 \over 2} + |V^m| +|h|_{m} \big).$$
     Consequently,  the only way  we can  control this boundary value is through the estimate
    $$
    \eps^{1 \over 2 }  \int_{0}^t  |(Z^\alpha \omega)^b|^2 \leq  \eps \int_{0}^t  \| \nabla V^m\|^2 + \int_{0}^t \Lambda_{\infty, 6}\big(
     \|V^m\|^2 + |h|_{m}^2 \big)
    $$
     and to  we see that  the left hand side can be estimated by using  Proposition  \ref{conormv}.
     Nevertheless, it will be useful to keep the slightly sharper form of the above inequality which reads
     \beq
     \label{omegab2}
    \eps^{1 \over 2 }  \int_{0}^t  |(Z^\alpha \omega)^b|^2   \leq  \Lambda_{\infty, 6}\Big( \int_{0}^t  \sqrt{\eps} \| \nabla V^m\|\, \|V^m\| + 
     \|V^m\|^2 + |h|_{m}^2 \Big).
     \eeq 
      The main difficulty will be to handle this non-homogeneous boundary value problem for the convection-diffusion
      equation \eqref{eqomegaalpha} since because of \eqref{omegab2} the boundary value is at a  low level of regularity
       (it is $L^2_{t,y}$ and no more).
              
    To split the difficulty, we set
    \beq
    \label{splitomega}
    Z^\alpha \omega= \omega^\alpha_{h}+ \omega^{\alpha}_{nh}
    \eeq
    where $\omega_{nh}$ solves in $\mathcal{S}$  the equation \eqref{eqomegaalpha}, that is to say
    \beq
    \label{eqomeganh}
    \D_{t}  \omega^\alpha_{nh} + v \cdot \nabla^\varphi   \omega^\alpha_{nh} - \eps \Delta^\varphi  \omega^\alpha_{nh}= F\eeq
     with the initial and boundary conditions
    \beq
    \label{omeganhc}
    (\omega^\alpha_{nh})^b= 0, \quad  (\omega^\alpha_{nh})_{t=0}= \omega_{0}.\eeq
    while $\omega_{h}$ will solves in $\mathcal{S}$ the homogeneous equation
    \beq
    \label{eqomegah}
    \D_{t}   \omega^\alpha_{h} + v \cdot \nabla^\varphi   \omega^\alpha_{h} - \eps \Delta^\varphi  \omega^\alpha_{h}= 0
    \eeq
    with the initial and boundary conditions
    \beq
    \label{omegahc}
  ( \omega^\alpha_{h})^b= (Z^\alpha \omega)^b, \quad  (\omega^\alpha_{h})_{t=0}= 0.
    \eeq

 
The main result of this section is the following proposition: 
 \begin{prop}
 \label{propomega}
     For $T \in [0, T^\eps]$, $T^\eps \leq 1$,  assume that for $M>0$, the estimate
    \beq
    \label{hypomegahfin}
    \sup_{[0, T]} \Lambda_{\infty, 6}(t) + \eps \int_{0}^T \big( \eps \|\nabla V^6 \|^2 + \eps \| \nabla S_{\n}\|_{4}^2 \big) \leq M
    \eeq
   holds.  Then, there exists $\Lambda(M)$ such that  
    \begin{align*}
  \| \omega \|_{L^4([0, T], H^{m-1}_{co}(\mathcal{S}))}^2 
   &    \leq \Lambda_{0} \| \omega(0)\|_{m-1}^2
      + \Lambda(M)\int_{0}^t ( \|V^m\| ^2  + \| \omega\|_{m-1}^2 + |h|_{m}^2  + \eps |h|_{m+{1 \over 2}}^2  \big) \\
    &  \quad    + \Lambda_{0}  \big( \int_{0}^t   \eps \| \nabla S_{\n} \|_{m-2}^2 +
   \eps   \| \nabla V^m\|^2\big).
     \end{align*}
        \end{prop}
      The proof follows from 
  the splitting  \eqref{splitomega}  
and  the estimates of Proposition \ref{propomeganh} and  
\ref{propomegah}.   

 \subsection{Estimate of $\omega_{nh}^\alpha$}
   Let us first estimate $\omega^\alpha_{nh}$. We shall use the notation
   $$\|\omega_{nh}^{m-1}(t) \|^2= \sum_{|\alpha| \leq m-1} \| \omega_{nh}^\alpha \|^2, \quad \int_{0}^t\| \nabla  \omega_{nh}^{m-1} \|^2
    = \sum_{| \alpha| \leq m-1} \| \nabla \omega_{nh}^\alpha \|^2.$$
    \begin{prop}
    For $t \in [0, T^\eps]$ , the solution $\omega_{nh}^\alpha$ of \eqref{eqomeganh}, \eqref{omeganhc} satisfies the estimate:
   \label{propomeganh}
     \begin{align*}
    &  \|\omega_{nh}^{m-1}(t) \|^2 + \eps \int_{0}^t  \| \nabla  \omega_{nh}^{m-1} \|^2 \\
   &    \leq \Lambda_{0}\| \omega(0)\|_{m-1}^2
      + \int_{0}^t \Lambda_{\infty, 6}\big( \|V^m\| ^2  + \| \omega\|_{m-1}^2 + |h|_{m}^2  + \eps |h|_{m+{1 \over 2}}^2  \big)
       + \Lambda_{0} \int_{0}^t   \eps \| \nabla S_{\n} \|_{m-2}^2.
     \end{align*} 
     \end{prop}
  Note that the term $ \Lambda_{0} \int_{0}^t \eps \| \nabla^\varphi  S_{\n}\|_{m-2}^2$ 
  in the right-hand side of the above estimate
   can be estimated by using Proposition \ref{propdzvm-2}.
   \begin{proof}
   
      Since we have for $\omega^\alpha_{nh}$ an homogeneous Dirichlet boundary condition, we deduce from \eqref{eqomeganh}
    and a standard energy estimate that
\beq
\label{omeganh1}  {d \over dt} {1 \over 2} \int_{\mathcal{S}}  |\omega^\alpha_{nh}|^2 \, d\V    + \eps \int_{\mathcal{S}} | \nabla^\varphi \omega^\alpha_{nh}|^2 \, d\V
 = \int_{\mathcal{S}}   F \cdot \omega^\alpha_{nh}\, d\V 
 \eeq
    and we need to estimate the right hand side with $F$ given by \eqref{Fomega}. By using \eqref{gues}, we get
    that
   $$ \|Z^\alpha (\omega \cdot \nabla^\varphi v \big) \|
    \leq \Lambda_{\infty, 6} \big( \|\omega \|_{m-1}+ \| \nabla^\varphi v\|_{m-1}\big)
     \leq \Lambda_{\infty, 6}\big( \|\omega\|_{m-1} + \|v\|_{m} + |h|_{m-{1 \over 2}}\big).$$
     Next, to estimate $\mathcal{C}_{S}$, 
     we observe that we can use \eqref{CS2} to estimate the part involving $\mathcal{C}_{S}^2$.
      Indeed, it suffices, to change $S_{\n}$ into $\omega$ and $m$ into $m+1$ to obtain
     \begin{multline}
     \label{CS2omega}
  \Big| \int_{\mathcal{S}} \mathcal{C}_{S}^2 \cdot \omega^\alpha_{nh} d\V\Big|
  \leq  \Lambda_{0}\big( \eps^{1 \over 2} \| \nabla^\varphi  \omega^\alpha_{nh} \| + \|\omega \|_{m-1}\big) \\  \big( \eps^{1\over 2 }\|\nabla \omega\|_{m-2}+ \|\omega\|_{m-1} + \Lambda_{\infty, 6}( |h|_{m-{1 \over 2 }}+ \eps^{1 \over 2} |h|_{m+{1 \over 2}})\big).
  \end{multline}  
     To estimate $\mathcal{C}_{S}^1$,  we can use the same  decomposition of $\mathcal{C}_{S}^1$ as the one given after \eqref{CS}. For $\mathcal{C}_{Sy}$, 
      we can also use \eqref{CSy} with $S_{\n}$ changed into $\omega$ and $m$ into $m+1$. This yields
  \beq
  \label{CSyomega}
  \|C_{Sy}\|\leq \Lambda_{\infty, 6}\big( \|\omega\|_{m-1} + \|v\|_{E^{m-1}}\big) \leq \Lambda_{\infty, 6} \big( \|\omega\|_{m-1} 
   +  \|v\|_{m} + |h|_{m} \big)
  \eeq
  The commutator $\mathcal{C}_{Sz}$ requires more care. Indeed, in \eqref{CSz}, we remark that we cannot change $m$ into $m+1$
   since $h$ is not smooth enough uniformly in $\eps$: such an estimate would involve $|h|_{m+{1 \over 2}}$ (and without a gain of $\eps^{1 \over 2}$). 
   By using as in \eqref{CSz0} that this commutator can be expanded into a sum of terms under the form
   $$ c_{\tilde \beta }Z^{\tilde \beta }\big(  {1- z  \over {z }} V_{z} \big)   \, Z_{3} Z^\gamma \omega$$
   where the constraints are now  $|\tilde \beta| + |\tilde \gamma| \leq m-1$, $|\tilde\gamma | \leq m-2$.
    Since $V_{z}= v_{z} / \partial_{z}\varphi= (v\cdot \N - \partial_{t} \eta )/ \partial_{z} \varphi$,
     we use as before  \eqref{gues} and that  $V_{z}$ vanishes on the boundary to write
     \begin{align*}
       \|c_{\tilde \beta }Z^{\tilde \beta }\big(  {1- z  \over {z }} V_{z} \big)   \, Z_{3} Z^\gamma \omega\|
      &   \leq     \Lambda_{\infty, 6} \| \omega \|_{m-1}+  \| \omega \|_{L^\infty} +  \big| Z\big( {1 \over \partial_{z} \varphi}  {1 - z \over z} v_{z}\big) \big\|_{m-2} \\
      & \leq      \Lambda_{\infty, 6} \big(  \| \omega \|_{m-1} + |h|_{m-{1 \over 2}} + \big\|{ 1-z \over z} Zv_{z}\big\|_{m-2} \big)
      \end{align*}
  Next, by using Lemma  \ref{hardybis}, we obtain that
  $$  \big\|{ 1-z \over z}  Zv_{z}\big\|_{m-2} \lesssim  \|Zv_{z}\|_{m-2} + \|\partial_{z} Zv_{z}\|_{m-2}
   \lesssim  \Lambda_{\infty, 6}\big( \|v\|_{E^{m}} + \sum_{| \alpha |= m- 1} \| v \cdot \partial_{z}Z^\alpha \N - \partial_{z} Z^\alpha \partial_{t} \eta \|\big)$$
   and the last term requires some care.
   Since $\eta$ is given by  \eqref{eqeta}, we first note that
   $$ |Z_{3} \hat \eta| \lesssim  |\tilde \chi(|\xi| z ) \hat{h}|$$
   where $\tilde \chi$ has a slightly bigger support than $\chi$ 
   and thus  we get that $Z_{3}$ acts as a zero order operator:
   \beq
   \label{Z3mieux}
    \| \nabla Z_{3} \eta \| \lesssim |h|_{1\over 2}.
    \eeq
    This yields that if $\alpha_{3} \neq 0$, we gain at least one derivative and thus,  have 
 $$ \| v \cdot \partial_{z}Z^\alpha \N - \partial_{z} Z^\alpha \partial_{t} \eta \|  \leq \Lambda_{\infty, 6}\big(   |h|_{m-{1 \over 2}} + |\partial_{t} h|_{m-
 {3 \over 2}}\big) \leq \Lambda_{\infty, 6}\big( |h|_{m-{1 \over 2} } + |v^b|_{m-{3 \over 2}} \big)$$
  where the last estimate comes from the boundary condition \eqref{bord2}.  Hence by using \eqref{trace}, we get
  $$ \| v \cdot \partial_{z}Z^\alpha \N - \partial_{z} Z^\alpha \partial_{t} \eta \| \leq \Lambda_{\infty, 6}\big( |h|_{m-{1 \over 2}}
   +  \|v\|_{E^{m-1}}\big).$$
   Consequently, we only need to estimate 
   $$  \| v \cdot \partial_{z}Z^\alpha \N - \partial_{z} Z^\alpha \partial_{t} \eta \| $$
    when $| \alpha |=m-1$, $\alpha_{3}= 0$. By using the expression \eqref{etaconv}, we note that
    $$ v \cdot \partial_{z}Z^\alpha \N - \partial_{z} Z^\alpha \partial_{t} \eta= 
    v_{1} \,\partial_{z} \big(  \psi_{z} \star_{y} \partial_{1} Z^{\alpha} h \big)
     + v_{2} \,  \partial_{z} \big(  \psi_{z} \star_{y} \partial_{2} Z^{\alpha} h \big) 
      - \partial_{z}\big(  \psi_{z} \star_{y} \partial_{t} Z^{\alpha} h \big) := \mathcal{T}_{\alpha}.$$
      For $z \leq -1$ , we can use the smoothing effect of the convolution  to get that
      $$  \|\mathcal{T}_{\alpha}\|_{L^2(\mathcal{S} \cap |z| \geq 1)} \leq  \Lambda_{\infty, 6} \big( \big\| \partial_{z}\big( {1 \over z^{m-1}}\tilde{ \psi}_{z}
       \star_{y}\nabla h\big)\big\|_{L^2(\mathcal{S} \cap |z| \geq 1)} 
        +  \big\| \partial_{z}\big( {1 \over z^{m-1}} \tilde{\psi}_{z}
       \star_{y}\partial_{t} h\big)\big\|_{L^2(\mathcal{S} \cap |z| \geq 1)} $$
       where $\tilde{\psi}_{z}$ has the same properties as $\psi_{z}$. This yields
       \beq
       \label{Talpha0} \|\mathcal{T}_{\alpha}\|_{L^2(\mathcal{S} \cap |z| \geq 1)}  \leq \Lambda_{\infty, 6} \big(|h|_{1 \over 2} + \|v\|_{E^1}\big).\eeq
        For $|z| \leq 1$,  
        We shall rewrite $\mathcal{T}_{\alpha}$ as 
       \beq
       \label{Talphadecfin} \mathcal{T}_{\alpha}= v_{1}^b \,\partial_{z} \big(  \psi_{z} \star_{y} \partial_{1} Z^{\alpha} h \big)
     + v_{2}^b \,  \partial_{z} \big(  \psi_{z} \star_{y} \partial_{2} Z^{\alpha} h \big) 
      - \partial_{z}\big(  \psi_{z} \star_{y} \partial_{t} Z^{\alpha} h \big) +  \mathcal{R}\eeq
       where
       $$ | \mathcal{R}| \leq \Lambda_{\infty, 6}|z| \big| \partial_{z} \big( \psi_{z} \star  Z^\alpha \nabla h\big)\big|
        \leq | Z_{3} \big( \psi_{z} \star  Z^\alpha \nabla h \big)\big|.$$
        Consequently,  by using again  the observation \eqref{Z3mieux}, we get that
        \beq
        \label{Talphadecfin1} \| \mathcal{R}\| _{L^2(\mathcal{S} \cap |z| \leq 1)} 
         \leq \Lambda_{\infty, 6} |h|_{m-{1 \over 2}}.\eeq
         To estimate the first term in \eqref{Talphadecfin}, we write
       \begin{align*}  &  v_{1}^b \,\partial_{z} \big(  \psi_{z} \star_{y} \partial_{1} Z^{\alpha} h \big)
     + v_{2}^b \,  \partial_{z} \big(  \psi_{z} \star_{y} \partial_{2} Z^{\alpha} h \big) \\
      & = \partial_{z} \big( \psi_{z} \star_{y}\big( v_{1}^b\partial_{1} Z^\alpha h + v_{2}^b \partial_{2} Z^\alpha h - \partial_{t} Z^\alpha h \big) \\
      & \quad + 
       \partial_{z} \Big(\int_{\mathbb{R}^2}\Big(  \big(v_{1}(t,y,0) - v_{1}(t,y',0) \big) \psi_{z}(y-y') \partial_{1} Z^\alpha h(t,y')
     \\ & 
     \quad  \quad \quad \quad  + \big(v_{2}(t,y,0) - v_{2}(t,y',0)\big) \psi_{z}(y-y') \partial_{2} Z^\alpha h(t,y'). 
      \end{align*} 
         For the second term, we can use the Taylor formula for $v_{i}$ to get  that
         $$ | (v_{i}(t,y,0) - v_{i}(t,y', 0))\partial_{z} \psi_{z}(y-y') | \leq \Lambda_{\infty, 6}  {1 \over z^2} \tilde{\psi}({y-y' \over z}) $$
          where $\tilde{\psi} $ is still an $L^1$ function. This yields that
        \begin{align}
  \label{Talphadecfin2}   &\sup_{z \in (-1, 0)} \Big\|   \partial_{z} \Big(\int_{\mathbb{R}^2}\Big(  \big(v_{1}(t,y,0) - v_{1}(t,y',0) \big) \psi_{z}(y-y') \partial_{1} Z^\alpha h(t,y')
     \\ & 
  \nonumber   \quad  \quad \quad \quad  \quad \quad  + \big(v_{2}(t,y,0) - v_{2}(t,y',0)\big) \psi_{z}(y-y') \partial_{2} Z^\alpha h(t,y') \Big\|_{L^2(\mathbb{R}^2)}
      \leq \Lambda_{\infty, 6} |h|_{m}.   
        \end{align}
        For the first term, we shall  use  Lemma \ref{lembordh}. For $|\alpha|= m-1$, we get from Lemma \ref{lembordh} that
        $$    v_{1}^b\partial_{1} Z^\alpha h + v_{2}^b \partial_{2} Z^\alpha h - \partial_{t} Z^\alpha h  = - \mathcal{C}^\alpha(h) - (V^\alpha)^b \cdot \N
         - v_{3}^b$$ 
         and hence also  that
   $$   \|   v_{1}^b\partial_{1} Z^\alpha h + v_{2}^b \partial_{2} Z^\alpha h - \partial_{t} Z^\alpha h \|_{H^{1 \over 2}(\mathbb{R}^2)}
       \leq \Lambda_{\infty, 6} \big(  |v^b|_{m-{1 \over 2}} + |h|_{m- {1 \over 2}}\big).$$
       Consequently, we obtain that
     \begin{align*}
      \|\partial_{z} \big( \psi_{z} \star_{y}\big( v_{1}^b\partial_{1} Z^\alpha h + v_{2}^b \partial_{2} Z^\alpha h - \partial_{t} Z^\alpha h \big)\|_{L^2(\mathcal{S}) }  & \leq   |v_{1}^b\partial_{1} Z^\alpha h + v_{2}^b \partial_{2} Z^\alpha h - \partial_{t} Z^\alpha h \big)|_{1\over 2} \\
       &\leq  \Lambda_{\infty, 6} \big(  |v^b|_{m-{1 \over 2}} + |h|_{m- {1 \over 2}}\big). 
      \end{align*}
     By combining \eqref{Talphadecfin}, \eqref{Talphadecfin1}, \eqref{Talphadecfin2}  and the last estimate with the trace Theorem, 
      we finally obtain that
    $$ \| \mathcal{T}_{\alpha}\|_{L^2(\mathcal{S}) \cap |z| \leq 1} \leq \Lambda_{\infty, 6}\big(  \|v\|_{E^m} + |h|_{m} \big).$$
    Since we had already proven \eqref{Talpha0}, we finally obtain that
   \beq
   \label{CSzomega}
   \| \mathcal{C}_{Sz}\| \leq \Lambda_{\infty, 6}\big( \|v\|_{E^m} + |h|_{m}\big).
   \eeq
   
   To conclude the proof of Proposition \ref{propomeganh},  we  use the energy estimate \eqref{omeganh1}, 
     the commutator estimates   \eqref{CS2omega}, \eqref{CSyomega}, \eqref{CSzomega}, Lemma \ref{mingrad}
      and the Young inequality to obtain
    \begin{align*}
    &  \|\omega_{nh}^{m-1}(t) \|^2 + \eps \int_{0}^t  \| \nabla  \omega_{nh}^{m-1} \|^2 \\
   &    \leq \Lambda_{0}\| \omega(0)\|_{m-1}^2
      + \int_{0}^t \Lambda_{\infty, 6}\big( \|v\|_{E^m}^2  + \| \omega\|_{m-1}^2 + |h|_{m}^2  + \eps |h|_{m+{1 \over 2}}^2  \big)
       + \Lambda_{0} \int_{0}^t   \eps \| \nabla \omega \|_{m-2}^2.
     \end{align*}
        To end the proof, we  note that thanks to \eqref{gues}, we have
       $$ \sqrt{\eps} \| \nabla \omega \|_{m-2} \leq  \Lambda_{0}\big( \sqrt{\eps}\| \partial_{zz}v \|_{m-2} + \sqrt{\eps} \|\partial_{z}v\|_{m-1} \big) + \Lambda_{\infty}|h|_{m-{1\over 2}} $$
        and we use again  Lemma \ref{lemdzS} and \eqref{equiv1} to obtain that
       \begin{align*}
    &  \|\omega_{nh}^{m-1}(t) \|^2 + \eps \int_{0}^t  \| \nabla  \omega_{nh}^{m-1} \|^2 \\
   &    \leq \Lambda_{0}\| \omega(0)\|_{m-1}^2
      + \int_{0}^t \Lambda_{\infty, 6}\big( \|V^m\| ^2  + \| \omega\|_{m-1}^2 + |h|_{m}^2  + \eps |h|_{m+{1 \over 2}}^2  \big)
       + \Lambda_{0} \int_{0}^t   \eps \| \nabla S_{\n} \|_{m-2}^2.
     \end{align*} 
      
    This ends the proof of Proposition \ref{propomeganh}.

  \end{proof}
  
  \subsection{Estimate of $\omega_{h}^\alpha$}
  \label{sectionheat}
   It remains to get an estimate for $\int_{0}^t | \omega^{m-1}_{h}|^2$
    where $\omega^\alpha_{h}$ solves the homogeneous equation \eqref{eqomegah}, with the non-homogeneous  boundary condition \eqref{omegahc}.
    Note that the only estimate on $ (Z^\alpha \omega)^b$ that we have at our disposal is the $L^2$ in time
     estimate \eqref{omegab2}. This creates two difficulties, the first one is that we cannot  easily lift
      this boundary condition and  perform a standard energy estimate. The second one  is that due to this poor estimate on  the 
       boundary value, we cannot hope an estimate of 
      $\sup_{[0, T]} \| \omega^{m-1}_{h}\|$ independent of $\eps$.
     \subsubsection{A simple computation on the heat equation}
    To understand the difficulty,   let us consider the heat equation 
      \beq
      \label{exheat}
      \partial_{t} f - \eps \Delta f= 0, \quad   x= (y,z) \in \mathcal{S}
      \eeq 
      with zero initial data and the boundary condition $f(t,y,0) =f^b(t,y)$
    \begin{lem}
    \label{lemheat}
    The solution of the above equation satisfies the estimate:
  $$  \int_{0}^{+ \infty} e^{-2\gamma t } \big\|( \gamma + |\partial_{t}|)^{1\over 4}) f \big\|^2 \, dt
     \leq \sqrt{\eps} \int_{0}^{+ \infty} e^{-2 \gamma t}|f^b|_{L^2(\mathbb{R}^2)}^2 \, dt.$$
    \end{lem}
    Consequently, we see that we  get a control of $f$ which is  $H^{1\over 4}_{t}(L^2(\mathcal{S})).$
     Note that by Sobolev embedding, this implies that we  can expect a control of 
     $ \| f\|_{L^4_{t}(L^2(\mathcal{S}))}$ (and not of   $\| f\|_{L^\infty_{t}(L^2(\mathcal{S}))}$).
    \begin{proof}
    We can use a Laplace Fourier transform in time and $y$ to get the  equation
    $$ - \eps \partial_{zz} \hat f + (\gamma + i \tau + \eps | \xi |^2 ) \hat f= 0, \quad \hat{f}_{/z=0} = \hat{f}^b$$
    where 
    $$\hat{f}(\gamma, \tau, \xi, z)= \int_{0}^{+ \infty}\int_{\mathbb{R}^2} e^{-\gamma t - i \tau t - i \xi \cdot y} f(t,y,z) \, dt dy $$
    with $\gamma >0$.
    
    We can solve explicitely this equation to get
    $$\hat{f}= e^{ \big( \gamma + i \tau + \eps |\xi|^2\big)^{1 \over 2} {z\over \sqrt{\eps} } }\hat{f}^b, \quad z<0$$
     and hence, we find that
     $$ | \hat{f}(\gamma, \tau, \xi, \cdot ) |_{L^2_{z}}^2 \leq   { \sqrt{\eps} \over (\gamma + |\tau | + \eps |\xi|^2)^{1 \over 2}  }
     |\hat{f}^b|_{L^2_{z}}^2.$$ This  yields
     $$ (\gamma + |\tau|)^{1 \over 2}  | \hat{f}(\gamma, \tau, \xi, \cdot ) |_{L^2_{z}}^2 \leq \sqrt{\eps}   |\hat{f}^b|_{L^2_{z}}^2$$
      and  the result follows   from the Bessel-Parseval identity.
       \end{proof}

    \subsubsection{Statement of the estimate of $\omega_{h}^\alpha$}
    
    \begin{prop}
    \label{propomegah}
    For $T \in [0, T^\eps]$, $T^\eps \leq 1$,  assume that for $M>0$, the estimate
    \beq
    \label{hypomegah}
    \sup_{[0, T]} \Lambda_{\infty, 6}(t) +  \int_{0}^T \big( \eps \|\nabla V^6 \|^2 + \eps \| \nabla S_{\n}\|_{4}^2 \big) \leq M
    \eeq
   holds.  Then, there exists $\Lambda(M)$ such that for $|\alpha| \leq m-1$, we have  
    \begin{align*}
  \| \omega^\alpha_{h} \|_{L^4([0, T], L^2(\mathcal{S}))}^2 \leq  \Lambda(M)  \int_{0}^T \big( \|V^m \|^2 + \|h\|_{m}^2\big) +
   \eps  \int_{0}^t \| \nabla V^m\|^2.
     \end{align*}
        \end{prop}
        Again, note that the last term in the right hand side of the above inequality can be estimated by using
         Proposition \ref{conormv}.
         
   We will prove this estimate by using a  microlocal  symmetrizer.  
    In the convection diffusion equation \eqref{eqomegah}, the convection term   creates some 
 difficulties. Indeed, 
     the convection operator  
    $\partial_{t}^\varphi + v \cdot \nabla^\varphi= \partial_{t} + v_{y}\partial_{y}+ V_{z}\partial_{z}$ dominates in the low frequency regime.
     For this operator, the boundary is characteristic since 
     $V_{z}$ vanishes on the boundary and  the fact that $V_{z}$ is not uniformly zero in a vicinity of the boundary
       is not convenient when perfoming microlocal energy estimate 
 (see \cite{Majda-Osher}). 
       More importantly,  if we add a convection term to \eqref{exheat}, even constant,  that is  to say that we study
       $$ \partial_{t} f + c \cdot \nabla_{y} f - \eps \Delta f=0$$
       for $c \in \mathbb{R}^2$,  then the result of Lemma \ref{lemheat} becomes
       $$  \int_{0}^{+ \infty} e^{-2\gamma t } \big\|( \gamma + |\partial_{t} + c \cdot \nabla_{y}|)^{1\over 4}) f \big\|^2 \, dt
     \leq \sqrt{\eps} \int_{0}^{+ \infty} e^{-2 \gamma t}|f^b|_{L^2(\mathbb{R}^2)}^2 \, dt.$$
     In other words, the smoothing effect  does not occur in the direction $\partial_{t}$ but in the direction $\partial_{t}+ c\cdot  \nabla_{y} $.
      This will be much more difficult to detect when $c$ has variable coefficients. Consequently, in order to fix this
      difficulty, we shall use Lagrangian coordinates in order to eliminate the convection term and study a purely parabolic problem.

   \subsubsection{Lagrangian coordinates}
    Let us define  a parametrization of $\Omega_{t}$ by
    \beq
    \label{lagrange1}
    \partial_{t} X(t,x)= u(t, X(t,x))= v(t,\Phi(t, \cdot)^{-1}\circ X),  \quad X(0, x)= \Phi(0,x)
    \eeq
    where $\Phi(t, \cdot)^{-1}$ stands for the inverse of the map $\Phi(t, \cdot)$ defined by \eqref{diff}. 
    Note that the meaning of the initial condition is that  we choose the parametrization of $\Omega_{0}$ 
     which is given by \eqref{diff}.
     Let us also  define $J(t,x)= | \mbox{det }\nabla X(t,x)|$ the Jacobian of the change of variable.
     
     We have the following estimates for $X:$
     
     \begin{lem}
     \label{lemlagrange}
     Under the assumption of Proposition \ref{propomegah}, we have for $t \in[0, T]$ the estimates
     \begin{eqnarray}
     \label{J0}   & &|J(t,x)|_{W^{1, \infty}}+ |1/J(t,x)|_{W^{1, \infty}} \leq  \Lambda_{0}, \\
     \label{nablaX} & & \| \nabla X(t) \|_{L^\infty}  + \| \partial_{t} \nabla X \|_{L^\infty} \leq \Lambda_{0} e^{ t\Lambda(M)},  \\
     \label{dtnablaX} & & \| \nabla X\|_{1, \infty} +   \| \partial_{t} \nabla X\|_{1, \infty}  \leq \Lambda(M)e^{ t\Lambda(M)}, \\
     \label{ZnablaX}  & &
         \sqrt{\eps}\| \nabla^2 X \|_{1, \infty} 
        +   \sqrt{\eps}\| \partial_{t} \nabla^2 X \|_{L^\infty} \leq   \Lambda(M)(1+t^2) e^{t \Lambda(M)}.
        \end{eqnarray}
     
     \end{lem}
     
     \begin{proof}
       Since $u$ is divergence free  we get that $J(t,x)= J_{0}(x)$ and  the first estimate follows from Proposition \ref{propeta}.
             Next, by using the ordinary differential equation \eqref{lagrange1}, we get that
       \beq
       \label{eqDX}\partial_{t} DX= Dv D\Phi^{-1} DX\eeq
       where $D$ stands for the differential with respect to the space variable and hence we find
      $$ |\nabla X(t,\cdot )|_{L^\infty} \leq  \Lambda_{0}+ \Lambda_{0} \int_{0}^t \|v\|_{E^{1, \infty}} | \nabla X(t)|_{L^\infty}
       \leq \Lambda_{0} + \Lambda(M) \int_{0}^t | \nabla X(t)|_{L^\infty}.$$
       This yields the first part of  \eqref{nablaX} by using the Gronwall inequality. Moreover, by using again the equation \eqref{eqDX},  we get  \eqref{dtnablaX}.
       Next, by applying  one conormal derivative to \eqref{eqDX},  we also get  that
       $$ \| \nabla X \|_{1, \infty} \leq \Lambda_{0}+ \Lambda_{0} \int_{0}^t \|v\|_{E^{2, \infty}}  \| \nabla X \|_{1, \infty}$$
        and hence we  get  \eqref{dtnablaX} from the Gronwall inequality. The estimate for the time derivative
         follows again from the equation.  The estimate \eqref{ZnablaX} can be obtained in the same way.  By applying
          $\sqrt{\eps}\nabla$ to \eqref{eqDX}, we find
        $$ \sqrt{\eps } \| \nabla^2 X\|_{1, \infty}  \leq \Lambda(M) + \Lambda(M) \int_{0}^t \sqrt{\eps} \| \nabla^2 X\|_{1, \infty}
          + \Lambda(M) e^{\Lambda(M) t} \int_{0}^t   \sqrt{\eps} \| \nabla^2 v\|_{1, \infty}$$
           and hence, by using Corollary \ref{corLinfty0t} and the assumption \eqref{hypomegah}, we find
          $$  \sqrt{\eps } \| \nabla^2 X\|_{1, \infty}  \leq \Lambda(M) ( 1+t)^2 e^{\Lambda(M) t} + \Lambda(M) \int_{0}^t \sqrt{\eps} \| \nabla^2 X\|_{1, \infty}$$
           and the  first part of \eqref{ZnablaX} follows by using the Gronwall inequality. For the second part of \eqref{ZnablaX}, 
            it suffices to  apply $\sqrt{\eps}\nabla$ to \eqref{eqDX} and to  use \eqref{hypomegah} and that
             $ \sqrt{\eps}\| \nabla ^2 v\|_{L^\infty} \leq  \Lambda_{\infty, 6}$.
                 \end{proof}

  \subsubsection{Equation in Lagrangian coordinates}
  Let us set 
  \beq
  \label{Omegaalphadef} 
  \Omega^\alpha=e^{-\gamma t} \omega_{h}^\alpha (t, \Phi^{-1}\circ X)\eeq
 where $\gamma >0$ is a large parameter to be chosen. Then  $\Omega^\alpha$ solves in $\mathcal{S}$  the equation
  \beq
  \label{omegalagrange} a_{0}\big( \partial_{t}\Omega^\alpha+ \gamma \Omega^\alpha\big) - \eps \partial_{i}\big( a_{ij} \partial_{j}\Omega^\alpha \big) = 0 
  \eeq
  where we have used  the summation convention. Note that   $ a_{0}= |J_{0}|^{1\over 2}$ and that  the matrix $ (a_{ij})_{i,j}$ is defined by 
  $$  (a_{ij})_{ij} =  |J_{0}|^{1 \over 2} P^{-1}, \quad P_{ij}=  \partial_{i}X \cdot \partial_{j} X.$$
  Note that  the coefficients of the matrix $(a_{ij})$ will therefore match the estimates of Lemma \ref{lemlagrange}.
    Moreover, we also obtain thanks to Lemma \ref{lemlagrange}  that
   \beq
   \label{adefp}
    a(t,x)_{ij} \xi_{i}  \xi_{j} \geq  {1 \over \Lambda_{0}} e^{-t \Lambda(M) t} |\xi|^2 \geq c_{0}|\xi|^2
    \eeq
    where $c_{0}$ depends only on $M$ for $t \in [0, 1]$.
     On the boundary, we get that  
     \beq
     \label{bordOmega} \Omega^\alpha_{/z=0}= (\Omega^\alpha)^b:= e^{-\gamma t} \omega^\alpha_{nh}(t,\Phi^{-1}\circ X(t, y,0)).
     \eeq
     We shall prove that:
     \begin{theoreme}
     \label{theomicrolocal}
      There exists $\gamma_{0}$ which depends only on $M$ defined by \eqref{hypomegah} such that for $\gamma \geq \gamma_{0}$,
       the solution of \eqref{omegalagrange} with the boundary condition \eqref{bordOmega} satisfies the estimate
       $$ \big\| \Omega^{m-1}\|_{H^{1 \over 4}(0, T, L^2)}^2 \leq  \Lambda(M)
\sqrt{\eps} \int_{0}^T  | (\Omega^{m-1})^b|_{L^2(\mathbb{R})^2}^2$$
where $\Lambda(M)$ is uniformly bounded for $T\in [0, 1]$.     
     \end{theoreme}
     Note that  we define the norm  $H^{1 \over 4}((0,T), L^2))$ as
     \beq
     \label{H1/4}
    \big\|  f \|_{H^{1 \over 4}(0, T, L^2)}^2 = \inf \big\{ \|Pf\|_{H^{1 \over 4}(\mathbb{R}, L^2(\mathcal{S}))}, \quad Pf= f \mbox{ on } [0, T]
      \times \mathcal{S}\big\}\eeq
      and we define the norm on the whole space by Fourier transform in time.
     Let us first explain how to deduce the proof of Proposition \ref{propomegah} from the estimate of Theorem \ref{theomicrolocal}.
     We first observe that
     $$ \| \Omega^\alpha \|_{L^4([0, T], L^2(\mathcal{S}))}^2
      \leq    C \| \Omega^\alpha \|_{L^2(\mathcal{S}, L^4(0, T))}^2
       \leq C  \| \Omega^{m-1}\|_{H^{1 \over 4}(0, T, L^2)}^2.$$
      where the last estimate comes from the one-dimensional Sobolev embedding $H^{1 \over 4} \subset L^4$. 
       Note that with  the definition  \eqref{H1/4},   $C$ does not depend on $T$ since $C$ is given by the Sobolev
        embedding in the whole space. Next, we use \eqref{theomicrolocal} to get
        $$ \| \Omega^\alpha \|_{L^4([0, T], L^2(\mathcal{S}))}^2 \leq  \Lambda(M) \sqrt{\eps} \int_{0}^T e^{- \gamma t } | (\Omega^{m-1})^b|_{L^2(\mathbb{R})^2}^2$$ 
        and thanks to  \eqref{bordOmega} and \eqref{Omegaalphadef},  we obtain by change of variable
         (let us recall that $J$ satisfies the estimate \eqref{J0}) that
         $$  \| \omega^\alpha_{h} \|_{L^4([0, T], L^2(\mathcal{S}))}^2 \leq  \Lambda(M) \sqrt{\eps} \int_{0}^T  | (\omega_{h})^b|_{m-1}^2.$$
          We finally end the proof of  Proposition \ref{propomegah} by using \eqref{omegab2} and the Young inequality.
         
          \bigskip
          It remains to prove Theorem \ref{theomicrolocal}.  In the next subsection, we shall
 study a constant coefficient model.  In the analysis of this model,  we will  
  construct symbols which will allow us  to study  the original problem
            with the help of the paradifferential calculus described in section \ref{sectionparadiff}.
    \subsubsection{Symbolic Analysis}
    In this section, we shall perform a symbolic analysis. This will allow to get  our energy estimate for the equation
     \eqref{omegalagrange}, by using the semi-classical  paradifferential calculus with parabolic homogeneity described
      in section \ref{sectionparadiff}.
      
      Let us consider   $a=(a_{0}(z), a_{ij}(z))$  smooth in $z$  where $a_{ij}$ is a    symmetric matrix
     and $a_{0}$ a  function   with values in the compact set  $\mathcal{K}$ defined by 
       \beq
       \label{compact} a_{0}\geq m ,\, a_{3,3} \geq m,  \quad    |a| + \sqrt{\eps}| \partial_{z} a |  \leq M, \quad (a_{ij}) \geq c_{0} \mbox{Id}.
       \eeq
       where $m$, $c_{0}$ and $M$ are positive numbers.
       This means that we neglect the dependence in $t, \,y$ in the coefficients of \eqref{omegalagrange}.
      The symbolic version of \eqref{omegalagrange} becomes
      \beq
      \label{omegasymbolic}
     \big( {a_{0}\over a_{3,3}} (\gamma + i \tau ) + A_{y}(a, \sqrt{\eps} \xi )\big) \Omega  + A_{z}(a, \sqrt{\eps}\xi)\sqrt{\eps} \partial_{z} \Omega= \eps \partial_{zz} \Omega +
       R  + F, \quad z<0, \quad  \Omega(0)= \Omega^b
       \eeq
       where $A_{y}$ is homogeneous of degree two in $\xi$, $A_{z}$  is homogeneous of degree one in $\xi$ and
       \beq
       \label{defArond}
       A_{y}(a,   \xi)= \sum_{1\leq i,j \leq 2} {a_{ij} \over a_{3,3}}\xi_{i} \xi_{j}, \quad A_{z}(a,\xi)=  - 2 \sum_{1 \leq k \leq 2}{a_{k3}\over a_{3,3}}\,  i \xi_{k},
       \eeq
       and
       $$ 
        R=   i \sum_{1 \leq k \leq 2} \sqrt{\eps} {\partial_{z} a_{3k}\over a_{3,3}}  \sqrt{\eps}\xi_{k} \Omega+ \sqrt{\eps}
        {\partial_{z}a_{33}
        \over a_{3,3}}\, \sqrt{\eps}\partial_{z}\Omega.
        $$
        Note that we have incorporated  $F$ in \eqref{omegasymbolic} which  is a given source term.
         Since in the system  \eqref{omegalagrange}, the equations for the components of $\Omega$ are not coupled, we  can study
           separately each component and hence we shall assume that $\Omega \in \mathbb{R}$ in this subsection. 
        
       Thanks to \eqref{compact}, we have that
        \beq
        \label{rdexi}
        |R| \lesssim \sqrt{\eps} | \xi| | \Omega|+ \sqrt{\eps}|\partial_{z}\Omega|.
        \eeq
        Let us set $\zeta^\eps= (\gamma, \tau, \eps \xi)$   and $\langle \zeta^\eps \rangle= \big( \gamma^2 + \tau^2+ | \sqrt{\eps} \xi |^4 \big)^{1 \over 4}$.
        We rewrite \eqref{omegasymbolic} as a system  by setting
        \beq
        \label{Udef}
         U=\big(\Omega, \sqrt{\eps} \partial_{z}\Omega/  \langle \zeta^\eps \rangle \big)^t.
        \eeq
       This yields the system 
       \beq
       \label{systemOmega}
       \sqrt{\eps} \partial_{z} U= \langle  \zeta^\eps \rangle 
\mathcal{A}\big( a, \tilde{\zeta}^\eps\big) U   + \mathcal F
       \eeq
        where
        \beq
        \label{matrixOmega}
         \mathcal{A}\big(  a, \tilde{\zeta}\big)= \left(\begin{array}{cc}  0 & 1 \\   \ {a_{0}\over a_{3,3}}(\tilde \gamma +  i \tilde\tau)
          + A_{y}(a, \tilde{\xi}) &  A_{z}(a, \tilde{\xi})\end{array} \right)
          \eeq
          and  where we set  $\tilde \zeta  = ( \tilde{\gamma}, \tilde{\tau},  \tilde{\xi})$ 
           with
          $$ \tilde \gamma = \gamma / \langle \zeta \rangle^2, \quad \tilde \tau= \tau  / \langle \zeta \rangle^2, \quad
            \tilde \xi=  \xi/ \langle \zeta \rangle$$
            and  for  $\tilde{\zeta}^\eps$, we replace $(\gamma, \tau, \xi)$ by $(\gamma, \tau, \eps \xi)$. 
          Moreover,  the source term $\mathcal{F}$ is defined by 
          \beq
          \label{Fdef}
          \mathcal F={1 \over \langle \zeta^\eps \rangle} (0, R  + F).
          \eeq
          The boundary condition at $z=0$ becomes
          \beq
          \label{Gammadef}
          \Gamma U= \Omega^b,  \quad \Gamma (U_{1}, U_{2})^t= U_{1}
          \eeq
          where  the writting    $U=(U_{1}, U_{2})^t$ is related to the block structure  of \eqref{matrixOmega}.
      
      Note that $\tilde{\zeta}$ is in the compact set $\tilde\gamma^2 + \tilde\tau^2 + |\tilde \xi|^4 = 1, \, \tilde\gamma \geq 0.$
      
      \begin{prop}
      \label{propmodele}
      There exists $\gamma_{0}>0$ such that for every $\gamma \geq \gamma _{0}$, 
      we have for the solution of \eqref{systemOmega}, \eqref{Gammadef} the estimate
     $$  \langle \zeta^\eps \rangle |U|_{L^2_{z}}^2 + \sqrt{\eps}|U(0)|^2  \leq C \sqrt{\eps} |\Omega^b|^2 + \Big( {| F|_{L^2_{z}} \over \langle \zeta^\eps \rangle^{3\over 2}}\Big)^2. $$
      \end{prop}
      Note that $\gamma_{0}$ only depends on the estimates  \eqref{compact}.  In  terms of $\Omega$, the previous 
       proposition gives  the estimate
      \beq
      \label{estmodele}  \langle \zeta^\eps \rangle |U|_{L^2_{z}}^2 =    \langle \zeta^\eps\rangle  | \Omega|_{L^2_{z}}^2 +  { \eps \over \langle \zeta^\eps \rangle} |
        \partial_{z} \Omega|_{L^2_{z}}^2 \leq  C \sqrt{\eps} |\Omega^b|^2 + \Big( {| F|_{L^2_{z}} \over \langle \zeta^\eps \rangle^{3\over 2}}\Big)^2.
        \eeq 
       
      We shall prove this Proposition by using the symmetrizers method. 
   \begin{lem}
   \label{lemsym}
   There exists $\mathcal{S}\big( a, \tilde{\zeta}\big) $ symmetric,  smooth in its argument and $\kappa>0$ such that
   $$  \mathcal{S}\mathcal{A} + (\mathcal{S} \mathcal{A})^* \geq \kappa Id, \quad \mathcal{S} + \Gamma^* \Gamma \geq \kappa Id $$
   for every $(a,\tilde \zeta)  \in \mathcal{K}\times S_{+}$   where $\mathcal{K}$ is the compact set  defined by \eqref{compact} and
    $$S_{+}= \{\tilde{\zeta}, \,  \langle \tilde \zeta \rangle = 1,\,  
    \tilde{\gamma} \geq 0\}.$$ 
   \end{lem}

   The proof of the Proposition can be easily obtained from the Lemma.
    We multiply the equation by $\mathcal{S}\big( a(z), \tilde{\zeta}\big) U$ and we integrate in $z$. This yields by using Lemma \ref{lemsym}
    $$ \kappa \big( \langle \zeta^\eps \rangle |U|_{L^2_{z}}^2 + \sqrt{\eps}|U(0)|^2 \big)
     \leq  C\Big( |\mathcal F|_{L^2_{z}} |U|_{L^2_{z}}+ \sqrt{\eps}| \partial_{z} \mathcal{S}|_{L^\infty} |U|_{L^2_{z}}^2 + \sqrt{\eps} |\Gamma U(0)|^2\Big).$$  
     Note that $ \sqrt{\eps}| \partial_{z} \mathcal{S}|_{L^\infty}$ is uniformly bounded thanks to \eqref{compact}
      and hence  we get from Cauchy-Schwarz that
      $$  \langle \zeta^\eps \rangle |U|_{L^2_{z}}^2 + \sqrt{\eps}|U(0)|^2  \leq C \Big(  { |\mathcal F|_{L^2_{z}}^2 \over \langle \zeta^\eps \rangle }
       + |U|_{L^2_{z}}^2 + \sqrt{\eps} |\Gamma U(0)|^2 \Big).$$
       By using \eqref{Fdef},   we get that $|\mathcal F| \lesssim  |U|$
        and hence the result follows by choosing $\gamma$ sufficiently large.
        
       \bigskip

It remains  the proof of Lemma \ref{lemsym}.  It can be obtained from classical arguments.

We  first prove that  the eigenvalues of 
$ \mathcal{A}$ has nonzero real part.  If $X=(X_{1}, X_{2})$ is an eigenvector  of $\mathcal{A}(a, \tilde{\zeta})$ associated to the eigenvalue $\mu$, we get
 that $X_{2}= \mu X_{1}$  and that
 $$ a_{33}\mu^2 X_{1} =\Big( a_{0}(\tilde \gamma +i  \tilde \tau) +a_{3,3} A_{y}(\tilde{\xi}) + \mu a_{3,3}A_{z}(\tilde{\xi}) \Big)X_{1}.$$
 If we assume that $\mu = i \lambda$, this yields
 $$a_{0}(\tilde \gamma +  i  \tilde \tau) + \eta_{i}\eta_{j} a_{ij}= 0$$
 where $\eta=(\xi_{1}, \xi_{2}, \lambda)$. By using \eqref{compact}, we find by taking the real part
  that $\eta=0$ and $\tilde \gamma = 0$ and thus also $\tilde \tau=0$. This is impossible for
   $\langle \tilde \zeta \rangle = 1$. Consequently,  there is no eigenvalue on the imaginary axis. Moreover, we easily see 
   that there is  one eigenvalue with positive real  part $\mu_{+}$  and one with negative real part $\mu_{-}$.
     This yields that we can diagonalize  $\mathcal{A}$ in  a smooth way: there exists a smooth invertible matrix
      $\mathcal{P}(a, \tilde{\zeta})$ such that
      $$ \mathcal{A}= \mathcal{P}\left(\begin{array}{cc} \mu_{+} & 0 \\ 0 & \mu_{-} \end{array} \right)\mathcal{P}^{-1}.$$
      
      We  thus  choose $\mathcal{S}$ in a classical way  under the form
      $$ \mathcal{S}=  (\mathcal{P}^{-1})^*\left( \begin{array}{cc}  1 & 0 \\ 0 & -\delta \end{array} \right)\mathcal{P}^{-1}$$
      with $\delta>0$ to be chosen. The first property on $\mathcal{S}$ in Theorem \ref{lemsym} is therefore met.
       Note that  smooth   projections  on the subspace associated to the positive and respectively negative  eigenvalue of $\mathcal{A}$
        are given by 
        $$ \Pi_{+}= \mathcal{P} \left( \begin{array}{cc}  1 & 0 \\ 0 &  0  \end{array} \right)\mathcal{P}^{-1}, \quad
         \Pi_{-}= \mathcal{P} \left( \begin{array}{cc}  0 & 0 \\ 0 &  1  \end{array} \right)\mathcal{P}^{-1}$$

      To get the second one, we note that  if $X$ is an  eigenvector  of $\mathcal{A}(a,  \zeta)$   associated with the eigenvalue of positive real part (we recall
       that we study the equation for $z<0$)
       we necessarily have  $\Gamma X \neq 0$.
         In other words, we have  $\mbox{Ker }\Gamma \cap  \mbox{Ker }\Pi_{-}$.  This yields
            that  the map $X\in \mathbb{R}^2\mapsto  (\Gamma X, \Pi_{-} X)$  is invertible and hence by compactness, there exists $ \alpha >0$ such that for
              every $X \in \mathbb{R}^2$, every $a$ in the compact set \eqref{compact}   and $\tilde{\zeta}$, $\langle \tilde \zeta \rangle= 1$, $\tilde \gamma \geq 0$, we have 
             $$  | \Gamma X|^2 + | \Pi_{-}(a(0), \zeta )X |^2 \geq \alpha |X|^2.$$ 
              This allows to choose $\delta $ sufficiently small in order to get the second property in Lemma \ref{lemsym}.


     \subsection{Energy estimate via microlocal symmetrizer}
   We shall now give the proof of  Theorem  \ref{theomicrolocal}.
   Let us take $T\in [0, T^\eps]$. We  consider the solution of \eqref{omegalagrange}, \eqref{bordOmega} on $[0,T]$. Since
     the initial value is zero, we can assume that $\Omega^\alpha$ is zero for $t \leq 0$.
    Thanks to Lemma \ref{lemlagrange},  we can assume that $a_{0}$ and $a_{ij}$ verify the estimates  \eqref{compact}
     on $[0, T] \times \mathcal{S}$ by a suitable choice of the numbers $m$, $M$ and $c_{0}$
      (thus these numbers depend on $\Lambda_{\infty, T}$)
      Note that thanks to Lemma \ref{lemlagrange}, we  can also assume that 
      \beq
      \label{compact2} \|\partial_{t,y}( a_{0}, a_{ij}) \|_{L^\infty} + \sqrt{\eps}   \|\partial_{t,y}\nabla a_{ij}\|_{L^\infty} \leq M.
      \eeq
       
      We can   choose  extensions of these coefficients on
       $\mathbb{R} \times \mathcal{S}$ such that the new coefficients still satisfy \eqref{compact} and the above estimate.
        We shall not use a different notation for the extensions. We also extend the boundary value $(\Omega^\alpha)^b$
         in \eqref{bordOmega} by $0$ for $t\geq T$ and for $t \leq 0$ and denote by $g$ this new function. We shall consider
          the solution of
         \beq
         \label{etalagrange}
         a_{0}\big( \partial_{t}+ \gamma\big)\rho - \eps \partial_{i}\big( a_{ij} \partial_{j}\rho \big) = 0, \quad (t,x) \in \mathbb{R}\times \mathcal{S}
         , \quad \rho(t,y,0)=g
         \eeq
   which vanishes for $t \leq 0$. By using a standard uniqueness result for this parabolic equation, we get
   that $\rho = \Omega^\alpha$ on $[0, T] \times \mathcal{S}$ consequently, by using the notations of section \ref{sectionparadiff},  it  suffices to prove the
   estimate:
   \begin{prop}
   The solution of \eqref{etalagrange} satisfies the estimate
   \label{propparaeta}
   \beq
   \label{aprouver} \| \rho \|_{\mathcal{H}^{{1\over 2}, \gamma, \eps}}^2 \leq C  \sqrt{\eps} \|g\|_{L^2(\mathbb{R}^3)}^2.
   \eeq
     for $\gamma$ sufficiently large, where  $C$ depending only   of the parameters in \eqref{compact}, \eqref{compact2}
     (in particular it is independent of $\eps$)
     \end{prop}
     Indeed, let us assume that this last proposition is proven, then
     since
     $$  \|\rho\|_{H^{1\over 4}(\mathbb{R}, L^2(\mathcal{S}))} \leq \| \rho \|_{\mathcal{H}^{{1\over 2}, \gamma, \eps}} $$
    we find that 
     \beq
     \label{prouve}  \| \Omega^\alpha \|_{H^{1 \over 2 }([0, T], L^2(\mathcal{S}))}^2 \leq C   \sqrt{\eps} \|(\Omega^\alpha)^b\|_{L^2([0, T] \times \mathbb{R}^2)}^2.\eeq
   We shall thus focus on the proof of \eqref{aprouver}. Again, we shall consider that $\rho$ is a scalar function.
   By using the notations of the previous subsection, we can  define two symbols (with $z$ as parameter)  $a_{y}$ and $a_{z}$ by
   $$ a_{y}(X, \zeta,z)= A_{y}(a(t,y,z),  \xi), \quad  a_{z}(X, \zeta,z)= A_{z}(a(t,y,z),  \xi),$$ in such a way that
    we can rewrite \eqref{etalagrange} under the form
    \beq
    \label{etalagrange2}\eps \partial_{zz} \rho=  {a_{0} \over a_{33}}\big( \partial_{t} + \gamma ) \rho   +  a_{y}(X, \sqrt{\eps} \partial_{y}, z) \rho
      +a_{z}(X, \sqrt{\eps} \partial_{y}, z) \sqrt{\eps} \partial_{z} \rho +  R^1
      \eeq
      where 
      $$ R^1= {\sqrt{\eps}\partial_{i} a_{ij}\over a_{3,3}} \sqrt{\eps} \partial_{j} \rho.$$
      In view of the model estimate of Proposition \ref{propmodele}, we need to control $\| R^1 \|_{\mathcal H^{-{3 \over 2},  \gamma, \eps}}.$
 Let us set 
      $$b_{ij}= {\sqrt{\eps}\partial_{i} a_{ij}\over a_{3,3}}$$
      Let us start with the estimates of the terms  where $j\neq 3$. We first note that
      $$ \| b_{ij} \sqrt{\eps} \partial_{j} \rho\|_{\mathcal H^{-{3 \over 2},  \gamma, \eps}}
       \leq   \| \sqrt{\eps} \partial_{j}\big( b_{ij}  \rho\big)\|_{\mathcal H^{-{3 \over 2}, \gamma, \eps}}
        +   \| \sqrt{\eps}\partial_j b_{ij}  \rho\|_{\mathcal H^{-{3 \over 2},  \gamma, \eps}}.$$
        For the second term, thanks to the uniform estimate \eqref{compact} and  \eqref{compact2},  we can write
        $$  \| \sqrt{\eps}\partial_{j} b_{ij}  \rho\|_{\mathcal H^{-{3 \over 2}, \gamma, \eps}}
        \leq {1 \over \gamma^{3\over 4}}\| \sqrt{\eps}\partial_{j} b_{ij}  \rho\|_{L^2} \leq {C \over \gamma^{3\over 4}} \|\rho\|_{L^2}.$$
        For the first term thanks to the definition of the weighted  norm, we get
        $$  \| \sqrt{\eps} \partial_{j}\big( b_{ij}  \rho\big)\|_{\mathcal H^{-{3 \over 2},  \gamma, \eps}}
         \leq   \| b_{ij} \rho\big\|_{\mathcal H^{-{1 \over 2},  \gamma, \eps}}
          \leq {C \over \gamma^{1\over 4}} \| \rho \|_{L^2}.$$
        Consequently, we have proven that
      $$  \| b_{ij} \sqrt{\eps} \partial_{j} \rho\|_{\mathcal H^{-{3 \over 2},  \gamma, \eps}} \leq {C \over \gamma^{1\over 4} } \| \rho \|_{L^2}.$$
       for $j \neq 3$. It remains the case that $j=3$ which is more complicated. We shall use a duality argument and
        the paradifferential calculus of section \ref{sectionparadiff}. We first write
        $$ \| b_{ij} \sqrt{\eps} \partial_{z} \rho\|_{\mathcal H^{-{3 \over 2},  \gamma, \eps}} \leq {1 \over \gamma^{1 \over 2}}
          \| b_{ij} \sqrt{\eps} \partial_{z} \rho\|_{\mathcal H^{-{1 \over 2},  \gamma, \eps}}.$$
          Next for any  test function $f$ in the Schwartz class, we write
          $$ |\big( b_{ij} \sqrt{\eps} \partial_{z} \rho, f \big)_{L^2}| \leq  | \big(  \sqrt{\eps} \partial_{z} \rho, b_{ij} f \big)_{L^2}|
           \leq  \|\sqrt{\eps} \partial_{z} \rho\|_{\mathcal{H}^{-{1 \over 2}, \gamma, \eps}} \| b_{ij} f \|_{\mathcal{H}^{{1 \over 2}, \gamma, \eps}}.$$
         Next,  thanks to the estimate \eqref{compact}, \eqref{compact2}, we get in particular that 
          $\|b_{ij}\|_{L^\infty}$ and $\|\nabla_{t,y} b_{ij}\|_{L^\infty}$ are uniformly bounded, thus we find
          $$ \| b_{ij} f \|_{\mathcal{H}^{{1 \over 2}, \gamma, \eps}}
           \leq \|T^{\eps, \gamma}_{b_{ij}} f\|_{\mathcal{H}^{{1 \over 2}, \gamma, \eps}} +   \| (b_{ij}- T^{\eps, \gamma}_{b_{ij}}) f\|_{\mathcal{H}^{{1}, \gamma, \eps}} \leq C\|f\|_{\mathcal{H}^{{1 \over 2}, \gamma, \eps}} $$
            by using  in Theorem \ref{symbolic2} the estimate  (1)  and the first estimate in (5).
             This proves by duality  that
            $$   \| b_{ij}  \sqrt{\eps} \partial_{z} \rho \|_{\mathcal{H}^{ -{1 \over 2}, \gamma, \eps}} \leq C  \|\sqrt{\eps} \partial_{z} \rho\|_{\mathcal{H}^{-{1 \over 2}, \gamma, \eps}}$$
             and hence that
          \beq
          \label{bij1} \| b_{ij}  \sqrt{\eps} \partial_{z} \rho \|_{\mathcal{H}^{ -{3 \over 2}, \gamma, \eps}} \leq {C\over \gamma^{1\over 2}}  \|\sqrt{\eps} \partial_{z} \rho\|_{\mathcal{H}^{-{1 \over 2}, \gamma, \eps}}.\eeq
          Consequently, we have proven that
\beq
\label{R1para}
\|R^1 \|_{\mathcal{H}^{ -{3 \over 2}, \gamma, \eps}}
 \leq  {C\over \gamma^{1\over 4}} \big( \| \rho \|_{L^2} + \|\sqrt{\eps} \partial_{z} \rho\|_{\mathcal{H}^{-{1 \over 2}, \gamma, \eps}}\big).
 \eeq
 Note that in view of the left hand side in the model estimate of Proposition \ref{propmodele}, the above estimate is a good estimate
  since  for $\gamma$ sufficiently large, it  will be possible to absorb this term  by the principal term of our estimate.
       
   Next, we can replace products by paraproducts in the equation \eqref{etalagrange2} to rewrite it under the form
   \beq
   \label{etalagrange3}
   \eps \partial_{zz}\rho=  T^{\eps, \gamma}_{a_{0}/a_{33}}\big( \partial_{t} + \gamma \big)\rho
    + T^{\eps, \gamma}_{a_{y}} \rho  +  T^{\eps, \gamma}_{a_{z}}\sqrt{\eps} \partial_{z}\rho  +  R
    \eeq
    where
    \begin{eqnarray}
    \label{Rdefpara}
    & &  R= R^1 + R^2, \\
   \label{R2defpara}  & & R^2 = \big( {a_{0}\over a_{33}} - T^{\eps, \gamma}_{a_{0}/a_{33}}\big) \big( \partial_{t}+ \gamma\big)
      - \sum_{1 \leq i,  j \leq 2} \big( {a_{ij}\over a_{33}} - T^{\eps, \gamma}_{a_{ij}/a_{33}} \big) \eps\, \partial_{i} \partial_{j} \rho
       \\ 
     \nonumber  & & \quad  \quad \quad - 2 \sum_{1 \leq k \leq 2} \big( {a_{k3}\over a_{33}} - T^{\eps, \gamma}_{a_{k3}/a_{33}} \big) \eps  \partial_{k}
      \partial_{z} \rho.
  \end{eqnarray}
      As before, we need to estimate $ \|R^2 \|_{\mathcal{H}^{ -{3 \over 2}, \gamma, \eps}}$. For the first term in the definition of  $R^2$, 
       we  can use \eqref{compact} and \eqref{compact2} to write
       \begin{align*}
       &   \| \big( {a_{0}\over a_{33}} - T^{\eps, \gamma}_{a_{0}/a_{33}}\big)\big( \partial_{t}+ \gamma) \rho \|_{\mathcal{H}^{ -{3 \over 2}, \gamma, \eps}}
       \\
       & \leq \big\| \big( \partial_{t} + \gamma \big)  \big( {a_{0}\over a_{33}} - T^{\eps, \gamma}_{a_{0}/a_{33}}\big)\rho \big\|_{\mathcal{H}^{ -{3 \over 2},\gamma, \eps}} + \big\| \partial_{t}({a_{0}\over a_{33}}\big) \rho \|_{\mathcal{H}^{ -{3 \over 2}, \gamma, \eps}} +\big\| 
T^{\eps, \gamma}_{\partial_{t}(a_{0}/a_{33})}\rho \big\|_{\mathcal{H}^{ -{3 \over 2}, \gamma, \eps}} \\
&  \leq \big\|  \big( {a_{0}\over a_{33}} - T^{\eps, \gamma}_{a_{0}/a_{33}}\big)\rho \big\|_{\mathcal{H}^{ {1 \over 2}, \gamma, \eps}}
 + {C \over \gamma^{3 \over 4}} \| \rho \|_{L^2} \\
 &\leq {C \over \gamma^{1 \over 4}}   \| \rho \|_{L^2}
\end{align*}
   where we have  used the $L^2$ continuity of paraproducts (i.e. (1) with $\mu=0$ in Theorem \ref{symbolic2})
    and the first estimate in (5)  of Theorem \ref{symbolic2}
    to get the two  last lines.  For the second type of terms in $R^2$,  we can  proceed  in the same way
     (note that $\eps \partial_{ij}$ is  an operator of order $2$  and hence has the same order as  $\partial_{t}$ in our calculus) 
     to obtain that 
 $$ \big\| \big( {a_{ij}\over a_{33}} - T^{\eps, \gamma}_{a_{ij}/a_{33}} \big) \eps\, \partial_{i} \partial_{j} \rho\big\|_{\mathcal{H}^{ -{3 \over 2}, \gamma, \eps}}
   \leq {C \over \gamma^{1 \over 4}} \| \rho \|_{L^2}.$$ For the last  type of terms in $R^2$, we first proceed in the same way to write
    \begin{align*}
       &   \|  \big( {a_{k3}\over a_{33}} - T^{\eps, \gamma}_{a_{k3}/a_{33}}\big)
 \eps \partial_{k}\partial_{z} \rho\|_{\mathcal{H}^{ -{3 \over 2}, \gamma, \eps}}  \\
       & \leq \big\| \sqrt{\eps}\partial_{k}  \big( {a_{k3}\over a_{33}} - T^{\eps, \gamma}_{a_{k3}/a_{33}}\big) \sqrt{\eps}\partial_{z}\rho \big\|_{\mathcal{H}^{ -{3 \over 2},\gamma, \eps}} + \big\|\sqrt{\eps} \partial_{k}({a_{k3}\over a_{33}}\big) \sqrt{\eps}\partial_{z} \rho \|_{\mathcal{H}^{ -{3 \over 2}, \gamma, \eps}} +\big\| T^{\eps, \gamma}_{\sqrt{\eps}\partial_{k}(a_{k3}/a_{33})} \sqrt{\eps}\partial_{z}\rho \big\|_{\mathcal{H}^{ -{3 \over 2}, \gamma, \eps}} \\
&  \leq  {1 \over \gamma^{1 \over 4}} \big\|  \big( {a_{k3}\over a_{33}} - T^{\eps, \gamma}_{a_{k3}/a_{33}}\big) \sqrt{\eps}\partial_{z}\rho \big\|_{L^2}
 + {C \over \gamma^{1 \over 2}} \| \sqrt{\eps}  \partial_{z}\rho \|_{\mathcal{H}^{- {1 \over 2}, \gamma, \eps}}.
\end{align*}
Indeed, to get the last line,  we have used the continuity of the paraproduct and  that the 
 term $ \big\|\sqrt{\eps} \partial_{k}({a_{0}\over a_{33}}\big) \sqrt{\eps}\partial_{z} \rho \|_{\mathcal{H}^{ -{3 \over 2}, \gamma, \eps}} $
  can be estimated in  a similar way  as it was done to get \eqref{bij1}.
   To estimate the first term in  the right hand side of the above inequality, we can again use a duality argument. 
   For any test function $f$, we have
   $$\Big( \big( {a_{k3}\over a_{33}} - T^{\eps, \gamma}_{a_{k3}/a_{33}}\big) \sqrt{\eps}\partial_{z}\rho, f \Big)_{L^2}
   = \Big( \sqrt{\eps}\partial_{z}\rho,  \big( {a_{k3}\over a_{33}} - T^{\eps, \gamma}_{a_{k3}/a_{33}}\big) f \Big) -
    \Big( \sqrt{\eps}\partial_{z}\rho,  \big(  (T^{\eps, \gamma}_{a_{k3}/a_{33}})^* - T^{\eps, \gamma}_{a_{k3}/a_{33}}\big) f \Big)$$
     where we have used that $a_{k3}/a_{33}$ is a scalar real valued function. Consequently, by using again Theorem \ref{symbolic2}, 
      we obtain that
   $$ \Big| \Big( \big( {a_{k3}\over a_{33}} - T^{\eps, \gamma}_{a_{k3}/a_{33}}\big) \sqrt{\eps}\partial_{z}\rho, f \Big)_{L^2}\Big|
    \leq  C \| \sqrt{\eps}\partial_{z}\rho\|_{\mathcal{H}^{- {1}, \gamma, \eps}} \, \|f\|_{L^2}$$
     and hence that
     $$  \big\|  \big( {a_{k3}\over a_{33}} - T^{\eps, \gamma}_{a_{k3}/a_{33}}\big) \sqrt{\eps}\partial_{z}\rho \big\|_{L^2}
      \leq   C \| \sqrt{\eps}\partial_{z}\rho\|_{\mathcal{H}^{- {1}, \gamma, \eps}}.$$
       Consequently, we finally obtain that
    $$\big\| \big( {a_{ij}\over a_{33}} - T^{\eps, \gamma}_{a_{ij}/a_{33}} \big) \eps\, \partial_{i} \partial_{j} \rho\big\|_{\mathcal{H}^{ -{3 \over 2}, \gamma, \eps}}
\leq {C \over \gamma^{1 \over 2}} \| \sqrt{\eps}  \partial_{z}\rho \|_{\mathcal{H}^{- {1 \over 2}, \gamma, \eps}}$$
and hence by collecting the previous estimates and \eqref{R1para}, we get that
\beq
\label{Rpara}
\|R \|_{\mathcal{H}^{ -{3 \over 2}, \gamma, \eps}}
 \leq  {C\over \gamma^{1\over 4}} \big( \| \rho \|_{L^2} + \|\sqrt{\eps} \partial_{z} \rho\|_{\mathcal{H}^{-{1 \over 2}, \gamma, \eps}}\big).
\eeq
Next, we can rewrite \eqref{etalagrange3}  as a first order system by setting
\beq
\label{Udefpara} U=\big ( \rho,  T^{\eps, \gamma}_{ 1/\langle \zeta \rangle} \sqrt{\eps}\partial_{z} \rho \big)^t.\eeq
 Note that $T^{\eps, \gamma}_{ 1/\langle \zeta \rangle}$ is just the Fourier multiplier by $1/\langle \zeta^\eps \rangle$. This yields the system
 \beq
 \label{systpara}
 \sqrt{\eps} \partial_{z} U = T^{\eps, \gamma}_{M} U + \mathcal{F}, \quad z <0
 \eeq
 where the symbol $M\in \Gamma_{1}^1$ is given by 
 $$ M(X, \zeta, z)= \langle \zeta \rangle  \mathcal{A}(a(X,z), \tilde{\zeta}), \quad \tilde{\zeta}=(\tilde \gamma, \tilde \tau, \tilde \xi)=
 (\gamma/\langle \zeta \rangle^2, \tau/\langle \zeta \rangle^2, \xi/\langle \zeta \rangle)$$
 where $\mathcal{A}$ is defined in \eqref{matrixOmega} and  the source term $\mathcal{F}$ is defined by
 $$ \mathcal{F}= (0, T^{\eps, \gamma}_{{1\over \langle \zeta \rangle}} R + \mathcal{C}\big)^t$$
 where $\mathcal{C}$ is  a lower order commutator. By using \eqref{Rpara} and  (2) in Theorem \ref{symbolic2} to estimate $\mathcal C$, we get that
 \beq
 \label{Fpara}
  \|\mathcal F\|_{\mathcal H^{ -{1 \over 2}, \gamma, \eps}} \leq {C \over \gamma^{1 \over 4}} \|U\|_{\mathcal H^{ {1 \over 2}, \gamma, \eps}}.
 \eeq
 The boundary condition for  \eqref{systpara} becomes
 \beq
 \label{bordpara}
 \Gamma U_{/z=0}=g 
 \eeq
 where $\Gamma$ is still defined by \eqref{Gammadef}. We can then perform an energy estimate for the system
  \eqref{systpara}, \eqref{bordpara} by using the symmetrizer $\mathcal{S}$ constructed in Lemma \ref{lemsym}.
   Indeed let us define a symbol $S(X, \zeta, z) \in \Gamma_{1}^0$ by 
   $$ S(X, \zeta, z) = \mathcal S ( a(X,z), \tilde \zeta, z).$$
   From Lemma \ref{lemsym}, we get that at the level of symbols, we have
   $$ \mbox{Re } SM \geq \kappa \langle \zeta \rangle, \quad  \mbox{Re }S_{/z=0}+ \Gamma^* \Gamma \geq \kappa.$$
    Consequently, we can perform an energy estimate in a standard way  by taking the scalar product of \eqref{systpara} with
     $T^{\eps, \gamma}_{S} U$ and by using Theorem \ref{symbolic2} (in particular  the estimates (2), (3)
      and the Garding inequality (4)). This yields
      $$ \|U\|_{\mathcal H^{ {1 \over 2}, \gamma, \eps}}^2 + \sqrt{\eps}\|U(0)\|_{L^2(\mathbb{R}^3)}^2
       \leq {C \over \gamma^{1 \over 4}}\big(   \|U\|_{\mathcal H^{ {1 \over 2}, \gamma, \eps}} + \sqrt{\eps}\|U(0)\|_{L^2(\mathbb{R}^3)}^2\big)
        +  \sqrt{\eps} C\|g\|_{L^2(\mathbb{R}^3)}^2 + \big|\big(\mathcal{F}, T^{\eps, \gamma}_{S} U \big)_{L^2} |.$$ 
  Since  by using \eqref{Fpara}, we get 
  $$ \big|\big(\mathcal{F}, T^{\eps, \gamma}_{S} U \big)_{L^2} | \leq  \|   \mathcal{F}\|_{\mathcal H^{ -{1\over 2}, \gamma, \eps} } \, 
   \| U\|_{\mathcal H^{ {1 \over 2}, \gamma, \eps}} \leq {C\over \gamma^{1 \over 4}} 
    \| U\|_{\mathcal H^{ {1 \over 2}, \gamma, \eps}}^2.$$
    Consequently for $\gamma$ sufficiently large (with respect to $C$),  we obtain the estimate
    $$  \|U\|_{\mathcal H^{ {1 \over 2}, \gamma, \eps}}^2 \leq  \sqrt{\eps} C\|g\|_{L^2(\mathbb{R}^3)}^2.$$
    To conclude, we note that 
    $$  \|U\|_{\mathcal H^{ {1 \over 2}, \gamma, \eps}}^2= \|\rho\|_{\mathcal H^{ {1 \over 2}, \gamma, \eps}}^2
     +  \|\sqrt{\eps}\partial_{z}\rho\|_{\mathcal H^{- {1 \over 2}, \gamma, \eps}}^2.$$
      This ends the proof of Proposition \ref{propparaeta}.

      
    \section{Proof of  Theorem \ref{main}: uniform existence} 
    \label{sectionexist} 
    In this section, we shall  prove how we can combine all our energy estimates to get our uniform 
   existence  result.  Let us  fix  $m \geq 6$.
      We consider  initial data such that 
      $$\mathcal{I}_{m}(0) =  \| v_{0}\|_{E^m}  + |h_{0}|_{m} + \sqrt{\eps}|h_{0}|_{m+{1 \over 2}} + \|v_{0}\|_{E^{2,\infty}}+
     \eps \| \partial_{zz} v(0)\|_{L^\infty}<+\infty.$$
     For such data, we are not aware of a  local existence result. 
     We could prove it  by using our energy estimates and a classical iteration scheme. Nevertheless, we can also  avoid this by 
      using  the available classical existence results in Sobolev spaces (for example  \cite{Beale81}, \cite{Tani96}). Indeed,  we can first smooth
       the initial  velocity   and consider  a sequence  $v_{0}^\delta \in H^{r-1}$, $3<r<7/2$ ($\delta$  being a regularization parameter) 
        to meet the assumption of   \cite{Beale81}, \cite{Tani96}.     
  This allows to get a positive time $T^{\eps, \delta}$  for which a solution $v$
         associated to this initial data exists
        in the space  $\mathcal{K}^r ([0, T^{\eps, \delta}]\times \mathcal{S})= H^{r\over 2} ([0, T], L^2) \cap L^2 ([0, T], H^r)$. 
  Next,   
     we can   get by  standard parabolic energy estimates  
       that additional regularity  propagates from the initial data, that is to say that on $[0, T^{\eps, \delta}], $ we have 
      \begin{align}
      \label{defNm}
        \mathcal{N}_{m}(T) &  = \sup_{[0, T]}\big( \|v(t) \|_{m}^2 + \|\partial_{z}v\|_{m-2}^2 + |h(t)|_{m}^2 + \eps |h(t)|_{m+{1 \over 2}}^2
        + \eps \|\partial_{zz}v(t)\|_{L^\infty}^2 +  \|v(t)\|_{E^{2, \infty}}^2 \big) \\
        \nonumber & \quad  \quad \quad   + \|\partial_{z}v\|_{L^4([0,T], H^{m-1}_{co})}^2 
        + \eps \int_{0}^T \|\nabla v\|_{m}^2 + \eps \int_{0}^t  \|\nabla \partial_{z}v\|_{m-2}^2<+ \infty. 
       \end{align}
       Moreover, we can also get from the initial condition  that \eqref{apriori} is valid  on $[0, T^{\eps, \delta}]$  (possibly by taking $T^{\eps, \delta}$ smaller).  
       Note that since we do not  propagate any additional normal  regularity, we do not need additional compatibility conditions.
              We shall not detail this step since it can be  done by classical energy estimates (much simpler than the ones we have proven since in this stage we are not interested
           in estimates independent of $\eps$.) 
           
            An  important remark is that if $\mathcal{N}_{m}(T_{0}) <+\infty$, then the solution can be continued on 
             $[0, T_{1}]$, $T_{1}>T_{0}$ with $\mathcal{N}_{m}(T_{1})<+\infty$. Indeed  if $\mathcal{N}_{m}(T_{0})<+\infty$,
             we can use the parabolic regularity for the Stokes problem on $[T_{0}/2, T_{0}]$ to get  that the solution actually
              enjoys much more  standard Sobolev regularity on $[T_{0}/2, T_{0}]$ (note that we assume that
               the surface $h$ is $H^6$) and in particular, we find  that $u(T_{0})\in H^{r-1}$, $3<r<7/2$. 
               This allows to use again the result  of Beale \cite{Beale81}  to continue  the solution and the previous argument about the propagation of additional
                regularity to get our claim.

    Next  we want to  use this remark to prove  that the solution can be continued on an interval of time independent of $\eps$ and $\delta$.
    Towards this, we first note that it is equivalent to control $\mathcal{N}_{m}(T)$ and $\mathcal{E}_{m}(T)$
    where   
    $$
      \mathcal{E}_{m}(T)=  \sup_{[0, T]}  \mathcal{Q}_{m}(t) +  \mathcal{D}_{m}(T) + \| \omega \|_{L^4([0, T], H^{m-1}_{co})}^2      $$
       with  $\mathcal{Q}_{m}$   defined by \eqref{defQm} and 
       $$ \mathcal{D}_{m}(T) = \eps \int_{0}^T \big( \eps \|\nabla V^m\|^2 + \eps \| \nabla S_{\n}\|_{m-2}^2 \big).$$
    
    Indeed, the fact that $\mathcal{N}_{m}(T) \leq \Lambda\big( {1 \over c_{0}}, \mathcal{E}_{m}(T))$   is a consequence of \eqref{equiv1}, \eqref{vnorm}, 
     \eqref{dzvm-1O},  Lemma \ref{lemdzS} and Corollary \ref{corLinfty} while the reverse inequality is just a consequence of product estimates.
     
         For two parameters $R$ and $c_{0}$ to be chosen $1/c_{0}<<R$, we can thus define  a time $T^{\eps, \delta}_{*}$ 
        \begin{multline*}
        T^{\eps, \delta}_{*}= \sup\big\{T\in [0, 1],  \quad \mathcal{E}_{m}(t) \leq R, \quad 
        |h(t)|_{2, \infty} \leq 1/c_{0}, \quad \partial_{z}\varphi(t) \geq c_{0}, \\
         \quad g- \partial_{z}^\varphi q^E(t) \geq {c_{0} \over 2}, \quad \forall t \in [0, T]. 
        \big\}.
        \end{multline*}

      At first, let us notice that thanks to Corollary \ref{corLinfty}, we  have that  for $T \leq T^{\eps, \delta}_{*}$, 
      $$ \Lambda_{\infty, 6}(T) \leq \Lambda\big( R)$$
      where $\Lambda_{\infty, m}$ is defined by \eqref{deflambdainfty}.
     Thanks to Corollary \ref{corLinfty0t}, we also have
      $$ \int_{0}^T \sqrt{\eps} \| \nabla^2 v\|_{1, \infty} \leq \Lambda(R).$$
      This allows to use Proposition \ref{conormv}, Proposition \ref{propdzvinfty} ,   Proposition \ref{dzzvLinfty} and Proposition  \ref{propomega}  to get that 
   $$
    \mathcal{E}_{m}(T)
       \leq \Lambda\big( {1 \over c_{0}}, \mathcal{I}_{m}(0)\big) + \Lambda(R)\Big( T^{1 \over 2} + \Lambda(R)\int_{0}^T | (\partial_{z}\partial_{t} q^E)^b \|_{L^\infty}
        + \Lambda(R)\int_{0}^T \| \omega \|_{m-1}^2\Big).
   $$
     Consequently, from the Cauchy-Schwarz  inequality, we find that
    \beq
    \label{Em1}  
    \mathcal{E}_{m}(T)
       \leq \Lambda\big( {1 \over c_{0}}, \mathcal{I}_{m}(0)\big) + \Lambda(R)\Big( T^{1 \over 2} + \Lambda(R)\int_{0}^T | (\partial_{z}\partial_{t} q^E)^b \|_{L^\infty}\Big). \eeq
       To estimate the last term in the right hand side, we can use Proposition \ref{proptaylor} to find
    $$ \int_{0}^T | (\partial_{z}\partial_{t} q^E)^b \|_{L^\infty} \leq \Lambda(R)\big( T + \int_{0}^T \big( \eps \|\partial_{zz}v \|_{L^\infty} + 
     \eps \| \partial_{zz} v\|_{3} \big)\big)$$
      and hence thanks to \eqref{dzzvk}, we find
      \beq
      \label{taylorfin}  \int_{0}^T | (\partial_{z}\partial_{t} q^E)^b \|_{L^\infty} \leq \Lambda(R)\big( T + \int_{0}^T \eps  \| \partial_{z} S_{\n}\|_{3}\big)
       \leq \Lambda(R)\sqrt{T}
       \eeq
       where the last estimate comes  again from the Cauchy-Schwarz inequality. Consequently, we obtain from \eqref{Em1} that
       \beq
       \label{Em2}
        \mathcal{E}_{m}(T)
       \leq \Lambda\big( {1 \over c_{0}}, \mathcal{I}_{m}(0)\big) + \Lambda(R) T^{1 \over 2}.   
            \eeq 
      Moreover,  thanks to  the equation \eqref{bordv2}, we get
        that
        \beq
        \label{choixc01} |h(t)|_{2, \infty} \leq |h(0)|_{2, \infty}+ \Lambda(R) T, \quad \forall t\in [0, T]\eeq
          and also 
        \beq
        \label{choixc02} \partial_{z} \varphi(t) \geq  1  -  \int_{0}^t \| \partial_{t} \nabla \eta \|_{L^\infty} \geq 1 - \Lambda(R) T, \quad
         \forall t \in [0,T]\eeq
        since we have chosen $A$ so that \eqref{Adeb} is verified. Finally, since
        $$ g- (\partial_{z}^\varphi q^E)^b(t)=  (g- (\partial_{z}^\varphi q^E)^b_{/t=0} -\Lambda(R) \int_{0}^t\big( 1 +  | (\partial_{t} \partial_{z}q^E)^b |_{L^\infty}
        \big), $$
        we get from \eqref{taylorfin} that
        \beq
        \label{taylorcon} g- (\partial_{z}^\varphi q^E)^b \geq  (g- (\partial_{z}^\varphi q^E)^b_{/t=0} -  \sqrt{t} \Lambda(R), \quad \forall
         t \in [0,T].\eeq
        In view of 
         \eqref{taylorcon}, \eqref{choixc01}, \eqref{choixc02} and \eqref{Em2}, we can take
          $R= 2 \Lambda( |h(0)|_{2, \infty}, \mathcal{I}_{m}(0))$   to get that  there exists $T_{*}$ which depends only on 
           $ \mathcal{I}_{m}(0)$ (and hence does not depend on $\eps$ and $\delta$)  so that for $T \leq \mbox{Min }(T_{*}, T^{\eps, \delta}_{*})$, 
           we have
           $$ \mathcal{E}_{m}(T) \leq R/2, \quad 
        |h(t)|_{2, \infty} \leq {1\over  2 c_{0}}, \quad \partial_{z}\varphi(t) \geq 2c_{0}, \quad g- \partial_{z}^\varphi q^E(t) \geq {3 \over 4} c_{0}, \quad \forall t \in [0,T].$$
          This yields  $T^{\eps, \delta }_{*} \geq T_{*}$. Indeed otherwise, our criterion about the   continuation of  the solution
           would contradict the definition of $T^{\eps, \delta}_{*}$.
           
           We have thus proven that  for the  smoothed  initial data $v_{0}^\delta$,  there exists an interval of time $[0, T_{*}] $  independent of $\eps$ and
            $\delta$
            for which the solution exists and such that $\mathcal{N}_{m}(T^*) <+\infty$.  To get the existence of solution
             without the additional regularity for the initial data, it suffices to pass to the limit.  The fact that                
      $ \mathcal{N}_{m}(T^*)$ is uniformly bounded in $\delta$  allows to pass to the limit easily by using strong compactness arguments.
      We shall not give more details since the arguments are very close to the ones that allow to get the inviscid
       limit that we shall detail below.

       \section{Uniqueness}
      \label{sectionunique} 
       \subsection{Uniqueness for Navier-Stokes}
       In this section, we shall prove the uniqueness part of Theorem \ref{main}.
        We consider two solutions $(v^i, \varphi^i, q^i)$  with the same initial data of \eqref{NSv}, \eqref{bordv1}, \eqref{bordv2}
         which satisfy  on $[0, T^\eps]$, the estimate
         \beq
         \label{apriorifinal}
         \mathcal{N}_{m}^i(T^\eps) \leq R, \quad i=1, \, 2
         \eeq
       where $\mathcal{N}_{m}$ is defined  above in \eqref{defNm} and the superscript $i$ refers to  one of the two solutions.
       We also assume that  for each solution  \eqref{apriori} is verified. We set
        $v= v^1- v^2$, $h=h^1- h^2$, $q= q^1- q^2$.
       
       We shall first provide  a  simple  uniqueness proof for the Navier-Stokes equation.
               We will explain how  the proof has to be modified in order to get  uniqueness
         for the  Euler equation in  a second step.

         At first, by using \eqref{transportW} and \eqref{deltaphi}, we get that $v$ solves the equation
       \beq
             \label{eqdiff}
          \big( \partial_{t}  +  v_{y}^1\cdot \nabla_{y}+ V_{z}^1 \partial_z   \big)v + \nabla^{\varphi} q - \eps \Delta^{\varphi} v = \mathcal{F}
       \eeq
          with $\mathcal{F}$ given by  
      \begin{align}
      \label{sourcediff} 
       	 \mathcal{F}  & = ( v_{y}^1 - v_{y}^2) \cdot \nabla_{y}v^{2} + ( V_{z}^1 - V_{z}^2 ) \partial_{z}v^2  -
		 \eps \big( {1 \over \partial_{z} \varphi^1}- {1 \over \partial_{z} \varphi^2}\big)  \big( (P^1)^* \nabla q^2 \big) \\
      \nonumber   & \quad      +\eps  {1 \over \partial_{z} \varphi_{2}}  \big( (P^1- P^2)^* \nabla q^2\big) 
          +\eps \big( {1 \over \partial_{z} \varphi^1}- {1 \over \partial_{z} \varphi^2}\big)  \nabla \cdot \big( E^1 \nabla v^2 \big)
          +\eps  {1 \over \partial_{z} \varphi_{2}} \nabla \cdot \big( (E^1- E^2) \nabla v^2\big).
       \end{align}
       In  a similar way,  thanks to \eqref{graddiv}, we can write the divergence free condition under the form
       \beq
       \label{divdiff}
       \nabla^\varphi \cdot v = -  \eps \big( {1 \over \partial_{z} \varphi^1}- {1 \over \partial_{z} \varphi^2}\big)  \nabla \cdot \big( P^1 v^2 \big)
          -\eps  {1 \over \partial_{z} \varphi^{2}} \nabla \cdot \big( (P^1- P^2) v^2\big).
      \eeq 
       On the boundary, we obtain from \eqref{bordv1} \eqref{bordv2} that
       \beq
       \label{diffbord1}
       \partial_{t} h  +  (v^b)^1_{y} \cdot \nabla h - ((v^1)^b_{3}- (v^2)^b_{3})= 
        - ((v^1)^b - (v^2)^b)_{y}\cdot \nabla h^2
        \eeq 
        and
        \begin{align}
        \label{diffbord2}
       & q  \n^1 - 2 \eps S^\varphi v \n^1=  
       g h \n^1 + 2 \eps (S^{\varphi^1}- S^{\varphi^2})v^1 \n^1  
        + 2 \eps S^{\varphi^{2}} v^2 (\n^1 -  \n^2).
        \end{align} 
    At first,  as in the proof of Proposition \ref{heps} (see \eqref{h1}), we can show  from \eqref{diffbord1} that    
     \beq 
     \label{diffest1}   \eps\, |h(t )|_{{1 \over 2}}^2
      \leq  
         \eps \int_{0}^t \|  \nabla   v \|_{L^2(\mathcal{S})}^2    + \Lambda(R) \int_{0}^t
        \big( \| v \|_{L^2(\mathcal{S})}^2 + \eps\,|h|_{{1 \over 2}}^2  \big)\, d\tau.
        \eeq

       Next, we can use again a  standard energy estimate for \eqref{eqdiff}, by using Proposition \ref{propeta}
        and Proposition \ref{proppE}, Proposition \ref{propPNS},  from a lenghty but easy  computation, we obtain  for
         $v= v^1-v^2$, $h=h^1-h^2$ that
        $$ \|v(t)\|_{L^2(\mathcal{S})}^2 + 
         \|h(t)\|_{L^2(\mathcal{S})}^2 + \eps \int_{0}^t \|\nabla v \|_{L^2(\mathcal{S})}^2
          \leq \Lambda(R)\int_{0}^t \big(  |h|_{H^{1\over 2}(\mathbb{R}^2)}^2 +   \|v\|_{L^2(\mathcal{S})}^2+ \| \nabla (q^1- q^2) \|_{L^2(\mathcal{S})}
          \, \|v\|_{L^2(\mathcal{S})}\big).$$
       We shall not detail this estimate since most of the arguments have been already used, the only term for
        which one has to be careful, is the boundary term that involves one of the two last terms of the right hand side of
          \eqref{diffbord2}. 
           For example, we have to estimate the boundary term
           $$ \int_{z=0} 2 \eps (S^{\varphi^1}- S^{\varphi^2}) v^1 n^1 \cdot v \, dy$$
            and by  using the definition of $S^{\varphi^i} v^1$, $i=1,\, 2$,  we obtain integrals like
          $$  \eps \int_{/z=0} \partial_{z}v_{j}^1 \, \partial_{i} h \, v_{j}\, dy$$
          where $1\leq i \leq 2$.
           We can estimate it by
          $$ \Big| \eps \int_{/z=0} \partial_{z}v_{j}^1 \, \partial_{i} h \, v_{j}\, dy\Big|
           \leq  \eps\,  |\partial_{z}v_{j}^1 v_{j}|_{H^{1\over 2}(\mathbb{R}^2)} \, | \partial_{i}
           h|_{H^{-{1 \over 2} }(\mathbb{R}^2)}
            \leq \Lambda(R)\,  \eps \,\|v\|_{H^1(\mathcal{S})}\, |h|_{H^{1 \over 2 }(\mathbb{R}^2)} $$
            where the last estimate follows from \eqref{cont2D} and the trace Theorem. We can then use 
             the Young inequality.

        From the equation for the pressure, we can also thanks to the estimates of section \ref{sectionpressure} obtain  that
        $$ \| \nabla (q^1- q^2) \|_{L^2(\mathcal{S})}  \leq  \Lambda(R)\big( |h|_{H^{1 \over 2}} + \|v\|_{H^1(\mathcal(S)}).$$
        Note that there is no $\eps$ in front of $ |h|_{H^{1 \over 2}}$ in this estimate because of the Euler part of the pressure.
        This yields
        $$ \|v(t)\|_{L^2(\mathcal{S})}^2 + 
         |h(t)|_{L^2(\mathbb{R})^2}^2 + \eps |h(t)|_{1\over 2}^2 + \eps \int_{0}^t \|\nabla v \|_{L^2(\mathcal{S})}^2
          \leq \Lambda(R)\int_{0}^t \big(  |h|_{H^{1\over 2}(\mathbb{R}^2)}^2 +   \|v\|_{L^2(\mathcal{S})}^2\big).$$
Consequently, we get the uniqueness for Navier-Stokes by combining the last estimate and \eqref{diffbord1}.
            Note that the above estimate is not uniform in $\eps$ and thus does not allow to recover the uniqueness for Euler.
          
          \subsection{Uniqueness for Euler}
          \begin{prop}
          \label{propuniqueeuler}
          Consider  $(v^i, h^ i)$, $i=1, \, 2$ two solutions of \eqref{eulerint}, \eqref{eulerb}  defined on $[0, T]$ with the same initial data and
           the regularity stated in Theorem \ref{theoinviscid}. Then $v^1= v^2$ and  $h^1= h^2$.
          \end{prop}
          
          \begin{proof}
           We assume that
           \beq
           \label{borneunique1}
          \sup_{i, \, [0,T]}\big( \|v^i\|_{m} +\| \partial_{z} v^i\|_{m-2} + |h^i|_{m} +\| \partial_{z}v^i \|_{1, \infty}\big)\leq R.
          \eeq
          The proof relies almost only  on arguments that have been used previously, we shall consequently only give the main steps. 
          Let us also set $v= v^1- v^2$, $h= h^1 - h^2$. By using at first the same crude estimate as we have used above, we first obtain that
        \beq
        \label{uniqE1}
         \|v(t)\|_{L^2(\mathcal{S})}^2 + 
         \|h(t)\|_{L^2(\mathcal{S})}^2 
          \leq \Lambda(R)\int_{0}^t \big(  |h|_{1/2}^2 +   \|v\|_{1}^2 + \| \partial_{z}v \|_{L^2}^2\big).
          \eeq
          Note that we have used the estimates of Proposition \ref{proppE} to estimate the pressure.
          
         Next, se shall estimate $\|v\|_{1}$ and $|h|_{1}$. Let us set
         $$ \mathcal{E}(v,q, \varphi)= \big(\D_{t}+ v\cdot \nabla^\varphi \big) v + \nabla^\varphi q.$$
          We first apply $Z_{j}$ for $j=1, \, 2,\, 3$ to obtain:
          $$ D\mathcal{E}(v^i, q^i, \varphi^i) \cdot (Z_{j}v^i, Z_{j}q^i, Z_{j}\varphi^i)=0$$
          for $i=1, \, 2$.
          We thus obtain that
       $$ D\mathcal{E}(v^2, q^2, \varphi^2) \cdot (Z_{j}v, Z_{j}q, Z_{j}\varphi) + \big(D\mathcal{E}(v^1, q^1, \varphi^1)- D\mathcal{E}(v^2, q^2, \varphi^2) \big)
       \cdot( Z_j v^1, Z_{j}q^1, Z_{j}\varphi^1)=0$$
       where we also set $\varphi = \varphi^1 - \varphi^2$, $q= q^1 - q^2$. Consequently, by using Lemma \ref{lemal}, we can introduce the good unknowns
       $$ V_{j}= Z_{j}v - \partial_{z}^{\varphi^2}v^{2} Z_{j}\varphi, \quad Q_{j}= Z_{j} q -  \partial_{z}^{\varphi^2} q^2 Z_{j}\varphi$$ to  obtain that
       $$\big( \D_{t}+ v^2 \cdot \nabla^\varphi) V_{j}+ \nabla^{\varphi^{2}} Q_{j}= \mathcal{R}$$
       where 
       $$ \mathcal{R}=  -(V_{j}\cdot \nabla^{\varphi_{2}}v^2 + Z_{j}\varphi \big( \partial_{z}^{\varphi^2} v^2 \cdot \nabla^{\varphi^2}) v^2
       -  \big(D\mathcal{E}(v^1, q^1, \varphi^1)- D\mathcal{E}(v^2, q^2, \varphi^2) \big)
       \cdot( Z_j v^1, Z_{j}q^1, Z_{j}\varphi^1).$$
       Consequently, by using that $q^2$ verifies the Taylor sign condition on $[0,T]$, we can proceed as in the proof of Proposition 
       \ref{conormv} to get the estimate:
       \beq
       \label{uniqE2} \|V_{j}(t)\|_{L^2(\mathcal{S})}^2 +  |Z_{j}h|_{L^2(\mathbb{R}^2)}^2 
        \leq  \Lambda(R)\int_{0}^t \big(  |h|_{1}^2 +   \|v\|_{1}^2 + \| \partial_{z}v \|_{L^2}^2\big).
        \eeq
       In view of the estimates \eqref{uniqE1}, \eqref{uniqE2}, we still need to estimate $\|\partial_{z}v\|$ in order to conclude from the Gronwall Lemma.
       Let us set  $\omega^i= \nabla^{\varphi^i}\times v^i$, $i=1, \, 2$ and $\omega= \omega^1- \omega^2.$
       By using an estimate like \eqref{dzvm-1O}, we first obtain that
       \beq
       \label{uniqE3}  \| \partial_{z}v \|_{L^2} \leq \Lambda(R) \big( \|\omega \|_{L^2(\mathcal{S})}+ |h|_{1} + \|v\|_{1} \big)
       \eeq
        and hence we see that it only remains to estimate $  \|\omega \|_{L^2(\mathcal{S})}$.
         Since $\omega^i$ solves the equation
         $$\big( \partial_{t}^{\varphi^i}+ v^i \cdot \nabla^{\varphi^i}\big)\omega^i- \omega^i \cdot \nabla^{\varphi^i} v^i= 0,$$
         a standard estimate on the difference yields
         \beq
         \label{uniqE4}
         \|\omega(t)\|_{L^2(\mathcal{S})}^2 \leq \Lambda(R)\int_{0}^t \big(  |h|_{1}^2 +   \|v\|_{1}^2 +
            \| \partial_{z}v \|_{L^2}^2 +   \|\omega(t)\|_{L^2(\mathcal{S})}^2\big).
            \eeq
           It suffices to combine \eqref{uniqE1}, \eqref{uniqE2}, \eqref{uniqE3}, \eqref{uniqE4} to end the proof of 
           Proposition \ref{propuniqueeuler}.

           \end{proof}
       \section{Proof of Theorem \ref{theoinviscid}: inviscid limit}
       \label{sectioninviscid}
       From the uniform estimates of Theorem \ref{main},  we have that for $\eps \in (0, 1]$
        \begin{align}
      \label{uniformfin}
        \mathcal{N}_{m}(T) &  = \sup_{[0, T]}\big( \|v^\eps(t) \|_{m}^2 + \|\partial_{z}v^\eps\|_{m-2}^2 + |h(t)^\eps|_{m}^2 + \eps |h(t)^\eps|_{m+{1 \over 2}}^2
        + \eps \|\partial_{zz}v(t)^\eps\|_{L^\infty}^2 +  \|v(t)^\eps\|_{E^{2, \infty}}^2 \big) \\
        \nonumber & \quad  \quad \quad   + \|\partial_{z}v^\eps\|_{L^4([0,T], H^{m-1}_{co})}^2 
        + \eps \int_{0}^T \|\nabla v^\eps\|_{m}^2 + \eps \int_{0}^t  \|\nabla \partial_{z}v^\eps\|_{m-2}^2\leq R. 
       \end{align}
       In particular, we have  that 
        $h^\eps$ is bounded in $L^\infty([0,T], H^m)$, 
        that $v^\eps$ is bounded in $L^\infty ([0,T], H^m_{co})$ and $\nabla v^\eps$ is bounded in $L^\infty ([0,T], H^{m-2}_{co})$.
         Thus for every $t$, $v^\eps(t)$ is  compact in $H^{m-1}_{co, loc}$. 
        Next, by using the equation, we also get that $\partial_{t} v^\eps$ is bounded in $L^2([0, T], L^2)$
         and that $\partial_{t}h^\eps$ is bounded in $L^2([0, T], L^2)$. Moreover, from Proposition \ref{propPNS} and Proposition
          \ref{proppE}, we also get that $\nabla q^\eps$ is bounded in $L^2([0, T], L^2).$
          By classical arguments, we deduce that there exists a sequence $\eps_{n}$ which tends to zero and
           $(v, h, q)$ such that $v^{\eps_{n}}$ tends to $v$ strongly in  $\mathcal{C}([0,  T], H^{m-1}_{co, loc})$, 
           $h^{\eps_{n}}$ tends to $h$ strongly in  $\mathcal{C}([0,  T], H^{m-1}_{loc})$, 
             $\nabla q^\eps$  tends to $\nabla q$ weakly in $L^2([0, T] \times \mathcal{S}).$
            and
           \beq
           \label{uniflim}
          \sup_{[0, T]}\big( \|v \|_{m}^2 + \|\partial_{z}v\|_{m-2}^2 + |h(t)|_{m}^2 +  \|v^\eps(t)\|_{E^{2, \infty}}^2 \big)  + \|\partial_{z}v^\eps\|_{L^4([0,T], H^{m-1}_{co})}^2 \leq R.
          \eeq
            These convergences  allow to pass to the limit in the equations by classical arguments and hence, we find that
            $(v,h, \nabla q)$ solves weakly the Euler equation \eqref{eulerint} in $\mathcal{S}$. Since, we can also assume
             that the trace $v^{\eps_{n}}_{/z=0}$ converges weakly in $L^2([0,T] \times \mathbb{R}^2))$, we also get that the boundary condition 
             $\partial_{t}h= v \cdot N$ is verified in the weak sense.  To pass to the limit, in the boundary condition \eqref{bordv2}, 
             we note that  since the  Lipschitz norm of $v^\eps$ is uniformly bounded, it only remains in the limit
              $ q= gh$. 
             
             For the pressure, we note that  because of  Proposition  \ref{propPNS}, 
            we have  that $ \|\nabla (q^\eps)^{NS} \|_{L^2(\mathcal{S})} \leq \eps \Lambda(R)$  and hence tends  to zero strongly. 
            
            We have thus proven that $(v, h)$ is a  solution of the free surface Euler equation that satisfies the estimate \eqref{uniflim}.
             From the uniqueness for the Euler equation in this class, we obtain that the whole family $(u^\eps, h^\eps)$
              converges to $(v,h)$ as above. Note that this proves only local $L^2$ convergence of $v^\eps (t)$ and $h^\eps (t)$.
               To get strong convergence, we can use the energy identities. We shall first use the one for \eqref{NSv}, 
                thanks to proposition \ref{basicL2} and the estimates \eqref{uniformfin}, we get that for $t \in [0,T]$, we have 
             $$  \|v^\eps J^\eps (t) \|_{L^2(\mathcal{S})}^2 + g |h^\eps (t)|_{L^2(\mathbb{R}^2)}^2  - \big( \|v^\eps_{0} J^\eps_{0}\|_{L^2(\mathcal{S})}^2
              + |h^\eps_{0}|_{L^2(\mathbb{R}^2)}^2\big) \leq \eps \Lambda(R)$$
             where we have set $J^\eps (t)=( \partial_{z}\varphi^\eps )^{1 \over 2}.$ 
              This yields
              $$ \lim_{\eps \rightarrow 0}  \big(  \|v^\eps J^\eps (t) \|_{L^2(\mathcal{S})}^2 + g |h^\eps (t)|_{L^2(\mathbb{R}^2)}^2 \big)
               =   \lim_{\eps \rightarrow 0} \big( \|v^\eps_{0} J^\eps_{0}\|_{L^2(\mathcal{S})}^2
              +  g|h^\eps_{0}|_{L^2(\mathbb{R}^2)}^2\big)=   \|v_{0} J_{0}\|_{L^2(\mathcal{S})}^2
              +  g |h_{0}|_{L^2(\mathbb{R}^2)}^2.$$
              Note that the last equality  is a consequence of  \eqref{hypinviscid}  and the estimates \eqref{uniformfin}
               since
               $$\| \partial_{z}\varphi^\eps_{/t=0} - \partial_{z} \varphi_{/t=0} \|_{L^2(\mathcal{S})} \leq C  |h^\eps_{0} - h_{0}|_{H^{1 \over 2}(\mathbb{R}^2)}
                \leq C \Lambda(R)  |h^\eps_{0} - h_{0}|_{L^2(\mathbb{R}^2)}^{1 \over 2}
 .$$
               Next, we can use the conservation of energy for the solution $(v,h)$ of the  free surface Euler equation (that we formally get by taking
                $\eps=0$ in Proposition \ref{basicL2}) to get
                $$ \lim_{\eps \rightarrow 0}  \big(  \|v^\eps J^\eps (t) \|_{L^2(\mathcal{S})}^2 + g |h^\eps (t)|_{L^2(\mathbb{R}^2)}^2 \big)
                =  \|vJ (t)\|_{L^2(\mathcal{S})}^2
              + g|h(t) |_{L^2(\mathbb{R}^2)}^2.$$
              This yields that  $(v^\eps J^\eps (t), h^\eps(t))$   (for which we already had weak congergence) converge strongly in $L^2$
               to $(vJ, h)$. Since the strong convergence of $h$ gives the one of $J^\eps$, we finally obtain
                 by combining with the uniform bounds \eqref{uniformfin} that
                  $(v^\eps(t), h^\eps(t))$ converge strongly  towards $(v, h)$ in  $L^2(\mathcal{S}) \times L^2(\mathbb{R}^2)$.
                   The $L^\infty$ convergence can be finally obtained thanks to the $L^2$ convergence, 
                    the uniform bounds  \eqref{uniformfin} and the inequality \eqref{emb} (with $s_{2}=0$). This ends the proof
                     of Theorem \ref{theoinviscid}.

           \section{Proof of the technical Lemmas}
  \label{sectiontech}
  
  \subsection{Proof of Lemma \ref{lemFP0}}
  \label{sectionFP}

          The estimate of $ \| \rho \|_{L^\infty}$ and  $\|\partial_{i} \rho\|_{L^\infty}= \|Z_{i} \rho\|_{L^\infty} $, $i=1, \, 2$   can be easily obtained 
           from the maximum principle as in \eqref{Sninfty1}.
          Indeed, we get that $\partial_{i} \rho$ solves the equation
          $$  \partial_{t}  \partial_{i} \rho + w  \cdot \nabla \partial_{ i }  \rho = \eps \partial_{zz}
           \partial_{ i } \rho +   \partial_{i} \mathcal{H} 
           - \partial_{i} w\cdot  \nabla \rho$$
           still with an homogeneous Dirichlet boundary condition. Consequently, by using again the
            maximum principle, we find
          \beq
          \label{hetab} \|\nabla_y \rho \|_{L^\infty} \leq  \|\rho_{0}\|_{1, \infty}
           + \int_{0}^t \Big(  \|  \mathcal{H} \|_{1, \infty} 
            +  \| \partial_{i} w\cdot \nabla \rho \|_{L^\infty}\Big).\eeq
   
         To estimate the last term in the above expression, we use  that $u_{3}$ vanishes on the boundary  to get  
            \beq
            \label{hetab1}   \| \partial_{i} w\cdot \nabla \rho \|_{L^\infty}
             \lesssim \| w \|_{{1, \infty} } \| \rho\|_{1, \infty}
              +  \| \partial_{z} \partial_{i}w_{3}\|_{L^\infty}  \| Z_{3} \rho \|_{L^\infty}
               \lesssim \|w \|_{ E^{2, \infty}} \| \rho \|_{1, \infty}.\eeq

      It remains to estimate $\|Z_{3} \rho \|_{L^\infty} $  which is the most difficult term. We cannot
       use the same  method as previously due to the bad commutator between $Z_{3}$ and
        the Laplacian. We shall use  a more precise description of the solution of \eqref{eqetaFP0}.
           We shall first  rewrite the equation \eqref{eqetaFP0} as
        $$  \partial_{t} \rho + z \partial_{z}w_{3}(t,y,0) \partial_{z} 
         \rho + w_{y}(t,y,0) \cdot \nabla_{y} \rho - \eps \partial_{zz} \rho  =
          \mathcal{H} - R:=G$$
          where 
          $$ R=  \big(w_{y}(t,x)- w_{y}(t,y,0)\big) \cdot \nabla_{y} \rho +\big( w_{3}(t,x) -  z \partial_{z}w_{3}(t,y,0)
          \big) \partial_{z} \rho.$$
          The idea will be to use an exact representation of the Green's function of the operator
           in the left-hand side to perform the estimate.

          Let   $S(t, \tau)$ be the $C^0$  evolution operator generated by the left hand side of the above equation.
          This means that  $f(t,y,z)= S(t, \tau) f_{0}(y,z)$  solves the equation
   $$
     \partial_{t} f +   z \partial_{z}w_{3}(t,y,0)  \partial_{z} f  
         + w_{y}(t,y,0) \cdot \nabla_{y} f - \eps \partial_{zz} f= 0, \quad z>0, \, t>\tau, \quad f(t, y, 0) = 0. 
$$
with the initial condition $f(\tau, y, z) = f_{0}(y,z)$.
Then we have the following estimates:
         \begin{lem}
         \label{FP}
         There exists $C>0$ independent of $ \eps $   such that
       \begin{eqnarray}
       \label{FP1} & &  \big\|  z\partial_{z}  S(t, \tau) f_{0} \|_{L^\infty}
         \leq C \big( \|f_{0}\|_{L^\infty} + \|  z \partial_{z}  f_{0} \|_{L^\infty} \big), \quad \forall t \geq \tau \geq 0.
       \end{eqnarray} 
           \end{lem}
           
           We shall  postpone the proof of the Lemma until the end of the section.
           
       By using Duhamel formula, we deduce that 
      \beq
      \label{Duh}
      \rho(t) = S(t, \tau) \rho_{0} + \int_{0}^t S(t, \tau) G(\tau) \, d\tau.
      \eeq
      Consequently, by using \eqref{FP1} in  Lemma \ref{FP}, we obtain
      $$ \|  Z_{3} \rho \|_{L^\infty} \lesssim
        \Big( \|\rho_{0}\|_{L^\infty}
        + \|   z\partial_{z}  \rho_{0} \|_{L^\infty}
         + \int_{0}^t  \big( \|G \|_{L^\infty}
        + \|  z\partial_{z}  G \|_{L^\infty}\big) \Big).$$
        Since $\rho$ and $G$ are compactly supported, we obtain
      \beq
      \label{etab1} 
       \|  Z_{3} \rho \|_{L^\infty} \lesssim
        \Big( \|\rho_{0}\|_{1, \infty}
                + \int_{0}^t   \|G \|_{1, \infty} \Big).
                \eeq
        It remains to estimate the right hand side.
    First, let      us estimate the term involving $R$.
           Since $u_{3}(t,y,0)= 0$, we have
          $$ \| R \|_{L^\infty} \lesssim \|w_y\|_{L^\infty} \|\nabla_y \rho \|_{L^\infty} + \| \partial_{z} w_{3}\|_{L^\infty}
            \| Z_{3} \rho \|_{L^\infty} \lesssim  \|w\|_{E^{1, \infty}} \, \| \rho \|_{1, \infty}.$$
            Next, in a similar way, we  get that 
         $$ \| Z R\|_{L^\infty} \lesssim  \| w \|_{2, \infty}  \| \rho \|_{1, \infty}
          + \Big\|  \big(w_y(t,x)- w_y(t,y,0)\big) \cdot Z \nabla_y \rho \Big\|_{L^\infty}
           + \Big\|  \big( w_{3}(t,x) -  z \partial_{z}w_{3}(t,y,0)
          \big) Z \partial_{z} \rho\Big\|_{L^\infty}$$
          By using the Taylor formula and
            the fact that $\rho$ is compactly supported in  $z$,  this yields 
            $$ \| ZR \|_{L^\infty} \lesssim 
            \| w \|_{2, \infty}  \| \rho \|_{1, \infty} +
             \| \partial_{z}w_y\|_{L^\infty} \| \varphi(z)   Z \nabla_y \rho \|_{L^\infty}+
             \|\partial_{zz} w_{3}\|_{L^\infty} \| \varphi^2(z) Z \partial_{z} \rho \|_{L^\infty} $$
             with $\varphi(z)= z/(1-z)$
             and hence, we obtain
            $$ \| R \|_{1, \infty} \lesssim \big(  \|w\|_{E^{2, \infty}} + \| \partial_{zz} w_{3}  \|_{1, \infty}\big) \big( \|  \rho\|_{1, \infty}
             + \| \varphi(z) \rho \|_{2, \infty}\big).$$
             The additional factor $\varphi$ in the last term is crucial to close our estimate.
              Indeed,  by  the Sobolev embedding \eqref{sob}, we  have  that  for $|\alpha |=2$
               $$ \| \varphi Z^\alpha \eta \|_{L^\infty} 
                \lesssim  \| Z^\alpha \eta \|_{s_{2}} + \| \partial_{z} \big ( \varphi Z^\alpha \eta\big)  \|_{s_{1}} $$
                with $s_{1}+ s_{2}>2$, thus, with $s_{1}= 1,$  $s_{2}= 2$,
                 we obtain from  definition of $Z_{3}$ that
            \beq
            \label{trick}   \| \varphi Z^\alpha \eta \|_{L^\infty}  \lesssim  \|\eta \|_{4}, 
             \quad | \alpha | = 2.\eeq
             Consequently, we  get
             that
                         \beq
            \label{Rest}
            \| R(t) \|_{1, \infty} \lesssim
            \big(  \|w\|_{E^{2, \infty}} + \| \partial_{zz} w_{3}  \|_{L^\infty}\big) \big( \|  \rho\|_{1, \infty}
             + \|  \rho \|_{4}\big).
            \eeq 
          Finally,  the proof of      
      Proposition \ref{etainfty} follows from the last estimate and \eqref{etab1}.

     It remains to prove Lemma \ref{FP}.

    \subsubsection*{Proof of Lemma \ref{FP}}
    Let us set $f(t,y,z)= S(t, \tau) f_{0}(y,z)$, then $f$ solves the equation
   $$
     \partial_{t} f +   z \partial_{z}w_{3}(t,y,0)  \partial_{z} f  
         + w_{y}(t,y,0) \cdot \nabla_y f - \eps \partial_{zz} f= 0, \quad z>0,  \quad f(t, y, 0) = 0. 
$$ 
We can first transform the problem into a problem in the whole space.
 Let us  define $\tilde{f}$ by 
 \beq
 \label{tildef} \tilde{f}(t,y,z)= f(t,y,z), \, z>0, \quad \tilde{f}(t,y,z) = - f(t,y,-z), \, z<0\eeq
 then $\tilde{f}$ solves
 \beq
 \label{FP2}
   \partial_{t} \tilde{f} +   z \partial_{z}w_{3}(t,y,0)  \partial_{z} \tilde{f}  
         + w_{y}(t,y,0) \cdot \nabla_y \tilde{f} - \eps \partial_{zz} \tilde{f}= 0,  \quad  z \in \mathbb{R}  
 \eeq 
with the initial condition $\tilde{f}(\tau, y, z)= \tilde{f}_{0}(y,z)$.  

%
%
We shall get the estimate by using an exact
   representation of the solution.

To solve \eqref{FP2}, we can first define
\beq
\label{tildeg} g(t,y,z)= f(t, \Phi(t, \tau, y), z)\eeq

where $\Phi$ is the solution of 
$$
\partial_{t} \Phi = w_{y}(t,\Phi, 0), \quad \Phi(\tau, \tau, y)= y.
$$
Then, $g$ solves the equation
$$ \partial_{t}g +  z \gamma(t,y) \partial_{z}g - \eps \partial_{zz} g = 0, \quad z \in \mathbb{R}, \quad
 g(\tau, y, z)= \tilde{f}_{0}(y,z)$$
 where
 \beq
 \label{Gamma}
 \gamma(t,y)= \partial_{z} w_{3}(t, \Phi(t, \tau, y), 0)
 \eeq
 which is a one-dimensional  Fokker-Planck type equation (note that now $y$ is only a parameter
  in the problem). By a simple computation in Fourier space, we find the explicit representation
 \begin{eqnarray} 
\nonumber g(t,x) & =  &   \int_{\mathbb{R}}
  {1 \over \sqrt{ 4 \pi \eps  \int_{\tau}^t  e^{2 \eps ( \Gamma(t) - \Gamma(s) ) }\, ds}} 
   \exp \Big(  - {  (z- z')^2 \over  4 \eps  \int_{\tau}^t  e^{2 \eps ( \Gamma(t) - \Gamma(s) ) }\, ds} \Big)
     \tilde{f}_{0}(y,  e^{- \Gamma(t) }z')\, dz' \\
   \label{duhamel2}  & = &  \int_{\mathbb{R}} k(t, \tau, y, z-z') \tilde{f}_{0}  (y,  e^{- \Gamma(t) }z')\, dz'
     \end{eqnarray} 
     where $\Gamma(t)= \int_{\tau}^t \gamma(s,y)\, ds$ (note that $\Gamma$ depends
      on $y$ and $\tau$, we do not write down explicitely this dependence for notational convenience).
      
      Note that $k$ is non-negative and that $\int_{\mathbb{R}} k(t, \tau, y, z) \, dz= 1$, thus, 
      we immediately recover that
      $$ \|g \|_{L^\infty} \leq \|\tilde{f}_{0}\|_{L^\infty}.$$
      Next, we observe that   we  can write 
   \beq
   \label{trickconorm} z \partial_{z}k (t,\tau, z-z')=\big( z - z' \big) \partial_{z} k  -  z'  \partial_{z'}k
     (t, \tau, z-z')\eeq
     with
     $$  \int_{\mathbb{R} } \big| \big( z - z' \big) \partial_{z} k \big| dz' \lesssim 1$$
    and thus by using an integration by parts, we find
    $$ \|z \partial_{z} g\|_{L^\infty} \lesssim 
     \| \tilde{f} \|_{L^\infty} + \Big\| e^{- \Gamma(t) } \int_{\mathbb{R}} k(t,\tau, y, z') z' \partial_{z} \tilde{f}_{0}(y, e^{- \Gamma(t)}
      z')  dz'  \Big\|_{L^\infty}.$$
      By using \eqref{Gamma}, this yields
    $$   \| z  \partial_{z} g\|_{L^\infty} \lesssim 
     \| \tilde{f}_{0} \|_{L^\infty} + \| z \partial_{z} \tilde{f}_{0} \|_{L^\infty}.$$
     By using \eqref{tildef} and \eqref{tildeg},  we obtain
     $$ \| z  \partial_{z} f \|_{L^\infty} \lesssim     \| z  \partial_{z} \tilde{f} \|_{L^\infty}
      \lesssim  \| \tilde{f}_{0} \|_{L^\infty} +   \| z \partial_{z} \tilde{f}_{0} \|_{L^\infty}
       \lesssim   \| f_{0} \|_{L^\infty} +   \| z \partial_{z} f_{0} \|_{L^\infty}  .$$
       This ends the proof of Lemma \ref{FP}. 
      
  \subsection{Proof of Lemma \ref{lemFP1}}
  \label{sectionFP1}
   We use the same idea as in the proof of Lemma \ref{lemFP0}. We first estimate
    $ \sqrt{\eps } \| \partial_{z} Z_{i} \rho\|_{L^\infty}, $   $i=1, \, 2$. We get for $\partial_{i} \rho$ the equation
    $$ \partial_{t} \partial_{i} \rho + w \cdot \nabla \partial_{i} \rho = \eps \partial_{zz} \partial_{i} \rho + \partial_{i}\mathcal{H}
     - \partial_{i} w \cdot \nabla \rho$$
      that we rewrite as
     $$ \partial_{t} \partial_{i}\rho + z \partial_{z}w_{3}(t,y,0) \partial_{z} 
        \partial_{i} \rho + w_y(t,y,0) \cdot \nabla_y \partial_{i }\rho - \eps \partial_{zz}\partial_{i} \rho  =
         \partial_{i} \mathcal{H}  - \partial_{i}w \cdot \nabla \rho- R:=G$$
          where 
          $$ R^1=  \big(w_y(t,x)- w_y(t,y,0)\big) \cdot \nabla_y\partial_{i} \rho +\big( w_{3}(t,x) -  z \partial_{z}w_{3}(t,y,0)
          \big) \partial_{z}\partial_{i} \rho.$$
    By using the notations before Lemma \ref{FP}, we obtain
    $$ \partial_{i}\rho=    S(t, \tau) \partial_{i}\rho_{0} + \int_{0}^t S(t, \tau) G(\tau) \, d\tau$$
    and we shall use the following semigroup estimates for $S$:
    \begin{lem}
    \label{lemFP1bis}
    Under the assumption of Lemma \ref{lemFP1} on $w$, we have that for $0 \leq \tau \leq t \leq T$: 
    $$ \sqrt{\eps}\| \partial_{z} S (t,\tau) f_{0}\|_{L^\infty}  \leq {\Lambda(M) \over \sqrt{t}} \|f_{0}\|_{L^\infty}, \quad
     \sqrt{\eps} \| \partial_{z}\big( z \partial_{z} S(t,\tau) f_{0}\big) \|_{L^\infty} \leq
       {\Lambda(M) \over \sqrt{t} }\big( \|f_{0}\|_{L^\infty} + \|z\partial_{z} f_{0} \|_{L^\infty}\big)$$
       where $\Lambda(M)$ does not depend on $\eps$.
    \end{lem}
    Let us postpone the proof of this Lemma until the end of the section.
     By using Lemma \ref{FP1}, we thus get
     that
     $$ \sqrt{\eps} \| \partial_{z} \partial_{i} \rho(t) \|_{L^\infty}
      \leq  \Lambda(M)\Big(  {1 \over \sqrt{t} } \|\partial_{i} \rho (0) \|_{L^\infty} + \int_{0}^t {1 \over \sqrt{ t- \tau} }
       \big( \| \mathcal{H}\|_{1,\infty} + \|\partial_{i} w \cdot \nabla \rho \|_{L^\infty} + \|R^1\|_{L^\infty}\big)\Big).$$
      Next, we can use \eqref{hetab1} and \eqref{Rest} to get that
      $$ \| R^1\|_{L^\infty} \leq \Lambda(M) \big( \|\rho\|_{1, \infty} + \|\rho\|_{4}\big).$$
       This yields
       $$ \sqrt{\eps} \| \partial_{z} \partial_{i} \rho(t) \|_{L^\infty}
      \leq  \Lambda(M)\Big(  {1 \over \sqrt{t} } \|\partial_{i} \rho (0) \|_{L^\infty} + \int_{0}^t {1 \over \sqrt{ t- \tau} }
       \big( \| \mathcal{H}\|_{1,\infty} +  \|\rho\|_{1, \infty} + \|\rho\|_{4}\big)\Big).$$
     The estimate of $\| \partial_{z} \big( Z_{3} \rho\big)\|_{L^\infty}$, we directly  use the Duhamel formula \eqref{Duh}
      and we use the second estimate given in Lemma \ref{FP1}.
    This ends the proof of Lemma \ref{lemFP1bis}. It only remains to prove Lemma \ref{FP1}:
     \subsubsection*{Proof of Lemma \ref{FP}}
     We can use  again  that the solution of the equation is given by the representation \eqref{tildeg}, \eqref{duhamel2}
      and it suffices to notice that  the kernel $k$ has the property
      $$ | \partial_{z} k | \leq { \Lambda(M) \over \sqrt{\eps(t- \tau)}}\, | k| $$
       and  the result follows from standard convolution estimates. 
  \subsection{Proof of Proposition \ref{Korn}}
  \label{sectionKorn}
   Thanks to Lemma \ref{mingrad}, the estimate  \eqref{estKorn}  for $v$ is actually equivalent to the standard Korn inequality
   in  $\Omega_{t}$  for $ u= v  \circ \Phi^{-1}$. For the sake of completeness, we shall 
   sketch the argument.
   
    We first note that
  $$ \int_{\mathcal{S}} |S^\varphi v|^2 \, d\V  \geq c_{0} \|S^\varphi v\|_{L^2(\mathcal{S})}^2.$$
  Next, we shall reduce the problem to the classical Korn inequality in $\mathcal{S}$ for an auxiliary
    vector field. Let us set
   $$   w_{i }= v_{i}+ \partial_{i}\varphi  \, v_{3},\, i= 1, \, 2,  \quad w_{3}=\partial_{z}\varphi v_{3}.$$
   We note that for $ 1 \leq i , \, j \leq 2$, we have
   \begin{align*}
   2 (Sw)_{ij}=  & \D_{i}v_{j} + \D_{j} v_{i} + \partial_{j}\varphi\big( \D_{z}v_{i} + \D_{i} v_{3}\big)
    +   \partial_{i}\varphi\big( \D_{z}v_{j} + \D_{j} v_{3}\big)  \\
   &  +  2 \partial_{i} \varphi \partial_{j}   \varphi \, \D_{z} v_{3}  +2 \partial^2_{ij}\varphi \, v_{3}
    \end{align*}
    and hence that
    $$ (Sw)_{ij}= (S^\varphi v)_{ij}+ \partial_{j}\varphi (S^\varphi v)_{i,3}+ \partial_{j}\varphi(S^\varphi v)_{j,3}
    + \partial_{i}\varphi \partial_{j}\varphi S^\varphi(v)_{3, 3}+  \partial^2_{ij} \varphi \, v_{3}.$$
  In a similar way, we have that
  \begin{eqnarray*} (Sw)_{i,3} & = &  \partial_{z}\varphi \, (S^\varphi v)_{i,3} + \partial_{i}\varphi \, \partial_{z}\varphi
   (S^\varphi v)_{3, 3}+ \partial^2_{i,3} \varphi\,  v_{3}, \quad i= 1, \, 2, \\
    (Sw)_{3, 3}& = & ( \partial_{z}\varphi)^2 (S^\varphi v)_{3, 3}+ \partial^2_{zz} \varphi\, v_{3}.
    \end{eqnarray*}
   This yields 
    \beq
    \label{korn1}
     \| S w \|_{L^2(\mathcal{S})} \leq \Lambda_{0}\big( \|S^\varphi v \|_{L^2(\mathcal{S})} +  \|v\|_{L^2(\mathcal{S})} \big).
     \eeq
         Next,  the Korn inequality in $\mathcal{S}$ for $w$,  yields for some $C>0$: 
     $$ \|\nabla w \|_{L^2(\mathcal{S})} \leq  C\big(   \| S  w \|_{L^2(\mathcal{S})} + \|w\|_{L^2(\mathcal{S})}.$$  
     Consequently, we obtain that
     $$  \|\nabla w \|_{L^2(\mathcal{S})} \leq \Lambda_{0}\big( \|S^\varphi v \|_{L^2(\mathcal{S})} +  \|v\|_{L^2(\mathcal{S})}\big) + C \| w \|_{L^2(\mathcal{S})}.$$
   Moreover,   since $\partial_{z} \varphi \geq \eta$, we have  from the definition of $w$ that 
     $$ \|w \|_{L^2(\mathcal{S})} \leq \Lambda_{0} \| v\|_{L^2(\mathcal{S})}, \quad \|\nabla v \|_{L^2(\mathcal{S})}
      \leq \Lambda_{0}\big( \|\nabla w \|_{L^2(\mathcal{S})}  + \|v \|_{L^2(\mathcal{S})} \big),$$
       the result follows by combining these inequalities.
      
      Finally, let us recall the proof of \eqref{korn1}.
      We can define an extension of $w$ in $\mathbb{R}^3$ by $\tilde{w}= w, \, z<0$ and
      $$  \tilde{w}_{i}(y,z)= 2 w_{i}(y, -z) - w_{i}(y,-3z), \quad \tilde{w}_{3}(y,z)= -2w_{3} (y, -z) + 3u_{3}(y, -z), \quad z>0.$$
       Since, 
      $$ \|S \tilde w \|_{L^2(\mathbb{R}^3)} \leq  5 \| S w \|_{L^2(\mathcal{S})},$$
 In $\mathbb{R}^3$, we obviously have that 
 $$ \|S \tilde w \|_{L^2(\mathbb{R}^3)}^2= {1 \over 2 } \| \nabla \tilde{w}\|_{L^2(\mathbb{R}^3)}^2 + {1 \over 2}
  \| \nabla \cdot \tilde{w} \|_{L^2(\mathbb{R}^3)}^2  \geq  {1 \over 2 } \| \nabla \tilde{w}\|_{L^2(\mathbb{R}^3)}^2$$
  and  the conclusion follows  from the remark that
  $$ \|\nabla w \|_{L^2(\mathcal{S})} \leq  \| \nabla \tilde w \|_{L^2(\mathbb{R}^3)}.$$
  
  \subsection{Proof of Lemma \ref{hardybis}}
  \label{prooflemmahardy}
  Let us start with the first inequality.  For any test function $f$, we  get  by an integration by parts
  $$ \int_{-\infty}^0 {1 \over z(1-z)} f \partial_{z}f\, dz={1 \over 2} \int_{-\infty}^0 { 1 - 2z \over z^2 (1-z)^2} |f|^2\, dz \geq {1 \over 2}
   \int_{-\infty}^0 { 1 \over z^2 (1-z)^2} |f|^2\, dz$$
    and we get the result from the Cauchy-Schwarz inequality.
    
    For the second inequality, we just write
    $$ \int_{-\infty}^0\big( {1 -z \over z}\big)^ 2 f(z)^2 \, dz \leq 4  \int_{-1}^0   {1 \over z^2}  f(z)^2 \, + 4  \int_{-\infty}^{- 1}  |f(z)|^2\, dz$$
     and we use the standard Hardy inequality on $(-1, 0)$ to estimate the first term.
     
  \section{Results on paradifferential calculus}
  \label{sectionparadiff}
  In this section, we  shall state  the results  on paradifferential calculus with parabolic homogeneity that we need for our estimates 
   of section \ref{sectionnorm2}. We shall  first  define a calculus as in
 the Appendix B of  \cite{Metivier-Zumbrun}  without the semiclassical parameter
    $\eps$.
  In a second step, we shall deduce the result for  the  "partially" semiclassical
       version (that is to say with weight $\langle \zeta^\eps \rangle= \big( \gamma^2 + \tau^2 + \eps^2 | \xi |^4 \big)^{1 \over 4}$) that we need.
       
     We first define operators acting on functions defined on $\mathbb{R}^3_{(t,y)}$. The calculus for functions defined on $\mathbb{R}_{t}\times
\mathcal{S}$ will immediately follow since the $z$ variable  will only be  a parameter. 

For notational convenience, we use in this section the notation $X=(t,y)=(x_{0}, x_{1}, x_{2})$ for the "space" variable and $(\tau, \xi)$
 for the corresponding Fourier  variables, we also use the notation $\zeta=(\gamma, \tau, \xi)= (\gamma, \eta)\in \mathbb{R}^4_{+}= [1, + \infty[ \times \mathbb{R}^3$
  where $\gamma \geq 1$ will be  a  parameter. We define the weight:
 $$ \langle \zeta \rangle=\big(\gamma^2 + \tau^2 + |\xi|^4\big)^{1\over 4}$$
 with $|\cdot|$ the Euclidian norm of $\mathbb{R}^2$.
 Note that this weight corresponds to the quasihomogeneous weight with $p=2$, $p_{0}= 1$, $p_{1}= p_{2}= 2$ in the general framework
  of \cite{Metivier-Zumbrun}. We also point out that since we shall only consider  
 $\gamma \geq 1$, we do  not need to  make any difference
    between $\langle \zeta\rangle$ and $\Lambda(\zeta)= \big( 1+ \langle \zeta \rangle^{2p})^{1 \over 2p}$ in the notation of \cite{Metivier-Zumbrun}.
    By using this weight, we define a scale of modified Sobolev type spaces by
    $$ \mathcal{H}^{s, \gamma}(\mathbb{R}^3)= \{ u \in \mathcal{S}'(\mathbb{R}^3), \quad \|u\|_{\mathcal H^{s, \gamma}}^2 <+\infty\}$$
     where
     $$ \|u\|_{\mathcal{H}^{s, \gamma}}^2 = \int_{\mathbb{R}^3} \langle \zeta \rangle^{2s} |\hat u (\tau, \xi)|^2\, d\tau d\xi,$$
     $\hat u$ being the Fourier transform of $u$.
     
     Next, we define our class of symbols:
     \begin{defi}
     For $\mu \in \mathbb{R}$
     \begin{itemize}
     \item  the class $\Gamma^\mu_{0}$  denotes the space of locally bounded matrices  $a(X, \zeta)$ on $\mathbb{R}^3\times \mathbb{R}^4_{+}$
      which are $\mathcal{C}^\infty$ with respect to  $(\tau, \xi)= \eta$ and such that for all $\alpha \in \mathbb{N}^d$, there is a constant
       $C_{\alpha}$ such that 
       $$ \forall (X, \zeta), \quad | \partial_{\eta}^\alpha a(X, \zeta)|=  C_{\alpha} \langle \zeta \rangle^{\mu - \langle \alpha \rangle}$$
       where $\langle \alpha \rangle =  2 \alpha_{0}+ \alpha_{1}+ \alpha_{2}$.
       \item $\Gamma^\mu_{1}$ denotes the set of symbols $a \in \Gamma^\mu_{0}$ such that $\partial_{i}a \in \Gamma^{\mu}_{0}$,  for $ 
        i=0, \, 1, \, 2.$ 
     \end{itemize}  
     \end{defi}
     The spaces $\Gamma_{0}^\mu$ are equipped with the seminorms
     $$ \| a \|_{(\mu, N)}=  \sup_{\langle \alpha \rangle \leq N} \sup_{\mathbb{R}^3 \times \mathbb{R}^4_{+}} \langle \zeta \rangle^{\langle \alpha \rangle - \mu} |\partial_{\eta}^\alpha a(X, \zeta)|.$$
     
    Paradiffential operators are defined as  pseudodifferential operators associated to  a suitable regularization of   the above symbols.
    For a symbol $a$  in the above class, we define a smooth symbol $\sigma_{a}( X, \zeta)$ defined by 
    $$ \mathcal{F}_{X}\sigma_{a}(\hat X, \zeta)= \psi(\hat X, \zeta) \mathcal{F}_{X}a(\hat X, \zeta)$$
    where $\mathcal{F}_{X}$ stands for the partial Fourier transform with respect to the $X$ variable and $\psi$ is an admissible cut-off
     function. We say that  a smooth function $\psi$ is admissible if  there exists $0<\delta_{1}<\delta_{2}<1$ such that
     $$ \psi(\hat X, \zeta)= 1,\mbox{ for } \langle \hat X \rangle \leq \delta_{1} \langle \zeta\rangle, \quad
       \psi(\hat X, \zeta)= 0,\mbox{ for } \langle \hat X \rangle \geq \delta_{2} \langle \zeta\rangle$$
       where $\langle \hat X \rangle= \big( \gamma^2 + \hat{X}_{0}^2 + |(\hat{X}_{1}, \hat{X}_{2})|^4 \big)^{1\over 4}$ and  that
        for every $\alpha, \, \beta \in \mathbb{N}^d$, there exists $C_{\alpha, \beta}$ such that
        $$ \forall \hat X,\, \zeta, \quad  |\partial_{\eta}^\alpha \partial_{\hat X}^\beta \psi(\hat X, \zeta)| \leq C_{\alpha, \beta} \langle \zeta\rangle^{-
        \langle \alpha \rangle - \langle \beta \rangle}.$$
        
       For a given admissible cut-off function,  we then define an operator associated to the symbol $a$,  $T_{a}^\gamma$ by 
        $$ T_{a}^\gamma u= Op(\sigma_{a}) u$$
        where $Op(\sigma_{a})$ is the pseudodifferential operator  defined by
        \beq
        \label{defpseudo} Op(\sigma_{a}) u(X) = (2 \pi)^{-3} \int_{\mathbb{R}^3} e^{i X \cdot \eta}  \sigma_{a}(X, \zeta) \hat{u}(\eta) \, d\eta.
        \eeq
        
        In this definition, the operator $T_{a}^\gamma$  depends on the choice of the admissible cut-off function.  Nevertheless, it can be
         shown that the difference between two operators defined with the same symbol but with two different admissible cut-off function
          is a lower order operator. We refer to \cite{Metivier-Zumbrun} for more details.
     Note that  viewing  $a(X) \in L^\infty$ as a symbol in $\Gamma^0_{0}$,  the operator $T_{a}^\gamma$ can be related to Bony's
      paraproduct (\cite{Bony}).
      
      The main interest of this class of operators is that it enjoys a  nice symbolic calculus.
      \begin{theoreme}
      \label{symbolic}
      We have the following results:
      \begin{enumerate}
      \item {\bf Boundedness} For  $a \in \Gamma^\mu_{0}$, and every $s\in \mathbb{R}$,  there exists $C>0$ (which depends only on semi norms of $a$) such that for every $\gamma \geq 1$
       and $u \in \mathcal H^{s+\mu, \gamma}$, we have
       $$ \|T_{a}^\gamma u\|_{\mathcal H^{s, \gamma}} \leq C \| u \|_{\mathcal{H}^{s+\mu, \gamma}}.$$
       \item {\bf Product} For $a \in \Gamma^\mu_{1}$, $b \in \Gamma^{\mu'}_{1}$ and $s \in \mathbb{R}$, we have
       $$ \|T_{a}^\gamma T_{b}^\gamma u - T_{ab}^\gamma u \|_{\mathcal H^{s, \gamma}} \leq
        C\Big( \|a\|_{(\mu, N)}\, \|\nabla_{X}b\|_{(\mu', N)} + \|\nabla_{X} a \|_{(\mu, N)} \|b\|_{(\mu', N)}\Big) \|u\|_{\mathcal{H}^{s+\mu+\mu'-1, \gamma}}$$
        where $C$, $\mu$ and $N$ depend only on $s$, $\mu$, $\mu'$.
        \item {\bf Adjoint} For $a \in \Gamma^\mu_{1}$ (recall that we allow $a$ to be matrix valued), denote by $(T_{a}^\gamma)^*$ the adjoint of $a$ 
         and  $a^*(X, \zeta)$ the adjoint matrix of $a$, then we have
         $$ \|(T_{a}^\gamma)^* - T_{a^*}^\gamma\|_{\mathcal{H}^{s, \gamma}} \leq C \|\nabla_{X}a\|_{(\mu, N)} \, \|u\|_{\mathcal H^{s+\mu-1, \gamma}}$$
         where $C$ and $N$ only depend on $\mu$ and $s$.
         \item {\bf Garding inequality} For $a \in \Gamma_{1}^\mu$, assume that there exists $c>0$
           such that
          $$ \forall X, \, \zeta, \quad    \mbox{Re }a(X, \zeta) \geq c  \langle \zeta \rangle^\mu.$$
           Then  there exists $C>0$ (depending only on $\|a,  \nabla_{X} a  \|_{(\mu, N)}$)  such that for $\gamma \geq C$, we have
         $$  {c \over 2} \| u\|_{\mathcal H^{{\mu \over 2}, \gamma}}^2 \leq \mbox{Re } \big(T_{a}^\gamma  u, 
           u \big)_{L^2}$$
         \item {\bf Paraproduct} If $a(X) \in L^\infty$, $\nabla_{X}a \in L^\infty$, there exists $C>0$ such that for every $u$, we have the estimates
         \begin{eqnarray*}
          \| au - T^\gamma_{a} u \|_{\mathcal H^{1, \gamma}} \leq C \|\nabla_{X} a \|_{L^\infty} \|u\|_{\mathcal{H}^{0, \gamma}}, \\
          \gamma  \| au - T_{a}^\gamma u \|_{\mathcal{H}^{0, \gamma}} \leq C \| \nabla_{X} a \|_{L^\infty} \| u\|_{\mathcal{H}^{1, \gamma}}, \\
           \| a \partial_{j}u - T_{a}^\gamma\partial_{j}u \|_{\mathcal H^{0, \gamma}} \leq C \| \nabla_{X} a \|_{L^\infty} \|u\|_{\mathcal{H}^{  {2 \over p_{j} } - 1, \gamma}}
           \end{eqnarray*}
           where $\partial_{0}= \partial_{t}$, $p_{0}=1$  and $p_{j}= 2$, $j=1, \, 2$. 
      \end{enumerate}
      \end{theoreme}
  For the proof of these results, we refer to \cite{Metivier-Zumbrun} Appendix B. 
  The meaning of the assumption in the Garding inequality is that $\mbox{Re} \,a : = {1 \over 2}(a+ a^*)$ for a matrix is bounded from
   below on the support of $w$.
  
  When dealing with paraproducts,   it is usefull to recall that
   thanks to the definition of the operators, we have when $a = a(X)$ that 
   $$ \partial_{j}\big( T^\gamma_{a} u \big) = T^\gamma_{a} \partial_{j} u + T^\gamma _{\partial_{j} a } u.$$
   Another usefull remark about products is that if $a= a(X, \zeta)$ but $b=b(\zeta)$ does not depend on X, then we have
   \beq
   \label{remarkproduit}
    T_{ab}^\gamma = T_{a}^\gamma T_{b}^\gamma.
   \eeq
   
   To conclude the description of the calculus, let us notice that if $a(X, z, \zeta)$ and $u(X,z)$ depends also on $z \in (-\infty,0) $ which plays
    the role  of a parameter, all the results
    of Theorem \ref{symbolic} remain true by  adding a supremum in $z$ in  the definition of the  $L^\infty$ norms and by defining
     the $\|\cdot\|_{\mathcal{H}^{s, \gamma}}$ norm of  a function defined on $\mathbb{R}\times \mathcal{S}$ by 
     $$ \|u\|_{\mathcal{H}^{s, \gamma}}^2 = \int_{-\infty}^0 \int_{\mathbb{R}^3} \langle \zeta \rangle^{2s}| \mathcal{F}_{t,y}u(\tau, \xi, z)|^2 \, d\tau
      d\xi dz.$$
      
    Let us turn to the semiclassical version of the calculus that we need. We  shall  use the semiclassical type  weight
     $\langle \zeta^\eps \rangle= \big( \gamma^2 + \tau^2 +| \sqrt{\eps} \xi|^4 \big)^{1\over 4}$
     where  $\eps \in (0,1)$.
      We thus define a new weighted norm as
    $$ \|u \|_{\mathcal{H}^{s, \gamma, \eps}}^2 = \int_{\mathbb{R}^3} \langle \zeta^\eps \rangle^{2s} |\hat u (\tau, \xi)|^2\, d\tau d\xi,$$ 
     For pseudodifferential operators, the semiclassical quantization corresponding to \eqref{defpseudo} for a symbol $\sigma(X, \zeta)$ is 
 $$ Op^\eps(\sigma) u(X) = (2 \pi)^{-3} \int_{\mathbb{R}^3} e^{i X \cdot \eta}  \sigma(X, \gamma, \tau, \sqrt{\eps}\, \xi) \hat{u}(\eta) \, d\eta, \quad
  \eta=(\tau, \xi)$$
 Note that this definition is different from the quantization used in \cite{Metivier-Zumbrun}, this is due to the fact that
  we shall study a problem of purely parabolic nature.
   Let us define the "scaling map" $H^\eps$ by
   $$ H^\eps u= \sqrt{\eps} u (t, \sqrt{\eps}y).$$
   Then, we note that
   $$   Op^\eps(\sigma) = H_{\eps}^{-1} Op(\sigma^\eps) H^\eps $$
   where $\sigma^\eps$ is given by $\sigma^\eps(X, \zeta)= \sigma(t, \sqrt{\eps} y, \zeta)$.
    This motivates the following definition: for a symbol $a(X, \zeta) \in \Gamma^\mu_{0}$, we  define a semiclassical paradifferential calculus by 
    $$ T^{\eps, \gamma}_{a} u = H_{\eps}^{-1} T^\gamma_{a^\eps} H_{\eps} u$$
    with $a^\eps(X, \zeta)= a(t, \sqrt{\eps} y, \zeta)$.
    
    Next, we observe that 
   $$ \|H^\eps u \|_{\mathcal H^{s, \gamma}}= \|u\|_{\mathcal{H}^{s, \gamma, \eps}}$$
  and that  if $a$ is  in $\Gamma^\mu_{0}$, then the family $(a^\eps)_{\eps \in (0, 1)}$ is uniformly bounded in $\Gamma^\mu_{0}$.
   Moreover, if $a \in \Gamma^\mu_{1}$, then  $a^\eps$ and $\nabla_{X} a^\eps$ are uniformly bounded in $\Gamma^\mu_{0}$.
    This allows to get from Theorem \ref{symbolic} the following properties of the semiclassical calculus:
     \begin{theoreme}
      \label{symbolic2}
      We have the following results:
      \begin{enumerate}
      \item {\bf Boundedness} For  $a \in \Gamma^\mu_{0}$, and every $s\in \mathbb{R}$,  we have
       $$ \|T_{a}^{\eps, \gamma} u\|_{\mathcal H^{s, \gamma, \eps}} \leq C \| u \|_{\mathcal{H}^{s+\mu, \gamma, \eps}}.$$
       \item {\bf Product} For $a \in \Gamma^\mu_{1}$, $b \in \Gamma^{\mu'}_{1}$ and $s \in \mathbb{R}$, we have
       $$ \|T_{a}^{\eps, \gamma} T_{b}^{\eps, \gamma} u - T_{ab}^{\eps, \gamma} u \|_{\mathcal H^{s, \gamma, \eps}} \leq
        C \|u\|_{\mathcal{H}^{s+\mu+\mu'-1, \gamma, \eps}}$$
        \item {\bf Adjoint} For $a \in \Gamma^\mu_{1}$,  we have
         $$ \|(T_{a}^{\eps, \gamma})^* - T_{a^*}^{\eps, \gamma}\|_{\mathcal{H}^{s, \gamma, \eps}} \leq C \|u\|_{\mathcal H^{s+\mu-1, \gamma, \eps}}$$
         \item {\bf Garding inequality}  For $a \in \Gamma_{1}^\mu$,  assume that there exists $c>0$
          such that 
          $$ \forall X, \, \zeta, \quad    \mbox{Re }a(X, \zeta) \geq c  \langle \zeta \rangle^\mu.$$
           Then  there exists $C>0$   such that for $\gamma \geq C$, we have
         $$  {c \over 2} \| T_{w}^{\eps, \gamma}  u\|_{\mathcal H^{{\mu \over 2}, \gamma, \eps}}^2 \leq \mbox{Re } \big(T_{a}^{\eps, \gamma}  u, 
           u \big)_{L^2}.$$
         \item {\bf Paraproduct} If $a(X) \in L^\infty$, $\nabla_{X}a \in L^\infty$,  then  we have 
         \begin{eqnarray*}
          \| au - T^{\eps, \gamma}_{a} u \|_{\mathcal H^{1, \gamma, \eps}} \leq C  \|u\|_{\mathcal{H}^{0, \gamma, \eps}}, \\
          \gamma  \| au - T_{a}^{\eps, \gamma} u \|_{\mathcal{H}^{0, \gamma}} \leq C\| u\|_{\mathcal{H}^{1, \gamma, \eps}}, \\
           \| a  \sqrt{\eps}^{\alpha_{j}}\partial_{j}u - T_{a}^{\eps, \gamma}\sqrt{\eps}^{\alpha_{j}}\partial_{j}u \|_{\mathcal H^{0, \gamma, \eps}} \leq C \|u\|_{\mathcal{H}^{{2 \over p_{j}}- 1 , \gamma, \eps}}
           \end{eqnarray*}

           where $\partial_{0}= \partial_{t}$, $p_{0}=1$  and $p_{j}= 2$, $j=1, \, 2$ and $\alpha_{0}= 0$, $\alpha_{j}= 1$, $j \geq 1$. 
      \end{enumerate}
     As before,  in all the above estimates $C$ only depends on seminorms of the symbols and  in particular, it is independent of $\eps$ for $\eps \in (0,1)$.
      \end{theoreme} 
   Note that in the above result, due to the anisotropy in our definition of the semiclassical calculus we do not gain powers of  $\eps$ in 
    the  product and adjoint estimates.
    
Finally, if  the symbols and $u$ depend on the parameter $z$ all the above results remain true  as before with the definition
$$ \|u\|_{\mathcal{H}^{s, \gamma, \eps}}^2 = \int_{-\infty}^0 \int_{\mathbb{R}^3} \langle \zeta^\eps \rangle^{2s}| \mathcal{F}_{t,y}u(\tau, \xi, z)|^2 \, d\tau
      d\xi dz.$$

  \vspace{0.5cm}      

\subsection*{Acknowledgements}
N. M.  was partially supported by an NSF-DMS grant  0703145
 and F. R.  by the ANR project ANR-08-JCJC-0104-01.
This work was carried out during visits to the University  of 
 Rennes 1 and the Courant Institute.
The hospitality of these institutions is highly acknowledged.

\end{document}